\documentclass[reqno,11pt]{amsart}
\usepackage{amsfonts,amsmath,amssymb}

\usepackage{hyperref}

\usepackage[margin=1.317in]{geometry}
\numberwithin{equation}{section}

\def\phii{\widetilde{\varphi}}

\def\eps{\epsilon}

\def\si{\sigma}

\def\varep{\varepsilon}

\def\D{\mathcal{D}}

\newtheorem{theorem}{Theorem}[section]
\newtheorem{lemma}[theorem]{Lemma}
\newtheorem{corollary}[theorem]{Corollary}
\newtheorem{proposition}[theorem]{Proposition}
\newtheorem{definition}[theorem]{Definition}
\newtheorem{remark}[theorem]{Remark}

\begin{document}

\title[The Euler--Maxwell system for electrons: global solutions in $2D$]{The Euler--Maxwell system for electrons: global \\ solutions in $2D$}

\author{Yu Deng}
\address{Princeton University}
\email{yudeng@math.princeton.edu}

\author{Alexandru D. Ionescu}
\address{Princeton University}
\email{aionescu@math.princeton.edu}

\author{Benoit Pausader}
\address{Princeton University}
\email{pausader@math.princeton.edu}

\thanks{The first author was supported in part by a Jacobus Fellowship from Princeton University. 
The second author was supported in part by NSF grant DMS-1265818. The third author was supported in part by NSF grants 
DMS-1069243 and DMS-1362940, and a Sloan fellowship.}

\begin{abstract}
A basic model for describing plasma dynamics is given by the Euler-Maxwell system, in which 
compressible ion and electron fluids interact with their own self-consistent electromagnetic field. In this paper we consider the ``one-fluid'' Euler--Maxwell model for electrons, in 2 spatial dimensions, and prove global 
stability of a constant neutral background. 

In 2 dimensions our global solutions have relatively slow (strictly less than $1/t$) pointwise decay and the system has a 
large (codimension 1) set of quadratic time resonances. The issue in such a situation is to solve the ``division problem''. 
To control the solutions we use a combination of improved energy estimates in the Fourier space, 
an $L^2$ bound on an oscillatory integral operator, and Fourier analysis of the Duhamel formula. 
\end{abstract}

\maketitle

\setcounter{tocdepth}{1}

\tableofcontents

\section{Introduction}\label{Eqs}

A plasma is a collection of fast-moving charged particles. It is believed that more than 90\% of the matter in the universe 
is in the form of plasma, from sparse intergalactic plasma, to the interior of stars to neon signs. In addition, 
understanding of the instability formation in plasma is one of the main challenges for nuclear fusion, in which charged 
particles are accelerated at high speed to create energy. We refer to \cite{Bit,DelBer} for physics references in book form.

At high temperature and velocity, ions and electrons in a plasma tend to become two separate fluids due to their different 
physical properties (inertia, charge). One of the basic fluid models for describing plasma dynamics is the 
so-called \textquotedblleft two-fluid\textquotedblright\ model, in which two compressible ion and electron fluids interact with 
their own self-consistent electromagnetic field. In 3 dimensions, nontrivial global solutions of the full two-fluid system were 
constructed for the first time by Guo--Ionescu--Pausader \cite{GuIoPa} (small irrotational perturbations of constant solutions), 
following earlier partial results in simplified models in \cite{Guo,GuPa,GeMa,IoPa2}. See also the introduction of \cite{GuIoPa} for a longer discussion of the Euler--Maxwell system in 3D, and its connections to many other models in mathematical physics, such as the Euler--Poisson model, the Zakharov system, the KdV, and the NLS.

A simplification of the full system is the ``one-fluid" model, which accounts for the interaction of electrons and the electromagnetic field, but neglects the dynamics of the ion fluid. Under suitable irrotationality assumptions, this model can be reduced to a coupled system of two Klein--Gordon equations with different speeds and no null structure. While global results are classical
in the case of scalar wave and Klein--Gordon equations, see for example \cite{Jo,JoKl, Kl2, KlVf,  Kl, Kl4, Ch, Ch2, Sh, Si, DeFa, DeFaXu, Al, Alin1, Alin2, Alin3}, it was pointed out by Germain \cite{Ge} that there are key 
new difficulties in the case of two Klein--Gordon equations with different speeds. In this case, the classical vector-field  method does not 
seem to work well, and there are large sets of resonances that contribute in the analysis.  

The one-fluid Euler--Maxwell model in 3D was analyzed in \cite{GeMa}, using the ``space-time resonance method'', and the authors proved global existence and scattering, with weak decay like $t^{-1/2}$. 
A more robust result for this problem, which gives time-integrability of the solution in $L^\infty$, for all parameters, 
was obtained by two of the authors in \cite{IoPa2}.

In this paper we consider the one-fluid Euler--Maxwell model\footnote{The simpler one-fluid Euler--Poisson model in 2D, which can be reduced to a single Klein--Gordon equation, was considered previously in \cite{IoPa1} and \cite{LiWu}.} in 2D. As in dimension 3, in the irrotational case this can still be reduced to a quasilinear coupled system of two Klein--Gordon equations with different speeds and no null structure. At the analytical level, one has, of course, all the difficulties of the 3D problem, such as large sets of resonances. In addition, there is one critical new difficulty, namely the slow decay of solutions, as it was observed by Bernicot--Germain \cite{BeGe} that the nonlinear solutions cannot have the ``almost" integrable $1/t$ decay, due to strong resonant quadratic interactions. This slow decay and the presence of large sets of resonances require new ideas to control the growth of the solution over time.

\subsection{The one-fluid Euler--Maxwell system in 2D and the Main Theorem}\label{EMMain} 

For any vector-field $X=(X_1,X_2)$ and any function $f$ defined in a domain of $\mathbb{R}^2$ let
\begin{equation}\label{conventions}
\begin{split}
&\nabla f:=(\partial_1f,\partial_2f),\qquad \nabla^\perp f:=(-\partial_2f,\partial_1f),\\
&\hbox{div}(X):=\partial_1X_1+\partial_2X_2,\qquad\hbox{curl}(X):=\partial_1X_2-\partial_2X_1,\qquad X^\perp:=(-X_2,X_1).
\end{split}
\end{equation}

The Euler--Maxwell system for electrons is the coupled system
\begin{equation}\label{EuMa1}
\begin{cases}
& \partial _{t}n_{e}+\hbox{div}(n_{e}v_{e})=0, \\
& n_{e}m_{e}\left( \partial _{t}v_{e}+v_{e}\cdot \nabla v_{e}\right) +\nabla p_{e}=-n_{e}e\left( E-\frac{bv_{e}^\perp}{c}\right) , \\
& \partial _{t}b+c\cdot\hbox{curl}(E)=0, \\
& \partial _{t}E+c\nabla^\perp b=4\pi en_{e}v_{e},
\end{cases}
\end{equation}
where $n_e,b:\mathbb{R}^2\times I\to\mathbb{R}$ and $v_e,E:\mathbb{R}^2\times I\to\mathbb{R}^2$ are $C^2$ functions, and $e,m_e,c$ are strictly positive constants. The electrons have charge $-e$, density $n_{e}$, mass $m_{e}$, velocity $v_{e}$, 
and pressure $p_{e}$, $c$ denotes the speed of light, and $E,b$ denote the electric and magnetic field. The system has a family of equilibrium solutions $(n_e,v_e,b,E)=(n_0,0,0,0)$, where $n_0>0$ is a constant. The goal of this paper is to investigate their stability properties.

We consider a barotropic pressure law given by
\begin{equation*}
p_e=p_e(n_e),\qquad p_e^\prime>0,
\end{equation*}
where $p_e$ is assumed to be analytic in a neighborhood of $n_0$. We assume that the electric field and the velocity field satisfy the compatibility and irrotationality equations
\begin{equation}\label{compatibility}
\hbox{div}(E)=4\pi e(n_0-n_e),\qquad\hbox{curl}(v_{e})=\frac{e}{m_{e}c}b.
\end{equation}
These two equations are propagated by the dynamic flow if they are satisfied at the initial time.

The system \eqref{EuMa1}-\eqref{compatibility} is a complicated nonlinear system of six scalar evolution equations with two constraints. To study it, we first want to simplify it and reduce it to a system of unconstrained evolution equations (see \eqref{system3prime}). We start by nondimensionalizing the system: let\footnote{$\beta$ is the ``electron plasma frequency'' and $d$ is the ratio of the speed of sound over the speed of light.}
\begin{equation}\label{variables8}
\begin{split}
n_e(x,t):=n_0\cdot (1+\rho(\lambda x,\beta t)),\quad&\quad v_e(x,t):=c\cdot u(\lambda x,\beta t),\\
b(x,t):=c\sqrt{4\pi n_0m_e}\cdot\widetilde{b}(\lambda x,\beta t),\quad&\quad E(x,t):=c\sqrt{4\pi n_0m_e}\cdot\widetilde{E}(\lambda x,\beta t),\\
\lambda:=\frac{1}{c}\sqrt{\frac{4\pi e^2n_0}{m_e}},\quad\beta:=\sqrt{\frac{4\pi e^2n_0}{m_e}},\quad&\quad h'(y):=\frac{1}{m_ec^2}\frac{p'_e(n_0(1+y))}{1+y},\quad d:=h'(0).
\end{split}
\end{equation}
In the physically relevant case $0<d<1$. We can rewrite the equations as
\begin{equation}\label{system2}
\begin{cases}
\partial_t\rho+\hbox{div}[(1+\rho)u]&=0,\\
\partial_tu+u\cdot\nabla u+h'(\rho)\nabla \rho&=-\widetilde{E}+\widetilde{b}u^\perp,\\
\partial_t\widetilde{b}+\hbox{curl}(\widetilde{E})&=0,\\
\partial_t\widetilde{E}+\nabla^\perp\widetilde{b}&=(1+\rho)u.
\end{cases}
\end{equation}

The structure condition (\ref{compatibility}) becomes
\begin{equation}\label{condition}
\hbox{div}(\widetilde{E})+\rho=0,\qquad \hbox{curl}(u)-\widetilde{b}=0.
\end{equation}
Notice that, as a consequence of \eqref{system2},
\begin{equation*}
\partial_t\big(\hbox{div}(\widetilde{E})+\rho\big)=0,\qquad\partial_t\big(\hbox{curl}(u)-\widetilde{b}\big)+\hbox{div}\big(u\cdot (\hbox{curl}(u)-\widetilde{b})\big)=0.
\end{equation*}
Therefore the identities \eqref{condition} are propagated by the flow if they are satisfied at time $t=0$.

In terms of the variables $(\rho,u,\widetilde{E},\widetilde{b})$, our main theorem can be stated as follows:

\begin{theorem}
\label{main} Assume that $\overline{c}>0$, $h:(-\overline{c},\overline{c})\to\mathbb{R}$ is an analytic function, and $d=h'(0)\in(0,1)$. Assume that $(\rho_0,u_0,\widetilde{E}_0,\widetilde{b}_0):\mathbb{R}^2\to\mathbb{R}\times\mathbb{R}^2\times\mathbb{R}^2\times\mathbb{R}$ is small, smooth, and localized data that satisfies the irrotationality assumption
\begin{equation*}
\hbox{div}(\widetilde{E}_0)+\rho_0=0,\qquad \hbox{curl}(u_0)-\widetilde{b}_0=0,
\end{equation*}
and the smallness condition
\begin{equation}
\Vert (\rho_0,u_0,\widetilde{E}_0,\widetilde{b}_0)\Vert_{H^{N_{0}}\cap H_{\Omega}^{N_1}}+\Vert (\rho_0,u_0,\widetilde{E}_0,\widetilde{b}_0)\Vert _{Z}=\eps_{0}\leq \overline{\eps}.
\label{maincond2}
\end{equation}
Here $\overline{\eps}>0$ is sufficiently small, $N_{0},N_1$ are sufficiently large, see Definition \ref{MainZDef} for the precise description of the norms. 

Then there exists a unique global solution $(\rho,u,\widetilde{E},\widetilde{b})\in C([0,\infty ):{H^{N_{0}}}\cap H_{\Omega}^{N_1})$ of the system \eqref{system2}--\eqref{condition} with initial data $(\rho(0),u(0),\widetilde{E}(0),\widetilde{b}(0))=(\rho_0,u_0,\widetilde{E}_0,\widetilde{b}_0)$. Moreover, the system returns to equilibrium, in the sense that for any $t\in [0,\infty)$,
\begin{equation}
\Vert (\rho(t),u(t),\widetilde{E}(t),\widetilde{b}(t))\Vert_{H^{N_{0}}\cap H_{\Omega}^{N_1}}+\sup_{|\alpha |\leq 10}(1+t)^{0.999}\Vert D_{x}^{\alpha }(\rho(t),u(t),\widetilde{E}(t),\widetilde{b}(t))\Vert _{L^{\infty }}\lesssim \eps_{0}.
\label{mainconcl2.1}
\end{equation}
\end{theorem}

We remark that under the constraints \eqref{condition}, we may rewrite \eqref{system2} solely in terms of $(u,\widetilde{E})$ as
\begin{equation}\label{system3prime}
\begin{cases}
\partial_tu+h'(\rho)\nabla\rho+\widetilde{E}&=-(1/2)\nabla\vert u\vert^2,\\
\partial_t\widetilde{E}+\nabla^\perp\hbox{curl}(u)-u&=-\hbox{div}(\widetilde{E})u.
\end{cases}
\end{equation}
This is an equivalent formulation of the system, as we can define $(\rho,b)$ through \eqref{condition}.

Theorem \ref{main} can be used to generate smooth, global solutions of the Euler--Maxwell system \eqref{EuMa1} starting from data in a full neighborhood of the constant solution $(n_e,v_e,E,B)=(n_0,0,0,0)$, according to the linear transformations in \eqref{variables8}. 

The rest of the paper is concerned with the proof of Theorem \ref{main}.

\subsection{Challenges and main ingredients}

The system \eqref{ener4} is a quasilinear time-reversible system, with no dissipation and no relaxation effects.\footnote{When dissipation or relaxation is present, one expects stronger decay, even at the level of the $L^2$-norm, see e.g. \cite{ChJeWa,Peng} and the references therein. In our case however, the evolution is time-reversible and we need a different mechanism of decay based on dispersion.} Starting with the seminal works of John, Klainerman, Christodoulou and Shatah \cite{Jo,Kl2,JoKl,KlVf,Kl,Sh,Kl4,Ch,ChKl}, the main strategy to prove long-time regularity, relies on an interplay between 
\begin{enumerate}
\item control of high order energies;
\item dispersion and decay estimates.
\end{enumerate}

In this paper we use a combination of energy estimates and the Fourier transform method to accomplish these goals. The situation we consider here is substantially more difficult than in other quasilinear evolutions, due to the combination of the following  factors:

\setlength{\leftmargini}{1.8em}
\begin{itemize}
\item Less than $1/t$ pointwise decay of the solutions. The decay of the linear solutions of the system of Klein--Gordon equations is $t^{-1}$ as $t\to\infty$ in 2D. However, this optimal pointwise decay cannot be propagated by the nonlinear flow, even in simpler semilinear evolutions, due to the presence of a large set of space-time resonances. This was pointed out by Bernicot--Germain \cite{BeGe} who found a logarithmic loss. In this paper we prove $t^{-1+\kappa}$ pointwise decay of the nonlinear solution $(U_e,U_b)$, for certain $\kappa>0$ small.

\item Large set of time resonances. In certain cases, particularly in dimension $1$, one can overcome the slow pointwise decay using the method of normal forms of Shatah \cite{Sh}. The critical ingredient needed is the absence of time resonances (or at least a 
suitable ``null structure" of the quadratic nonlinearity matching the set of time resonances), see \eqref{Tres} for the precise definition. Our system, however, has a full (codimension 1) set of time resonances\footnote{This situation arises also in semilinear equations and is called the ``division problem". The issue is to control 
bilinear operators defined by singular kernels (denominators) in the Fourier space. This motivated the introduction of the $X^{s,b}$ 
spaces, as an elegant framework for analysis. However, the semilinear  mechanism based on perturbative analysis in $X^{s,b}$ 
spaces does not seem to work in quasilinear problems, due to the unavoidable loss of derivative.}, and no meaningful null 
structures.
\end{itemize}

To address these issues we use a combination of improved energy estimates and Fourier analysis. The natural framework to carry out the analysis is the $(\xi,t)$ space, where $\xi$ is the frequency corresponding to the physical variable and $t$ is the time variable. The analysis we perform here (based essentially on an improved energy identity and Fourier analysis in the $(\xi,t)$ space) is, in certain ways, reminiscent of the type of analysis performed in semilinear problems using the I-method of Colliander--Keel--Staffilani--Takaoka--Tao \cite{CoKeStTaTa}. 

The proof of the main theorem has two main components:

\begin{itemize}

\item Energy estimates to control the increment of high order Sobolev norms and weighted norms. Many of the new difficulties of the problem are present at this stage and are caused by the combination of slow decay (which prevents direct estimates of the energy increment) and large time resonant sets (which prevent the use of global normal forms).    

\item Semilinear dispersive analysis using a suitable $Z$ norm, to provide a precise description of the nonlinear solution at a lower regularity level, including pointwise decay. The main difficulty here is caused by the presence of a large set of space-time resonances and the slow dispersion/decay in dimension 2.  
\end{itemize}
We discuss the main ideas in more detail below.

\subsubsection{Energy estimates} The dynamics of the evolution can be described as a coupled quasilinear system for two complex-valued variables $U_e$ and $U_b$, see \eqref{variables4} for the precise definition. To prove energy estimates we start with a paradifferential reduction of the system, which allows us to isolate the skew-symmetric quasilinear terms and the perturbative nonlinearities. Then we construct suitable energy functionals, which control for example the $H^N$ norm of the solution, and prove energy identities of the form
\begin{equation}\label{eta1}
\partial_tE_N(t)\approx\langle D\rangle^NU\ast\langle D\rangle^NU\ast\langle D\rangle^2U. 
\end{equation}
The cubic terms in the right-hand side do not lose derivatives, so such an identity can be used to prove local regularity. However, in our problem, the solution is expected to have strictly less that $1/t$ pointwise decay (as pointed out in \cite{BeGe}). As a result, an identity like \eqref{eta1} cannot be used directly to control the long-term growth of the high order energy.

To construct long-term solutions we need to understand better the cubic terms in the right-hand side of \eqref{eta1}. By adjusting the definition of the energy functionals, we notice that these terms have special structure. More precisely, when written in the Fourier space, the energy increment $|\mathcal{E}_N(t)-\mathcal{E}_N(0)|$ can be estimated by a sum of space-time integrals of the form
\begin{equation}\label{OPT}
\begin{split}
&\int_0^t\langle T[\langle D\rangle^2V_\mu(s),\langle D\rangle^NV_\nu(s)],\langle D\rangle^NV_\sigma(s)\rangle ds,\\
&\widehat{T[f,g]}(\xi):=\int_{\mathbb{R}^{2}}e^{is\Phi (\xi
,\eta )}m(\xi ,\eta )\widehat{f}(\xi -\eta)\widehat{g}(\eta)\,d\eta.
\end{split}
\end{equation}
Here $V_\sigma(s):=e^{is\Lambda_\sigma}U_\sigma(s)$, $\sigma\in\{e,b\}$, are the {\it{profiles}} of the nonlinear solutions $U_{\sigma}$, see \eqref{variables4}, and the oscillatory phases $\Phi$ are specific to each interaction and are of the form
\begin{equation}\label{ModelPhase}
\Phi(\xi,\eta)=\Lambda_0(\xi)\pm\Lambda_{1}(\xi-\eta)\pm\Lambda_{2}(\eta),\quad\Lambda_j\in\{\Lambda_e,\Lambda_b\}.
\end{equation} 

The key algebraic property is that all the resulting cubic space-time integrals are of one of the following two special types:

\begin{enumerate}
\item Nonresonant type: the quadratic interaction phase is bounded away from $0$, i.e. $\vert\Phi\vert\gtrsim 1$;
\item Strongly semilinear type: the interaction is smoothing, i.e.
\begin{equation}\label{SSTerm}
\vert m(\xi,\eta)\vert\lesssim (1+\vert\xi\vert+\vert\eta\vert)^{-1}.
\end{equation}
\end{enumerate} 
This special structure is connected to the Hamiltonian structure of the problem, in particular to the conserved physical energy defined by $K''=2h'$, $K(0)=K'(0)=0$, and  
\begin{equation}\label{Econserved}
\mathcal{E}_{conserved}:=\int_{\mathbb{R}^2}[K(\rho)+(1+\rho)|u|^2+|\widetilde{E}|^2+|\widetilde{b}|^2]\,dx.
\end{equation}

The contribution of the nonresonant integrals can be estimated using the method of normal forms of Shatah \cite{Sh} (integration by 
parts in time). This is somewhat delicate in quasilinear problems, due to the potential loss of derivatives. However, it has been done 
recently in some cases, for example either by using carefully constructed nonlinear changes of variables 
(as in \cite{WuAG}), or the ``iterated energy method" as in \cite{GeMa}, or the ``paradifferential normal 
form method" as in \cite{AlDe}, or the ``modified energy method" as in \cite{HuIfTa}. The method 
we use here, which is based on energy estimates in the Fourier space in the spirit of the ``I-method", draws inspiration from these works as well as from the previous work of one of the authors \cite{IoPu4}.

The main remaining issue is to control the strongly semilinear terms. This requires a new idea as normal forms cannot be used in this case, because the time resonant sets
\begin{equation}\label{Tres}
\mathcal{R}_{\Phi}:=\{(\xi,\eta)\in\mathbb{R}^2\times\mathbb{R}^2:\,\Phi (\xi,\eta )=0\},
\end{equation} 
are large, codimension 1 sets, for certain phases $\Phi$ as in \eqref{ModelPhase}. We notice however that the gain of one derivative in \eqref{SSTerm}, together with the fact that we anticipate proving almost optimal decay, allows us to restrict to frequencies that are very small relative to the time variable, i.e.
\begin{equation}\label{SemiCase}
|\xi|+|\eta|\lesssim (1+|t|)^\beta,\qquad \beta\ll 1.
\end{equation}

The key observation we need to bound the contribution of the frequencies in \eqref{SemiCase} is that the resonant sets $\mathcal{R}_\Phi$ satisfy a crucial {\it restricted nondegeneracy property}, namely the function
\begin{equation}\label{RNDC}
\Upsilon(\xi,\eta):=\nabla^2_{\xi,\eta}\Phi(\xi,\eta)[\nabla^\perp_\xi\Phi(\xi,\eta),\nabla^\perp_\eta\Phi(\xi,\eta)]
\end{equation}
can vanish only up to finite order on the resonant set $\mathcal{R}_\Phi$. Using this nondegeneracy property and a $TT^\ast$ argument we can show that localized Fourier integral operators of the form
\begin{equation}\label{OPT2}
Lf(\xi):=\int_{\mathbb{R}^2}e^{it\Phi(\xi,\eta)}a(\xi,\eta)\varphi(2^\lambda\Phi(\xi,\eta))f(\eta)d\eta,
\end{equation}
where $1\ll 2^\lambda\approx |t|^{0.999}$, satisfy nontrivial $L^2$ bounds of the form
\begin{equation}\label{OPT3}
\big\Vert Lf\big\Vert_{L^2}\lesssim 2^{-\lambda}2^{-\lambda/250}\Vert f\Vert_{L^2}.
\end{equation}
See Lemma \ref{L2EstLem} for a precise version at a suitable level of generality. The point is the strong (better than $|t|^{-1}$) gain in the $L^2$ bound.

Given these ingredients we can finish the proof of the energy estimates: we decompose the strongly semilinear integrals in \eqref{OPT} dyadically over the size of the modulation $|\Phi(\xi,\eta)|$. Then we estimate the integral corresponding to small modulation $|\Phi(\xi,\eta)|\leq (1+|t|)^{-0.99999}$ using the $L^2$ bound \eqref{OPT3}. To estimate the higher modulation contributions we integrate by parts in time again (normal forms) and gain time integrability; the potential loss of derivatives is not an issue here because of \eqref{SemiCase}.

The energy analysis in this paper is simplified by the fact that we anticipate proving almost optimal $|t|^{-1+\kappa}$ pointwise decay for the nonlinear solution, for very small values of $\kappa$. This allows us to take $\beta$ very small in \eqref{SemiCase} and simplifies significantly the analysis of the strongly semilinear integrals. However, in future work we will show that these ideas can be expanded to prove global regularity in other problems that involve strictly less than $|t|^{-1}$ decay and large sets of time resonances, such as certain water wave models in 3 dimensions.

\subsubsection{Dispersive analysis} The goal of the dispersive analysis is to give a precise description of the nonlinear solution, at a lower level of regularity. This description is encoded in the $Z$ norm, and the goal is to prove a partial bootstrap estimate for the $Z$ norm, of the form
\begin{equation}\label{OPT4}
\begin{split}
\text{ if }\,\,
&\sup_{t\in[0,T]}\big[\|(U_e(t),U_b(t))\|_{H^{N_{0}}\cap H^{N_1}_\Omega}+\|(V_e(t),V_b(t))\|_{Z}\big]\leq \epsilon_{1}\\
\text{ then }\,\,&\sup_{t\in[0,T]}\|(V_e(t),V_b(t))\|_{Z}\lesssim \eps_0+\epsilon_1^{2},
\end{split}
\end{equation} 
where $\eps_0\ll\eps_1$ is the size of the initial data. Loss of derivatives is not an issue here, so we can use the Duhamel formula, written in terms of the profiles $V_e$ and $V_b$,
\begin{equation}\label{OPT5}
\widehat{V_{\sigma}}(\xi,t)=\widehat{V_{\sigma}}(\xi,0)+\sum_{\mu,\nu\in\{e,b\}}\int_0^t\int_{\mathbb{R}^2}e^{is\Phi(\xi,\eta)}\mathfrak{m}(\xi,\eta)\widehat{V_{\pm\mu}}(s,\xi-\eta)\widehat{V_{\pm\nu}}(s,\eta)\,d\eta ds,
\end{equation}
where the sum is over all possible quadratic interactions, and, for simplicity, we ignore here the higher order interactions. The main contributions in this integral come from the set of {\it{space-time resonances}}, which is the set of points where the function $\Phi$ is stationary in all variables, 
\begin{equation}
\nabla _{(t,\eta )}[t\Phi (\xi ,\eta )]=0,\qquad \hbox{i.e.}\qquad \Phi (\xi
,\eta )=0\,\,\text{ and }\,\,\nabla _{\eta }\Phi (\xi ,\eta )=0.  \label{STRes}
\end{equation}
Understanding these contributions forms 
the basis of the ``method of space-time resonances'', as it was highlighted by 
Germain, Masmoudi, and Shatah \cite{Ge, GeMa,GeMaSh,GeMaSh2} in several problems. In our case the space-time resonant 
set is a finite union of spheres in $\mathbb{R}^2\times\mathbb{R}^2$, of the form
\begin{equation}\label{OPT6}
\{(\xi,\eta)=(R_j\omega,r_j\omega):\omega\in\mathbb{S}^1\},
\end{equation}
for finitely many pairs $(R_j,r_j)\in (0,\infty)^2$.

As in semilinear problems, having an effective $Z$ norm is critical in order to prove a bound like \eqref{OPT4}. To construct such a norm we can gain some intuition by substituting Schwartz functions, independent of $s$, as inputs $V_{\pm\mu}$ and $V_{\pm\nu}$ in the right-hand side of \eqref{OPT5}. An important observation is that the space-time resonant points are {\it{nondegenerate}} (according to the terminology introduced in \cite{IoPa2}), in the sense that the Hessian of the matrix $\nabla_{\eta\eta}^2\Phi(\xi,\eta)$ is non-singular at these points. Assume that $s\approx 2^m\gg 1$. Integration by parts in $\eta$ and $s$ shows that the main contribution comes from a small neighborhood of the stationary points where $|\nabla_{\eta}\Phi(\xi,\eta)|\leq 2^{-m/2+\delta m}$ and $|\Phi(\xi,\eta)|\leq 2^{-m+\delta m}$ (the space-time resonant points as in \eqref{STRes}--\eqref{OPT6}). A simple calculation shows that this main contribution is of the type
\begin{equation*}
\widehat{V}(\xi)\approx {\sum}_j c_j(\xi)\varphi_{\leq -m}(|\xi|-R_j),
\end{equation*}
up to factors of $2^{\delta m}$, where the functions $c_j$ are smooth.

We can now describe more precisely the choice of the $Z$ space. We use the framework introduced by two of the authors in \cite{IoPa1}, which was later refined in \cite{IoPa2,GuIoPa,De}. 
The idea is to decompose the profile as a superposition of atoms, using localization in both space and frequency,
\begin{equation*}
f={\sum}_{j,k}Q_{jk}f,\qquad Q_{jk}f=\varphi_j(x)\cdot P_kf(x).
\end{equation*} 
The $Z$ norm is then defined by measuring suitably every atom.

In our case, the $Z$ space should include all Schwartz functions. It also has to include functions like $\varphi_{\leq -m}(|\xi|-R_j)$, due to the considerations above, for any $m$ large. It should measure localization in both space and frequency, and be strong enough, at least, to recover the $t^{-1+\kappa}$, $\kappa\ll 1$, pointwise decay. A space with these properties is proposed in Definition \ref{MainZDef}; we notice that the $Z$ space depends in a significant way on the location and the shape of the set of space-time resonance outputs. 

Once the norm is defined we prove the conclusion of \eqref{OPT4} by careful Fourier analysis of the Duhamel formula: we decompose our profiles in space and frequency, localize to small sets in the frequency space, 
keeping track in particular of the frequencies around the space-time resonance sets, use integration by parts in $s$ and $\eta$ to bound nonresonant interactions, use integration by parts in $\xi$ to control the location of the output, and use multilinear H\"{o}lder-type estimates to bound $L^2$ norms. We emphasize that the semilinear analysis in this paper is more difficult than in our earlier papers \cite{IoPa1,IoPa2,GuIoPa}. Some of these new difficulties are:   

\begin{itemize}

\item The slow decay of the solution, which prevents simple estimates even in nonresonant cases.

\item The derivatives of the profiles $\partial_tV_e$ and $\partial_tV_b$ contain additional secondary oscillations that need to be properly accounted for in the normal form transformation. They also contain secondary resonance terms, and time derivatives of functions with better estimates. See Lemma \ref{dtfLem} for the complex description of the derivatives $\partial_tV_e$ and $\partial_tV_b$.  

\item The argument in the most difficult resonant cases relies on exploiting a key algebraic property of iterated resonances, 
which is proved in Proposition \ref{Separation2}. Roughly speaking, when considering second iterates, this property implies that 
there can be no ``accidental" 2-cascades when the output of a space-time resonance interacts in a coherent and resonant way with another wave.
\end{itemize}

To deal with these issues it is important to be able to restrict to a suitable class of ``almost radial" functions. This is possible because of the rotation invariance of the system, as long as we propagate control of energy norms of the type $H^{N_0}\cap H^{N_1}_\Omega$ containing both a high order Sobolev norm and a high order weighted $L^2$ norm defined by the rotation vector-field $\Omega=x_1\partial_2-x_2\partial_1$. We notice that the linear estimates in Lemma \ref{LinEstLem} and many of the bilinear estimates are much stronger because the functions we consider are almost radial, in a suitable quantitative sense. See also \cite{De}, where several techniques related to ``almost radiality'' are developed and used in dimension 3 in order to control certain types of degenerate resonances.

\subsection{Organization of the paper} In section \ref{NotationsF} we summarize the main notation, define precisely the main norms, and set up the main bootstrap argument (Proposition \ref{bootstrap}). In section \ref{LemSec}, we collect several important lemmas that are used in the paper. These lemmas include the linear estimates in Lemma \ref{LinEstLem}, some elements of paradifferential calculus in Lemmas \ref{PropProd}--\ref{PropHHSym}, integration by parts bounds in Lemmas \ref{tech5}--\ref{RotIBP}, and a localization bound in Lemma \ref{PhiLocLem}. 

Sections \ref{ENERGY0} and \ref{SSterm} contain the improved energy estimates. The main components of the proof are Proposition \ref{ChangeUnknownsProp} (paralinearization of the system), Proposition \ref{IncrementEnergyProp} (the basic energy estimate), Lemma \ref{BootstrapEE2} (control of the nonresonant space-time integrals), and Lemma \ref{BootstrapEE1} (control of the strongly semilinear space-time integrals). The proof of this last lemma relies on a key $L^2$ bound on localized Fourier integral operators, proved in Lemma \ref{L2EstLem}.

In sections \ref{partialt} and \ref{Sec:Z1Norm} we prove the dispersive estimates. The main results are Lemmas \ref{dtfLemPrelim} and \ref{dtfLem} (precise descriptions of the functions $\partial_tV_e$ and $\partial_tV_b$), and Lemmas \ref{ZNormEstSimpleLem1}--\ref{ResLem} (the core bilinear estimates, divided in several cases).

In section \ref{phacolle} we collect all the estimates related to the phase functions $\Phi$. The proofs require very precise information about these functions, including bounds on sub-level sets, structure and separation of resonances, a slow propagation property of iterated resonances, and a restricted nondegeneracy property of the time resonant set.

\section{Functions spaces and the main proposition}\label{NotationsF}

\subsection{Notation, atomic decomposition, and the $Z$-norm}\label{defznorm} We start by summarizing our main definitions and notations.

\subsubsection{Littlewood-Paley projections}

We fix $\varphi:\mathbb{R}\to[0,1]$ an even smooth 
function supported in $[-8/5,8/5]$ and equal to $1$ in $[-5/4,5/4]$. For simplicity of notation, we also 
let $\varphi:\mathbb{R}^2\to[0,1]$ denote the corresponding radial function on $\mathbb{R}^2$. Let
\begin{equation*}
\varphi_{k}(x):=\varphi(|x|/2^k)-\varphi(|x|/2^{k-1})\text{ for any }k\in\mathbb{Z},\qquad \varphi_I:=\sum_{m\in I\cap\mathbb{Z}}\varphi_m\text{ for any }I\subseteq\mathbb{R}.
\end{equation*}
For any $B\in\mathbb{R}$ let 
\begin{equation*}
\varphi_{\leq B}:=\varphi_{(-\infty,B]},\quad\varphi_{\geq B}:=\varphi_{[B,\infty)},\quad\varphi_{<B}:=\varphi_{(-\infty,B)},\quad \varphi_{>B}:=\varphi_{(B,\infty)}.
\end{equation*}
For any $a<b\in\mathbb{Z}$ and $j\in[a,b]\cap\mathbb{Z}$ let
\begin{equation}\label{Alx80}
\varphi^{[a,b]}_j:=
\begin{cases}
\varphi_{j}\quad&\text{ if }a<j<b,\\
\varphi_{\leq a}\quad&\text{ if }j=a,\\
\varphi_{\geq b}\quad&\text{ if }j=b.
\end{cases}
\end{equation}

For any $x\in\mathbb{Z}$ let $x_{+}=\max(x,0)$ and $x_-:=\min(x,0)$. Let
\begin{equation*}
\mathcal{J}:=\{(k,j)\in\mathbb{Z}\times\mathbb{Z}_+:\,k+j\geq 0\}.
\end{equation*}
For any $(k,j)\in\mathcal{J}$ let
\begin{equation*}
\phii^{(k)}_j(x):=
\begin{cases}
\varphi_{\leq -k}(x)\quad&\text{ if }k+j=0\text{ and }k\leq 0,\\
\varphi_{\leq 0}(x)\quad&\text{ if }j=0\text{ and }k\geq 0,\\
\varphi_j(x)\quad&\text{ if }k+j\geq 1\text{ and }j\geq 1,
\end{cases}
\end{equation*}
and notice that, for any $k\in\mathbb{Z}$ fixed, $\sum_{j\geq-\min(k,0)}\phii^{(k)}_j=1$. 

Let $P_k$, $k\in\mathbb{Z}$, denote the operator on $\mathbb{R}^2$ defined by the Fourier multiplier $\xi\to \varphi_k(\xi)$. 
Let $P_{\leq B}$ (respectively $P_{>B}$) denote the operators on $\mathbb{R}^2$ defined by the Fourier 
multipliers $\xi\to \varphi_{\leq B}(\xi)$ (respectively $\xi\to \varphi_{>B}(\xi)$). For $(k,j)\in\mathcal{J}$ let $Q_{jk}$ denote the operator
\begin{equation}\label{qjk}
(Q_{jk}f)(x):=\phii^{(k)}_j(x)\cdot P_kf(x).
\end{equation}
In view of the uncertainty principle the operators $Q_{jk}$ are relevant only when $2^j2^k\gtrsim 1$, which explains the definitions above.

\subsubsection{Quadratic phases, linear profiles and norms}

We fix the expansion of the function $h$ defined in \eqref{variables8} (note that we may assume that $h(0)=0$):
\begin{equation}\label{Exph}
h^\prime(\rho)=d+\kappa\rho+h_2(\rho),\qquad\vert h_2(\rho)\vert\lesssim\rho^2.
\end{equation}

An important role will be played by the functions $U_e,U_b,V_e,V_b$ defined by
\begin{equation}\label{variables4}
\begin{split}
V_e(t)&:=e^{it\Lambda_e}U_e(t),\qquad V_b(t):=e^{it\Lambda_b}U_b(t),\\
U_e(t)&:=\vert\nabla\vert^{-1}\hbox{div}(u)-i\vert\nabla\vert^{-1}\Lambda_e\hbox{div}(\widetilde{E}),\qquad\Lambda_e:=\sqrt{1+d\vert\nabla\vert^2}\\
U_b(t)&:=\vert\nabla\vert^{-1}\Lambda_b\hbox{curl}(u)-i\vert\nabla\vert^{-1}\hbox{curl}(\widetilde{E}),\qquad\Lambda_b:=\sqrt{1+\vert\nabla\vert^2}
\end{split}
\end{equation}

With $\Lambda_e=\sqrt{1-d\Delta}$ and $\Lambda_{b}:=\sqrt{1-\Delta}$ as in \eqref{variables4}, we define
\begin{equation}U_{-e}:=\overline{U_{e}},\quad U_{-b}:=\overline{U_{b}};\qquad V_{-e}:=\overline{V_{e}},\quad V_{-b}:=\overline{V_{b}};\qquad\Lambda_{-e}:=-\Lambda_{e},\quad\Lambda_{-b}:=-\Lambda_{b}.\label{notation}\end{equation} 
Let \begin{equation}\label{symbol0}\mathcal{P}:=\{e,b,-e,-b\}.\end{equation}
For $\sigma,\mu,\nu\in \mathcal{P}$, we define the associated phase function
\begin{equation}\label{phasedef}
\Phi_{\sigma\mu\nu}(\xi,\eta):=\Lambda_{\sigma}(\xi)-\Lambda_{\mu}(\xi-\eta)-\Lambda_{\nu}(\eta),
\end{equation} 
and the corresponding function
\begin{equation*}
\begin{split}
\Phi^{+}_{\sigma\mu\nu}(\alpha,\beta):=\Phi_{\sigma\mu\nu}(\alpha e,\beta e)=\lambda_{\sigma}(\alpha)-\lambda_{\mu}(\alpha-\beta)-\lambda_{\nu}(\beta),\\
\lambda_e(r)=-\lambda_{-e}(r):=\sqrt{1+dr^2},\qquad\lambda_b(r)=-\lambda_{-b}(r):=\sqrt{1+r^2},
\end{split}
\end{equation*}
where $e\in\mathbb{S}^{1}$ and $\alpha,\beta\in\mathbb{R}$.
If $\nu+\mu\neq 0$, by Proposition \ref{spaceres} for any $\xi\in\mathbb{R}^2$ there exists a unique $\eta=p(\xi)\in\mathbb{R}^2$ so that $(\nabla_{\eta}\Phi_{\sigma\mu\nu})(\xi,\eta)=0$. We define 
\begin{equation}\label{psidag}
\Psi_{\sigma\mu\nu}(\xi):=\Phi_{\sigma\mu\nu}(\xi,p(\xi)),\qquad \Psi_{\sigma}^{\dagger}(\xi):=2^{\mathcal{D}_0}(1+|\xi|)\inf_{\mu,\nu\in\mathcal{P};\nu+\mu\neq 0}|\Psi_{\sigma\mu\nu}(\xi)|,
\end{equation}
and notice that these functions are radial. The functions $\Psi_{e}^{\dagger}$ and $\Psi_{b}^{\dagger}$ are described in Remark \ref{largeres}; in particular, by setting $\D_0$ sufficiently large, $\Psi_e^{\dagger}\geq 10$ while $\Psi_b^{\dagger}$ vanishes on two spheres $|\xi|=\gamma_{1,2}=\gamma_{1,2}(d)\in(0,\infty)$. These spheres correspond to space-time resonances. For $n\in\mathbb{Z}$ we define the operators $A_{n}^{\sigma}$ by
\begin{equation}\label{aop}
\widehat{A_{n}^{\sigma}f}(\xi):=\varphi_{-n}(\Psi_{\sigma}^{\dagger}(\xi))\cdot\widehat{f}(\xi),
\end{equation}
for $\sigma\in\{e,b\}$. Given an integer $j\geq 0$ we define the operators $A^\sigma_{n,(j)}$, $n\in\{0,\ldots,j+1\}$, by
\begin{equation*}
A_{0,(j)}^\sigma:=\sum _{n'\leq 0}A_{n'}^\sigma,\qquad A_{j+1,(j)}^\sigma:=\sum _{n'\geq j+1}A_{n'}^\sigma,\qquad A_{n,(j)}^\sigma:=A_{n}^\sigma\,\,\text{ if }\,\,0<n<j+1.
\end{equation*}
We fix a constant $\D\geq 1$ sufficiently large depending only on the smoothness of $h$ and the parameter $d\in(0,1)$. 

We are now ready to define the main norms.

\begin{definition}\label{MainZDef}
Assume that $N_0,N_1$ are sufficiently large and $\delta$ is sufficiently small (for example $\delta:=4\cdot 10^{-7},\,N_{1}:=8/\delta^2,\,N_{0}:=20/\delta^2$ would work\footnote{The numbers $N_0$ and $N_1$ can, of course, be improved substantially by reexamining and reworking parts of the argument. However, our primary objective in this paper is to show that the global existence argument can be closed in {\it{some topology}}, and we will not deal with such additional challenges here.}). Let 
\begin{equation}\label{OmegaVF}
\Omega:=x_1\partial_2-x_2\partial_1
\end{equation}
denote the rotation vector-field, and define
\begin{equation}\label{Z2norm}
H_\Omega^{N_1}:=\{f\in L^2(\mathbb{R}^2):\,\|f\|_{H_\Omega^{N_1}}:=\sup_{m\leq N_1}\|\Omega^{m}f\|_{L^2}<\infty\}.
\end{equation}
For $\sigma\in\{e,b\}$ we define
\begin{equation}\label{sec5}
Z_1^\sigma:=\{f\in L^2(\mathbb{R}^2):\,\|f\|_{Z_1^\sigma}:=\sup_{(k,j)\in\mathcal{J}}2^{6k_+}\|Q_{jk}f\|_{B^\sigma_{j}}<\infty\},
\end{equation}
where
\begin{equation}\label{znorm2}
\|g\|_{B_j^{\sigma}}:=2^{(1-20\delta)j}\sup_{0\leq n\leq j+1}2^{-(1/2-19\delta)n}\|A_{n,(j)}^{\sigma}g\|_{L^{2}}.
\end{equation}
Finally, we define
\begin{equation}\label{znorm}
Z:=\big\{(f_e,f_b)\in L^2\times L^2:\,\|(f_e,f_b)\|_{Z}:=\sup_{m\leq N_1/2}\big[\|\Omega^mf_e\|_{Z_1^e}+\|\Omega^mf_b\|_{Z_1^b}\big]<\infty\big\}.
\end{equation}
\end{definition}

Notice that, when $\sigma=e$ we have the simpler formula, 
\begin{equation*}
\|g\|_{B_j^{e}}\approx 2^{(1-20\delta)j}\|g\|_{L^{2}}.
\end{equation*}
Similarly if $j\lesssim 1$ then $\|g\|_{B_j^{b}}\approx\|g\|_{L^{2}}$. The operators $A_{n,(j)}^{\sigma}$ are relevant only when $\sigma=b$ and $j\gg 1$, to localize to thin neighborhoods of the space-time resonant sets.

\subsection{The main bootstrap proposition}\label{bootstrap0} Our main result is the following proposition:

\begin{proposition}\label{bootstrap} Suppose $(\rho,u,\widetilde{E},\widetilde{b})$ is a solution to \eqref{system2}--\eqref{condition} on some time 
interval $[0,T]$, $T\geq 1$, with initial data $(\rho_0,u_0,\widetilde{E}_0,\widetilde{b}_0)$, and define $(U_e,U_b)$, $(V_e,V_b)$ as in \eqref{variables4}. Assume that 
\begin{equation}\label{bootstrap1}
\|(\rho_0,u_0,\widetilde{E}_0,\widetilde{b}_0)\|_{H^{N_{0}}\cap H^{N_1}_\Omega}+\|(V_e(0),V_b(0))\|_{Z}\leq\epsilon_{0}\ll 1
\end{equation} 
and, for any $t\in[0,T]$,
\begin{equation}\label{bootstrap2}
\|(\rho(t),u(t),\widetilde{E}(t),\widetilde{b}(t))\|_{H^{N_{0}}\cap H^{N_1}_\Omega}+\|(V_e(t),V_b(t))\|_{Z}\leq \epsilon_{1}\ll 1.
\end{equation} 
Then, for any $t\in[0,T]$,
\begin{equation}\label{bootstrap3}
\|(\rho(t),u(t),\widetilde{E}(t),\widetilde{b}(t))\|_{H^{N_{0}}\cap H^{N_1}_\Omega}+\|(V_e(t),V_b(t))\|_{Z}\lesssim \eps_0+\epsilon_1^{3/2}.
\end{equation} 
\end{proposition}

Given Proposition \ref{bootstrap}, Theorem \ref{main} follows using a local existence result and a continuity argument. 
See \cite[Sections 2 and 3]{IoPa2} (in particular Proposition 2.1 and Proposition 2.4) for similar arguments in dimension 3.
 
The rest of this paper is concerned with the proof of Proposition \ref{bootstrap}. This proposition follows from 
Proposition \ref{BootstrapEE} and Proposition \ref{BootstrapZnorm}.

\section{Main lemmas}\label{LemSec}

In this section we collect several important lemmas which are used often in the proofs in the next sections. Let $\Phi=\Phi_{\sigma\mu\nu}$ as in \eqref{phasedef}.

\subsection{Operator Estimates}

We define the class of Calder\'{o}n--Zygmund symbols $\mathcal{S}^n$, $n\ge 1$, by
\begin{equation*}
\mathcal{S}^n=\{q:\mathbb{R}^2\to\mathbb{C}:\,\,\Vert q\Vert_{\mathcal{S}^n}:=\sup_{\xi\ne0}\sup_{\vert\alpha\vert\le n}\vert\xi\vert^{\vert\alpha\vert}\vert D^\alpha_\xi q(\xi)\vert\le 1\}.
\end{equation*}

\begin{lemma}\label{lem:CZ}
Assume that $q\in\mathcal{S}^{N_1}$, $1\le p\le\infty$, $k\in\mathbb{Z}$ then
\begin{equation*}
\Vert \mathcal{F}^{-1}q\mathcal{F}P_kf\Vert_{L^p}\lesssim\Vert P_kf\Vert_{L^p},\qquad\Vert \mathcal{F}^{-1}q\mathcal{F}f\Vert_{Z^\sigma_1}\lesssim\Vert f\Vert_{Z^\sigma_1},\qquad\Vert \mathcal{F}^{-1}q\mathcal{F}f\Vert_{H^{N_1}_\Omega}\lesssim\Vert f\Vert_{H^{N_1}_\Omega}
\end{equation*}
uniformly in $q$, $p$, $k$.
\end{lemma}
We refer to \cite[Lemma 5.1.]{IoPa2} for a similar proof. To bound multilinear operators, we will often use the following simple lemma:
\begin{lemma}\label{L1easy}
Assume $l\geq 2$, $f_1,\ldots,f_l,f_{l+1}\in L^2(\mathbb{R}^2)$, and $M:(\mathbb{R}^2)^l\to\mathbb{C}$ is a continuous compactly supported function. Then
\begin{equation}\label{ener62}
\begin{split}
\Big|\int_{(\mathbb{R}^2)^l}M(\xi_1,\ldots,\xi_l)\widehat{f_1}(\xi_1)\cdot\ldots\cdot\widehat{f_l}(\xi_l)\cdot\widehat{f_{l+1}}(-\xi_1-\ldots-\xi_l)\,d\xi_1\ldots d\xi_l\Big|\\
\lesssim \big\|\mathcal{F}^{-1}M\big\|_{L^1}\|f_1\|_{L^{p_1}}\cdot\ldots\cdot\|f_{l+1}\|_{L^{p_{l+1}}},
\end{split}
\end{equation}
for any exponents $p_1,\ldots p_{l+1}\in[1,\infty]$ satisfying $1/p_1+\ldots+1/p_{l+1}=1$. As a consequence
\begin{equation}\label{ener62.1}
\begin{split}
\Big\|\mathcal{F}_{\xi}^{-1}\Big\{\int_{(\mathbb{R}^2)^{l-1}}M(\xi,\eta_2\ldots,\eta_l)\widehat{f_2}(\eta_2)\cdot\ldots\cdot\widehat{f_l}(\eta_l)\cdot\widehat{f_{l+1}}(-\xi-\eta_2\ldots-\eta_l)\,d\eta_2\ldots d\eta_l\Big\}\Big\|_{L^{q}}\\
\lesssim \big\|\mathcal{F}^{-1}M\big\|_{L^1((\mathbb{R}^2)^l)}\|f_2\|_{L^{p_2}}\cdot\ldots\cdot\|f_{l+1}\|_{L^{p_{l+1}}},
\end{split}
\end{equation}
if $q,p_2\ldots p_{l+1}\in[1,\infty]$ satisfy $1/p_2+\ldots+1/p_{l+1}=1/q$.
\end{lemma}

Given a continuous compactly supported $M:(\mathbb{R}^2)^2\to\mathbb{C}$ we define
\begin{equation}\label{ener65}
\|M\|_{S^\infty}:=\big\|\mathcal{F}^{-1}M\big\|_{L^1}
\end{equation}
and notice that
\begin{equation}\label{ener65.5}
\|M_1\cdot M_2\|_{S^\infty}\lesssim \|M_1\|_{S^\infty}\|M_2\|_{S^\infty}.
\end{equation}
Given integers $k,k_1,k_2\in\mathbb{Z}$, we define
\begin{equation}\label{ener66}
\|M(\xi,\eta)\|_{S^\infty_{kk_1k_2}}:=\big\|\mathcal{F}^{-1}\big[M(\xi,\eta)\varphi_{k}(\xi)\varphi_{k_{1}}(\xi-\eta)\varphi_{k_{2}}(\eta)\big]\big\|_{L^1}.
\end{equation}

We will often use the following simple lemma to estimate the $S^\infty$ norm of symbols:

\begin{lemma}\label{Sinfinity}
If $f:\mathbb{R}^2\times\mathbb{R}^2\to\mathbb{C}$ is a bounded function, $\chi,\chi'\in\mathcal{S}$ and $l_1,l_2\in\mathbb{Z}$ then
\begin{equation}\label{ener66.5}
\begin{split}
\Big\|\int_{\mathbb{R}^2\times\mathbb{R}^2}&e^{ix\cdot\alpha}e^{iy\cdot \beta}f(\alpha,\beta)\chi(2^{-l_1}\alpha)\chi'(2^{-l_2}\beta)\,d\alpha d\beta\Big\|_{L^1_{x,y}}\\
&\lesssim \sum_{m=0}^3\Big[2^{ml_1}\|\partial_\alpha^m f\|_{L^\infty}+2^{ml_2}\|\partial_\beta^m f\|_{L^\infty}\Big].
\end{split}
\end{equation}
\end{lemma}

\begin{proof} By rescaling we may assume that $l_1=l_2=0$. Letting
\begin{equation*}
F(x,y):=\int_{\mathbb{R}^2\times\mathbb{R}^2}e^{ix\cdot\alpha}e^{iy\cdot \beta}f(\alpha,\beta)\chi(\alpha)\chi'(\beta)\,d\alpha d\beta
\end{equation*}
we have, using Plancherel theorem and integration by parts
\begin{equation*}
\|F\|_{L^2}\lesssim \|f\|_{L^\infty}\qquad\text{ and }\qquad\|(|x|+|y|)^3F\|_{L^2}\lesssim\sum_{m=0}^3\Big[\|\partial_\alpha^m f\|_{L^\infty}+\|\partial_\beta^m f\|_{L^\infty}\Big].
\end{equation*}
The desired $L^1$ bound follows, with an implicit constant that depends only on $\chi,\chi'$.
\end{proof}

\subsection{Weyl paradifferential calculus}\label{ParaDiffCalc} We recall first the definition of paradifferential operators (Weyl quantization): given a symbol $a=a(x,\zeta):\mathbb{R}^2\times\mathbb{R}^2\to\mathbb{C}$, we define the operator $T_a$ by
\begin{equation}\label{Tsigmaf2}
\begin{split}
\mathcal{F}\left\{T_{a}f\right\}(\xi)=\frac{1}{4\pi^2}\int_{\mathbb{R}^2}\chi\Big(\frac{\vert\xi-\eta\vert}{\vert\xi+\eta\vert}\Big)\widetilde{a}(\xi-\eta,(\xi+\eta)/2)\widehat{f}(\eta)d\eta,
\end{split}
\end{equation}
where $\widetilde{a}$ denotes the partial Fourier transform of $a$ in the first coordinate and $\chi=\varphi_{\le-2\mathcal{D}}$.

We will use a simple norm to estimate symbols: for $q\in [1,\infty]$ and $l\in \mathbb{R}$ we define
\begin{equation}\label{nor1}
\begin{split}
&\|a\|_{\mathcal{L}^q_l}:=\sup_{\zeta\in\mathbb{R}^2}(1+|\zeta|^2)^{-l/2}\|\,|a|(.,\zeta)\|_{L^q_x},\\
&|a|(x,\zeta):=\sum_{|\beta|\leq 20,\,|\alpha|\leq 2}|\zeta|^{|\beta|}|(D^\beta_\zeta D^\alpha_x a)(x,\zeta)|.
\end{split}
\end{equation}
The index $l$ is called the {\it{order}} of the symbol, and it measures the contribution of the symbol in terms of derivatives on $f$. Notice that we have the simple product rule
\begin{equation}\label{nor2}
\|ab\,\|_{\mathcal{L}^p_{l_1+l_2}}\lesssim \|a\|_{\mathcal{L}^q_{l_1}}\|b\|_{\mathcal{L}^r_{l_2}},\qquad 1/p=1/q+1/r.
\end{equation}

An important property of paradifferential operators is that they behave well with respect to products. More precisely:

\begin{lemma}\label{PropProd} (i) If $1/p=1/q+1/r$ and $k\in\mathbb{Z}$, and $l\in [-10,10]$ then
\begin{equation}
\label{LqBdTa}
\Vert P_kT_af\Vert_{L^p}\lesssim 2^{lk_+}\Vert a\Vert_{\mathcal{L}^q_l}\Vert P_{[k-2,k+2]}f\Vert_{L^r}.
\end{equation}

(ii) For any symbols $a,b$ let $E(a,b):=T_aT_b-T_{ab}$. If $1/p=1/q_1+1/q_2+1/r$, $k\in\mathbb{Z}$, and $l_1,l_2\in [-4,4]$ then 
\begin{equation}
\label{LqBdTa2}
2^{k_+}\Vert P_kE(a,b)f\Vert_{L^p}\lesssim (2^{l_1k_+}\Vert a\Vert_{\mathcal{L}^{q_1}_{l_1}})(2^{l_2k_+}\Vert b\Vert_{\mathcal{L}^{q_2}_{l_2}})\cdot\Vert P_{[k-4,k+4]}f\Vert_{L^r}.
\end{equation}
\end{lemma} 

\begin{proof} The proof follows directly from the definitions \eqref{Tsigmaf2} and \eqref{nor1}. Indeed, for (i) we write 
\begin{equation*}
\begin{split}
\langle P_kT_af,g\rangle&=C\int_{\mathbb{R}^4} \overline{g}(x)f(y)I(x,y)dxdy,\\
I(x,y)&=\int_{\mathbb{R}^6} a(z,(\xi+\eta)/2)e^{i\xi\cdot (x-z)}e^{i\eta\cdot (z-y)}\chi\Big(\frac{\vert \xi-\eta\vert }{\vert\xi+\eta\vert}\Big)\varphi_k(\xi) \,d\eta d\xi dz\\
&=\int_{\mathbb{R}^6} a(z,\xi+\theta/2)e^{i\theta\cdot(z-y)}e^{i\xi\cdot (x-y)}\chi\Big(\frac{\vert \theta\vert }{\vert 2\xi+\theta\vert}\Big)\varphi_k(\xi)\varphi_{\leq k-10}(\theta) \,d\xi d\theta dz.
\end{split}
\end{equation*}
We observe that
\begin{equation*}
\begin{split}
(1+2^{2k}\vert x-y\vert^2)^2&I(x,y)=\int_{\mathbb{R}^6} \frac{a(z,\xi+\theta/2)}{(1+2^{2k}\vert z-y\vert^2)^2}\chi\Big(\frac{\vert \theta\vert }{\vert 2\xi+\theta\vert}\Big)\varphi_k(\xi)\varphi_{\leq k-10}(\theta)\\
&\times \left[(1-2^{2k}\Delta_\theta)^2(1-2^{2k}\Delta_\xi)^2\{e^{i\theta\cdot(z-y)}e^{i\xi\cdot (x-y)}\}\right] \,d\xi d\theta dz.
\end{split}
\end{equation*}
By integration by parts in $\xi$ and $\theta$ it follows that
\begin{equation*}
(1+2^{2k}\vert x-y\vert^2)^2|I(x,y)|\lesssim \int_{\mathbb{R}^6} \frac{|a|(z,\xi+\theta/2)}{(1+2^{2k}\vert z-y\vert^2)^2}\varphi_{[k-2,k+2]}(\xi)\varphi_{\leq k-8}(\theta)\,d\xi d\theta dz,
\end{equation*}
where $|a|$ is defined as in \eqref{nor1}. The bound \eqref{LqBdTa} now follows since $\|(1+2^{2k}|y|^2)^{-2}\|_{L^1(\mathbb{R}^2)}\lesssim 2^{-2k}$.

The proof of part (ii) is similar. The point is, of course, the gain of one derivative in the left-hand side for the operator $T_aT_b-T_{ab}$. We remark that one could in fact gain two derivatives by subtracting the contribution of the Poisson bracket between the symbols $a$ and $b$, defined by
\begin{align*}
\{ a,b \} := \nabla_x a \nabla_\zeta b - \nabla_\zeta a \nabla_x b.
\end{align*}
We do not need a refinement of this type in this paper.
\end{proof}

An important feature of the Weyl paradifferential calculus is self-adjointness of operators defined by real-valued symbols. We record below some properties that follow easily from the definition (for \eqref{Alu3} one also needs the observation that $(\Omega_{\xi}+\Omega_{\eta})\big[\chi\big(\frac{\vert\xi-\eta\vert}{\vert\xi+\eta\vert}\big)\big]\equiv 0$).

\begin{lemma}\label{PropSym} 

(i) If $\|a\|_{\mathcal{L}_0^\infty}<\infty$ is real-valued then $T_a$ is a bounded self-adjoint operator on $L^2$.

(ii) We have
\begin{equation}\label{Alu2}
\overline{T_af}=T_{a'}\overline{f},\quad\text{ where }\quad a'(y,\zeta):=\overline{a(y,-\zeta)}
\end{equation} 
and
\begin{equation}\label{Alu3}
\Omega (T_af)=T_a(\Omega f)+T_{\Omega_{x,\zeta}a}f\quad \text{ where }\quad \Omega_{x,\zeta}a(x,\zeta)=(\Omega_xa)(x,\zeta)+(\Omega_{\zeta}a)(x,\zeta).
\end{equation} 
\end{lemma}

The paradifferential calculus is useful to linearize products and compositions.

\begin{lemma}\label{PropHHSym}
(i) If $f,g\in L^2$ then
\begin{equation}\label{paracomp0}
fg=T_fg+T_gf+\mathcal{H}(f,g)
\end{equation}
where $\mathcal{H}$ is smoothing in the sense that
\begin{equation*}
\Vert P_k\mathcal{H}(f,g)\Vert_{L^q}\lesssim \sum_{k',k''\geq k-3\D,\,|k'-k''|\leq 3\D}\min\big(\Vert P_{k'}f\Vert_{L^q}\Vert P_{k''}g\Vert_{L^\infty},\Vert P_{k'}f\Vert_{L^\infty}\Vert P_{k''}g\Vert_{L^q}\big).
\end{equation*}

(ii) Assume that $c>0$ and $F(z)=z+h(z)$, where $h$ is analytic for $\vert z\vert <c\leq 1$ and satisfies $\vert h(z)\vert\lesssim \vert z\vert^3$. If $\Vert \Omega^au\Vert_{W^{5,\infty}}\le c/2$, $a\in[0,N_1/2]$, then
\begin{equation}\label{Paracomp}
\begin{split}
F(u) & = T_{F^\prime(u)}u + E^3(u),
\\
\Vert (1-\Delta)E^3(u)\Vert_{H^{N_0}\cap H^{N_1}_\Omega}&\lesssim (\Vert u\Vert_{W^{5,\infty}}+\Vert\Omega^{N_1/2}u\Vert_{W^{5,\infty}})^2\Vert u\Vert_{H^{N_0}\cap H^{N_1}_\Omega}.
\end{split}
\end{equation}
\end{lemma}

\begin{proof} Part (i) follows directly from definition \eqref{Tsigmaf2}. For part (ii) we expand $h(z)=\sum_{n\geq 3}a_nz^n$ and apply \eqref{Alu3} and standard Littlewood-Paley analysis for each monomial at a time. 
\end{proof}

The paradifferential calculus and the results in this subsection are used mostly in the derivation of the quasilinear equations in Proposition \ref{ChangeUnknownsProp} below.

\subsection{Integration by parts} In this subsection we prove two lemmas that are used in the paper in integration by parts arguments. We start with an oscillatory integral estimate. See \cite[Lemma 5.4]{IoPa2} for the proof.

\begin{lemma}\label{tech5} (i) Assume that $0<\eps\leq 1/\eps\leq K$, $N\geq 1$ is an integer, and $f,g\in C^N(\mathbb{R}^2)$. Then
\begin{equation}\label{ln1}
\Big|\int_{\mathbb{R}^2}e^{iKf}g\,dx\Big|\lesssim_N (K\eps)^{-N}\big[\sum_{|\alpha|\leq N}\eps^{|\alpha|}\|D^\alpha_xg\|_{L^1}\big],
\end{equation}
provided that $f$ is real-valued, 
\begin{equation}\label{ln2}
|\nabla_x f|\geq \mathbf{1}_{{\mathrm{supp}}\,g},\quad\text{ and }\quad\|D_x^\alpha f \cdot\mathbf{1}_{{\mathrm{supp}}\,g}\|_{L^\infty}\lesssim_N\eps^{1-|\alpha|},\,2\leq |\alpha|\leq N.
\end{equation}

(ii) Similarly, if $0<\rho\leq 1/\rho\leq K$ then
\begin{equation}\label{ln3}
\Big|\int_{\mathbb{R}^2}e^{iKf}g\,dx\Big|\lesssim_N (K\rho)^{-N}\big[\sum_{m\leq N}\rho^{m}\|\Omega^m g\|_{L^1}\big],
\end{equation}
provided that $f$ is real-valued, 
\begin{equation}\label{ln4}
|\Omega f|\geq \mathbf{1}_{{\mathrm{supp}}\,g},\quad\text{ and }\quad\|\Omega^m f \cdot\mathbf{1}_{{\mathrm{supp}}\,g}\|_{L^\infty}\lesssim_N\rho^{1-m},\,2\leq m\leq N.
\end{equation}
\end{lemma}

We will need another result about integration by parts using the vector-field $\Omega$. 

\begin{lemma}\label{RotIBP}
Assume that $t\in[2^{m}-1,2^{m+1}]$, $m\geq 0$, $k,k_1,k_2\in\mathbb{Z}$, $L\le 1\le U$, and
\begin{equation*}
2^{-2m/5}\le L\le 2^{k_1},\qquad 2^k+2^{k_1}+2^{k_2}\le U\leq U^2\le 2^{2m/5}L.
\end{equation*}
Assume that $A\geq 1+2^{-k_1}$ and
\begin{equation}\label{hypo}
\begin{split}
\sup_{0\le a\le 100}\big[\Vert\Omega^ag\Vert_{L^2}+\Vert\Omega^af\Vert_{L^2}\big]+\sup_{0\leq\vert\alpha\vert\le N}A^{-\vert\alpha\vert}\Vert D^\alpha\widehat{f}\,\Vert_{L^2}&\le 1,\\
\sup_{\xi,\eta}\sup_{\vert\alpha\vert\le N}(2^{-m/2}\vert\eta\vert)^{\vert\alpha\vert}\vert D^\alpha_\eta m(\xi,\eta)\vert&\le 1.
\end{split}
\end{equation}
Fix $\xi\in\mathbb{R}^2$ and $\Phi=\Phi_{\sigma\mu\nu}$ as in \eqref{phasedef}, and let
\begin{equation*}
I_p=I_p(f,g):=\int_{\mathbb{R}^2}e^{it\Phi(\xi,\eta)}m(\xi,\eta)\varphi_p(\Omega_\eta\Phi(\xi,\eta))\varphi_k(\xi)\varphi_{k_1}(\xi-\eta)\varphi_{k_2}(\eta)\widehat{f}(\xi-\eta)\widehat{g}(\eta)d\eta.
\end{equation*}
If $U^42^{-m}\leq 2^{2p}\leq 1$ and $AL^{-1}U^2\leq 2^m$ then
\begin{equation}\label{OmIBP}
\vert I_p\vert\lesssim_N  (2^{p+m})^{-N} \big[2^{m/2}+U^42^{-p}+U^2L^{-1}A2^{p}\big]^N+2^{-10m}.
\end{equation}
In addition, assuming that $(1+\delta/4)\nu\geq -m$, the same bounds holds when $I_p$ is replaced by
\begin{equation*}
\widetilde{I_p}:=\int_{\mathbb{R}^2}e^{it\Phi(\xi,\eta)}\varphi_\nu(\Phi(\xi,\eta))m(\xi,\eta)\varphi_p(\Omega_\eta\Phi(\xi,\eta))\varphi_k(\xi)\varphi_{k_1}(\xi-\eta)\varphi_{k_2}(\eta)\widehat{f}(\xi-\eta)\widehat{g}(\eta)d\eta.
\end{equation*}
\end{lemma}

A slightly simpler version of this integration by parts lemma was used recently by one of the authors \cite{De}. The main interest of this lemma is that we have essentially no assumption on $g$ and very mild assumptions on $f$. We will often use this lemma with
\begin{equation*}
U,L\approx 1,\qquad A\le 2^{(1-\delta^2)m},
\end{equation*}
in which case we obtain 
\begin{equation*}
\vert I_p\vert\lesssim_N 2^{-(N/2)(m+2p)}+(2^{-m}A)^{-N}.
\end{equation*}
However, our more precise version allows us to consider some cases when a frequency is small.

\begin{proof}[Proof of Lemma \ref{RotIBP}] We decompose first $f=\mathcal{R}_{\leq m/10}f+[I-\mathcal{R}_{\leq m/10}]f$, $g=\mathcal{R}_{\leq m/10}g+[I-\mathcal{R}_{\leq m/10}]g$, where the operators $\mathcal{R}_{\leq L}$ are defined in polar coordinates by
\begin{equation}\label{RadOp}
(\mathcal{R}_{\leq L}h)(r\cos\theta,r\sin\theta):=\sum_{n\in\mathbb{Z}}\varphi_{\leq L}(n)h_n(r)e^{in\theta}\quad\text { if }\quad h(r\cos\theta,r\sin\theta):=\sum_{n\in\mathbb{Z}}h_n(r)e^{in\theta}.
\end{equation}
Since $\Omega$ corresponds to $d/d\theta$ in polar coordinates, using \eqref{hypo} we have,
\begin{equation*}
\big\|[I-\mathcal{R}_{\leq m/10}]f\big\|_{L^2}+\big\|[I-\mathcal{R}_{\leq m/10}]g\big\|_{L^2}\lesssim 2^{-10m}.
\end{equation*}
Therefore, using the H\"{o}lder inequality,
\begin{equation*}
\vert I_p\big([I-\mathcal{R}_{\leq m/10}]f,g\big)\vert+\vert I_p\big(\mathcal{R}_{\leq m/10}f,[I-\mathcal{R}_{\leq m/10}]g\big)\vert\lesssim 2^{-10m}.
\end{equation*}

It remains to prove a similar inequality for $I_p:=I_p\big(f_1,g_1\big)$, where $f_1:=\varphi_{[k_1-2,k_1+2]}\cdot \mathcal{R}_{\leq m/10}f$, $g_1:=\varphi_{[k_2-2,k_2+2]}\cdot \mathcal{R}_{\leq m/10}g$. Integration by parts gives
\begin{equation*}
I_p=c\int_{\mathbb{R}^2}e^{it\Phi(\xi,\eta)}\Omega_\eta\Big\{\frac{1}{t\Omega_\eta\Phi(\xi,\eta)}m(\xi,\eta)\varphi_p(\Omega_\eta\Phi(\xi,\eta))\varphi_k(\xi)\varphi_{k_1}(\xi-\eta)\varphi_{k_2}(\eta)\widehat{f_1}(\xi-\eta)\widehat{g_1}(\eta)\Big\}\,d\eta.
\end{equation*}
Iterating $N$ times, we obtain an integrand made of a linear combination of terms like
\begin{equation*}
\begin{split}
e^{it\Phi(\xi,\eta)}&\varphi_k(\xi)\left(\frac{1}{t\Omega_\eta\Phi(\xi,\eta)}\right)^N\times\Omega^{a_1}_\eta\left\{m(\xi,\eta)\varphi_{k_1}(\xi-\eta)\varphi_{k_2}(\eta)\right\}\\
&\times\Omega_\eta^{a_2}\widehat{f}(\xi-\eta)\cdot\Omega_\eta^{a_3}\widehat{g}(\eta)\cdot\Omega_\eta^{a_4}\varphi_p(\Omega_\eta\Phi(\xi,\eta))\cdot \frac{\Omega^{a_5+1}_\eta\Phi}{\Omega_\eta\Phi}\dots\frac{\Omega_\eta^{a_q+1}\Phi}{\Omega_\eta\Phi},
\end{split}
\end{equation*}
where $\sum a_i=N$. Then \eqref{OmIBP} follows from the pointwise bounds
\begin{equation}\label{BdsIBP}
\begin{split}
\left\vert\Omega^{a}_\eta\left\{m(\xi,\eta)\varphi_{k_1}(\xi-\eta)\varphi_{k_2}(\eta)\right\}\right\vert&\lesssim 2^{am/2},\\
\left\vert \Omega^{a}_\eta\varphi_p(\Omega_\eta\Phi(\xi,\eta))\right\vert+ \left\vert\frac{\Omega^{a+1}_\eta\Phi}{\Omega_\eta\Phi}\right\vert&\lesssim U^{4a}2^{-pa},\\
\end{split}
\end{equation}
which hold in the support of the integral, and the $L^2$ bounds
\begin{equation}\label{l2ibp}
\begin{split}
\Vert \Omega^{a}_\eta \widehat{g_1}(\eta)\Vert_{L^2}&\lesssim 2^{am/2}\\
\Vert \Omega^a_\eta\widehat{f_1}(\xi-\eta)\varphi_{k}(\xi)\varphi_{k_1}(\xi-\eta)\varphi_{k_2}(\eta)\varphi_{\leq p+4}(\Omega_\eta\Phi(\xi,\eta))\Vert_{L^2_\eta}&\lesssim U^{2a}\big(L^{-1}A2^p\big)^a+2^{am/2}.
\end{split}
\end{equation}

The first bound in \eqref{BdsIBP} is direct. For the second bound we notice that
\begin{equation}\label{OmegaBounds}
\begin{split}
\Omega_\eta\vert\xi-\eta\vert&=-\frac{\xi\cdot\eta^\perp}{\vert\xi-\eta\vert},\qquad\Omega_\eta(\xi\cdot\eta^\perp)=-\xi\cdot\eta,\qquad\Omega_\eta(\xi\cdot\eta)=\xi\cdot\eta^\perp,\\
\Omega_\eta\Phi(\xi,\eta)&=\frac{\lambda^\prime_\mu(\vert\xi-\eta\vert)}{\vert\xi-\eta\vert}(\xi\cdot\eta^\perp),\qquad\vert\Omega_\eta^a\Phi(\xi,\eta)\vert\lesssim U^{2a}.
\end{split}
\end{equation}

The first bound in \eqref{l2ibp} follows from the hypothesis. To prove the second bound, we may assume that $U2^p\ll L2^{\min(k,k_2)}$, and $\xi=(s,0)$, $s\approx 2^k$. The identities \eqref{OmegaBounds} show that $\varphi_{\leq p+4}(\Omega_\eta\Phi(\xi,\eta))\neq 0$ only if $|\xi\cdot\eta^\perp|\lesssim 2^p U$, which gives $|\eta_2|\lesssim 2^pU2^{-k}$. Therefore, $|\eta_2|\ll L$, so we may assume that $|\eta_1-s|\approx 2^{k_1}$.

We write now
\begin{equation*}
-\Omega_\eta\widehat{f_1}(\xi-\eta)=(\eta_1\partial_2\widehat{f_1}-\eta_2\partial_1\widehat{f_1})(\xi-\eta)=\frac{\eta_1}{s-\eta_1}(\Omega\widehat{f_1})(\xi-\eta)-\frac{s\eta_2}{s-\eta_1}(\partial_1\widehat{f_1})(\xi-\eta).
\end{equation*}
The second bound in \eqref{l2ibp} follows by iterating this identity and using the bounds on $f$ in \eqref{hypo} and the bounds proved earlier, $|s\eta_2|\lesssim 2^pU$, $|\eta_1-s|\approx 2^{k_1}$.

The last claim follows from the formula
\begin{equation*}
\varphi_\nu(\Phi(\xi,\eta))=c2^{\nu}\int_{\mathbb{R}}\widehat{\varphi_0}(2^\nu\lambda)e^{i\lambda\Phi(\xi,\eta)}d\lambda
\end{equation*}
using an argument as in the proof of the localization Lemma \ref{PhiLocLem} below.
\end{proof}

\subsection{Schur test and localization}\label{ShurSubSec}

We will often use the classical Schur's test.

\begin{lemma}[Schur's test]\label{ShurLem}
Consider the operator $T$ given by
\begin{equation*}
Tf(\xi)=\int_{\mathbb{R}^2}K(\xi,\eta)f(\eta)d\eta.
\end{equation*}
Assume that
\begin{equation*}
\sup_\xi\int_{\mathbb{R}^2}\vert K(\xi,\eta)\vert d\eta\le K_1,\qquad \sup_\eta\int_{\mathbb{R}^2}\vert K(\xi,\eta)\vert d\xi\le K_2.
\end{equation*}
Then
\begin{equation*}
\Vert Tf\Vert_{L^2}\lesssim \sqrt{K_1K_2}\Vert f\Vert_{L^2}.
\end{equation*}
\end{lemma}

Our second lemma in this subsection shows that localization with respect to the phase is often a bounded operation:

\begin{lemma}\label{PhiLocLem}
Let $s\in[2^{m}-1,2^m]$, $m\geq 0$, and $(1+\varepsilon)\nu\le m$ for some $\varepsilon>0$. Let $\Phi=\Phi_{\sigma\mu\nu}$ as in \eqref{phasedef} and assume that $1/2=1/q+1/r$ and $\chi$ is a Schwartz function. Then
\begin{equation}\label{PhiLocLem2}
\begin{split}
\Big\Vert \varphi_{\leq 10m}(\xi)\int_{\mathbb{R}^2}e^{is\Phi(\xi,\eta)}&\chi(2^\nu\Phi(\xi,\eta))\widehat{f}(\xi-\eta)\widehat{g}(\eta)d\eta\Big\Vert_{L^2_\xi}\\
&\lesssim\sup_{t\in[s/10,10s]}\Vert e^{it\Lambda_\mu}f\Vert_{L^q}\Vert e^{it\Lambda_\nu}g\Vert_{L^r}+2^{-10m}\Vert f\Vert_{L^2}\Vert g\Vert_{L^2},
\end{split}
\end{equation}
where the constant in the inequality only depends on $\varepsilon$ and the function $\chi$.
\end{lemma}

\begin{proof}
We use the Fourier transform to write
\begin{equation*}
\chi(2^\nu\Phi(\xi,\eta))=c\int_{\mathbb{R}}e^{i\lambda 2^\nu\Phi(\xi,\eta)}\widehat{\chi}(\lambda)d\lambda.
\end{equation*}
We substitute this formula in the left-hand side of \eqref{PhiLocLem2}. The part of the integral over $|\lambda 2^\nu|\leq s/2$ is dominated by the first term in the right-hand side. Moreover, the integral over $|\lambda|\geq 2^{-\nu} s/2$ is negligible (thus dominated by $C2^{-10m}\Vert f\Vert_{L^2}\Vert g\Vert_{L^2}$), since $\widehat{\chi}$ is a Schwartz function.
\end{proof}

\subsection{Linear estimates}\label{LinearEstSubSec} We note first the straightforward estimates, $\sigma\in\{e,b\}$,
\begin{equation}\label{Straightforward}
\Vert P_kf\Vert_{L^2}\lesssim\min\{2^{(1-20\delta)k},2^{-N_0k}\}\Vert f\Vert_{Z_1^\sigma\cap H^{N_0}}.
\end{equation}
We prove now several linear estimates for functions in the $Z$ space.

\begin{lemma}\label{LinEstLem}
Assume that $\sigma\in\{e,b\}$ and 
\begin{equation}\label{Zs}
\Vert f\Vert_{Z_1^\sigma\cap H^{N_1/8}_\Omega}\leq 1.
\end{equation}
For any $(k,j)\in\mathcal{J}$ and $n\in\{0,\ldots,j+1\}$ let (recall the notation \eqref{Alx80})
\begin{equation}\label{Alx100}
f_{j,k}:=P_{[k-2,k+2]}Q_{jk}f,\qquad \widehat{f_{j,k,n}}(\xi):=\varphi_{-n}^{[-j-1,0]}(\Psi^\dagger_\sigma(\xi))\widehat{f_{j,k}}(\xi).
\end{equation}
For any $\xi_0\in\mathbb{R}^2\setminus\{0\}$ and $\kappa,\rho\in[0,\infty)$ let $\mathcal{R}(\xi_0;\kappa,\rho)$ denote the rectangle
\begin{equation}\label{Alx100.1}
\mathcal{R}(\xi_0;\kappa,\rho):=\{\xi\in\mathbb{R}^2:\big|(\xi-\xi_0)\cdot \xi_0/|\xi_0|\big|\leq\rho,\,\big|(\xi-\xi_0)\cdot \xi_0^\perp/|\xi_0|\big|\leq\kappa\}.
\end{equation}
Then, for any $(k,j)\in\mathcal{J}$ and $n\in\{0,\ldots,j+1\}$,
\begin{equation}\label{RadL2}
\big\Vert \sup_{\theta\in\mathbb{S}^1}|\widehat{f_{j,k,n}}(r\theta)|\,\big\Vert_{L^2(rdr)}+\big\Vert \sup_{\theta\in\mathbb{S}^1}|f_{j,k,n}(r\theta)|\,\big\Vert_{L^2(rdr)}\lesssim 2^{-5k_+}2^{(1/2-19\delta)n-(1-20\delta)j+2\delta^2 j},
\end{equation}
\begin{equation}\label{FL1bd}
\sup_{\kappa+\rho\leq 2^{k-10}}\int_{\mathbb{R}^2}|\widehat{f_{j,k,n}}(\xi)|\mathbf{1}_{\mathcal{R}(\xi_0;\kappa,\rho)}(\xi)\,d\xi\lesssim 2^{-5k_+}2^{-j+21\delta j}2^{-19\delta n}\kappa 2^{-k/2}\min(1,2^n\rho)^{1/2},
\end{equation}
\begin{equation}\label{FLinftybd}
\|\widehat{f_{j,k,n}}\|_{L^\infty}\lesssim 
\begin{cases}
2^{2\delta n}2^{-(1/2-21\delta)(j-n)}\,\,&\text{ if }\,\,2^k\approx 1,\\
2^{-5k_+}2^{-21\delta k}2^{-(1/2-21\delta)(j+k)}\,\,&\text{ if }\,\,2^k\gg 1\text{ or }2^k\ll 1,
\end{cases}
\end{equation}
\begin{equation}\label{FLinftybdDER}
\|D^\alpha\widehat{f_{j,k,n}}\|_{L^\infty}\lesssim_{|\alpha|}
\begin{cases}
2^{|\alpha|j}2^{2\delta n}2^{-(1/2-21\delta)(j-n)}\,\,&\text{ if }\,\,2^k\approx 1,\\
2^{|\alpha|j}2^{-5k_+}2^{-21\delta k}2^{-(1/2-21\delta)(j+k)}\,\,&\text{ if }\,\,2^k\gg 1\text{ or }2^k\ll 1.
\end{cases}
\end{equation}
Moreover, if $m\geq 0$ and $|t|\in[2^{m}-1,2^{m+1}]$ then
\begin{equation}\label{LinftyBd}
2^{\delta^2n}\left\| e^{-it\Lambda_\sigma}f_{j,k,n}\right\|_{L^\infty}\lesssim 
\begin{cases}
2^{-m+20\delta j}&\hbox{ if }2^k\ll 1,\\
2^{-j+20\delta j}&\hbox{ for all }j,k,m,\\
2^{-m+2\delta m}2^{-3/4k_-}&\hbox{ if }j\le (1-\delta^2)m+k_-.
\end{cases}\\
\end{equation}
In particular
\begin{equation}\label{LinftyBd2}
\Vert e^{-it\Lambda_\sigma}P_kf\Vert_{L^\infty}\lesssim (1+\vert t\vert)^{-1+21\delta}.
\end{equation}
\end{lemma}

\begin{proof} Recall that $N_1\geq 8/\delta^2$. The hypothesis gives 
\begin{equation}\label{Alx101}
\Vert f_{j,k,n}\Vert_{L^2}\lesssim 2^{-6 k_+}2^{(1/2-19\delta)n-(1-20\delta)j},\qquad \big\Vert \Omega^{N_1/8}f_{j,k,n}\big\Vert_{L^2}\lesssim \Vert \Omega^{N_1/8}f\Vert_{L^2}\lesssim 1.
\end{equation} 
The first inequality in \eqref{RadL2} follows by interpolation, using also Sobolev embedding,
\begin{equation*}
\begin{split}
\big\Vert \sup_{\theta\in\mathbb{S}^1}|\widehat{f_{j,k,n}}(r\theta)|\,\big\Vert_{L^2(rdr)}
&\lesssim \Vert \widehat{f_{j,k,n}}\Vert_{L^2}+\Vert \Omega\widehat{f_{j,k,n}}\Vert_{L^2}.
\end{split}
\end{equation*}
The second inequality follows similarly. 

Inequality \eqref{FL1bd} follows from \eqref{RadL2}. Indeed, the left-hand side is dominated by
\begin{equation*}
C(\kappa 2^{-k})\sup_{\theta\in\mathbb{S}^1}\int_{\mathbb{R}}|\widehat{f_{j,k,n}}(r\theta)|\mathbf{1}_{\mathcal{R}(\xi_0;\kappa,\rho)}(r\theta)\,rdr\lesssim \sup_{\theta\in\mathbb{S}^1}\big\Vert \widehat{f_{j,k,n}}(r\theta)\,\big\Vert_{L^2(rdr)}(\kappa 2^{-k})[2^k\min(\rho,2^{-n})]^{1/2},
\end{equation*}
which gives the desired result.

We now consider \eqref{FLinftybd}. For $\theta\in\mathbb{S}^1$ fixed we estimate
\begin{equation*}
\begin{split}
\|\widehat{f_{j,k,n}}(r,\theta)\|_{L^\infty_r}&\lesssim 2^{j/2}\|\widehat{f_{j,k,n}}(r,\theta)\|_{L^2_r}+2^{-j/2}\|(\partial_r\widehat{f_{j,k,n}})(r,\theta)\|_{L^2_r}\\
&\lesssim 2^{j/2}2^{-k/2}\|\widehat{f_{j,k,n}}(r,\theta)\|_{L^2(rdr)},
\end{split}
\end{equation*}
using the support property of $Q_{jk}f$ in the physical space. The desired bound follows using \eqref{RadL2} and the observation that $f_{j,k,n}=0$ if $n\geq 1$ and $|k|\geq\D$. The bounds in \eqref{FLinftybdDER} follow as well, since differentiation in $\xi$ corresponds to multiplication by $2^j$ factors, due to space localization.

We now turn to \eqref{LinftyBd}. We estimate first, using \eqref{FL1bd},
\begin{equation}\label{LinEstLemA}
\Vert e^{-it\Lambda_\sigma}f_{j,k,n}\Vert_{L^\infty}\lesssim \Vert \widehat{f_{j,k,n}}\Vert_{L^1}\lesssim \min\{2^k,2^{-3k}\}2^{-(1-20\delta)j}2^{-19\delta n}.
\end{equation}
Similarly, if $2^k\ll 1$ then the usual dispersion estimate gives
\begin{equation}\label{LinEstLemA1}
\Vert e^{-it\Lambda_\sigma}f_{j,k,n}\Vert_{L^\infty}\lesssim 2^{-m}\Vert f_{j,k,n}\Vert_{L^1}\lesssim 2^{20\delta j-m}.
\end{equation}
The bounds in the first two lines of \eqref{LinftyBd} follow.

Assume now that $j\le (1-\delta^2)m+k_-$, $1\ll 2^m$, $k\in[-25\delta j-\D,m/4]$. Using Lemma \ref{tech5} (i)  it is easy to see that
\begin{equation*}
\big|\big(e^{-it\Lambda_\sigma}f_{j,k,n}\big)(x)\big|\lesssim 2^{-2m} \qquad \text{ unless }|x|\approx 2^{m+k_-}.
\end{equation*}
Also, letting $f'_{j,k,n}:=\mathcal{R}_{\leq m/10}f_{j,k,n}$, see \eqref{RadOp}, we have $\|f_{j,k,n}-f'_{j,k,n}\|_{L^2}\lesssim 2^{-4m}$ therefore
\begin{equation*}
\big\|e^{-it\Lambda_\sigma}(f_{j,k,n}-f'_{j,k,n})\big\|_{L^\infty}\lesssim 2^{-2m}.
\end{equation*} 
On the other hand, if $|x|\approx 2^{m+k_-}$ then, using again Lemma \ref{tech5} and \eqref{FLinftybdDER},
\begin{equation}\label{Alx90.4}
\begin{split}
\big(e^{-it\Lambda_\sigma}f'_{j,k,n}\big)(x)&=C\int_{\mathbb{R}^2}e^{i\Psi}\varphi(\kappa_r^{-1}\nabla_\xi\Psi)\varphi(\kappa_\theta^{-1}\Omega_\xi\Psi)\widehat{f'_{j,k,n}}(\xi)d\xi+O(2^{-2m}),\\
\Psi&:=-t\Lambda_\sigma(\xi)+x\cdot\xi,\quad\kappa_r:=2^{\delta^2m}\big(2^{m/2}+2^{j}\big),\quad\kappa_\theta:=2^{\delta^2m}2^{(m+k+k_-)/2}.
\end{split}
\end{equation}
If $j\le m/2$, the second cut-off is not needed and a simple estimate using \eqref{FLinftybd} gives
\begin{equation}\label{LinEstLemB}
\big\vert  \big(e^{-it\Lambda_\sigma}f'_{j,k,n}\big)(x)\big\vert\lesssim 2^{-2m}+ \kappa_r^22^{-2m}2^{4k_+}\cdot\Vert \widehat{f'_{j,k,n}}\Vert_{L^\infty}\lesssim 2^{-m+2\delta m-\delta^2m}.
\end{equation}
On the other hand, if $m/2\le j\le (1-\delta^2)m$, then
\begin{equation}\label{LinEstLemC}
\begin{split}
\left\vert  \left(e^{-it\Lambda_\sigma}Q_{jk}f\right)(x)\right\vert&\lesssim 2^{-2m}+ (2^{-m-k_--k}\kappa_\theta)\cdot\sup_\theta\big\Vert [\varphi(\kappa_r^{-1}\nabla_\xi\Psi)\widehat{f'_{j,k,n}}](r\theta)\big\Vert_{L^1(rdr)}\\
&\lesssim
2^{-2m}+(2^{-m-k_--k}\kappa_\theta)(2^{-m}\kappa_r2^{4k_+}2^{k-})^{1/2}\cdot\sup_\theta\Vert \widehat{f'_{j,k,n}}(r\theta)\Vert_{L^2(rdr)}\\
&\lesssim 2^{2\delta m-m-\delta^2m}2^{-3k_-/4}.
\end{split}
\end{equation}
Estimate \eqref{LinftyBd} follows from \eqref{LinEstLemA}, \eqref{LinEstLemA1}, \eqref{LinEstLemB} and \eqref{LinEstLemC}. 

Estimate \eqref{LinftyBd2} follows from \eqref{LinftyBd}.
\end{proof}

\section{Energy estimates, I: quasilinear energies}\label{ENERGY0}

In the next two sections we prove our main energy estimate:

\begin{proposition}\label{BootstrapEE}
Assume that $(\rho,u,\widetilde{E},\widetilde{b})$ is a solution to \eqref{system2}-\eqref{condition} on a time interval $[0,T]$, $T\ge 1$, satisfying the hypothesis \eqref{bootstrap1}-\eqref{bootstrap2} in Proposition \ref{bootstrap}. Then, for any $t\in[0,T]$,
\begin{equation}\label{bootstrapEE}
\Vert (\rho(t),u(t),\widetilde{E}(t),\widetilde{b}(t))\Vert_{H^N\cap H_\Omega^{N_1}}\lesssim\epsilon_0+\epsilon_1^{3/2}.
\end{equation}
\end{proposition}

We prove Proposition \ref{BootstrapEE} in several steps. We start by switching to new unknowns which respect better the quasilinear nature of our system, Proposition \ref{ChangeUnknownsProp} recasts the main system \eqref{system2}-\eqref{condition} as a quasilinear system and allows us to obtain suitable bounds on the energy increment (see Proposition \ref{IncrementEnergyProp}). We then bound these increments in Lemmas \ref{BootstrapEE2} and \ref{BootstrapEE1}. 

We often use the bounds in Proposition \ref{BootSimple} below. These bounds follow from \eqref{bootstrap2}, Lemma \ref{LinEstLem}, Lemma \ref{dtfLemPrelim}, and Corollary \ref{AlxCoro}.

\begin{proposition}\label{BootSimple}
Let $(U_e,U_b), (V_e,V_b)$ be defined as in \eqref{variables4}, and assume that the hypothesis of Proposition \ref{bootstrap} holds. Then, for $\sigma\in\{e,b\}$, $k\in\mathbb{Z}$, and $t\in[0,T]$,
\begin{equation}\label{BootSimple1}
\|U_\sigma(t)\|_{H^{N_{0}}\cap H^{N_1}_\Omega}\lesssim\eps_1,
\end{equation}
\begin{equation}\label{BootSimple2}
\sup_{|\mu|\leq (1+t)2^{-\D}}\sum_{a\leq N_1/2}\|e^{-i\mu\Lambda_\sigma}P_k\Omega^aU_\sigma(t)\|_{L^{\infty}}\lesssim\eps_1(1+t)^{-1+21\delta},
\end{equation}
\begin{equation}\label{Alx2}
\big\|P_{\leq k}(\partial_t+i\Lambda_\sigma)U_{\sigma}(t)\big\|_{H^{N_0}}+\sum_{a\leq N_1}\big\|P_{\leq k}\Omega^a(\partial_t+i\Lambda_\sigma)U_{\sigma}(t)\big\|_{L^2}\lesssim \eps_1^22^{k_+}(1+t)^{-1+22\delta}.
\end{equation}
Moreover, for $0\le a\le N_1/2$ we can decompose $\Omega^a(\partial_t+i\Lambda_\sigma)U_{\sigma}(t)=G_2(t)+G_{\infty}(t)$ such that
\begin{equation}\label{Alx3}
\begin{split}
\sup_{|\mu|\leq (1+t)^{1-\delta/4}}\|e^{-i\mu\Lambda_\sigma}P_kG_{\infty}(t)\|_{L^{\infty}}&\lesssim\eps_1^2(1+t)^{-2+50\delta},\\
\|P_kG_{2}(t)\|_{L^{2}}&\lesssim\eps_1^2(1+t)^{-3/2+60\delta}.
\end{split}
\end{equation}
\end{proposition}

\subsection{Quasilinear variables} We now introduce the Hodge decomposition
\begin{equation*}
\begin{split}
&P=\vert\nabla\vert^{-1}\hbox{div},\qquad Q=\vert\nabla\vert^{-1}\hbox{curl},\qquad R_j=\partial_j\vert\nabla\vert^{-1},\qquad Id=-R_jP+\in_{jk}R_kQ,\\
&A=Pu,\qquad C=P\widetilde{E},\qquad B=Qu,\qquad D=Q\widetilde{E},
\end{split}
\end{equation*}
where we use the standard convention that repeated indices are summed. We define the nonlinear dispersion relation
\begin{equation}\label{DispersionRelations}
\Sigma_e:=\sqrt{(1+\rho)(1+h^\prime(\rho)\vert\zeta\vert^2)},
\end{equation}
and the new unknowns
\begin{equation}\label{NewUnknowns}
\mathcal{U}_e:=T_{\sqrt{1+\rho}}A-iT_{\Sigma_e}T_{1/\sqrt{1+\rho}}C,\qquad \mathcal{U}_b:=\Lambda_bB-iD.
\end{equation}
Recall also \eqref{variables4}
\begin{equation*}
U_e=A-i\Lambda_eC,\qquad U_b=\Lambda_bB-iD,\qquad \Lambda_e=\sqrt{1+d|\nabla|^2},\qquad \Lambda_b=\sqrt{1+|\nabla|^2}.
\end{equation*}
In particular $\mathcal{U}_b=U_b$ and the main unknowns $(A,B,C,D)$ can be recovered by
\begin{equation}\label{ABCD}
\begin{split}
&2A=\big[U_e+\overline{U_e}\big],\qquad 2C=i\Lambda_e^{-1}\big[U_e-\overline{U_e}\big],\\
&2B=\Lambda_b^{-1}\big[U_b+\overline{U_b}\big],\qquad 2D=i\big[U_b-\overline{U_b}\big],\\
&\rho=-\vert\nabla\vert C,\qquad u_j=-R_jA+\in_{jk}R_kB.
\end{split}
\end{equation}

With the notation in subsection \ref{ParaDiffCalc}, see in particular \eqref{nor1} and Lemmas \ref{PropProd} and \ref{PropSym}, and using also the bounds \eqref{BootSimple1}--\eqref{BootSimple2} on $U_e$, we have
\begin{equation}\label{nor5}
\begin{split}
&\sup_{a\leq N_1}\big\|\Omega^a_{x,\zeta}\big[\Sigma_e(x,\zeta)-\Lambda_e(\zeta)\big]\big\|_{\mathcal{L}_1^2}\lesssim \eps_1,\\
&\sup_{a\leq N_1/2}\big\|\Omega^a_{x,\zeta}\big[\Sigma_e(x,\zeta)-\Lambda_e(\zeta)\big]\big\|_{\mathcal{L}_1^\infty}\lesssim \eps_1(1+t)^{-99/100}.
\end{split}
\end{equation}

We derive now the main system for the quasilinear variables $\mathcal{U}_e$ and $\mathcal{U}_b$. The main point of this derivation is to write the quadratic components of the nonlinearities as sums of (1) {\it{strongly semilinear}} terms, which gain one derivative relative to the main variables, and (2) {\it{nonresonant}} terms, for which the corresponding phases do not vanish so we can later estimate them using normal forms. More precisely:

\begin{proposition}\label{ChangeUnknownsProp}
With the definitions above and the assumptions of Proposition \ref{BootstrapEE} we have
\begin{equation}\label{EquivalenceNorms}
\big\|\mathcal{U}_e-U_e\big\|_{H^{N_0}\cap H^{N_1}_\Omega}\lesssim \eps_1^2(1+t)^{-99/100},\qquad \sup_{a\leq N_1/2}\big\|\Omega^a[\mathcal{U}_e-U_e]\big\|_{W^{100,\infty}}\lesssim \eps_1^2(1+t)^{-99/50}.
\end{equation}
The variables $\mathcal{U}_e$ and $\mathcal{U}_b$ satisfy the equations
\begin{equation}\label{ParaUnknowns}
\begin{split}
\left(\partial_t+iT_{\Sigma_e}+iT_{u\cdot\zeta}\right)\mathcal{U}_e&=\mathcal{SS}_e+\mathcal{N}_e+\mathcal{C}_e,\\
\left(\partial_t+i\Lambda_b\right)\mathcal{U}_b&=\mathcal{SS}_b+\mathcal{N}_b+\mathcal{C}_b,
\end{split}
\end{equation}
where 

$\bullet\,\,$ $\mathcal{N}_e$, $\mathcal{N}_b$ are quadratic semilinear nonresonant terms, i.e. linear combinations of operators of the form
\begin{equation}\label{NRMult1}
\mathcal{N}_\sigma=\sum_{\mu,\nu}N_{\sigma\mu\nu}[U_\mu,U_\nu],\qquad \mathcal{F}\big\{N_{\sigma\mu\nu}[f,g]\big\}(\xi)=\frac{1}{4\pi^2}\int_{\mathbb{R}^2}n_{\sigma\mu\nu}(\xi,\eta)\widehat{f}(\xi-\eta)\widehat{g}(\eta)\,d\eta,
\end{equation}
where the symbols $n_{\sigma\mu\nu}$ of the bilinear operators $N_{\sigma\mu\nu}$ satisfy
\begin{equation}\label{NRMult0}
n_{\sigma\mu\nu}\equiv 0\qquad\text{ unless }\qquad\sigma=e\text{ or }\nu\in\{e,-e\}
\end{equation}
and, with $n^a_{\sigma\mu\nu}:=(\Omega_\xi+\Omega_\eta)^an_{\sigma\mu\nu}$ and recalling \eqref{ener66},
\begin{equation}\label{NRMult}
\sup_{a\le N_1}\big[\Vert n^a_{\sigma\mu\nu}\Vert_{S^\infty_{kk_1k_2}}+2^{k_2}\Vert (\partial_\xi+\partial_\eta) n^a_{\sigma\mu\nu}\Vert_{S^\infty_{kk_1k_2}}\big]\lesssim (1+2^{3k_1})\mathbf{1}_{\{k_2\ge \max(k_1,0)+\D/2\}};
\end{equation}

$\bullet\,\,$ $\mathcal{C}_e$, $\mathcal{C}_b$ are higher order semilinear terms satisfying
\begin{equation}\label{CubBoun}
\Vert \mathcal{C}_e\Vert_{H^{N_0}\cap H^{N_1}_\Omega}+\Vert \mathcal{C}_b\Vert_{H^{N_0}\cap H^{N_1}_\Omega}\lesssim \eps_1^3(1+t)^{-9/5};
\end{equation}

$\bullet\,\,$ $\mathcal{SS}_e$, $\mathcal{SS}_b$ are quadratic strongly semilinear terms, i.e. linear combinations of operators smoothing of order $1$, namely
\begin{equation}\label{SSMult}
\mathcal{SS}_\sigma=\sum_{\mu,\nu}SS_{\sigma\mu\nu}[U_{\mu},U_{\nu}],\qquad \mathcal{F}\big\{SS_{\sigma\mu\nu}[f,g]\big\}(\xi)=\frac{1}{4\pi^2}\int_{\mathbb{R}^2}ss_{\sigma\mu\nu}(\xi,\eta)\widehat{f}(\xi-\eta)\widehat{g}(\eta)\,d\eta,
\end{equation}
where, with $ss^a_{\sigma\mu\nu}:=(\Omega_\xi+\Omega_\eta)^ass_{\sigma\mu\nu}$,
\begin{equation}\label{SSMult2}
\begin{split}
&\sup_{a\le N_1}\Vert ss^a_{\sigma\mu\nu}\Vert_{S^\infty_{kk_1k_2}}\lesssim 2^{-\max(k,k_1,k_2,0)}(1+2^{4\min(k,k_1,k_2)}),
\end{split}
\end{equation}
\end{proposition}

\begin{proof} We remark first that \eqref{NRMult1}--\eqref{NRMult} guarantee that the modulation $|\Phi_{\sigma\mu\nu}(\xi,\eta)|$ is bounded from below in the support of the multiplier $n_{\sigma\mu\nu}$, due to Lemma \ref{PhaseClass}.

The bounds in \eqref{EquivalenceNorms} follows easily from Lemma \ref{PropSym} and \eqref{BootSimple1}--\eqref{BootSimple2}. To prove \eqref{ParaUnknowns}, using Lemma \ref{PropHHSym} and \eqref{Exph} we observe that
\begin{equation*}
h(\rho)=T_{h^\prime(\rho)}\rho+\kappa \mathcal{H}(\rho,\rho)/2+E^3(\rho),\qquad \Vert (1-\Delta)E^3(\rho)\Vert_{H^{N_0}\cap H^{N_1}_\Omega}\lesssim \eps_1^3(1+t)^{-9/5}.
\end{equation*}
We apply $P$ to the first equation in \eqref{system3prime}, so
\begin{equation*}
\begin{split}
\partial_tA-\vert\nabla\vert\big[-T_{h^\prime(\rho)}\vert\nabla\vert C&+\kappa\mathcal{H}(\rho,\rho)/2+E^3(\rho)\big]+C=\vert\nabla\vert\big[T_{u_j}u_j+\mathcal{H}(u_j,u_j)/2\big]\\
&=-\vert\nabla\vert T_{u_j}R_jA+\in_{jk}\vert\nabla\vert T_{u_j}R_kB+\vert\nabla\vert\mathcal{H}(u_j,u_j)/2.
\end{split}
\end{equation*}
Therefore
\begin{equation*}
\begin{split}
\partial_tA+&T_{1+h^\prime(\rho)\vert\zeta\vert^2}C+iT_{u\cdot\zeta}A=\mathcal{SS}^1_e+\mathcal{N}^1_e+\mathcal{C}^1_e,\\
\mathcal{SS}^1_e&:=\vert\nabla\vert\mathcal{H}(u_j,u_j)/2+\kappa \vert\nabla\vert \mathcal{H}(\rho,\rho)/2,\\
\mathcal{N}^1_e&:=\kappa(T_{\rho\vert\zeta\vert^2}-\vert\nabla\vert T_{\rho}\vert\nabla\vert)C-[\vert\nabla\vert, T_{u_j}]R_jA+\big(iT_{u\cdot\zeta}A-|\nabla|T_{u_j}R_jA\big)+\in_{jk}\vert\nabla\vert T_{u_j}R_kB,\\
\mathcal{C}^1_e&:=(T_{h_2(\rho)\vert\zeta\vert^2}-\vert\nabla\vert T_{h_2(\rho)}\vert\nabla\vert)C+\vert\nabla\vert E^3(\rho).
\end{split}
\end{equation*}

Using the second equation in \eqref{system3prime} we obtain
\begin{equation*}
\partial_tC=P\left[(1+\rho)u\right]=T_{1+\rho}A+[R_j,T_{\rho}]u_j-R_jT_{u_j}\vert\nabla\vert C+R_j\mathcal{H}(\rho,u_j),
\end{equation*}
which we rewrite as
\begin{equation*}
\begin{split}
\partial_tC-T_{1+\rho}A+iT_{u\cdot\zeta}C&=\mathcal{SS}^2_e+\mathcal{N}^2_e\\
\mathcal{SS}^2_e&:=R_j\mathcal{H}(\rho,u_j),\\
\mathcal{N}^2_e&:=\big[iT_{u\cdot\zeta}C-R_jT_{u_j}\vert\nabla\vert C\big]+[R_j,T_{\rho}]u_j.
\end{split}
\end{equation*}
Therefore we obtain the system
\begin{equation*}
\begin{split}
\left(\partial_t+iT_{u\cdot\zeta}\right)A+T_{1+h^\prime(\rho)\vert\zeta\vert^2}C&=\mathcal{SS}^1_e+\mathcal{N}^1_e+\mathcal{C}^1_e,\\
\left(\partial_t+iT_{u\cdot\zeta}\right)C-T_{1+\rho}A&=\mathcal{SS}^2_e+\mathcal{N}^2_e.
\end{split}
\end{equation*}
Using \eqref{NewUnknowns}, this leads to
\begin{equation}\label{zaw1}
\begin{split}
\left(\partial_t+iT_{\Sigma_e}+iT_{u\cdot\zeta}\right)\mathcal{U}_e&=i[T_{u\cdot\zeta},T_{\sqrt{1+\rho}}]A + i(T_{\Sigma_e}T_{\sqrt{1+\rho}}-T_{\Sigma_e}T_{1/\sqrt{1+\rho}}T_{1+\rho})A\\
&+T_{\partial_t\sqrt{1+\rho}}A+T_{\sqrt{1+\rho}}(\mathcal{SS}^1_e+\mathcal{N}^1_e+\mathcal{C}^1_e)\\
&+[T_{u\cdot\zeta},T_{\Sigma_e}T_{1/\sqrt{1+\rho}}]C+(T_{\Sigma_e}T_{\Sigma_e}T_{1/\sqrt{1+\rho}}-T_{\sqrt{1+\rho}}T_{1+h^\prime(\rho)\vert\zeta\vert^2})C\\
&-i[\partial_t,T_{\Sigma_e}T_{1/\sqrt{1+\rho}}]C-iT_{\Sigma_e}T_{1/\sqrt{1+\rho}}(\mathcal{SS}_e^2+\mathcal{N}^2_e).
\end{split}
\end{equation}

To see that the right-hand side of \eqref{zaw1} has the required form we expand
\begin{equation}\label{ApproxSigma}
\begin{split}
\Sigma_e&=\Lambda_e\Big[1+\rho(1+\frac{\kappa\vert\zeta\vert^2}{1+d\vert\zeta\vert^2})+(h_2(\rho)(1+\rho)+\kappa\rho^2)\frac{\vert\zeta\vert^2}{1+d\vert\zeta\vert^2}\Big]^{1/2}=\Lambda_e+\Sigma_e^{\ge 1},\\
\Sigma_e^{\ge1}&:=\Sigma_e^1+\Sigma_e^{\ge 2},\qquad \Sigma_e^1=\rho\Lambda_e^1=\frac{\rho}{2}\Big[\sqrt{1+d\vert\zeta\vert^2}+\frac{\kappa\vert\zeta\vert^2}{\sqrt{1+d\vert\zeta\vert^2}}\Big].
\end{split}
\end{equation}
where $\Sigma^1_e$ is a symbol of order $1$ depending linearly on $\rho$, and $\Sigma^{\ge 2}_e$ is a symbol of order $1$ depending at least quadratically on $\rho$. 

Recall that $\partial_t\rho=-|\nabla|A-|\nabla|P(\rho u)$. With $E(a,b)=T_aT_b-T_{ab}$ as in Lemma \ref{PropProd}, we define
\begin{equation*}
\begin{split}
\mathcal{Q}_e&:=T_{-|\nabla|A/2}A+\mathcal{N}^1_e+[T_{u\cdot\zeta},\Lambda_e]C\\
&+d[T_{\rho\vert\zeta\vert^2}-\frac{1}{2}(\vert\nabla\vert^2T_\rho+T_\rho\vert\nabla\vert^2)]C+[E(\Sigma_e^1,\Lambda_e)+E(\Lambda_e,\Sigma_e^1)]C\\
&-\frac{i}{2}(T_{\Lambda_e(\zeta) (-|\nabla|A)}+\kappa T_{\Lambda_e(\zeta)^{-1}\vert\zeta\vert^2(-|\nabla|A)}-\Lambda_eT_{-|\nabla|A})C-i\Lambda_e\mathcal{N}_e^2,\\
\mathcal{SS}_e&:=\mathcal{SS}^1_e-i\Lambda_e\mathcal{SS}_e^2+P_{\le \mathcal{D}}\mathcal{Q}_e,\\
\mathcal{N}_e&:=P_{\geq \mathcal{D}+1}\mathcal{Q}_e,
\end{split}
\end{equation*}
and
\begin{equation*}
\begin{split}
\mathcal{C}_e&:=i[T_{u\cdot\zeta},T_{\sqrt{1+\rho}}]A + i(T_{\Sigma_e}T_{\sqrt{1+\rho}}-T_{\Sigma_e}T_{1/\sqrt{1+\rho}}T_{1+\rho})A,\\
&+T_{\partial_t\sqrt{1+\rho}+|\nabla|A/2}A+T_{\sqrt{1+\rho}-1}(\mathcal{SS}^1_e+\mathcal{N}^1_e)+T_{\sqrt{1+\rho}}\mathcal{C}^1_e\\
&+[T_{u\cdot\zeta},T_{\Sigma_e^{\ge1}}T_{1/\sqrt{1+\rho}}+T_{\Lambda_e}T_{1/\sqrt{1+\rho}-1}]C\\
&+[T_{\Sigma_e^2}T_{1/\sqrt{1+\rho}}-T_{\sqrt{1+\rho}}T_{1+h^\prime(\rho)\vert\zeta\vert^2}-d(T_{\rho\vert\zeta\vert^2}-\frac{1}{2}(\vert\nabla\vert^2T_\rho+T_\rho\vert\nabla\vert^2))]C\\
&+[E(\Sigma_e^{\ge 1},\Sigma_e^{\ge 1})+E(\Lambda_e,\Sigma_e^{\ge 2})+E(\Sigma^{\ge 2}_e,\Lambda_e)+E(\Sigma_e,\Sigma_e)T_{1/\sqrt{1+\rho}-1}]C\\
&-i\Big([\partial_t,T_{\Sigma_e}T_{1/\sqrt{1+\rho}}]-\frac{1}{2}(T_{\Lambda_e(\zeta) (-|\nabla|A)}+\kappa T_{\Lambda_e(\zeta)^{-1}\vert\zeta\vert^2(-|\nabla|A)}-\Lambda_eT_{-|\nabla|A})\Big)C\\
&-iT_{\Sigma_e^{\geq 1}}(\mathcal{SS}^2_e+\mathcal{N}^2_e)-iT_{\Sigma_e}T_{1/\sqrt{1+\rho}-1}(\mathcal{SS}^2_e+\mathcal{N}^2_e).
\end{split}
\end{equation*}

It is easy to see that $\mathcal{Q}_e$ is a quadratic term in $\rho,A,C,u_j,B$ that does not lose derivatives; in view of \eqref{ABCD} this can be written as a sum of bilinear operators in terms of $U_e,\overline{U_e},U_b,\overline{U_b}$. Moreover, all the interactions are between high and low frequencies, due to presence of the paradifferential operators. These interactions are automatically nonresonant when $\sigma=e$. The desired conclusions on the quadratic terms $\mathcal{SS}_e$ and $\mathcal{N}_e$ follow.

The term $\mathcal{C}_e$ contains cubic and higher order terms that do not lose derivatives. The desired bound in \eqref{CubBoun} follows using Lemmas \ref{PropProd}--\ref{PropHHSym} and the bounds \eqref{BootSimple1}--\eqref{BootSimple2}.

We now turn to magnetic variables. Applying $Q$ to \eqref{system3prime}, we obtain
\begin{equation*}
\begin{split}
&\partial_tB+D=0,\\
&\partial_tD+\Delta B-B=Q(\rho u)=QT_{\rho}u+Q\mathcal{H}(\rho,u)+QT_u\rho.
\end{split}
\end{equation*}
Therefore, recalling that $\mathcal{U}_b=\Lambda_bB-iD$ and $u_k=-R_kA+\in_{kn}R_nB$,
\begin{equation*}
\begin{split}
(\partial_t+i\Lambda_b)\mathcal{U}_b&=(-i)[QT_{\rho}u+Q\mathcal{H}(\rho,u)+QT_u\rho]=\mathcal{SS}_b+\mathcal{N}_b,\\
\mathcal{SS}_b:&=iP_{\leq \D}(\in_{jk}R_jT_{\rho}R_kA-\in_{jk}R_jT_{u_k}\rho)+iR_jT_{\rho}R_jB-i\in_{jk}R_j\mathcal{H}(\rho,u_k),\\
\mathcal{N}_b:&=iP_{\geq \D+1}(\in_{jk}R_jT_{\rho}R_kA-\in_{jk}R_jT_{u_k}\rho).
\end{split}
\end{equation*}
We notice that $B=\Lambda_b^{-1}[U_b+\overline{U}_b]$ gains a derivative, and the terms $R_jT_{\rho}R_kA$ and $R_jT_{u_k}\rho$ are nonresonant when restricted to high frequencies and $\sigma=b$. The desired conclusion follows.
\end{proof}

As a consequence, we can obtain our first high order energy estimates. Let $\langle\nabla\rangle:=(1-\Delta)^{1/2}$, $O^0:=\langle\nabla\rangle^{N_0}$, and $O^p:=\Omega^p$ for $p\in\{1,\ldots,N_1\}$. For $\sigma\in\{e,b\}$ and $p\in\{0,\ldots,N_1\}$ let
\begin{equation}\label{nor20}
W_\sigma^p:=O^p\mathcal{U}_{\sigma}.
\end{equation}

\begin{proposition}\label{IncrementEnergyProp} We define the energy functional
\begin{equation}\label{nor21}
\mathcal{E}(t):=(1/2)\sum_{\sigma\in\{e,b\},\,p\in\{0,\ldots,N_1\}}\|W_\sigma^p(t)\|_{L^2}^2.
\end{equation}
Under the hypothesis of Proposition \ref{BootstrapEE}, we have
\begin{equation}\label{EquivNorm2}
\Vert (\rho,u,\widetilde{E},\widetilde{b})(t)\Vert_{H^{N_0}\cap H_\Omega^{N_1}}^2\lesssim \mathcal{E}(t)+\varepsilon_1^3,\qquad \mathcal{E}(t)\lesssim \Vert (\rho,u,\widetilde{E},\widetilde{b})(t)\Vert_{H^{N_0}\cap H_\Omega^{N_1}}^2+\varepsilon_1^3,
\end{equation}
and
\begin{equation}\label{Increments}
\begin{split}
\partial_t\mathcal{E}=\mathcal{B}_{SS}+\mathcal{B}_{NR}+\mathcal{B}_{hot},\qquad\vert\mathcal{B}_{hot}(t)\vert\lesssim \varepsilon_1^3(1+t)^{-3/2}.
\end{split}
\end{equation}
Here the strongly semilinear increment is 
\begin{equation}\label{BSS}
\mathcal{B}_{SS}:=\sum_{\sigma\in\{e,b\}}\sum_{p\in\{0,\ldots,N_1\}}\Re\langle O^p\mathcal{SS}_\sigma,W_{\sigma}^p\rangle+\sum_{p\in\{0,\ldots,N_1\}}\Re\langle i[T_{\Sigma_e^1+u\cdot\zeta},O^p]U_e,P_{\leq\D}W_{e}^p\rangle,
\end{equation}
and the nonresonant increment is 
\begin{equation}\label{BNR}
\mathcal{B}_{NR}:=\sum_{\sigma\in\{e,b\}}\sum_{p\in\{0,\ldots,N_1\}}\Re\langle O^p\mathcal{N}_\sigma,W_{\sigma}^p\rangle+\sum_{p\in\{0,\ldots,N_1\}}\Re\langle i[T_{\Sigma_e^1+u\cdot\zeta},O^p]U_e,P_{\geq \D+1}W_{e}^p\rangle.
\end{equation}
Moreover, for $p\in\{0,\ldots,N_1\}$,
\begin{equation}\label{EqW}
\begin{split}
&\left(\partial_t+iT_{\Sigma_e}+iT_{u\cdot\zeta}\right)W_e^p=Q_e^p,\qquad \left(\partial_t+i\Lambda_b\right)W_b^p=Q_b^p,\\
&\Vert W_e^p(t)\Vert_{L^2}+\Vert W_b^p(t)\Vert_{L^2}\lesssim \eps_1,\qquad \Vert Q_e^p(t)\Vert_{L^2}+\Vert Q_b^p(t)\Vert_{L^2}\lesssim \eps_1^2(1+t)^{-9/10}.
\end{split}
\end{equation}
\end{proposition}

\begin{proof} In view of \eqref{ABCD}, $\Vert (\rho,u,\widetilde{E},\widetilde{b})\Vert_{H^{N_0}\cap H_\Omega^{N_1}}\approx \|(U_e,U_b)\|_{H^{N_0}\cap H_\Omega^{N_1}}$, and \eqref{EquivNorm2} follows using also \eqref{EquivalenceNorms}. For the remaining claims, we start from the equations \eqref{ParaUnknowns} and write
\begin{equation}\label{nor30}
\begin{split}
&(\partial_t+iT_{\Sigma_e}+iT_{u\cdot\zeta})W_e^p=Q_e^p:=O^p(\mathcal{SS}_e+\mathcal{N}_e+\mathcal{C}_e)+i[T_{\Sigma_e+u\cdot\zeta},O^p]\mathcal{U}_e,\\
&(\partial_t+i\Lambda_b)W_b^p=Q_b^p:=O^p(\mathcal{SS}_b+\mathcal{N}_b+\mathcal{C}_b),
\end{split}
\end{equation}
 The bounds in the second line of \eqref{EqW} follow from \eqref{BootSimple1}--\eqref{BootSimple2} and Proposition \ref{ChangeUnknownsProp}. To prove \eqref{Increments} we start by writing
\begin{equation*}
\partial_t\{(1/2)\|W_\sigma^p(t)\|_{L^2}^2\}=\Re\langle\partial_tW_\sigma^p(t),W_\sigma^p(t)\rangle.
\end{equation*}
Since the operators $T_{\Sigma}$, $T_{u\cdot\zeta}$, and $\Lambda_b$ are self-adjoint, see Lemma \ref{PropSym} (i), we have
\begin{equation*}
\Re\langle iT_{\Sigma_e}W_e^p(t),W_e^p(t)\rangle=\Re\langle iT_{u\cdot\zeta}W_e^p(t),W_e^p(t)\rangle=\Re\langle i\Lambda_bW_b^p(t),W_b^p(t)\rangle=0.
\end{equation*}
Then we recall \eqref{ApproxSigma}, \eqref{BootSimple1}--\eqref{BootSimple2}, and \eqref{EquivalenceNorms}. Therefore $[T_{\Sigma_e^{\geq 2}},O^p]\mathcal{U}_e$ and $[T_{\Sigma_e^1+u\cdot\zeta},O^p](\mathcal{U}_e-U_e)$ are cubic and higher order terms which do not lose derivatives,
\begin{equation*}
\big\|[T_{\Sigma_e^{\geq 2}},O^p]\mathcal{U}_e\big\|_{L^2}+\big\|[T_{\Sigma_e^1+u\cdot\zeta},O^p](\mathcal{U}_e-U_e)\big\|_{L^2}\lesssim \eps_1^3(1+t)^{-3/2}.
\end{equation*} 
The desired decomposition follows.
\end{proof}

\subsection{Proof of Proposition \ref{BootstrapEE}} In view of \eqref{EquivNorm2}, it suffices to prove that $|\mathcal{E}(t)-\mathcal{E}(0)|\lesssim\eps_1^3$ for any $t\in[0,T]$. Using \eqref{Increments}, it suffices to prove that
\begin{equation*}
\Big|\int_0^t[\mathcal{B}_{NR}(s)+\mathcal{B}_{SS}(s)+\mathcal{B}_{hot}(s)]\,ds\Big|\lesssim \eps_1^3
\end{equation*} 
Since $|\mathcal{B}_{hot}(s)|\lesssim \eps_1^3(1+s)^{-3/2}$, the contribution of this term can be easily bounded. We consider the other two terms. Given $t\in[0,T_0]$, we fix a suitable decomposition of the function $\mathbf{1}_{[0,t]}$, i.e. we fix functions $q_0,\ldots,q_{L+1}:\mathbb{R}\to[0,1]$, $|L-\log_2(2+t)|\leq 2$, with the properties
\begin{equation}\label{nh2}
\begin{split}
&\mathrm{supp}\,q_0\subseteq [0,2], \qquad \mathrm{supp}\,q_{L+1}\subseteq [t-2,t],\qquad\mathrm{supp}\,q_m\subseteq [2^{m-1},2^{m+1}],\\
&\sum_{m=0}^{L+1}q_m(s)=\mathbf{1}_{[0,t]}(s),\qquad q_m\in C^1(\mathbb{R})\text{ and }\int_0^t|q'_m(s)|\,ds\lesssim 1\text{ for }m=1,\ldots,L.
\end{split}
\end{equation}

Using Proposition \ref{IncrementEnergyProp}, it suffices to prove the following two results uniformly in $m\ge 0$:

\begin{lemma}\label{BootstrapEE2}
Under the hypothesis of Proposition \ref{BootstrapEE}, we have
\begin{equation}\label{nor11}
\Big| \int_0^tq_m(s)\mathcal{B}_{NR}(s)ds\Big|\lesssim \epsilon_1^{3}2^{-\delta^2m}.
\end{equation}
\end{lemma}

\begin{lemma}\label{BootstrapEE1}
Under the hypothesis of Proposition \ref{BootstrapEE}, we have
\begin{equation}\label{nor12}
\Big|\int_0^tq_m(s)\mathcal{B}_{SS}(s)ds\Big|\lesssim \epsilon_1^{3}2^{-\delta^2m}.
\end{equation}
\end{lemma}

We prove these two lemmas in subsection \ref{nor7} and section \ref{SSterm} below. For Lemma \ref{BootstrapEE2} we exploit the nonresonant condition and use normal forms: we integrate by parts in time, use again the equations \eqref{ParaUnknowns}, and symmetrize. As a result, the space-time integral in \eqref{nor11} can be bounded by a quartic expression, and then estimated using Proposition \ref{BootSimple}.

The proof of Lemma \ref{BootstrapEE1} is more difficult. Normal form transformations are not possible in this case, due to the vanishing of the resulting denominators on a large set (this is the "division problem" discussed in the introduction). To prove Lemma \ref{BootstrapEE1} we combine a partial normal norm transformation, to bound the contribution of sufficiently large modulations, and a crucial $L^2$ estimate on a localized Fourier integral operator, to bound the contribution of small modulations. This $L^2$ estimate is proved in Lemma \ref{L2EstLem} and depends on the nondegeneracy property of the resonant set \eqref{RNDC}.

\subsection{Proof of Lemma \ref{BootstrapEE2}}\label{nor7}

We consider the terms in \eqref{BNR} individually, and use also \eqref{NRMult1}. Let $n:\mathbb{R}^2\to\mathbb{R}^2$ denote symbols satisfying bounds similar to \eqref{NRMult},\begin{equation}\label{Ngen}
\Vert n\Vert_{S^\infty_{kk_1k_2}}+2^{k_2}\Vert (\partial_\xi+\partial_\eta)n\Vert_{S^\infty_{kk_1k_2}}\lesssim (1+2^{4k_1})\mathbf{1}_{\{k_2\ge \max(k_1,0)+\D/2\}},
\end{equation} 
for any $k,k_1,k_2\in\mathbb{Z}$. In particular $n(\xi,\eta)=0$ unless $|\eta|\gg 1+|\xi-\eta|$. Let 
\begin{equation}\label{Ngen2}
\mathcal{I}[f,g,h]:=\langle N[f,g],h\rangle=C\iint_{\mathbb{R}^2\times\mathbb{R}^2} n(\xi,\eta)\widehat{f}(\xi-\eta)\widehat{g}(\eta)\overline{\widehat{h}(\xi)}d\eta d\xi.
\end{equation} 
For \eqref{nor11} it suffices to prove that
\begin{equation}\label{Ngen3}
\Big|\int_0^tq_m(s)\mathcal{I}[\Omega^{p_1}U_\mu(s),\langle\nabla\rangle^\iota \Omega^{p_2}\mathcal{U}_\nu(s),W_\sigma^p(s)]\,ds\Big|\lesssim \eps_1^32^{-\delta^2m},
\end{equation}
provided that, as in \eqref{NRMult0}, either $\sigma=e$ or $\nu=\pm e$, and 
\begin{equation*}
\begin{split}
\text{ either } &p\in[1,N_1],\,p_1,p_2\in [0,N_1],\,p_1+p_2\leq p,\,\iota\in\{0,1\},\,\iota+p_2\leq p;\\
\text{ or } &p=p_1=p_2=0,\,\iota=N_0.
\end{split} 
\end{equation*}
Notice that we replaced $U_\nu$ with $\mathcal{U}_\nu$ in passing from \eqref{nor11} to \eqref{Ngen3} (with the same convention as before, $\mathcal{U}_{-\sigma}=\overline{\mathcal{U}_\sigma}$, $W_{-\sigma}=\overline{W_\sigma}$, $\sigma\in\{e,b\}$) which is acceptable due to \eqref{EquivalenceNorms}. In proving \eqref{Ngen3} we may also assume that $m\in[10,L]$, since $q_{L+1}$ is supported in an interval of length $\lesssim 1$.

To prove \eqref{Ngen3} we integrate by parts in $s$. We start from the observation that
\begin{equation*}
\begin{split}
0&=\int_{\mathbb{R}}\frac{d}{ds}\Big\{q_m(s)\widetilde{\mathcal{I}}[\Omega^{p_1}U_\mu(s),\langle\nabla\rangle^\iota \Omega^{p_2}\mathcal{U}_\nu(s),W_\sigma^p(s)]\Big\}ds,\\
\widetilde{\mathcal{I}}[f,g,h]&:=\iint_{\mathbb{R}^2\times\mathbb{R}^2} \frac{n(\xi,\eta)}{i\Phi_{\sigma\mu\nu}(\xi,\eta)}\widehat{f}(\xi-\eta)\widehat{g}(\eta)\overline{\widehat{h}(\xi)}d\eta d\xi.
\end{split}
\end{equation*}
This gives
\begin{equation*}
\begin{split}
-\int_{\mathbb{R}}\iint_{\mathbb{R}^2\times\mathbb{R}^2}q_m(s)&\mathcal{I}[\Omega^{p_1}U_\mu(s),\langle\nabla\rangle^\iota \Omega^{p_2}\mathcal{U}_\nu(s),W_\sigma^p(s)]ds\\
&=\int_{\mathbb{R}}q'_m(s)J_1(s)+\int_{\mathbb{R}}q_m(s)J_2(s)\,ds+\int_{\mathbb{R}}q_m(s)J_3(s)\,ds,
\end{split}
\end{equation*}
where
\begin{equation}\label{ResultNFBootstrapEE2}
\begin{split}
&J_1:=\widetilde{\mathcal{I}}[\Omega^{p_1}U_\mu,\langle\nabla\rangle^\iota \Omega^{p_2}\mathcal{U}_\nu(s),W_\sigma^p],\\
&J_2:=\widetilde{\mathcal{I}}[(\partial_s+i\Lambda_\mu)\Omega^{p_1}U_\mu,\langle\nabla\rangle^\iota \Omega^{p_2}\mathcal{U}_\nu,W_\sigma^p],\\
&J_3:=\widetilde{\mathcal{I}}[\Omega^{p_1}U_\mu,(\partial_s+i\Lambda_\nu)\langle\nabla\rangle^\iota \Omega^{p_2}\mathcal{U}_\nu,W_\sigma^p]+\widetilde{\mathcal{I}}[\Omega^{p_1}U_\mu,\langle\nabla\rangle^\iota \Omega^{p_2}\mathcal{U}_\nu,(\partial_s+i\Lambda_\sigma)W_\sigma^p].
\end{split}
\end{equation}
For \eqref{Ngen3} it remains to show that for any $s\in I_m=\mathrm{supp}(q_m)$,
\begin{equation}\label{Ngen4}
|J_1(s)|+2^m|J_2(s)|+2^m|J_3(s)|\lesssim \eps_1^32^{-\delta m}.
\end{equation}

The main observation is that the function $\Phi_{\sigma\mu\nu}(\xi,\eta)$ does not vanish in the support of $n$. More precisely, since $|\Phi_{\sigma\mu\nu}(\xi,\eta)|\gtrsim (1+|\xi-\eta|)^{-1}$, see \eqref{ener42.5}, we have  
\begin{equation}\label{ControlPhase}
\Big\Vert \frac{n(\xi,\eta)}{\Phi_{\sigma\mu\nu}(\xi,\eta)}\Big\Vert_{S^\infty_{kk_1k_2}}\lesssim (1+2^{20k_1})\mathbf{1}_{\{k_2\ge \max(k_1,0)+\mathcal{D}/2\}},
\end{equation}
as a consequence of Lemma \ref{Sinfinity} (with $f(\alpha,\beta)=[\Phi_{\sigma\mu\nu}(\alpha,\alpha-\beta)]^{-1}$) and \eqref{Ngen}. Moreover, if $\nu\neq\sigma$ then the stronger lower bound $|\Phi_{\sigma\mu\nu}(\xi,\eta)|\gtrsim 1+|\xi|+|\eta|$ holds, and therefore
\begin{equation}\label{ControlPhaseS}
\Big\Vert \frac{n(\xi,\eta)}{\Phi_{\sigma\mu\nu}(\xi,\eta)}\Big\Vert_{S^\infty_{kk_1k_2}}\lesssim (1+2^{20k_1})2^{-k_2}\mathbf{1}_{\{k_2\ge \max(k_1,0)+\mathcal{D}/2\}}.
\end{equation}

To prove the bounds on $|J_1(s)|$ and $|J_2(s)|$ in \eqref{Ngen4} we use Lemma \ref{L1easy}. If $p_1\leq N_1/2$ then
\begin{equation*}
\begin{split}
|J_1|&\lesssim \sum_{k,k_1,k_2\in\mathbb{Z}}\Big\Vert \frac{n(\xi,\eta)}{\Phi_{\sigma\mu\nu}(\xi,\eta)}\Big\Vert_{S^\infty_{kk_1k_2}}\|P_{k_1}\Omega^{p_1}U_\mu\|_{L^\infty}\|P_{k_2}\langle\nabla\rangle^\iota \Omega^{p_2}\mathcal{U}_\nu\|_{L^2}\|P_kW_\sigma^p\|_{L^2}\\
&\lesssim \sum_{k_2\geq \max(k_1,0)+\D/2,\,|k_2-k|\leq 10}2^{20\max(k_1,0)}\|P_{k_1}\Omega^{p_1}U_\mu\|_{L^\infty}\|P_{k_2}\mathcal{U}_{\nu}\|_{H^{N_0}\cap H^{N_1}_{\Omega}}\|P_k\mathcal{U}_\sigma\|_{H^{N_0}\cap H^{N_1}_{\Omega}}\\
&\lesssim \eps_1^32^{-m/2},
\end{split}
\end{equation*}
where we used \eqref{ControlPhase}, \eqref{BootSimple1}--\eqref{BootSimple2}, and \eqref{EquivalenceNorms}. A similar estimate gives the bound $|J_2(s)|\lesssim \eps_1^42^{-3m/2}$ if $p_1\leq N_1/2$, because, as a consequence of \eqref{Alx2}--\eqref{Alx3},
\begin{equation*}
\sum_{k_1\in\mathbb{Z}}2^{20\max(k_1,0)}\|P_{k_1}(\partial_s+i\Lambda_\mu)\Omega^{p_1}U_\mu(s)\|_{L^\infty}\lesssim \eps_1^22^{-3m/2}.
\end{equation*}
Moreover, the case $p_1\geq N_1/2$ (therefore $p_2\leq N_1/2$) is similar by estimating $\Omega^{p_1}U_\mu$ and $(\partial_s+i\Lambda_\mu)\Omega^{p_1}U_\mu$ in $L^2$, and estimating $\langle\nabla\rangle^\iota \Omega^{p_2}\mathcal{U}_\nu$ in $L^\infty$. Thus $|J_1(s)|+2^m|J_2(s)|\lesssim \eps_1^32^{-\delta m}$ as desired.

Some more care is needed for the bound on $|J_3(s)|$ in \eqref{Ngen4} because of the possible loss of derivatives. It follows from \eqref{EqW} that, for $q\in\{0,\ldots,N_1\}$,
\begin{equation}\label{nor50}
(\partial_s+i\Lambda_e)W_e^q=Q_e^q-iT_{\Sigma_e^{\geq 1}+u\cdot\zeta}W_e^q,\qquad (\partial_s+i\Lambda_e)W_b^q=Q_b^q.
\end{equation}
Using also \eqref{nor5} and \eqref{LqBdTa}, it follows that
\begin{equation}\label{nor51}
\|\langle\nabla\rangle^{-1}(\partial_s+i\Lambda_\mu)W_\mu^q\|_{L^2}\lesssim \eps_1^2(1+t)^{-9/10},\qquad \text{ for }\mu\in\{e,b,-e,-b\},\,q\in\{0,\ldots,N_1\}.
\end{equation}
Moreover, for $q\in[0,N_1/2]$,  
\begin{equation}\label{nor52}
\|(\partial_s+i\Lambda_\mu)\Omega^q\mathcal{U}_{\mu}\|_{L^\infty}\lesssim \eps_1^2(1+t)^{-9/5},\qquad \text{ for }\mu\in\{e,b,-e,-b\},
\end{equation}
as a consequence of Proposition \ref{ChangeUnknownsProp} and Lemma \ref{L1easy}. The desired bound on $|J_3(s)|$ follows by the same $L^2\times L^2\times L^\infty$ argument as before if $\nu\neq\sigma$, using the stronger bound \eqref{ControlPhaseS} instead of \eqref{ControlPhase} to recover the derivative loss.

Finally, assume that $\nu=\sigma=e$ (one cannot have $\nu=\sigma=b$ in this lemma, due to \eqref{NRMult0}). We may also assume that $p_1=0$ and $\langle\nabla\rangle^\iota \Omega^{p_2}\mathcal{U}_\nu=W_e^p$, since in the other cases the derivative loss can be recovered (notice that $\|\langle\nabla\rangle^2\Omega^q\mathcal{U}_{\nu}\|_{L^2}\lesssim \eps_1$ if $q\leq N_1-1$, due to \eqref{BootSimple} and the assumption $N_0\geq 2N_1$). Also, the contribution of $Q_e^p$ when applying \eqref{nor50} does not lose derivatives, due to \eqref{EqW}. For \eqref{Ngen4} it remains to prove that 
\begin{equation}\label{nor53}
\big|\widetilde{\mathcal{I}}[U_\mu,iT_{\Sigma^{\geq 1}_e+u\cdot\zeta}W_e^p,W_e^p]+\widetilde{\mathcal{I}}[U_\mu,W_e^p,iT_{\Sigma^{\geq 1}_e+u\cdot\zeta}W_e^p]\big|\lesssim \eps_1^32^{-m-\delta m}.
\end{equation}

We start with the contribution of $T_{u\cdot\zeta}$ and compute, for $k_1,k_2,k\in\mathbb{Z}$,
\begin{equation}\label{nor57}
\begin{split}
\mathcal{I}_{kk_1k_2}&:= \widetilde{\mathcal{I}}[P_{k_1}U_\mu,iP_{k_2}T_{u\cdot\zeta}W^p_e,P_kW^p_e]+\widetilde{\mathcal{I}}[P_{k_1}U_\mu,P_{k_2}W^p_e,iP_kT_{u\cdot\zeta}W^p_e]\\
&=C\int_{\mathbb{R}^6}\widehat{P'_{k_2}W^p_e}(\eta)\overline{\widehat{P'_{k}W^p_e}}(\xi)\widehat{P'_{k_1}U_\mu}(\xi-\eta-\theta)\widehat{u}_j(\theta)m^1_j(\xi,\eta,\theta)\,d\eta d\xi d\theta,
\end{split}
\end{equation}
\begin{equation*}
\begin{split}
m^1_j(\xi,\eta,\theta)&:=\frac{n_{kk_1k_2}(\xi,\eta+\theta)}{\Phi(\xi,\eta+\theta)}\chi(\frac{\theta}{\vert 2\eta+\theta\vert})\frac{(2\eta+\theta)_j}{2}-\frac{n_{kk_1k_2}(\xi-\theta,\eta)}{\Phi(\xi-\theta,\eta)}\chi(\frac{\theta}{\vert 2\xi-\theta\vert})\frac{(2\xi-\theta)_j}{2}\\
&=\frac{(\xi+\eta)_j}{2}\Big[\frac{n_{kk_1k_2}(\xi,\eta+\theta)}{\Phi(\xi,\eta+\theta)}\chi(\frac{\theta}{\vert 2\eta+\theta\vert})-\frac{n_{kk_1k_2}(\xi-\theta,\eta)}{\Phi(\xi-\theta,\eta)}\chi(\frac{\theta}{\vert 2\xi-\theta\vert})\Big]\\
&-\frac{(\xi-\eta-\theta)_j}{2}\Big[\frac{n_{kk_1k_2}(\xi,\eta+\theta)}{\Phi(\xi,\eta+\theta)}\chi(\frac{\theta}{\vert 2\eta+\theta\vert})+\frac{n_{kk_1k_2}(\xi-\theta,\eta)}{\Phi(\xi-\theta,\eta)}\chi(\frac{\theta}{\vert 2\xi-\theta\vert})\Big],
\end{split}
\end{equation*}
where $n_{kk_1k_2}(\xi,\eta):=n(\xi,\eta)\varphi_k(\xi)\varphi_{k_1}(\xi-\eta)\varphi_{k_2}(\eta)$ and, for simplicity of notation $\Phi:=\Phi_{e\mu e+}$ and $P'_l:=P_{[l-4,l+4]}$. We observe that
\begin{equation*}
\begin{split}
\frac{1}{\Phi(\xi,\eta+\theta)}-\frac{1}{\Phi(\xi-\theta,\eta)}&=\frac{\Lambda_e(\xi-\theta)-\Lambda_e(\xi)+\Lambda_e(\eta+\theta)-\Lambda_e(\eta)}{\Phi(\xi,\eta+\theta)\Phi(\xi-\theta,\eta)}\\
&=C\frac{\int_0^1\int_0^1\theta_j(\xi-\eta-\theta)_l(\partial_j\partial_l\Lambda_e)(\eta+x\theta+y(\xi-\eta-\theta))\,dxdy}{\Phi(\xi,\eta+\theta)\Phi(\xi-\theta,\eta)}.
\end{split}
\end{equation*}
We have a similar identity for $n_{kk_1k_2}(\xi,\eta+\theta)-n_{kk_1k_2}(\xi-\theta,\eta)$, which gains one factor of $2^{k_2}$, due to \eqref{Ngen}. Combining these identities with \eqref{ControlPhase} we recover the derivative loss, i.e.
\begin{equation*}
\big\|\mathcal{F}^{-1}\{m^1_j(\xi,\eta,\theta)\varphi_{k_3}(\theta)\}\big\|_{L^1(\mathbb{R}^6)}\lesssim (1+2^{k_1})^{50}(1+2^{k_3})^{10}.
\end{equation*}
Therefore, using Lemma \ref{L1easy}, 
\begin{equation*}
\begin{split}
\vert \mathcal{I}_{kk_1k_2}\vert&\lesssim \sum_{k_3\leq k_2-100}\Vert P'_{k_2}W_e^p\Vert_{L^2}\Vert P'_kW_e^p\Vert_{L^2}\Vert P'_{k_1}U_\mu\Vert_{L^\infty}\Vert P'_{k_3}u_j\Vert_{L^\infty}\cdot (1+2^{k_1})^{50}(1+2^{k_3})^{10}.
\end{split}
\end{equation*}
The desired conclusion follows by summing over $k,k_1,k_2$ with $|k-k_2|\leq 10$, $k_2\geq \max(k_1,0)+100$.

The contribution of $T_{\Sigma_e^{\geq 1}}$ in \eqref{nor53} can be bounded in a similar way. We first decompose 
\begin{equation*}
\begin{split}
&\Sigma_e^{\geq 1}=a_{(1)}+a_{(0)},\\
&a_{(1)}(x,\zeta):=\sqrt{1+d|\zeta|^2}\sum_{n\geq 1}d_n\Big(\frac{F(x)|\zeta|^2}{1+d|\zeta|^2}\Big)^n,\qquad F:=(d+\kappa)\rho+\kappa\rho^2+h_2(\rho)(1+\rho),
\end{split}
\end{equation*} 
compare with \eqref{ApproxSigma}. Here $a_{(0)}$ is a symbol of order $0$, whose contribution can be estimated directly without using symmetrization, and $d_n$ are the coefficients in the Taylor expansion of $\sqrt{1+x}-1$ around $0$. Due to \eqref{BootSimple1}--\eqref{BootSimple2}, we notice that $F$ satisfies the same $L^\infty$ bounds as $u_j$,
\begin{equation*}
\sum_{l\in\mathbb{Z}}2^{20\max(l,0)}\|P_lF\|_{L^\infty}\lesssim \eps_1(1+t)^{-9/10}.
\end{equation*}
The same symmetrization argument as in \eqref{nor57}, shows that the contribution of $T_{a_{(1)}}$ can be bounded as claimed in \eqref{nor53}. This completes the proof of Lemma \ref{BootstrapEE2}.

\section{Energy estimates, II: strongly semilinear terms}\label{SSterm}

We now turn to the proof of Lemma \ref{BootstrapEE1}. The idea is to first use the strongly semilinear nature to restrict to frequencies that are not too large. Then we use the main $L^2$ bound in Lemma \ref{L2EstLem} below to control the contribution of small modulations, and a partial normal form transformation to control the larger modulations.

\subsection{The main $L^2$ lemma}\label{L2Lemma} Assume $\Phi=\Phi_{\sigma\mu\nu}$ is as in \eqref{phasedef}, for some choice of $\sigma,\mu,\nu\in\mathcal{P}$.

\begin{lemma}\label{L2EstLem}
Assume that $T$ is given by
\begin{equation*}
Tf(\xi):=\int_{\mathbb{R}^2}e^{is\Phi(\xi,\eta)}a(\xi,\eta)\varphi(2^\lambda\Phi(\xi,\eta))f(\eta)d\eta,
\end{equation*}
with $2^m-1\le \vert s\vert\le 2^{m+1}$, $m\in\mathbb{Z}_+$, and
\begin{equation*}
5m/6\leq \lambda\le m(1-\delta_1),
\end{equation*}
where $\delta_1:=10^{-4}$. We assume that the function $a$ is supported in the set $B(0,2^k)\times B(0,2^k)$, where $k\geq 0$ is an integer, and satisfies the symbol-type estimates
\begin{equation}\label{BdOnAL2Lem}
\sup_{\xi,\eta\in\mathbb{R}^2}\big|D_{\xi,\eta}^\alpha a(\xi,\eta)\big|\lesssim_\alpha 2^{\vert\alpha\vert m/2}.
\end{equation}
Then 
\begin{equation}\label{L2Bd}
\big\Vert Tf\big\Vert_{L^2}\lesssim 2^{12k}2^{-1.004\lambda}\Vert f\Vert_{L^2}.
\end{equation}
\end{lemma}

The rest of this subsection is concerned with the proof of Lemma \ref{L2EstLem}. We may assume that $\lambda\geq 12k+\D$. Let
\begin{equation}\label{cas70}
R:=2^{-17\lambda/242}.
\end{equation}
and fix a smooth function $\chi:\mathbb{R}^2\to[0,1]$, supported in $B(0,2)$ and satisfying
\begin{equation*}
\sum_{i\in\mathbb{Z}^2}\chi(x-i)=1\qquad\text{ for any }x\in\mathbb{R}^2.
\end{equation*}
For $j\in\mathbb{Z}^2$ let
\begin{equation*}
v^j:=\frac{R}{2^\D}j,\qquad \chi_j(x):=\chi\Big(\frac{2^\D}{R}x-j\Big)=\chi\Big(\frac{2^\D}{R}(x-v^j)\Big).
\end{equation*}

We decompose
\begin{equation}\label{cas71}
\begin{split}
&T=\sum_{(i,j)\in\mathbb{Z}^2\times\mathbb{Z}^2}T_{ij},\\
&T_{ij}f(\xi):=\int_{\mathbb{R}^2}e^{is\Phi(\xi,\eta)}a(\xi,\eta)\varphi(2^\lambda\Phi(\xi,\eta))f(\eta)\chi_i(\xi)\chi_j(\eta)d\eta.
\end{split}
\end{equation}
Our analysis of the operator $T$ is based on the size of the smooth function $\Upsilon:\mathbb{R}^2\times\mathbb{R}^2\to\mathbb{R}$,
\begin{equation}\label{cas72}
\Upsilon(\xi,\eta)=\nabla_{\xi,\eta}^{2}\Phi(\xi,\eta)\left[\nabla_{\xi}^{\perp}\Phi(\xi,\eta),\nabla_{\eta}^{\perp}\Phi(\xi,\eta)\right].
\end{equation}
Using the formula \eqref{cas13} it is easy to see that
\begin{equation}\label{cas72.5}
\sup_{\xi,\eta\in\mathbb{R}^2,|\alpha|\geq 0}\big|D_{\xi,\eta}^\alpha \Upsilon(\xi,\eta)\big|\lesssim_\alpha 1.
\end{equation}

Let
\begin{equation}\label{cas73}
\begin{split}
&V^1:=\{(i,j)\in\mathbb{Z}^2\times\mathbb{Z}^2:\Upsilon(v^i,v^j)< 2^\D R\},\qquad T^1:=\sum_{(i,j)\in V^1}T_{ij},\\
&V^2:=\{(i,j)\in\mathbb{Z}^2\times\mathbb{Z}^2:\Upsilon(v^i,v^j)\geq 2^\D R\}.
\end{split}
\end{equation}
We show in Lemma \ref{ElemOp2} and Lemma \ref{ElemOpLem} below that
\begin{equation*}
\big\|T^1\big\|_{L^2\to L^2}\lesssim 2^{12k}2^{-\lambda}R^{1/17}\qquad \text{ and }\qquad\sup_{(i,j)\in V^2}\big\|T_{ij}\big\|_{L^2\to L^2}\lesssim 2^{-5\lambda/4}R^{-3/2}.
\end{equation*}
Assuming these bounds it follows that
\begin{equation*}
\begin{split}
\|Tf\|_{L^2}^2&\lesssim \|T^1f\|_{L^2}^2+\sum_{i\in\mathbb{Z}^2}\Big\|\sum_{j\in\mathbb{Z}^2,\,(i,j)\in V^2}T_{ij}f\Big\|_{L^2}^2\\
&\lesssim (2^{12k}2^{-\lambda}R^{1/17})^2\|f\|_{L^2}^2+(2^{2k}R^{-2})^22^{-5\lambda/2}R^{-3}\|f\|_{L^2}^2\\
&\lesssim \|f\|_{L^2}^2\big[2^{24k}2^{-2\lambda}R^{2/17}+2^{4k}2^{-5\lambda/2}R^{-7}\big].
\end{split}
\end{equation*}
The desired conclusion \eqref{L2Bd} follows, using the definition $R=2^{-17\lambda/242}$, see \eqref{cas70}.

\begin{lemma}\label{ElemOp2}
With the definition above,
\begin{equation*}
\big\|T^1\big\|_{L^2\to L^2}\lesssim 2^{12k}2^{-\lambda}R^{1/17}.
\end{equation*}
\end{lemma}

\begin{proof} We use Schur's test (see Lemma \ref{ShurLem}). Notice that, using the definitions and \eqref{cas72.5},
\begin{equation*}
\begin{split}
&|T^1f(\xi)|\lesssim \int_{\mathbb{R}^2}L(\xi,\eta)|f(\eta)|d\eta,\\
&L(\xi,\eta):=|a(\xi,\eta)|\varphi(2^\lambda\Phi(\xi,\eta))\varphi((2^{\D+2}R)^{-1}\Upsilon(\xi,\eta)).
\end{split}
\end{equation*}
We apply Proposition \ref{volume} with $\eps\approx 2^k2^{-\lambda}$, $\eps'\approx 2^{3k}R$. With $E'_1$ as in Proposition \ref{volume}, we have
\begin{equation*}
\begin{split}
&\sup_{\xi}\int_{\mathbb{R}^2}\mathbf{1}_{E'_1}(\xi,\eta)L(\xi,\eta)\,d\eta\lesssim 2^{14k}R^{1/8}2^{-\lambda}\lambda,\quad \sup_{\eta}\int_{\mathbb{R}^2}\mathbf{1}_{E'_1}(\xi,\eta)L(\xi,\eta)\,d\xi\lesssim 2^{8k}2^{-\lambda}\lambda.
\end{split}
\end{equation*}
Therefore, using Schur's test,
\begin{equation*}
\|T^1_1f\|_{L^2}\lesssim 2^{12k}R^{1/17}2^{-\lambda},\qquad T^1_1f(\xi):=\int_{\mathbb{R}^2}\mathbf{1}_{E'_1}(\xi,\eta)L(\xi,\eta)|f(\eta)|d\eta.
\end{equation*}
The contribution of the set $E'_2$ can be estimated in a similar way, and the conclusion of the lemma follows using also \eqref{res01}. 
\end{proof}

\begin{lemma}\label{ElemOpLem}
With the definitions above, for any $(i,j)\in V^2$,
\begin{equation*}
\big\|T_{ij}\big\|_{L^2\to L^2}\lesssim 2^{-5\lambda/4}R^{-3/2}.
\end{equation*}
\end{lemma}

\begin{proof} Since $T_{ij}$ and $T_{ij}^\ast$ obey similar estimates,  using \eqref{cas72} and \eqref{res01}, we may assume without loss of generality that
\begin{equation}\label{cas80}
\big\vert\nabla_\xi\Phi(v^i,v^j)\big\vert\gtrsim 1,\qquad \big\vert\nabla_\eta\Phi(v^i,v^j)\big\vert\approx \rho\gtrsim R, \qquad \big\vert\Upsilon(v^i,v^j)\big\vert\gtrsim R.
\end{equation}
In view of \eqref{cas72.5} and the definitions, we see that these inequalities remain true for all points $(\xi,\eta)$ in the support of the operator $T_{ij}$. 

We compute 
\begin{equation*}
\begin{split}
(T_{ij} T_{ij}^\ast) f(\xi)&=\int_{\mathbb{R}^2}K_{ij}(\xi,\xi^\prime)f(\xi^\prime)d\xi^\prime,\\
K_{ij}(\xi,\xi^\prime):&=\chi_i(\xi)\chi_i(\xi')\int_{\mathbb{R}^2}e^{is[\Phi(\xi,\eta)-\Phi(\xi^\prime,\eta)]}a(\xi,\eta)\overline{a}(\xi^\prime,\eta)\varphi(2^\lambda\Phi(\xi,\eta))\varphi(2^\lambda\Phi(\xi^\prime,\eta))\chi_j^2(\eta)d\eta.
\end{split}
\end{equation*}
Using Schur's test, it suffices to show that
\begin{equation}\label{SuffDelta}
\sup_{\xi}\int_{\mathbb{R}^2}\vert K_{ij}(\xi,\xi^\prime)\vert d\xi^\prime+\sup_{\xi^\prime}\int_{\mathbb{R}^2}\vert K_{ij}(\xi,\xi^\prime)\vert d\xi\lesssim 2^{-5\lambda/2}R^{-3}.
\end{equation}
By symmetry, it suffices to bound the first term.

We fix $\xi\in B(v^i,4R/2^{\D})$ and consider an additional partition of the $\eta$-space into small balls of radius $\approx 2^{-\lambda/2}$, centered at the points in the lattice $2^{-\lambda/2}\mathbb{Z}^2$. Recalling that $\big\vert\nabla_\eta\Phi(\xi,\eta)\big\vert\gtrsim R$ for all $\eta\in  B(v^j,4R/2^{\D})$, it is easy to see at most $\approx 2^{\lambda/2} R$ of these small balls can interact with the support of the integral defining $K_{ij}(\xi,\xi')$. It suffices to show that
\begin{equation}\label{SuffTTastA}
\int_{\mathbb{R}^2}\vert K_{\eta_0}(\xi,\xi^\prime)\vert d\xi^\prime\lesssim 2^{-3\lambda}R^{-4},
\end{equation}
for any choice of $\eta_0\in B(v^j,4R/2^{\D})$, where
\begin{equation}\label{Knorm}
\begin{split}
K_{\eta_0}(\xi,\xi^\prime):=\chi_i(\xi')\int_{\mathbb{R}^2}&e^{is[\Phi(\xi,\eta)-\Phi(\xi^\prime,\eta)]}a(\xi,\eta)\overline{a}(\xi^\prime,\eta)\\
&\varphi(2^\lambda\Phi(\xi,\eta))\varphi(2^\lambda\Phi(\xi^\prime,\eta))\chi_j^2(\eta)\chi(2^{-\lambda/2}(\eta-\eta_0))\,d\eta.
\end{split}
\end{equation}

We fix a point $\eta'_0\in B(\eta_0,2^{-\lambda/2+2})$ with the property that $\Phi(\xi,\eta'_0)\leq 2^{-\lambda+2}$; if such a point does not exist then $K_{\eta_0}\equiv 0$. We define two orthonormal basis, $(e_1,e_2)$ and $(V_1,V_2)$ such that
\begin{equation*}
e_1=\frac{\nabla_\xi\Phi(\xi,\eta'_0)}{\vert\nabla_\xi\Phi(\xi,\eta'_0)\vert},\qquad V_1=\frac{\nabla_\eta\Phi(\xi,\eta'_0)}{\vert\nabla_\eta\Phi(\xi,\eta'_0)\vert}
\end{equation*}
and decompose
\begin{equation*}
\eta-\eta'_0=\eta_1V_1+\eta_2V_2,\qquad \xi^\prime-\xi=\omega=\omega_1e_1+\omega_2e_2.
\end{equation*}
Using the Taylor expansion, we see that if $|\Phi(\xi,\eta)|\lesssim 2^{-\lambda}$ and $|\eta-\eta_0|\lesssim 2^{-\lambda/2}$ then
\begin{equation*}
O(2^{-\lambda})=\Phi(\xi,\eta)=\Phi(\xi,\eta'_0)+(\eta-\eta'_0)\cdot\nabla_\eta\Phi(\xi,\eta'_0)+O(2^{-\lambda}),
\end{equation*}
so that
\begin{equation}\label{SizeEta}
\eta-\eta'_0=\eta_1V_1+\eta_2V_2,\qquad\vert\eta_1\vert\lesssim \rho^{-1}2^{-\lambda},\qquad\vert\eta_2\vert\le 2^{-\lambda/2}.
\end{equation}
We also define the numbers
\begin{equation}\label{Sizeb}
b_0:=\left\vert \nabla^2_{\xi,\eta}\Phi(\xi,\eta'_0)[e_2,V_2]\right\vert\gtrsim \rho^{-1}R\gg 2^{-\lambda/4},\qquad b_1:=\vert\nabla_\xi\Phi(\xi,\eta'_0)\vert\approx 1.
\end{equation}

We would like to describe now the essential support of the kernel $K_{\eta_0}(\xi,.)$, assuming that $\xi$ and $\eta_0$ are fixed. We show first that 
\begin{equation}\label{LOClaim1B}
\hbox{if}\quad \vert\omega_1\vert\gtrsim 2^{-\lambda}+\omega_2^2\qquad\hbox{then}\qquad K_{\eta_0}(\xi,\xi^\prime)=0.
\end{equation}
Indeed, assume that $\vert\omega_1\vert\ge 2^{\D}(2^{-\lambda}+\omega_2^2)$ and $|\Phi(\xi,\eta)|\lesssim 2^{-\lambda}$ for some $\eta\in B(\eta_0,2^{-\lambda/2+2})$. Using Taylor expansion and \eqref{SizeEta} we estimate
\begin{equation*}
\begin{split}
\Phi(\xi^\prime,\eta)&=\Phi(\xi,\eta)+(\xi^\prime-\xi)\cdot\nabla_\xi\Phi(\xi,\eta)+O(\vert\xi^\prime-\xi\vert^2)\\
&=(\xi^\prime-\xi)\cdot\nabla_\xi\Phi(\xi,\eta'_0)+\nabla^2_{\xi,\eta}\Phi(\xi,\eta'_0)[\xi^\prime-\xi,\eta-\eta'_0]+O(2^{-\lambda}+\vert \xi^\prime-\xi\vert^2),\\
&=b_1\omega_1+\eta_2\nabla^2_{\xi,\eta}\Phi(\xi,\eta'_0)[\xi^\prime-\xi,V_2]+O(2^{-\lambda}+\rho^{-1}2^{-\lambda}(\vert\omega_1\vert+\vert\omega_2\vert)+\vert \omega_1\vert^2+\vert\omega_2\vert^2).
\end{split}
\end{equation*}
Therefore
\begin{equation*}
\vert \Phi(\xi^\prime,\eta)\vert\gtrsim \vert\omega_1\vert-C2^{-\lambda/2}(|\omega_1|+|\omega_2|)-C\omega_1^2\gtrsim \vert\omega_1\vert.
\end{equation*}
The desired conclusion \eqref{LOClaim1B} follows.

We show now that, for any integer $Q\geq 1$,
\begin{equation}\label{LOClaim2B}
\left\vert K_{\eta_0}(\xi,\xi^\prime)\right\vert\lesssim_Q \rho^{-1}2^{-3\lambda/2}\big[\big(1+2^{\lambda/2}b_0|\omega_2|\big)^{-Q}+2^{-Q(m-\lambda)}\big].
\end{equation}
Indeed, in view of \eqref{LOClaim1B}, we may assume that
\begin{equation}\label{ConsLOClaim2A}
\vert\omega_1\vert\lesssim 2^{-\lambda}+\omega_2^2,\qquad 2^{\D}2^{-\lambda/2}b_0^{-1}\leq \vert\omega_2\vert\leq 4R/2^{\D}.
\end{equation}
We would like to integrate by parts in the formula \eqref{Knorm}, in the direction of the vector-field $V_2$.
Letting $\Theta(\eta)=\Theta_{\xi,\xi'}(\eta):=\Phi(\xi,\eta)-\Phi(\xi^\prime,\eta)$, we see that
\begin{equation}\label{lowbounds1}
\begin{split}
V_2(\Theta)&=\nabla^2_{\xi,\eta}\Phi(\xi,\eta)[\xi^\prime-\xi,V_2]+O(\vert\xi-\xi^\prime\vert^2)\\
&=\omega_1\nabla^2_{\xi,\eta}\Phi(\xi,\eta)[e_1,V_2]+\omega_2\nabla^2_{\xi,\eta}\Phi(\xi,\eta)[e_2,V_2]+O(\vert\xi-\xi^\prime\vert^2)\\
&=\omega_2\big[b_0+O(2^{-\lambda/2})\big]+O(\vert \omega_1\vert+\vert\omega_2\vert^2).
\end{split}
\end{equation}
Therefore, in the support of the integral $K_{\eta_0}(\xi,\xi')$,
\begin{equation}\label{lowbounds2}
2b_0|\omega_2|\geq |V_2(\Theta)|\geq b_0|\omega_2|/2.
\end{equation}

We would like to apply Lemma \ref{tech5} with 
\begin{equation*}
K\approx 2^mb_0|\omega_2|\qquad\text{ and }\qquad\eps\approx \min\big(2^{-m/2},2^{-\lambda}b_0^{-1}|\omega_2|^{-1}\big),
\end{equation*}
which would give the desired bound \eqref{LOClaim2B}. For this we notice that, for any integer $q\geq 1$,
\begin{equation*}
\big|(V_2\cdot\nabla _\eta)^q\big\{a(\xi,\eta)\overline{a}(\xi^\prime,\eta)\chi_j^2(\eta)\chi(2^{-\lambda/2}(\eta-\eta_0))\big\}\big|\lesssim 2^{qm/2}.
\end{equation*}
Moreover, using also \eqref{lowbounds1}, in the support of the integral,
\begin{equation*}
\begin{split}
\left\vert (V_2\cdot\nabla_\eta)\left\{\varphi(2^\lambda\Phi(\xi,\eta))\right\}\right\vert&\lesssim 2^{\lambda/2},\\
\left\vert (V_2\cdot\nabla_\eta)\left\{\varphi(2^\lambda\Phi(\xi',\eta))\right\}\right\vert&\lesssim 2^\lambda b_0|\omega_2|,
\end{split}
\end{equation*}
and, for any integer $q\geq 2$,
\begin{equation*}
\begin{split}
\left\vert (V_2\cdot\nabla_\eta)^q\left\{\varphi(2^\lambda\Phi(\xi,\eta))\right\}\right\vert&\lesssim 2^{q\lambda/2},\\
\left\vert (V_2\cdot\nabla_\eta)^q\left\{\varphi(2^\lambda\Phi(\xi',\eta))\right\}\right\vert&\lesssim (2^\lambda b_0|\omega_2|)^q.
\end{split}
\end{equation*}
Lemma \ref{tech5} applies to complete the proof of \eqref{LOClaim2B}.

Given \eqref{LOClaim1B} and \eqref{LOClaim2B} (with $Q$ sufficiently large depending on $\delta_1$), it follows that
\begin{equation*}
\int_{\mathbb{R}^2}\vert K_{\eta_0}(\xi,\xi^\prime)\vert d\xi^\prime\lesssim \rho^{-1}2^{-3\lambda/2}\cdot b_0^{-1}2^{-\lambda/2}\cdot b_0^{-2}2^{-\lambda}\lesssim R^{-3}2^{-3\lambda},
\end{equation*}
which gives the desired bound \eqref{SuffTTastA}. This completes the proof of the lemma.
\end{proof}

\subsection{Proof of Lemma \ref{BootstrapEE1}}\label{SSTcontrol}

%

The rest of this subsection is concerned with the proof of Lemma \ref{BootstrapEE1}. We start with a lemma that connects our situation to the $L^2$ estimates in Lemma \ref{L2EstLem}.

\begin{lemma}\label{L2Lem1SmallMod}
Assume that $\delta_2:=10^{-3}$, $m\geq 0$, $s\in[2^{m}-1,2^{m+1}]$, and define
\begin{equation}\label{BAA}
\widehat{R[f,g]}(\xi):=\int_{\mathbb{R}^2}a(\xi,\eta)e^{is\Phi(\xi,\eta)}\varphi(2^{(1-\delta_2/2)m}\Phi(\xi,\eta))\widehat{f}(\xi-\eta)\widehat{g}(\eta)d\eta.
\end{equation}
Assume that $a$ satisfies the bounds \eqref{BdOnAL2Lem} and $k\geq 0$. Then
\begin{equation}\label{AA}
\big\|R\big[P_{\leq k}f,P_{\leq k}g\big]\big\|_{L^2}\lesssim 2^{12k}2^{-(1+\delta_2/2)m}\Vert f\Vert_{Z_1^b\cap H^{N_1/2}_\Omega}\Vert g\Vert_{L^2}.
\end{equation}
\end{lemma}

\begin{proof} Clearly we may assume $5k+\D\leq m$, $\|g\|_{L^2}=1$, and $\Vert f\Vert_{Z_1^b\cap H^{N_1/2}_\Omega}=1$. We decompose
\begin{equation*}
f=\sum_{(k_1,j_1)\in\mathcal{J}}f_{j_1,k_1}=\sum_{(k_1,j_1)\in\mathcal{J}}P_{[k_1-2,k_1+2]}Q_{j_1k_1}f.
\end{equation*}
Using \eqref{FLinftybdDER} and Lemma \ref{L2EstLem}, if $j_1\leq m/2$ then
\begin{equation}\label{Az1}
\big\|R\big[P_{\leq k}f_{j_1,k_1},P_{\leq k}g\big]\big\|_{L^2}\lesssim 2^{-\delta_2m/10}2^{12k}2^{-(1+\delta_2/2)m}\Vert g\Vert_{L^2}.
\end{equation}

On the other hand, if $j_1\geq m/2$ and $k_1\leq -m/10$ then $\|\widehat{f_{j_1,k_1}}\|_{L^\infty}\lesssim 2^{-21\delta k_1}2^{-(j_1+k_1)/4}$ (see \eqref{FLinftybd}). Using \eqref{cas4}, \eqref{cas5.55}, and Schur's lemma,
\begin{equation}\label{Az2}
\begin{split}
\big\|R\big[P_{\leq k}f_{j_1,k_1},P_{\leq k}g\big]\big\|_{L^2}&\lesssim 2^{k_1/4}2^{10k}2^{-m+\delta_2m}\|\widehat{f_{j_1,k_1}}\|_{L^\infty}\Vert g\Vert_{L^2}\lesssim 2^{-j_1/100}2^{12k}2^{-(1+\delta_2/2)m}.
\end{split}
\end{equation}

Finally, if $j_1\geq m/2$ and $k_1\in[-m/10,m/5]$ then we further decompose, as in \eqref{Alx100},
\begin{equation*}
f_{j_1,k_1}=\sum_{n_1\in\{0,\ldots,j_1+1\}}f_{j_1,k_1,n_1},\qquad \widehat{f_{j_1,k_1,n_1}}(\xi):=\varphi_{-n}^{[-j-1,0]}(\Psi^+_b(\xi))\widehat{f_{j_1,k_1}}(\xi).
\end{equation*}
Notice that $\|\widehat{f_{j_1,k_1,n_1}}\|_{L^\infty}\lesssim 2^{21\delta j_1}2^{-(j_1-n_1)/4}$, see \eqref{FLinftybd}. Using \eqref{cas4}, \eqref{cas5.5}, and Schur's lemma,
\begin{equation}\label{Az3}
\begin{split}
\big\|R\big[P_{\leq k}f_{j_1,k_1,n_1},P_{\leq k}g\big]\big\|_{L^2}&\lesssim 2^{n_1/8}2^{10k}2^{-m+\delta_2m}\|\widehat{f_{j_1,k_1,n_1}}\|_{L^\infty}\Vert g\Vert_{L^2}\\
&\lesssim 2^{-j_1/100}2^{12k}2^{-(1+\delta_2/2)m}.
\end{split}
\end{equation}
The desired conclusion follows from \eqref{Az1}--\eqref{Az3}.
\end{proof}

We prove now an estimate on certain trilinear integrals.

\begin{lemma}\label{MainL2Lem}
Assume that $\delta_2=10^{-3}$, $m\geq 1$, and $f,g,h\in C^1([2^{m-2},2^{m+2}]:L^2)$ satisfy
\begin{equation}\label{VFMainLemHyp}
\begin{split}
\Vert f(s)&\Vert_{Z_1^b\cap H^{N_1/2}_{\Omega}}+2^{(1-\delta_2/8)m}\Vert(\partial_sf)(s)\Vert_{L^2}\\
&+\Vert g(s)\Vert_{L^2}+\Vert h(s)\Vert_{L^2}+2^{(1-\delta_2/8)m}\big[\Vert(\partial_sg)(s)\Vert_{L^2}+\Vert(\partial_sh)(s)\Vert_{L^2}\big]\lesssim 1,
\end{split}
\end{equation}
for any $s\in [2^{m-2},2^{m+2}]$. Moreover, assume that
\begin{equation}\label{ImpDecay}
\begin{split}
&\Vert e^{-i(s+\lambda)\Lambda_\mu}f(s)\Vert_{L^\infty}\lesssim 2^{-(1-\delta_2/8)m},\qquad\partial_sf=F_2+F_\infty,\\
&\Vert e^{-i(s+\lambda)\Lambda_\mu}P_kF_{\infty}(s)\Vert_{L^\infty}\lesssim 2^{-2m+\delta_2m/4},\quad\Vert P_kF_{2}(s)\Vert_{L^2}\lesssim 2^{-3m/2+\delta_2m/8},
\end{split}
\end{equation}
for any $\lambda\in\mathbb{R}$ with $|\lambda|\leq 2^{m(1-\delta_2/10)}$ and any $k\in\mathbb{Z}$. Let 
\begin{equation*}
I[f,g,h]=\int_{\mathbb{R}}n(s)\int_{\mathbb{R}^2\times\mathbb{R}^2} a(\xi,\eta)e^{is\Phi(\xi,\eta)}\widehat{f}(\xi-\eta,s)\widehat{g}(\eta,s)\overline{\widehat{h}(\xi,s)}d\eta d\xi ds
\end{equation*}
where $n$ is a $C^1$ function supported in $[2^{m-2},2^{m+2}]$,
\begin{equation}\label{AssumptA}
\int_{\mathbb{R}}|n'(s)|\,ds\lesssim 1,\qquad\Vert\mathcal{F}^{-1}a\Vert_{L^1(\mathbb{R}^4)}\lesssim 1.
\end{equation}
Then, for any $k,k_1,k_2\in\mathbb{Z}$,
\begin{equation}\label{EstVF}
\big|I[P_{k_1}f,P_{k_2}g,P_kh]\big|\lesssim 2^{12\max(k,k_1,k_2,0)}2^{-\delta_2m/8}.
\end{equation}
\end{lemma}

\begin{proof} Let $\overline{k}:=\max(k,k_1,k_2,0)$. We may assume that $m\geq \D$ and $f=P_{k_1}f,\,g=P_{k_2}g,\,h=P_{k}h$. We may also assume that $a(\xi,\eta)=\varphi_{\leq \overline{k}}(\xi)\varphi_{\leq \overline{k}}(\eta)$; otherwise we write
\begin{equation*}
a(\xi,\eta)=\int_{\mathbb{R}^2\times\mathbb{R}^2}k(x,y)e^{ix\cdot\xi}e^{iy\cdot\eta},
\end{equation*}
where $\|k\|_{L^1(\mathbb{R}^4)}\lesssim 1$, and notice that the complex exponentials $e^{ix\cdot\xi}$ and $e^{iy\cdot\eta}$ can be combined with the functions $g$ and $h$, without affecting the hypothesis \eqref{VFMainLemHyp}. We define
\begin{equation*}
\begin{split}
\widehat{R_\tau[f,g]}(\xi)&:=\int_{\mathbb{R}^2}a(\xi,\eta)e^{is\Phi(\xi,\eta)}\varphi_{\tau}^{[\tau_-,0]}(\Phi(\xi,\eta))\widehat{f}(\xi-\eta)\widehat{g}(\eta)d\eta,
\end{split}
\end{equation*}
where $\tau_-:=\lfloor-(1-\delta_2/2)m\rfloor\leq \tau\leq 0$. We use Lemma \ref{L2Lem1SmallMod} to estimate the contribution of $\tau=\tau_-$,
\begin{equation*}
\left\vert\int_{\mathbb{R}}n(s)\langle R_{\tau_+}[f(s),g(s)],h(s)\rangle ds\right\vert\lesssim 2^{12\max(k,k_1,k_2,0)}2^{-\delta_2m/2}.
\end{equation*}

For $\tau_-<\tau<0$, we integrate by parts in time to get the missing decay. We find that
\begin{equation*}
\int_{\mathbb{R}}n(s)\langle R_\tau[f(s),g(s)],h(s)\rangle ds=C(I+II+III+IV),
\end{equation*}
where
\begin{equation*}
I=\int_{\mathbb{R}}n'(s)\int_{\mathbb{R}^4}a(\xi,\eta)e^{is\Phi(\xi,\eta)}\frac{\varphi_{\tau}(\Phi(\xi,\eta))}{\Phi(\xi,\eta)}\widehat{f}(\xi-\eta,s)\widehat{g}(\eta,s)\overline{\widehat{h}(\xi,s)}d\eta d\xi ds,
\end{equation*}
\begin{equation*}
II=\int_{\mathbb{R}}n(s)\int_{\mathbb{R}^4}a(\xi,\eta)e^{is\Phi(\xi,\eta)}\frac{\varphi_{\tau}(\Phi(\xi,\eta))}{\Phi(\xi,\eta)}(\partial_s\widehat{f})(\xi-\eta,s)\widehat{g}(\eta,s)\overline{\widehat{h}(\xi,s)} d\eta d\xi ds,
\end{equation*}
\begin{equation*}
III=\int_{\mathbb{R}}n(s)\int_{\mathbb{R}^4}a(\xi,\eta)e^{is\Phi(\xi,\eta)}\frac{\varphi_{\tau}(\Phi(\xi,\eta))}{\Phi(\xi,\eta)}\widehat{f}(\xi-\eta,s)(\partial_s\widehat{g})(\eta,s) \overline{\widehat{h}(\xi,s)}d\eta d\xi ds,
\end{equation*}
\begin{equation*}
IV=\int_{\mathbb{R}}n(s)\int_{\mathbb{R}^4}a(\xi,\eta)e^{is\Phi(\xi,\eta)}\frac{\varphi_{\tau}(\Phi(\xi,\eta))}{\Phi(\xi,\eta)}\widehat{f}(\xi-\eta,s)\widehat{g}(\eta,s) \overline{(\partial_s\widehat{h})(\xi,s)}d\eta d\xi ds.
\end{equation*}
Using the Fourier transform, we see that
\begin{equation}\label{Alx84}
\frac{\varphi_0(2^{-\tau}\Phi(\xi,\eta))}{\Phi(\xi,\eta)}=c2^{-\tau}\cdot 2^{\tau}\int_{\mathbb{R}}e^{i\lambda\Phi(\xi,\eta)}P(2^{\tau}\lambda)d\lambda,\qquad P=\mathcal{F}(\varphi_0(x)/x)\in\mathcal{S}.
\end{equation}
Thus
\begin{equation*}
\begin{split}
I&=c\int_{\mathbb{R}^2}n'(s)P(2^{\tau}\lambda)\int_{\mathbb{R}^4}a(\xi,\eta)e^{i(s+\lambda)\Phi(\xi,\eta)}\widehat{f}(\xi-\eta,s)\widehat{g}(\eta,s)\overline{\widehat{h}(\xi,s)}\,d\eta d\xi dsd\lambda,\\
III&=c\int_{\mathbb{R}^2}n(s)P(2^{\tau}\lambda)\int_{\mathbb{R}^4}a(\xi,\eta)e^{i(s+\lambda)\Phi(\xi,\eta)}\widehat{f}(\xi-\eta,s)(\partial_s\widehat{g})(\eta,s)\overline{\widehat{h}(\xi,s)}\,d\eta d\xi dsd\lambda,\\
IV&=c\int_{\mathbb{R}^2}n(s)P(2^{\tau}\lambda)\int_{\mathbb{R}^4}a(\xi,\eta)e^{i(s+\lambda)\Phi(\xi,\eta)}\widehat{f}(\xi-\eta,s)\widehat{g}(\eta,s)\overline{(\partial_s\widehat{h})(\xi,s)}\,d\eta d\xi dsd\lambda.\\
\end{split}
\end{equation*}
Using \eqref{VFMainLemHyp}--\eqref{AssumptA}, and Lemma \ref{L1easy}, we estimate
\begin{equation*}
\begin{split}
\big|I\big|&\lesssim\int_{\mathbb{R}}n'(s) \vert P(2^{\tau}\lambda)\vert\cdot\Vert e^{-i(s+\lambda)\Lambda_\mu}f(s)\Vert_{L^\infty}\Vert g(s)\Vert_{L^2}\Vert h(s)\Vert_{L^2} ds d\lambda\lesssim 2^{-\delta_2 m/4},\\
\big|III\big|&\lesssim \int_{\mathbb{R}}n(s) \vert P(2^{\tau}\lambda)\vert\cdot\Vert e^{-i(s+\lambda)\Lambda_\mu}f(s)\Vert_{L^\infty}\Vert (\partial_sg)(s)\Vert_{L^2}\Vert h(s)\Vert_{L^2}  ds d\lambda\lesssim 2^{-\delta_2 m/4},\\
\big|IV\big|&\lesssim \int_{\mathbb{R}}n(s) \vert P(2^{\tau}\lambda)\vert\cdot\Vert e^{-i(s+\lambda)\Lambda_\mu}f(s)\Vert_{L^\infty}\Vert g(s)\Vert_{L^2}\Vert (\partial_sh)(s)\Vert_{L^2}  ds d\lambda\lesssim 2^{-\delta_2 m/4}.
\end{split}
\end{equation*}
In these estimates we have also used $2^\tau\gtrsim 2^{-m(1-\delta_2/2)}$ and the fact that $P$ has rapid decay. 

To estimate $|II|$ decompose $\partial_s\widehat{f}=\widehat{F_2}+\widehat{F_\infty}$ according to \eqref{ImpDecay}. The contribution of $\widehat{F_\infty}$ can be estimated in the same way as above, using \eqref{Alx84} and $\Vert g(s)\Vert_{L^2}+\Vert h(s)\Vert_{L^2}\lesssim 1$. On the other hand, the contribution of $\widehat{F_2}$ can be estimated directly in the Fourier space in the integral defining $II$, using Schur's lemma and Proposition \ref{volume} (i).

The contribution $R_0$ can be estimated in exactly the same way. This finishes the proof.
\end{proof}

We can now complete the proof of our main lemma.

\begin{proof}[Proof of Lemma \ref{BootstrapEE1}] We examine the definition of $\mathcal{SS}_\sigma$ and $\mathcal{B}_{SS}$ in \eqref{SSMult} and \eqref{BSS}. It suffices to prove that if $W^1=W^{p_1}_{\pm\sigma_1}$, $W^2=W^{p_2}_{\pm\sigma_2}$ as in \eqref{nor20}, $\sigma_1,\sigma_2\in\{e,b\}$, $p_1,p_2\in[0,N_1]$, $f\in\big\{\Omega^a U_{\mu}:\,\mu\in\{\pm e,\pm b\},\,a\leq N_1/2\big\}$, 
and $m$ is a symbol that gains one derivative, i.e.
\begin{equation}\label{log6.0}
\Vert m\Vert_{S^\infty_{kk_1k_2}}\lesssim 2^{-\max(k,k_1,k_2,0)}(1+2^{10\min(k,k_1,k_2)}),\qquad\text{ for any }k,k_1,k_2\in\mathbb{Z},
\end{equation}
then, for any $m\geq 0$,
\begin{equation}\label{log6}
\Big|\int_{\mathbb{R}}q_m(s)\int_{\mathbb{R}^2\times\mathbb{R}^2}m(\xi,\eta)\overline{\widehat{W^2}}(\xi,s)\widehat{W^1}(\eta,s)\widehat{f}(\xi-\eta,s)\,d\xi d\eta ds\Big|\lesssim \eps_1^3 2^{-\delta^2m}.
\end{equation}

To prove this, for any let, for any $k,k_1,k_2\in\mathbb{Z}$,
\begin{equation*}
\mathfrak{I}_{k,k_1,k_2}:= \int_{\mathbb{R}\times\mathbb{R}^4}q_m(s)m(\xi,\eta)\overline{\widehat{P_{k}W^2}}(\xi,s)\widehat{P_{k_2}W^1}(\eta,s)\widehat{P_{k_1}f}(\xi-\eta,s)d\xi d\eta ds.
\end{equation*}
Using Proposition \ref{BootSimple} and $L^2\ast L^2\ast L^\infty$ estimates as in Lemma \ref{L1easy} we have
\begin{equation*}
|\mathfrak{I}_{k,k_1,k_2}|\lesssim \eps_1^32^{-\max(k,k_1,k_2,0)}2^{24\delta m}\min(1,2^{m+\min(k,k_1,k_2)}).
\end{equation*}
Therefore, it suffices to consider the case when $\max(k,k_1,k_2,0)$ is small relative to $m$, i.e.
\begin{equation}\label{endpr1}
-2m\le k,k_1,k_2\le 25\delta m.
\end{equation}
This reduction to relatively small frequencies is the crucial gain coming from the strongly semilinear structure of symbols \eqref{log6.0}.

It the case \eqref{endpr1} we use Lemma \ref{MainL2Lem} and the bounds in Proposition \ref{BootSimple} to verify the hypothesis. Since $\delta\leq 4\cdot 10^{-3}\delta_2$, it follows that $|\mathfrak{I}_{k,k_1,k_2}|\lesssim \eps_1^32^{-\delta m}$, and the conclusion \eqref{log6} follows.
\end{proof}

\section{Dispersive control I: the functions $\partial_t V_e$ and $\partial_t V_b$}\label{partialt}

In the next two sections we prove our main dispersive estimates:

\begin{proposition}\label{BootstrapZnorm}
Assume that $(\rho,u,\widetilde{E},\widetilde{b})$ is a solution to \eqref{system2}-\eqref{condition} on a time interval $[0,T]$, $T\ge 1$, satisfying the hypothesis \eqref{bootstrap1}-\eqref{bootstrap2} in Proposition \ref{bootstrap}. Then, for any $t\in[0,T]$,
\begin{equation}\label{bootstrapZ}
\Vert (V_e(t),V_b(t))\Vert_{Z}\lesssim\epsilon_0+\epsilon_1^{3/2}.
\end{equation}
\end{proposition}

\subsection{The Duhamel formula}

We will work mostly with the unknowns $(U_e,U_b)$ and $(V_e,V_b)$ defined in \eqref{variables4}. For these new unknowns, we see that \eqref{system2}-\eqref{condition} is equivalent to the system
\begin{equation}\label{ener4}
(\partial_t+i\Lambda_e)U_e=\mathcal{N}_e,\qquad (\partial_t+i\Lambda_b)U_b=\mathcal{N}_b,
\end{equation}
where, using \eqref{Exph}
\begin{equation}\label{ener4s}
\begin{cases}
\mathcal{N}_{e}&:=(1/2)|\nabla|\big(u_{1}^{2}+u_{2}^{2}\big)-i\Lambda_{e}\big(R_{1}(\rho u_{1})+R_{2}(\rho u_{2})\big)+(\kappa/2)\vert\nabla\vert(\rho^2)+\vert\nabla\vert H_3(\rho),\\
\mathcal{N}_{b}&:=-i\big(R_{1}(\rho u_{2})-R_{2}(\rho u_{1})\big),
\end{cases}
\end{equation}
where $H_3=H_3$ denotes the cubic antiderivative of $h_2$. The variables $\rho,u$ can be expressed in terms of the complex variables $U_e,U_b$,
\begin{equation}\label{variables5}
\rho=|\nabla|\Lambda_{e}^{-1}\Im(U_e),\quad u_{1}=-R_{1}\Re(U_e)+R_{2}\Lambda_{b}^{-1}\Re(U_b),\quad u_{2}=-R_{2}\Re(U_e)-R_{1}\Lambda_{b}^{-1}\Re(U_b).
\end{equation}

The system (\ref{ener4})-(\ref{variables5}) can be rewritten as
\begin{equation}\label{system8}
(\partial_{t}+i\Lambda_{\sigma})U_{\sigma}=\sum_{\mu,\nu\in\mathcal{P}}\mathcal{N}_{\sigma\mu\nu}(U_{\mu},U_{\nu})+\delta_{\sigma e}\vert\nabla\vert H_3(\rho)
\end{equation} for $\sigma\in\{e,b\}$, where the quadratic nonlinearities are defined by
\begin{equation}\label{system9}
\left(\mathcal{F}\mathcal{N_{\sigma\mu\nu}}(f,g)\right)(\xi)=\int_{\mathbb{R}^{2}}\mathfrak{m}_{\sigma\mu\nu}(\xi,\eta)\widehat{f}(\xi-\eta)\widehat{g}(\eta)\,d\eta.
\end{equation} 
The multipliers $\mathfrak{m}_{\sigma\mu\nu}$ are sums of functions of the form $m(\xi)m'(\xi-\eta)m''(\eta)$, satisfying suitable symbol-type estimates. We define $V_{\sigma}$ as in \eqref{variables4}, $V_{\sigma}(t)=e^{it\Lambda_{\sigma}}U_{\sigma}(t)$. The Duhamel formula is
\begin{equation}\label{duhamelDER}
\begin{split}
(\partial_t\widehat{V_{\sigma}})(\xi,s)&=[\partial_t\widehat{V_{\sigma}}]^{(2)}(\xi,s)+[\partial_t\widehat{V_{\sigma}}]^{(3)}(\xi,s),\\
[\partial_t\widehat{V_{\sigma}}]^{(2)}(\xi,s)&:=\sum_{\mu,\nu\in\mathcal{P}}\int_{\mathbb{R}^2}e^{is\Phi_{\sigma\mu\nu}(\xi,\eta)}\mathfrak{m}_{\sigma\mu\nu}(\xi,\eta)\widehat{V_{\mu}}(\xi-\eta,s)\widehat{V_{\nu}}(\eta,s)\,d\eta\\
[\partial_t\widehat{V_{\sigma}}]^{(3)}(\xi,s)&:=\delta_{\sigma e}\vert\xi\vert e^{is\Lambda_e}\widehat{H_3(\rho)}(\xi,s),
\end{split}
\end{equation}
or, in integral form,
\begin{equation}\label{duhamel}
\begin{split}
\widehat{V_{\sigma}}(\xi,t)&=\widehat{V_{\sigma}}(\xi,0)+\sum_{\mu,\nu\in\mathcal{P}}\int_{0}^{t}\int_{\mathbb{R}^2}e^{is\Phi_{\sigma\mu\nu}(\xi,\eta)}\mathfrak{m}_{\sigma\mu\nu}(\xi,\eta)\widehat{V_{\mu}}(\xi-\eta,s)\widehat{V_{\nu}}(\eta,s)\,d\eta ds\\
&+\delta_{\sigma e}\vert\xi\vert\int_0^te^{is\Lambda_e(\xi)}\widehat{H_3}(\rho)(\xi,s)ds .
\end{split}
\end{equation}

The rotation vector-field $\Omega$ acts according to the identity
\begin{equation*}
\begin{split}
\Omega_\xi&[\partial_t\widehat{V_{\sigma}}]^{(2)}(\xi,s)=\sum_{\mu,\nu\in\mathcal{P}}\int_{\mathbb{R}^2}(\Omega_\xi+\Omega_\eta)\big[e^{is\Phi_{\sigma\mu\nu}(\xi,\eta)}\mathfrak{m}_{\sigma\mu\nu}(\xi,\eta)\widehat{V_{\mu}}(\xi-\eta,s)\widehat{V_{\nu}}(\eta,s)\big]\,d\eta\\
&=\sum_{\mu,\nu\in\mathcal{P}}\sum_{a_1+a_2+a_3=1}\int_{\mathbb{R}^2}e^{is\Phi_{\sigma\mu\nu}(\xi,\eta)}(\Omega_\xi+\Omega_\eta)^{a_1}\mathfrak{m}_{\sigma\mu\nu}(\xi,\eta)(\Omega^{a_2}\widehat{V_{\mu}})(\xi-\eta,s)(\Omega^{a_3}\widehat{V_{\nu}})(\eta,s)\,d\eta.
\end{split}
\end{equation*}
Therefore, for $1\leq a\leq N_1$, letting $\mathfrak{m}_{\sigma\mu\nu}^{a_1}=(\Omega_\xi+\Omega_\eta)^{a_1}\mathfrak{m}_{\sigma\mu\nu}$, we have
\begin{equation}\label{DuhamelDER2}
\begin{split}
\Omega^a_\xi[\partial_t\widehat{V_{\sigma}}]^{(2)}(\xi,s)=\sum_{\mu,\nu\in\mathcal{P}}\sum_{a_1+a_2+a_3=a}\int_{\mathbb{R}^2}&e^{is\Phi_{\sigma\mu\nu}(\xi,\eta)}\mathfrak{m}_{\sigma\mu\nu}^{a_1}(\xi,\eta)\\
&\times(\Omega^{a_2}\widehat{V_{\mu}})(s,\xi-\eta)(\Omega^{a_3}\widehat{V_{\nu}})(s,\eta)\,d\eta.
\end{split}
\end{equation}
In integral form this becomes
\begin{equation}\label{duhamel2}
\begin{split}
\Omega^a_\xi\widehat{V_{\sigma}}&(\xi,t)=\Omega^a_\xi\widehat{V_{\sigma}}(\xi,0)+\sum_{\mu,\nu\in\mathcal{P}}\sum_{a_1+a_2+a_3=a}\int_0^t\int_{\mathbb{R}^2}e^{is\Phi_{\sigma\mu\nu}(\xi,\eta)}\mathfrak{m}_{\sigma\mu\nu}^{a_1}(\xi,\eta)\\
&\times(\Omega^{a_2}\widehat{V_{\mu}})(s,\xi-\eta)(\Omega^{a_3}\widehat{V_{\nu}})(s,\eta)\,d\eta ds+\delta_{\sigma e}\vert\xi\vert \int_0^te^{is\Lambda_e}\Omega_\xi^a\widehat{H_3(\rho)}(\xi,s)ds.
\end{split}
\end{equation}

\subsection{Control of the time derivatives}

In this section we prove several bounds on the functions $\partial_t V_\sigma$. These bounds rely on the assumptions \eqref{bootstrap2} and the Duhamel formula \eqref{duhamel}. Given $\Phi_{\sigma\mu\nu}$ as in \eqref{phasedef} we define
\begin{equation}\label{Alx200}
\Xi_{\mu\nu}(\xi,\eta):=(\nabla_\eta\Phi_{\sigma\mu\nu})(\xi,\eta)=-(\nabla\Lambda_\mu)(\eta-\xi)-(\nabla\Lambda_\nu)(\eta),\qquad  \Xi:\mathbb{R}^2\times\mathbb{R}^2\to\mathbb{R}.
\end{equation}

We start with a preliminary result.

\begin{lemma}\label{dtfLemPrelim}

Assume the hypothesis of Proposition \ref{bootstrap}, $m\geq 0$, $s\in [2^{m}-1,2^{m+1}]$, $k\in\mathbb{Z}$.
Then
\begin{equation}\label{Alx1}
\big\|P_{\leq k}(\partial_tV_{\sigma})(s)\big\|_{H^{N_0}}+\sum_{a\leq N_1}\big\|P_{\leq k}\Omega^a(\partial_tV_{\sigma})(s)\big\|_{L^2}\lesssim \eps_1^22^{k_+}2^{-m+22\delta m},
\end{equation}
\begin{equation}\label{dtFTotalBds}
\sum_{a\leq N_1/2}\|e^{-is\Lambda_\sigma}P_k\Omega^a(\partial_tV_{\sigma})(s)\|_{L^{\infty}}\lesssim\eps_1^2\min\{2^{-2m+43\delta m},2^{2k}\}.
\end{equation}
In addition, for any $a\in[0,N_1/2]\cap\mathbb{Z}$ we may decompose
\begin{equation}\label{DecdtfPrelim}
\Omega^a(\partial_tV_\sigma)(s)=\epsilon_1^2\{\widetilde{f_C}(s)+\widetilde{f_{NC}}(s)\},
\end{equation}
where, with $\Psi_{\sigma\mu\nu}$ as \eqref{psidag} and for any $k\in\mathbb{Z}$,
\begin{equation}\label{dtfType1Prelim}
\begin{split}
&\widehat{\widetilde{f_C}}(\xi,s)
=\sum_{\mu+\nu\ne0}e^{is\Psi_{\sigma\mu\nu}(\xi)}\sum_{0\le q\le m/2-10\delta m}g^q_{\sigma\mu\nu}(\xi,s),\\
&\varphi_k(\xi)g^q_{\sigma\mu\nu}(\xi,s)=0\quad\text{ if }\,q\geq 1\text{ and }k\leq-\D\quad\text{ or }\quad\text{ if }\,k\notin[-m/2+\delta^2 m/5,\delta^2 m/5],\\
&|\varphi_k(\xi)D^\alpha_\xi g^q_{\sigma\mu\nu}(\xi,s)|\lesssim 2^{-42\delta k_-}2^{-m+3\delta m}2^{-q+42\delta q}2^{(m/2+q+2\delta^2m)\vert\alpha\vert},\\
&\Vert \varphi_{k}\cdot\partial_sg^q_{\sigma\mu\nu}(s)\Vert_{L^\infty}\lesssim 2^{(6\delta-2)m}2^{q+42\delta q}\qquad\text{ if }\,k\geq -\D,
\end{split}
\end{equation}
and
\begin{equation}\label{dtfType2Prelim}
\begin{split}
&\Vert P_k\widetilde{f_{NC}}(s)\Vert_{L^2}\lesssim 2^{-3m/2+60\delta m}+2^{-m}(1+2^{m+k_--50\delta m})^{-1},\\
&\Vert \widehat{P_k\widetilde{f_{NC}}}(s)\Vert_{L^\infty}\lesssim 2^{60\delta m}\big(1+2^{m+k_-}\big)^{-1}.
\end{split}
\end{equation}
Moreover
\begin{equation}\label{OmegadtfPrelim}
\sup_{0\le b\le N_1/4}\big\{\Vert \Omega^{b}g^q_{\sigma\mu\nu}(s)\Vert_{L^2}+\Vert\Omega^b\widetilde{f_{NC}}(s)\Vert_{L^2}\big\}\lesssim 1.
\end{equation}
In particular, for any $a\in[0,N_1/2]\cap\mathbb{Z}$ and $k\in\mathbb{Z}$
\begin{equation}\label{dtFTotalBds2Prelim}
\Vert P_k\Omega^a(\partial_tV_\sigma)(s)\Vert_{L^2}\lesssim \epsilon_1^22^{-m+3\delta m+\delta^2m}.
\end{equation}
\end{lemma}

\begin{proof}

Using Lemma \ref{dtVCubic}, it suffices to treat $[\partial_tV_\sigma]^{(2)}$. The bounds \eqref{Alx1} and \eqref{dtFTotalBds} follow easily from the formula \eqref{DuhamelDER2} and the $L^\infty$ bound \eqref{LinftyBd2}. The desired decomposition \eqref{DecdtfPrelim} corresponds to the decomposition of $\partial_tV_\sigma$ into a coherent component $\widetilde{f_C}$ and a non coherent component $\widetilde{f_{NC}}$. We use \eqref{DuhamelDER2} and Lemma \ref{lem:CZ}. In addition, we examine the nonlinearities $\mathcal{N}_e$ and $\mathcal{N}_b$ in \eqref{ener4s} and notice that the quadratic interaction of $U_{\pm b}$ with itself is only present in the nonlinearity $\mathcal{N}_e$ and carries a gradient factor in front (a ``null structure"). 

It suffices to show the following: let $\mu,\nu\in\mathcal{P}$ and assume that
\begin{equation}\label{AssonFmunu}
\begin{split}
\Vert f^\mu\Vert_{H^{N_0/2}\cap Z_1^\mu\cap H^{N_1/2}_\Omega}+ \Vert f^\nu\Vert_{H^{N_0/2}\cap Z_1^\nu\cap H^{N_1/2}_\Omega}\le 1.
\end{split}
\end{equation}
Define $\beta_{\mu\nu}=1$ if $|\mu|=|\nu|=b$ and $\beta_{\mu\nu}=0$ otherwise. Define, for $\sigma\in\{e,b\}$,
\begin{equation}\label{Isigmamunu}
\begin{split}
\widetilde{I}^{\sigma\mu\nu}&:=\sum_{k,k_1,k_2\in\mathbb{Z}}2^{\beta_{\mu\nu}k_-}2^{\max\{k_1,k_2,0\}}P_kI^{\sigma\mu\nu}[P_{k_1}f^\mu,P_{k_2}f^\nu],\\
\mathcal{F}\left\{I^{\sigma\mu\nu}[f,g]\right\}(\xi,s)&:=\int_{\mathbb{R}^2}e^{is\Phi_{\sigma\mu\nu}(\xi,\eta)}\widehat{f}(\xi-\eta)\widehat{g}(\eta)d\eta.
\end{split}
\end{equation}
Then the bounds in Lemma \ref{dtfLemPrelim} hold for $\widetilde{I}^{\sigma\mu\nu}$ with $\epsilon_1=1$.

Using Lemma \ref{LinEstLem} and $L^2\times L^\infty$ estimates (using Lemma \ref{L1easy}), and recalling that $N_0=20/\delta^2$, it is easy to see that the contribution of $k,k_1,k_2$ with $\max(k_1,k_2)\geq \delta^2 m/5-\D$ or with $\min(k_1,k_2)\leq -3m$ is acceptable as required in \eqref{dtfType2Prelim}, as part of the component $\widetilde{f_{NC}}$. It remains to consider $P_kI^{\sigma\mu\nu}[P_{k_1}f^\mu,P_{k_2}f^\nu]$ uniformly in $-3m\le k_1,k_2\le \delta^2 m/5-\D$. We decompose
\begin{equation*}
\begin{split}
I^{\sigma\mu\nu}[P_{k_1}f^\mu,P_{k_2}f^\nu]&=\sum_{(k_1,j_1)\in\mathcal{J}}\sum_{(k_2,j_2)\in\mathcal{J}}I^{\sigma\mu\nu}[f^\mu_{j_1,k_1},f^\nu_{j_2,k_2}]=\sum_{(k_1,j_1)\in\mathcal{J}}\sum_{(k_2,j_2)\in\mathcal{J}}I^{\sigma\mu\nu}_{j_1,k_1,j_2,k_2}
\end{split}
\end{equation*}
where, as in \eqref{Alx100},
\begin{equation}\label{Alx8}
f^\mu_{j_1,k_1}=P_{[k_1-2,k_1+2]}Q_{j_1k_1}f^\mu,\qquad f^\nu_{j_2,k_2}=P_{[k_2-2,k_2+2]}Q_{j_2k_2}f^\nu.
\end{equation}

Lemma \ref{LinEstLem} and Plancherel theorem show that if $j_2=\max\{j_1,j_2\}\ge m-10\delta m-\D^2$ then
\begin{equation*}
\begin{split}
\Vert I^{\sigma\mu\nu}_{j_1,k_1,j_2,k_2}\Vert_{L^2}&\lesssim \Vert e^{-is\Lambda_\mu}f^\mu_{j_1,k_1}\Vert_{L^\infty}\Vert f^\nu_{j_2,k_2}\Vert_{L^2}\lesssim 2^{-(1-21\delta)m}2^{(\delta-1/2)j_2},\\
\Vert \mathcal{F}I^{\sigma\mu\nu}_{j_1,k_1,j_2,k_2}\Vert_{L^\infty}&\lesssim \Vert \widehat{f^\mu_{j_1,k_1}}\Vert_{L^\infty}\Vert \widehat{f^\nu_{j_2,k_2}}\Vert_{L^1}\lesssim 2^{-(1-45\delta)j_2}.
\end{split}
\end{equation*}
Summing in $j_1,j_2,k_1,k_2$, we obtain an acceptable contribution as part of the component $\widetilde{f_{NC}}$ in \eqref{dtfType2Prelim}. Therefore it suffices to obtain the desired decomposition for $I^{\sigma\mu\nu}_{j_1,k_1,j_2,k_2}$ uniformly in $k_1,k_2,j_1,j_2,m$ under the assumptions
\begin{equation}\label{Alx4}
-(1-10\delta)m+\D^2\le k_1,k_2\le\delta^2 m/5-\D,\qquad 0\le j_1\leq j_2\le (1-10\delta)m-\D^2.
\end{equation}

Let $\kappa_r:=2^{\delta^2m/2}\big(2^{-m/2}+2^{j_2-m}\big)$. With $\Xi_{\mu\nu}$ defined as in \eqref{Alx200}, we decompose
\begin{equation}\label{IBPdtfPrelim}
\begin{split}
I^{\sigma\mu\nu}_{j_1,k_1,j_2,k_2}&=\sum_{k\in\mathbb{Z}}\big[P_kI^{C}+P_kI^{NC}\big],\\
\widehat{I^{C}}(\xi)&:=\int_{\mathbb{R}^2}e^{is\Phi_{\sigma\mu\nu}(\xi,\eta)}\varphi(\kappa_r^{-1}\Xi_{\mu\nu}(\xi,\eta))\cdot\widehat{f^\mu_{j_1,k_1}}(\xi-\eta)\widehat{f^\nu_{j_2,k_2}}(\eta)d\eta,\\
\widehat{I^{NC}}(\xi)&:=\int_{\mathbb{R}^2}e^{is\Phi_{\sigma\mu\nu}(\xi,\eta)}(1-\varphi(\kappa_r^{-1}\Xi_{\mu\nu}(\xi,\eta)))\cdot\widehat{f^\mu_{j_1,k_1}}(\xi-\eta)\widehat{f^\nu_{j_2,k_2}}(\eta)d\eta.
\end{split}
\end{equation}
Integration by parts using Lemma \ref{tech5} and Lemma \ref{LinEstLem} shows that, for any $k\in\mathbb{Z}$,
\begin{equation}\label{Alx104}
\Vert \widehat{P_kI^{NC}}\Vert_{L^\infty}+\Vert P_kI^{NC}\Vert_{L^2}\lesssim 2^{-2m},
\end{equation}
which gives an acceptable contribution as part of $\widetilde{f_{NC}}$.

To understand the contributions of the functions $P_kI^C$ we consider several cases. 

{\bf{Case 1.}} We assume first that \eqref{Alx4} holds and, in addition,
\begin{equation}\label{Alx105}
\mu+\nu\ne 0\qquad\text{ and }\qquad 2^k\geq 2^{\delta m}\big(2^{-m/2}+2^{j_2-m}\big).
\end{equation}
This is the main case which produces the coherent contribution $\widetilde{f_C}$. We may assume that $2^{\min\{k_1,k_2\}}\gtrsim 2^{k_-}$ and consider the functions $\Psi_{\sigma\mu\nu}$, see Proposition \ref{spaceres}. For simplicity of notation, let $\Phi=\Phi_{\sigma\mu\nu}$, $\Psi=\Psi_{\sigma\mu\nu}$, and $\Xi=\Xi_{\mu\nu}$ in the rest of the proof. If $j_2\le m/2$ then let
\begin{equation}\label{Alx97.1}
g_{\sigma\mu\nu}(\xi):=\int_{\mathbb{R}^2}e^{is\left[\Phi(\xi,\eta)-\Psi(\xi)\right]}\varphi(\kappa_r^{-1}\Xi(\xi,\eta))\widehat{f^\mu_{j_1,k_1}}(\xi-\eta)\widehat{f^\nu_{j_2,k_2}}(\eta)d\eta.
\end{equation}
In view of Proposition \ref{spaceres} (iii), we notice that the support of the integral is included in the set $\{|\eta-p(\xi)|\lesssim \kappa_r2^{\delta^2m}\}$. In particular, using \eqref{FLinftybd},
\begin{equation*}
\Vert \varphi_k\cdot g_{\sigma\mu\nu}\Vert_{L^\infty}\lesssim 2^{2\delta^2m}\kappa_r^2\Vert\widehat{f^\mu_{j_1,k_1}}\Vert_{L^\infty}\Vert\widehat{f^\nu_{j_2,k_2}}\Vert_{L^\infty}\lesssim  2^{-(1-3\delta)m}2^{-42\delta k_-}.
\end{equation*}
In the support of the integral, using \eqref{cas10},
\begin{equation}\label{DerxiPhase}
\nabla_\xi\left[s(\Phi(\xi,\eta)-\Psi(\xi))\right]=s\left[\nabla_\xi\Phi(\xi,\eta)-\nabla_\xi\Phi(\xi,p(\xi))\right]\lesssim \vert s\vert \vert\eta-p(\xi)\vert\lesssim 2^m\kappa_r2^{\delta^2m},
\end{equation}
so we obtain an acceptable contribution as in \eqref{dtfType1Prelim} (corresponding to $q=0$). 

On the other hand, if
\begin{equation*}
m/2\le j_2\le (1-10\delta)m\quad\text{ and }\quad k\leq-\D
\end{equation*}
then we may assume that $2^{k_1}+2^{k_2}\ll 1$ and we can place the contribution in $\widetilde{f_{NC}}$:
\begin{equation*}
\begin{split}
\big\|P_kI^C\big\|_{L^2}&\lesssim 2^{-2m}+\Vert e^{-is\Lambda_\mu}f^\mu_{j_1,k_1}\Vert_{L^\infty}\Vert f^\nu_{j_2,k_2}\Vert_{L^2}\lesssim 2^{-3m/2+40\delta m},\\
\big\|\widehat{P_kI^C}\big\|_{L^\infty}&\lesssim 2^{\delta^2m}\kappa_r\Vert\widehat{f^\mu_{j_1,k_1}}\Vert_{L^\infty}\Vert\widehat{f^\nu_{j_2,k_2}}\Vert_{L^2}\lesssim  2^{-m+60\delta m}.
\end{split}
\end{equation*}

Finally, if
\begin{equation}\label{Alx104.1}
m/2\le j_2\le (1-10\delta)m\quad\text{ and }\quad k\geq-\D
\end{equation}
we may further decompose, with $\kappa_\theta:=2^{2\delta^2m-m/2}$,
\begin{equation*}
\begin{split}
I^C&=I^{||}+I^{\perp},\\
\widehat{I^{||}}(\xi)&:=\int_{\mathbb{R}^2}e^{is\Phi(\xi,\eta)}\varphi(\kappa_r^{-1}\Xi(\xi,\eta))\varphi(\kappa_\theta^{-1}\Omega_\eta\Phi(\xi,\eta))\widehat{f^\mu_{j_1,k_1}}(\xi-\eta)\widehat{f^\nu_{j_2,k_2}}(\eta)d\eta,\\
\widehat{I^{\perp}}(\xi)&:=\int_{\mathbb{R}^2}e^{is\Phi(\xi,\eta)}\varphi(\kappa_r^{-1}\Xi(\xi,\eta))(1-\varphi(\kappa_\theta^{-1}\Omega_\eta\Phi(\xi,\eta)))\widehat{f^\mu_{j_1,k_1}}(\xi-\eta)\widehat{f^\nu_{j_2,k_2}}(\eta)d\eta.\\
\end{split}
\end{equation*}
Integration by parts using Lemma \ref{RotIBP} shows that
\begin{equation*}
\Vert \widehat{P_kI^\perp}\Vert_{L^\infty}+\Vert P_kI^\perp\Vert_{L^2}\lesssim 2^{-2m},
\end{equation*}
which gives a contribution as in \eqref{dtfType2Prelim}. Now, let $q=j_2-m/2$ and
\begin{equation}\label{Alx99.8}
g^q_{\sigma\mu\nu}(\xi):=\int_{\mathbb{R}^2}e^{is\left[\Phi(\xi,\eta)-\Psi(\xi)\right]}\varphi(\kappa_r^{-1}\Xi(\xi,\eta))\varphi(\kappa_\theta^{-1}\Omega_\eta\Phi(\xi,\eta))\widehat{f^\mu_{j_1,k_1}}(\xi-\eta)\widehat{f^\nu_{j_2,k_2}}(\eta)d\eta.
\end{equation}
We claim that these functions lead to acceptable contributions as in \eqref{dtfType1Prelim}. Indeed, the volume of $\eta$-integration is essentially a $\kappa_\theta\times\kappa_r$ box around $p(\xi)$ and we estimate, using \eqref{FL1bd},
\begin{equation}\label{Alx99.81}
\begin{split}
\Vert \varphi_k\cdot g^q_{\sigma\mu\nu}\Vert_{L^\infty}&\lesssim \Vert \widehat{f^\mu_{j_1,k_1}}\Vert_{L^\infty}\cdot 2^{-j_2+21\delta j_2}\kappa_\theta 2^{-19\delta(m-j_2)}2^{4\delta^2m}\lesssim 2^{-m+2.9\delta m}2^{-q+42\delta q}.
\end{split}
\end{equation}
The $\xi$ derivatives of $g^q_{\sigma\mu\nu}$ can be estimated in a similar way, using \eqref{DerxiPhase}. The bound on $\partial_sg^q_{\sigma\mu\nu}$ is proved later, see \eqref{Alx99.6}. This completes the analysis in this case.

{\bf{Case 2.}} We assume now that
\begin{equation}\label{Alx108}
\mu+\nu=0\qquad\text{ and }\qquad 2^k\geq 2^{\delta m}\big(2^{-m/2}+2^{j_2-m}\big).
\end{equation}
We notice that, in the support of the integral $\widehat{I^C}$,
\begin{equation*}
\vert\nabla_\eta\Phi(\xi,\eta)\vert=\vert\nabla\Lambda_\mu(\eta-\xi)-\nabla\Lambda_\mu(\eta)\vert\gtrsim (1+2^{3k_1}+2^{3k_2})^{-1} 2^k,
\end{equation*}
therefore $P_kI^C=0$ in this case.

{\bf{Case 3.}} We assume now that
\begin{equation}\label{Alx111}
\mu+\nu\neq 0\qquad\text{ and }\qquad 2^k\leq 2^{\delta m}\big(2^{-m/2}+2^{j_2-m}\big).
\end{equation}
We may assume that $2^{k_1}+2^{k_2}\lesssim 2^{\delta m}\big(2^{-m/2}+2^{j_2-m}\big)$ and estimate first, using also \eqref{Alx4},
\begin{equation*}
\begin{split}
\Vert \widehat{P_kI^C}\Vert_{L^\infty}&\lesssim 2^{2\delta^2m}\min\big\{\kappa_r\Vert \widehat{f^\mu_{j_1,k_1}}\Vert_{L^\infty}\Vert f^\nu_{j_2,k_2}\Vert_{L^2},\,\,\kappa_r^2\Vert \widehat{f^\mu_{j_1,k_1}}\Vert_{L^\infty}\Vert \widehat{f^\nu_{j_2,k_2}}\Vert_{L^\infty}\big\}\lesssim 2^{-m+50\delta m}.
\end{split}
\end{equation*}
This gives an acceptable $\widetilde{f_{NC}}$-type contribution if $k\leq -m/2+5\delta m$. On the other hand, if $k\geq -m/2+5\delta m$ then $j_2\geq m/2+4\delta m-2$ (using \eqref{Alx111}) and we estimate
\begin{equation*}
\Vert P_kI^C\Vert_{L^2}\lesssim 2^{-2m}+\Vert e^{-is\Lambda_\mu}f^\mu_{j_1,k_1}\Vert_{L^\infty}\Vert f^\nu_{j_2,k_2}\Vert_{L^2}\lesssim 2^{-3m/2+40\delta m}.
\end{equation*}
This gives an acceptable $\widetilde{f_{NC}}$-type contribution, as desired.

{\bf{Case 4.}} We assume now that
\begin{equation}\label{Alx115}
\mu+\nu=0\qquad\text{ and }\qquad 2^k\leq 2^{\delta m}\big(2^{-m/2}+2^{j_2-m}\big).
\end{equation}
Notice that
\begin{equation}\label{PhaseMuNu}
\begin{split}
\vert\nabla_\eta\Phi(\xi,\eta)\vert=\vert\nabla\Lambda_\mu(\eta-\xi)-\nabla\Lambda_\mu(\eta)\vert\gtrsim (1+2^{3k_1}+2^{3k_2})^{-1}\vert\xi\vert,\qquad \vert D^\alpha_\eta\Phi(\xi,\eta)\vert\lesssim_\alpha \vert\xi\vert.
\end{split}
\end{equation}
Lemma \ref{tech5} (i) shows that $|\mathcal{F}I^{\sigma\mu\nu}_{j_1,k_1,j_2,k_2}(\xi)|\lesssim 2^{-2m}$ if $2^m|\xi|\geq 2^{j_2}2^{2\delta^2m}$. Therefore we may assume
\begin{equation}\label{Alx115.1}
m+k\leq j_2+3\delta^2m.
\end{equation}
The desired $L^\infty$ bound in \eqref{dtfType2Prelim} follows using \eqref{FL1bd},
\begin{equation*}
\big\|\widehat{P_kI^{\sigma\mu\nu}_{j_1,k_1,j_2,k_2}}\big\|_{L^\infty}\lesssim \Vert\widehat{f^\mu_{j_1,k_1}}\Vert_{L^\infty}\Vert\widehat{f^\nu_{j_2,k_2}}\Vert_{L^1}\lesssim 2^{-(1-21\delta)j_2}2^{21\delta m}\lesssim 2^{50\delta m}\big(1+2^{m+k_-}\big)^{-1}.
\end{equation*}

The $L^2$ bound also follows if $k_-\leq -m$ directly from the $L^\infty$ bound. On the other hand, if $-m\leq k_-\leq -9\delta m$ then we further decompose $f^\nu_{j_2,k_2}=\sum_{n_2=0}^{j_2+1} f^\nu_{j_2,k_2,n_2}$ as in \eqref{Alx100}. Then
\begin{equation}\label{Alx116}
\begin{split}
\Vert P_kI^{\sigma\mu\nu}[f_{j_1,k_1}^\mu,f_{j_2,k_2,n_2}^\nu\Vert_{L^2}&\lesssim \Vert e^{-is\Lambda_\mu}f^\mu_{j_1,k_1}\Vert_{L^\infty}\Vert f^\nu_{j_2,k_2,n_2}\Vert_{L^2}\\
&\lesssim 2^{-m}\min\big(2^{21\delta m},2^{j_1/2+\delta j_1}\big)2^{-j_2+20\delta j_2}2^{n_2/2-19\delta n_2}.
\end{split}
\end{equation}
Recalling \eqref{Alx115.1}, this suffices to prove the $L^2$ bound in \eqref{dtfType2Prelim} if $n_2=0$. On the other hand, the contribution for $n_2\geq 1$ is nontrivial only if $|\nu|=b$, and in this case we can use the null structure provided by the factor $2^{\beta_{\mu\nu}k_-}$ in \eqref{Isigmamunu}. The desired $L^2$ bound follows in this case as well.

{\bf{The remaining bounds.}} The bound \eqref{dtFTotalBds2Prelim} follows since we have good bounds for each term. On the other hand, \eqref{OmegadtfPrelim} follows by inspection of the defining formulas above using the assumptions \eqref{AssonFmunu} and the identities
\begin{equation}\label{Alx99.7}
\begin{split}
&\left\{\Omega_\xi+\Omega_\eta\right\}\chi(\Phi(\xi,\eta))\equiv0,\qquad \left\{\Omega_\xi+\Omega_\eta\right\}\chi(\kappa^{-1}\Omega_\eta\Phi(\xi,\eta))\equiv0,\\
&\left\{\Omega_\xi+\Omega_\eta\right\}\chi(\kappa^{-1}\nabla_{\eta}\Phi(\xi,\eta))=\kappa^{-1}\nabla^{\perp}_\eta\Phi(\xi,\eta)\cdot\nabla\chi(\kappa^{-1}\nabla_\eta\Phi(\xi,\eta)).
\end{split}
\end{equation}

Finally we prove the claim about the time derivative $\partial_sg^q_{\sigma\mu\nu}$,
\begin{equation}\label{Alx99.6}
\Vert \varphi_{\geq -\D}\partial_sg^q_{\sigma\mu\nu}(s)\Vert_{L^\infty}\lesssim 2^{(6\delta-2)m}2^{q+42\delta q}.
\end{equation}
Assume, for example, that $j_2\geq m/2$ and examine the defining formula \eqref{Alx99.8}. We may assume that $\min\{k_1,k_2\}\geq -2\D$. Notice that, in the support of integration,
\begin{equation*}
\left\vert\Phi(\xi,\eta)-\Psi(\xi)\right\vert=\left\vert\Phi(\xi,\eta)-\Phi(\xi,p(\xi))\right\vert\lesssim 2^{\delta^2m}\kappa_r^2.
\end{equation*}
Therefore, estimating as in \eqref{Alx99.81},
\begin{equation*}
\begin{split}
\Big\vert\int_{\mathbb{R}^2}e^{is\left[\Phi(\xi,\eta)-\Psi(\xi)\right]}\left[\Phi(\xi,\eta)-\Psi(\xi)\right]\varphi(\kappa_\theta^{-1}\Omega_\eta\Phi(\xi,\eta))\varphi(\kappa_r^{-1}\Xi(\xi,\eta))\widehat{f^{\mu}_{j_1,k_1}}(\xi-\eta)\widehat{f^{\nu}_{j_2,k_2}}(\eta)\,d\eta\Big\vert\\
\lesssim 2^{3\delta m-2m}2^{q+42\delta q}.
\end{split}
\end{equation*}
To estimate the remaining contributions we use the decomposition proved earlier,
\begin{equation*}
\begin{split}
\partial_sf^\mu_{j_1,k_1}=\widetilde{g_C^\mu}+\widetilde{g_{NC}^\mu},\qquad \partial_sf^\nu_{j_2,k_2}=\widetilde{g_C^\nu}+\widetilde{g_{NC}^\nu}.
\end{split}
\end{equation*}
The terms containing $\widetilde{g_C^\mu}$ and $\widetilde{g_C^\nu}$ in the integrals defining $\partial_s g_{\sigma\mu\nu}^q$ are dominated by
\begin{equation*}
\begin{split}
C2^{4\delta^2m}\kappa_r\kappa_\theta(\Vert \mathcal{F}\widetilde{g^{\mu}_C}\Vert_{L^\infty}\Vert \widehat{f^{\nu}_{j_2,k_2}}\Vert_{L^\infty}+\Vert \widehat{f^{\mu}_{j_1,k_1}}\Vert_{L^\infty}\Vert\mathcal{F}\widetilde{g^{\mu}_C}\Vert_{L^\infty})\lesssim 2^{5\delta m-2m+q}.
\end{split}
\end{equation*}
On the other hand, using also \eqref{OmegadtfPrelim}, the terms containing $\widetilde{g_{NC}^\mu}$ and $\widetilde{g_{NC}^\nu}$ in the integrals defining $\partial_s g_{\sigma\mu\nu}^q$ are dominated by
\begin{equation*}
\begin{split}
C2^{4\delta^2m}\kappa_\theta\kappa_r^{1/2}\big[\sup_{\theta\in\mathbb{S}^1}\Vert \mathcal{F}\widetilde{g^{\mu}_{NC}}(r\theta)\Vert_{L^2(rdr)}\Vert \widehat{f^{\nu}_{j_2,k_2}}\Vert_{L^\infty}+\Vert \widehat{f^{\mu}_{j_1,k_1}}\Vert_{L^\infty}\sup_{\theta\in\mathbb{S}^1}\Vert\mathcal{F}\widetilde{g^{\nu}_{NC}}(r\theta)\Vert_{L^2(rdr)}\big]\\
\lesssim
2^{-2m+q/2}.
\end{split}
\end{equation*}
The desired bound \eqref{Alx99.6} follows. The case $j_2\leq m/2$ is similar, using the simpler formula \eqref{Alx97.1} instead of \eqref{Alx99.8}, and this completes the proof of the lemma.
\end{proof}

\begin{corollary}\label{AlxCoro}
For any $s\in[2^{m}-1,2^{m+1}]\cap [0,T]$ and $a\in[0,N_1/2]$ we can decompose $\Omega^a(\partial_tV_{\sigma})(s)=G_2(s)+G_{\infty}(s)$ such that, for any $k\in\mathbb{Z}$,
\begin{equation}\label{Alx90}
\begin{split}
\sup_{|\lambda|\leq 2^{m(1-\delta/10)}}\|e^{-i(s+\lambda)\Lambda_\sigma}P_kG_{\infty}(s)\|_{L^{\infty}}&\lesssim\eps_1^22^{-2m+50\delta m},\\
\|P_kG_{2}(s)\|_{L^{2}}&\lesssim\eps_1^22^{-3m/2+60\delta m}.
\end{split}
\end{equation}
\end{corollary}

\begin{proof} We may assume $m\geq\D$ and define 
\begin{equation*}
G_2(s):=P_{>\delta^2m}[\Omega^a(\partial_tV_{\sigma})(s)]+P_{\leq\delta^2m}[\eps_1^2P_{\geq -m/2}\widetilde{f_{NC}}(s)].
\end{equation*}
The desired $L^2$ bound in \eqref{Alx90} follows from \eqref{Alx1} and \eqref{dtfType2Prelim}. Let
\begin{equation*}
G_\infty^1(s):=P_{\leq\delta^2m}[\eps_1^2P_{<-m/2}\widetilde{f_{NC}}(s)],\qquad G_\infty^2(s):=P_{\leq\delta^2m}[\eps_1^2\widetilde{f_{C}}(s)].
\end{equation*}
The desired bound on $G_\infty^1$ is a consequence of \eqref{dtFTotalBds} and the $L^\infty$ boundedness of the operator $P_{\leq -m/2+10}e^{-i\lambda\Lambda_\sigma}$. To prove the bound for $G_\infty^2(s)$ we examine  \eqref{dtfType1Prelim}. It suffices to prove that, for $k\in[-m/2+\delta^2m/5,2\delta^2m]$, $x\in\mathbb{R}^2$, and $\mu,\nu\in\mathcal{P}$, $\mu+\nu\neq 0$, 
\begin{equation}\label{Alx90.1}
\Big|\int_{\mathbb{R}^2}e^{ix\cdot\xi}e^{-i\lambda\Lambda_\sigma(\xi)-is[\Lambda_{\mu}(\xi-p_{\mu\nu}(\xi))+\Lambda_\nu(p_{\mu\nu}(\xi))]}\varphi_k(\xi)g_{\sigma\mu\nu}^q(\xi,s)\Big|\lesssim 2^{-2m+50\delta m}.
\end{equation}

Let $\Gamma_{\mu\nu}(\xi):=\Lambda_{\mu}(\xi-p_{\mu\nu}(\xi))+\Lambda_\nu(p_{\mu\nu}(\xi))$ and notice that
\begin{equation}\label{Alx91}
(\nabla\Gamma_{\mu\nu})(\xi)=(\nabla\Lambda_\mu)[\xi-p_{\mu\nu}(\xi)]=(\nabla\Lambda_\nu)[p_{\mu\nu}(\xi)].
\end{equation}
The desired estimate \eqref{Alx90.1} follows if $q=0$ by a standard stationary phase argument, using the bounds $\|\varphi_kD^\alpha g^0_{\sigma\mu\nu}(s)\|_{L^\infty}\lesssim 2^{-42\delta k_-}2^{-m+3\delta m}2^{(m/2+2\delta^2m)\vert\alpha\vert}$ and $|p'_{+;\mu\nu}(|\xi|)|\gtrsim 2^{-3\delta^2m}$, see \eqref{cas10.1}.

On the other hand, if $q\geq 1$ then we may assume $k\geq -\D$ and the desired bound \eqref{Alx90.1} follows by an argument similar to the one used in the proof of \eqref{LinftyBd}: up to negligible errors we may assume that $|x|\approx 2^m$ and restrict the integral to a small $2^{-m/2+2\delta^2m}\times 2^{q-m/2+2\delta^2m}$ box, as in \eqref{Alx90.4}. The contribution from this box can be then estimated using the bound $\|\varphi_k g^q_{\sigma\mu\nu}(s)\|_{L^\infty}\lesssim 2^{-m+3\delta m}2^{-q+42\delta q}$, which leads to the desired bound \eqref{Alx90.1}.
\end{proof}

\subsection{Precise estimates on the time derivative}

We prove now more precise estimates on the time derivative which are needed in the proof of the more difficult Lemma \ref{ResLem}.

\begin{lemma}\label{dtfLem}

Assume the hypothesis of Proposition \ref{bootstrap}, $m\geq 0$, $s\in [2^m-1,2^{m+1}]$, $\sigma\in\{e,b\}$.
For any $a\in[0,N_1/4]\cap\mathbb{Z}$ we may decompose
\begin{equation}\label{Decdtf}
\Omega^a(\partial_tV_\sigma)(s)=\epsilon_1^2\left\{f_C(s)+f_{SR}(s)+f_{NC}(s)+\partial_sF_C(s)+\partial_sF_{NC}(s)+\partial_sF_{LO}(s)\right\},
\end{equation}
where we have coherent inputs,
\begin{equation}\label{dtfType1}
\begin{split}
\widehat{f_C}(\xi,s)
=\sum_{\mu,\nu\in\mathcal{P},\,\mu+\nu\neq 0}e^{is\Psi_{\sigma\mu\nu}(\xi)}\sum_{0\le q\le (1/2-40\delta)m}g^q_{\sigma\mu\nu}(\xi,s)\varphi_{\leq-3\delta m}(\Psi_{\sigma\mu\nu}(\xi)),\\
\Vert D^\alpha_\xi g^q_{\sigma\mu\nu}(s)\Vert_{L^\infty}\lesssim 2^{-m-q+8\delta^2m}2^{(m/2+q+3\delta^2m)\vert\alpha\vert},\qquad
\Vert \partial_sg^q_{\sigma\mu\nu}(s)\Vert_{L^\infty}\lesssim 2^{(5\delta-2)m+ q};
\end{split}
\end{equation}
secondary resonances
\begin{equation}\label{dtfSR}
\begin{split}
\Vert D_\xi^\alpha \widehat{f_{SR}}(s)\Vert_{L^\infty}&\lesssim 2^{-3m/2+75\delta m}2^{(1-300\delta) m\vert\alpha\vert},\qquad f_{SR}=\widetilde{P}_{SRI}f_{SR};
\end{split}
\end{equation}
and nonresonant contributions
\begin{equation}\label{dtfType2}
\begin{split}
\Vert f_{NC}(s)\Vert_{L^2}&\lesssim 2^{-19m/10}.
\end{split}
\end{equation}
In addition
\begin{equation}\label{dtfType3}
\begin{split}
\Vert \widehat{F_C}(s)\Vert_{L^\infty}\lesssim 2^{3\delta m-m+10\delta^2m},&\qquad\Vert F_{NC}(s)\Vert_{L^2}\lesssim 2^{-41m/40},\\
\Vert (1+2^m\vert\xi\vert)\widehat{F_{LO}}(\xi,s)\Vert_{L^\infty_\xi}\lesssim 2^{5\delta m},&\qquad P_{\ge-13/15m}F_{LO}\equiv 0,
\end{split}
\end{equation}
and
\begin{equation}\label{Omegadtf}
\sup_{0\le b\le N_1/4}\big\{\Vert \Omega^{b}g^q_{\sigma\mu\nu}\Vert_{L^2}+\Vert\Omega^bf_{SR}(s)\Vert_{L^2}+\Vert\Omega^bf_{NC}(s)\Vert_{L^2}\big\}\lesssim 1.
\end{equation}
\end{lemma}

In \eqref{dtfSR}, $\widetilde{P}_{SRI}$ denotes the projection on a small $2^{-\D}$ neighborhood of the space-time resonance inputs $\alpha_{1,e}$, $\alpha_{2,e}$, and $\alpha_{2,b}$ defined in Remark \ref{spaceres} (ii).

\begin{remark}
Note that this is a weak form of Proposition \ref{ZNormProp}. It decomposes the time derivative into a main coherent contribution, which behaves in a similar way to $t^{-1}f$, but with different time oscillation, a secondary-resonance contribution, which is smoother than expected, a non-resonant contribution, which behaves like $\partial_tF$, and an error term which is very small in $L^2$.
\end{remark}

\begin{proof}[Proof of Lemma \ref{dtfLem}] Again, using Lemma \ref{dtVCubic}, it suffices to consider $[\partial_tV_\sigma]^{(2)}$. For simplicity of notation, we sometimes write $I^{\sigma\mu\nu}=I$, $\Phi_{\sigma\mu\nu}=\Phi$ and $\Psi_{\sigma\mu\nu}=\Psi$ as before.
We may use the same simplification as in the beginning of the proof of Lemma \ref{dtfLemPrelim} and in particular assume $\epsilon_1=1$ and reduce to $\widetilde{I}^{\sigma\mu\nu}$ together with 
\begin{equation}\label{Alx95}
\begin{split}
\Vert f^\mu\Vert_{H^{N_0/2}\cap Z_1^\mu\cap H^{N_1/2}_\Omega}+ \Vert f^\nu\Vert_{H^{N_0/2}\cap Z_1^\nu\cap H^{N_1/2}_\Omega}\le 1.
\end{split}
\end{equation}
Up to acceptable $f_{NC}$ errors we may assume $m\geq \D^2$ and restrict the sum in \eqref{Isigmamunu} to the range
\begin{equation}\label{Alx9}
-3m\le k,k_1,k_2\le \delta^2m/10-\D^2.
\end{equation} 

We decompose, with $l_0:=\lfloor-3\delta m-4\delta^2m\rfloor$, $\varphi_{hi}(x):=\varphi_{>l_0}(x)$, $\varphi_{lo}(x):=\varphi_{\leq l_0}(x)$,
\begin{equation*}
\begin{split}
I^{\sigma\mu\nu}[P_{k_1}f^\mu,P_{k_2}f^\nu]&=I^{hi}[P_{k_1}f^\mu,P_{k_2}f^\nu]+I^{lo}[P_{k_1}f^\mu,P_{k_2}f^\nu],\\
\mathcal{F}\{I^{\ast}[f,g]\}(\xi)&:=\int_{\mathbb{R}^2}e^{is\Phi(\xi,\eta)}\varphi_{\ast}(\Phi(\xi,\eta))\widehat{f}(\xi-\eta)\widehat{g}(\eta)d\eta,\qquad\ast\in\{hi,lo\}.
\end{split}
\end{equation*}

{\bf{Contribution of $I^{hi}$.}} We may rewrite
\begin{equation*}
\begin{split}
I^{hi}[P_{k_1}f^\mu,P_{k_2}f^\nu]&=\partial_s\left\{J[P_{k_1}f^\mu,P_{k_2}f^\nu]\right\}-J[P_{k_1}\partial_sf^\mu,P_{k_2}f^\nu]-J[P_{k_1}f^\mu,P_{k_2}\partial_sf^\nu],\\
\mathcal{F}\{J[f,g]\}(\xi)&:=\int_{\mathbb{R}^2}e^{is\Phi(\xi,\eta)}\frac{\varphi_{hi}(\Phi(\xi,\eta))}{i\Phi(\xi,\eta)}\widehat{f}(\xi-\eta)\widehat{g}(\eta)d\eta.
\end{split}
\end{equation*}
Using \eqref{LinftyBd2}, \eqref{dtFTotalBds2Prelim}, and Lemma \ref{PhiLocLem}, we easily see that
\begin{equation*}
\begin{split}
\Vert J[P_{k_1}\partial_sf^\mu,P_{k_2}f^\nu]\Vert_{L^2}+\Vert J[P_{k_1}f^\mu,P_{k_2}\partial_sf^\nu]\Vert_{L^2}&\lesssim 2^{(50\delta-2)m},\\
\end{split}
\end{equation*}
which gives an acceptable $f_{NC}$-type contribution as in \eqref{dtfType2}.

We define $f^\mu_{j_1,k_1}$ and $f^\nu_{j_2,k_2}$ as in \eqref{Alx8}. If $j_2=\max(j_1,j_2)\ge  m/18$ then, using Lemma \ref{PhiLocLem},
\begin{equation*}
\begin{split}
\Vert J[f^{\mu}_{j_1,k_1},f^\nu_{j_2,k_2}]\Vert_{L^2}
&\lesssim 2^{-(1/2-\delta)j_2}2^{-m+50\delta m}.
\end{split}
\end{equation*}
This gives an acceptable $F_{NC}$-type contribution as in \eqref{dtfType3}. On the other hand, if
\begin{equation*}
j_2=\max\{j_1,j_2\}\le m/18,\qquad\mu+\nu\ne 0,
\end{equation*}
then integration by parts, using Lemma \ref{tech5} (i), shows that
\begin{equation*}
\begin{split}
&\big\vert \mathcal{F}\big\{J[f^{\mu}_{j_1,k_1},f^{\nu}_{j_2,k_2}]-J^S\big\}(\xi)\big\vert\lesssim 2^{-2m},\qquad\kappa_r:=2^{\delta^2m-m/2},\\
&\widehat{J^S}(\xi):=\int_{\mathbb{R}^2}e^{is\Phi(\xi,\eta)}\frac{\varphi_{hi}(\Phi(\xi,\eta))}{i\Phi(\xi,\eta)}\varphi(\kappa_r^{-1}\nabla_\eta\Phi(\xi,\eta))\widehat{f^{\mu}_{j_1,k_1}}(\xi-\eta)\widehat{f^{\nu}_{j_2,k_2}}(\eta)d\eta.
\end{split}
\end{equation*}
In view of \eqref{cas10}, the $\eta$-integral in the definition of $\widehat{J^S}$ takes place over a ball of radius $\lesssim 2^{\delta^2m}\kappa_r$ centered at $p(\xi)$. If $|\Phi(\xi,\eta)|\geq 2^{-\D}$ in this ball then $|\widehat{J^S}(\xi)|\lesssim 2^{(3\delta-1) m}$, using \eqref{FLinftybd}. On the other hand, if $|\Phi(\xi,\eta)|\leq 2^{-\D+10}$ in this ball then we use Proposition \ref{separation1} (i) and (iii). In view of \eqref{FLinftybd}, we have $|\widehat{f^{\mu}_{j_1,k_1}}(\xi-\eta)|+|\widehat{f^{\nu}_{j_2,k_2}}(\eta)|\lesssim 1$ in the support of the integral, so we can conclude that $|\widehat{J^S}(\xi)|\lesssim 2^{(3\delta-1) m+8\delta^2m}$. In all cases we get an acceptable $F_{C}$-type contribution.

Finally, assume that
\begin{equation*}
j_2=\max\{j_1,j_2\}\le m/18,\qquad\mu+\nu= 0.
\end{equation*}
In this case, we use again \eqref{PhaseMuNu} as before, and integrate by parts using Lemma \ref{tech5} to conclude that $\big\|\mathcal{F}P_kJ[f^{\mu}_{j_1,k_1},f^{\nu}_{j_2,k_2}]\big\|_{L^\infty}\lesssim 2^{-2m}$ if $m+k\geq j_2+\delta^2m$. On the other hand, if $m+k\leq j_2+\delta^2m$ then we get an $F_{LO}$-type contribution. Indeed we may assume $2^{k_1}\approx 2^{k_2}$ and estimate
\begin{equation*}
\big|\varphi_k(\xi)\mathcal{F}J[f^{\mu}_{j_1,k_1},f^\nu_{j_2,k_2}](\xi)\big|\lesssim 2^{3.1\delta m}\Vert \widehat{f^\mu_{j_1,k_1}}\Vert_{L^\infty}\Vert \widehat{f^\nu_{j_2,k_2}}\Vert_{L^1}\lesssim 2^{3.1\delta m}2^{2\delta j_1}2^{-j_2+21\delta j_2}\lesssim 2^{4.8\delta m}2^{-j_2},
\end{equation*}
using \eqref{FL1bd} and \eqref{FLinftybd}. This suffices since $2^{-j_2}\lesssim 2^{\delta^2m}\big(1+2^{m+k})^{-1}$.

{\bf{Contribution of $I^{lo}$, $\underline{k}$ small.}} We consider now $P_{k}I^{lo}[P_{k_1}f^\mu,P_{k_2}f^\nu]$ in the case 
\begin{equation*}
\underline{k}=\min\{k,k_1,k_2\}\le -\D.
\end{equation*}
We define $f^\mu_{j_1,k_1},f^\mu_{j_2,k_2},f^\mu_{j_1,k_1,n_1},f^\mu_{j_2,k_2,n_2}$ as in \eqref{Alx100}. Using Lemma \ref{separation1}, we may assume that
\begin{equation*}
f^\mu_{j_1,k_1}=f^\mu_{j_1,k_1,0},\qquad f^\nu_{j_2,k_2}=f^\nu_{j_2,k_2,0},\qquad -\D\le\max\{k,k_1,k_2\}\le\D.
\end{equation*}
If $j_2=\max\{j_1,j_2\}\ge(1-\delta)m$ then, using \eqref{LinftyBd2} and Lemma \ref{PhiLocLem}, we see that
\begin{equation*}
\begin{split}
\big\Vert P_kI^{lo}[f^\mu_{j_1,k_1},f^\nu_{j_2,k_2}]\big\Vert_{L^2}&\lesssim 2^{-m+21\delta m}2^{-j_2+20\delta j_2}\lesssim 2^{-2m+50\delta m},
\end{split}
\end{equation*}
which gives an acceptable $f_{NC}$-type contribution as in \eqref{dtfType2}. On the other hand, using Proposition \ref{spaceres} (iv), $\vert\nabla_\eta\Phi(\xi,\eta)\vert\gtrsim 1$ on the support of the integral. Therefore, if $\max\{j_1,j_2\}\le(1-\delta)m$ then we can integrate by parts using Lemma \ref{tech5} to obtain again an acceptable $f_{NC}$-type contribution as in \eqref{dtfType2}.

{\bf{Contribution of $I^{lo}$, $\underline{k}$ not small.}} From now on, we may assume that
\begin{equation}\label{MedFreq}
\min\{k,k_1,k_2\}\ge -\D,\qquad j_1\leq j_2.
\end{equation}
Assume first that
\begin{equation*}
j_1=\min\{j_1,j_2\}\ge (1-200\delta)m,\qquad \min\{k,k_1,k_2\}\ge \D.
\end{equation*}
In this case, using Lemma \ref{Shur2Lem} and \eqref{RadL2}, we find that
\begin{equation*}
\big\Vert P_kI^{lo}[f^\mu_{j_1,k_1,n_1},f^\nu_{j_2,k_2,n_2}]\big\Vert_{L^2}\lesssim 2^{-\frac{n_1+n_2}{2}}\big\|\sup_{\theta}|\widehat{f^{\mu}_{j_1,k_1,n_1}}(r\theta)|\big\|_{L^2}\big\|\sup_{\theta}|\widehat{f^{\nu}_{j_2,k_2,n_2}}(r\theta)|\big\|_{L^2}\lesssim 2^{-1.95m},
\end{equation*}
for any $j_1,j_2,n_1,n_2$. This gives an acceptable $f_{NC}$-type contribution as in \eqref{dtfType2}. To control  the remaining contributions of $P_kI^{lo}$ we consider two cases.

{\bf{Case 1.}} Assume first that
\begin{equation}\label{dtfAss1}
\min\{k,k_1,k_2\}\ge -\D,\quad j_1\le(1-200\delta)m,\quad
j_2\ge(1-50\delta)m,
\end{equation}
and decompose, with $\kappa_\theta:=2^{-m/2+2\delta^2m}$,
\begin{equation*}
\begin{split}
P_{k}I^{lo}[f^\mu_{j_1,k_1},f^\nu_{j_2,k_2}]&=I^\perp+I^{||}[f^\mu_{j_1,k_1},f^\nu_{j_2,k_2}],\\
\widehat{I^\perp}(\xi)&:=\int_{\mathbb{R}^2}e^{is\Phi(\xi,\eta)}\varphi_{lo}(\Phi(\xi,\eta))(1-\varphi(\kappa_\theta^{-1}\Omega_\eta\Phi(\xi,\eta)))\widehat{f^\mu_{j_1,k_1}}(\xi-\eta)\widehat{f^\nu_{j_2,k_2}}(\eta)d\eta,\\
\mathcal{F}\{I^{||}[f,g]\}(\xi)&:=\int_{\mathbb{R}^2}e^{is\Phi(\xi,\eta)}\varphi_{lo}(\Phi(\xi,\eta))\varphi(\kappa_\theta^{-1}\Omega_\eta\Phi(\xi,\eta))\widehat{f}(\xi-\eta)\widehat{g}(\eta)d\eta.\\
\end{split}
\end{equation*}
Integration by parts using Lemma \ref{RotIBP} shows that $I^\perp$ yields an acceptable $f_{NC}$-type contribution as in \eqref{dtfType2}. Moreover, notice that, using Lemma \ref{PhiLocLem}, Lemma \ref{LinEstLem}, and Proposition \ref{separation1} (ii),
\begin{equation*}
\begin{split}
\Vert P_kI^{||}[f^\mu_{j_1,k_1},f^\nu_{j_2,k_2,n_2}]\Vert_{L^2}&\lesssim 2^{-m+21\delta m}2^{-j_2+20\delta j_2}2^{n_2/2}+2^{-4m}\lesssim 2^{-2m+100\delta m}2^{n_2/2},\\
P_kI^{||}[f^\mu_{j_1,k_1,n_1},f^\nu_{j_2,k_2,n_2}]&\equiv 0\qquad\text{ if }\,\,n_1,n_2\geq 0.
\end{split}
\end{equation*}
These are acceptable $f_{NC}$-type contributions if $n_2\leq m/20$. Thus, in the following, we may assume that
\begin{equation}\label{sw}
f^\mu_{j_1,k_1}=f^\mu_{j_1,k_1,0},\qquad |\nu|=b,\,\,f^\nu_{j_2,k_2}=\sum_{n_2\geq m/20}f^\nu_{j_2,k_2,n_2}.
\end{equation}

Using Schur's test and Lemma \ref{LinEstLem}, we see that
\begin{equation*}
\begin{split}
\Vert I^{||}[f^\mu_{j_1,k_1,0},f^\nu_{j_2,k_2,n_2}]\Vert_{L^2}&\lesssim \kappa_\theta\cdot \sup_{\theta\in\mathbb{S}^1}\Vert \widehat{f^\nu_{j_2,k_2,n_2}}(r\theta)\Vert_{L^1(rdr)} \cdot \Vert f^\mu_{j_1,k_1,0}\Vert_{L^2}\lesssim 2^{-3m/2+100\delta m}2^{-j_1+20\delta j_1}.
\end{split}
\end{equation*}
Therefore we may assume that $j_1\le m/2$. In this case, we further decompose
\begin{equation}\label{DecFSR}
\begin{split}
I^{||}[f^\mu_{j_1,k_1},f^\nu_{j_2,k_2}]&=I^C[f^\mu_{j_1,k_1},f^\nu_{j_2,k_2}]+I^{NC}[f^\mu_{j_1,k_1},f^\nu_{j_2,k_2}],\\
\widehat{I^{C}[f,g]}(\xi)&:=\int_{\mathbb{R}^2}e^{is\Phi(\xi,\eta)}\varphi_{lo}(\Phi(\xi,\eta))\varphi(\kappa_\theta^{-1}\Omega_\eta\Phi(\xi,\eta))\varphi_{Lo}(\xi,\eta)\widehat{f}(\xi-\eta)\widehat{g}(\eta)d\eta,\\
\widehat{I^{NC}[f,g]}(\xi)&:=\int_{\mathbb{R}^2}e^{is\Phi(\xi,\eta)}\varphi_{lo}(\Phi(\xi,\eta))\varphi(\kappa_\theta^{-1}\Omega_\eta\Phi(\xi,\eta))\varphi_{Hi}(\xi,\eta)\widehat{f}(\xi-\eta)\widehat{g}(\eta)d\eta,
\end{split}
\end{equation}
\begin{equation*}
\begin{split}
\varphi_{Lo}(\xi,\eta)&:=\varphi_{\le-400\delta m}(\nabla_\xi\Phi(\xi,\eta))\varphi_{\ge -\D}(\nabla_\eta\Phi(\xi,\eta)),\\
\varphi_{Hi}(\xi,\eta)&:=\varphi_{>-400\delta m}(\nabla_\xi\Phi(\xi,\eta))\varphi_{\ge -\D}(\nabla_\eta\Phi(\xi,\eta)).
\end{split}
\end{equation*}
Notice that $1=\varphi_{Lo}(\xi,\eta)+\varphi_{Hi}(\xi,\eta)$ in the support of the integral, in view of the assumption \eqref{sw} on $f^\nu_{j_2,k_2}$ and Proposition \ref{separation1} (iii).

We first consider the integrals $I^C$, which produce the secondary resonances $f_{SR}$. Indeed using \eqref{RadL2}, \eqref{FLinftybd} and \eqref{MedFreq}, we estimate
\begin{equation*}
\begin{split}
\Vert \widehat{I^C}\Vert_{L^\infty}&\lesssim\kappa_\theta\Vert\widehat{f^\mu_{j_1,k_1}}\Vert_{L^\infty}\cdot\big\Vert\sup_{\theta}|\widehat{f^\nu_{j_2,k_2}}(r\theta)|\big\Vert_{L^1(rdr)}\lesssim 2^{-3m/2+74\delta m},
\end{split}
\end{equation*}
The derivatives can be estimated in the same way, given that $j_1\leq m/2$ and the definition of the cutoff $\varphi_{Lo}$, while the support property is a consequence of the smallness of both $\Phi$ and $\nabla_\xi\Phi$.

{\bf{The integral $I^{NC}$.}} We show now that $I^{NC}$ gives acceptable contributions, i.e,
\begin{equation}\label{INCOK}
\begin{split}
&P_kI^{NC}[f^\mu_{j_1,k_1},f^\nu_{j_2,k_2}]=\partial_sF_1+F_2,\\
&\Vert F_1\Vert_{L^2}\lesssim 2^{-36m/35},\qquad\Vert F_2\Vert_{L^2}\lesssim 2^{-21 m/11}.
\end{split}
\end{equation}
Recall that $l_0=\lfloor-3\delta m-4\delta^2m\rfloor$, let $l_-:=\lfloor-14/15 m\rfloor$, and decompose further
\begin{equation*}
\begin{split}
&I^{NC}[f^\mu_{j_1,k_1},f^\nu_{j_2,k_2}]=\sum_{l_-\le l\le l_0}I^{NC}_{l},\\
&\widehat{I^{NC}_{l}}(\xi):=\int_{\mathbb{R}^2}e^{is\Phi(\xi,\eta)}\varphi_{l}^{[l_-,l_0+1]}(\Phi(\xi,\eta))\varphi(\kappa_\theta^{-1}\Omega_\eta\Phi(\xi,\eta)))\varphi_{Hi}(\xi,\eta)\widehat{f^{\mu}_{j_1,k_1}}(\xi-\eta)\widehat{f^{\nu}_{j_2,k_2}}(\eta)d\eta.\\
\end{split}
\end{equation*}
Using Schur's test with Lemma \ref{Shur3Lem} (i), we find that
\begin{equation*}
\begin{split}
\Vert P_kI^{NC}_{l_-}\Vert_{L^2}&\lesssim 2^{l_-}\kappa_\theta 2^{1000\delta m}\Vert f^{\nu}_{j_2,k_2}\Vert_{L^2}\lesssim 2^{-39m/20},
\end{split}
\end{equation*}
which is acceptable for \eqref{INCOK}. 

On the other hand, for $l_-<l\leq l_0$ we write
\begin{equation*}
\begin{split}
iI^{NC}_{l}&=\partial_sJ_l-\mathcal{A}_l-\mathcal{B}_l,\\
\widehat{J_l}(\xi)&:=\int_{\mathbb{R}^2}e^{is\Phi(\xi,\eta)}\widetilde{\varphi_{l}}(\Phi(\xi,\eta))\varphi(\kappa_\theta^{-1}\Omega_\eta\Phi(\xi,\eta)))\varphi_{Hi}(\xi,\eta)\widehat{f^{\mu}_{j_1,k_1}}(\xi-\eta)\widehat{f^{\nu}_{j_2,k_2}}(\eta)d\eta,\\
\widehat{\mathcal{A}_l}(\xi)&:=\int_{\mathbb{R}^2}e^{is\Phi(\xi,\eta)}\widetilde{\varphi_{l}}(\Phi(\xi,\eta))\varphi(\kappa_\theta^{-1}\Omega_\eta\Phi(\xi,\eta)))\varphi_{Hi}(\xi,\eta)\partial_s\widehat{f^{\mu}_{j_1,k_1}}(\xi-\eta)\widehat{f^{\nu}_{j_2,k_2}}(\eta)d\eta,\\
\widehat{\mathcal{B}_l}(\xi)&:=\int_{\mathbb{R}^2}e^{is\Phi(\xi,\eta)}\widetilde{\varphi_{l}}(\Phi(\xi,\eta))\varphi(\kappa_\theta^{-1}\Omega_\eta\Phi(\xi,\eta)))\varphi_{Hi}(\xi,\eta)\widehat{f^{\mu}_{j_1,k_1}}(\xi-\eta)\partial_s\widehat{f^{\nu}_{j_2,k_2}}(\eta)d\eta,
\end{split}
\end{equation*}
where $\widetilde{\varphi_{l}}(x):=x^{-1}\varphi_l(x)$. Recall \eqref{sw}. Lemma \ref{Shur3Lem} (i) and Lemma \ref{LinEstLem} give
\begin{equation*}
\begin{split}
\Vert J_{l}\Vert_{L^2}&\lesssim \sum_{n_2\geq 1}2^{2\delta^2m}\kappa_\theta\cdot 2^{-l}\cdot 2^{\frac{l-n_2}{2}+400\delta m}\Vert \widehat{f^{\mu}_{j_1,k_1}}\Vert_{L^\infty}\Vert f^\nu_{j_2,k_2,n_2}\Vert_{L^2}\lesssim 2^{-\frac{m+l}{2}+(1000\delta-1)m},
\end{split}
\end{equation*}
which gives an acceptable contribution as in \eqref{INCOK}. We can estimate $\mathcal{A}_l$ similarly,
\begin{equation*}
\begin{split}
\Vert \mathcal{A}_l\Vert_{L^2}&\lesssim \sum_{n_2\geq 1}2^{2\delta^2m}\kappa_\theta\cdot 2^{-\frac{l+n_2}{2}+400\delta m}\cdot \Vert \widehat{\partial_sf^{\mu}_{j_1,k_1}}\Vert_{L^\infty}\Vert f^{\nu}_{j_2,k_2,n_2}\Vert_{L^2}\lesssim 2^{-2m},
\end{split}
\end{equation*}
 using also Lemma \ref{dtfLemPrelim}, which gives an acceptable contribution as in \eqref{INCOK}. 

To control the term $\mathcal{B}_l$, we need more precise estimates. We use Lemma \ref{dtfLemPrelim} to decompose
\begin{equation*}
\partial_sf^{\nu}_{j_2,k_2}=\widetilde{g_C}+\widetilde{g_{NC}},
\end{equation*}
and let $\widehat{\mathcal{B}_l}=\widehat{\mathcal{B}_l^1}+\widehat{\mathcal{B}_l^2}$ denote the corresponding decomposition of $\widehat{\mathcal{B}_l}$. Using Schur's test (Lemma \ref{Shur3Lem} (i)) again, we find that
\begin{equation}\label{Alx48.2}
\Vert \mathcal{B}_l^2\Vert_{L^2}\lesssim \kappa_\theta\cdot 2^{400\delta m}\cdot \Vert \widetilde{g_{NC}}\Vert_{L^2}\lesssim 2^{-2m+1000\delta m},
\end{equation}
which gives an acceptable contribution as in \eqref{INCOK}. Moreover, using \eqref{dtfType1Prelim}, we can write $\widehat{\mathcal{B}_l^1}$ as a sum over $q\in[0,m/2-10\delta m]$ and over $\theta,\kappa\in\mathcal{P}$, $\theta+\kappa\neq 0$, of integrals of the form
\begin{equation}\label{Alx48}
\begin{split}
\mathcal{C}_l(\xi):=\int_{\mathbb{R}^2}e^{is[\Phi_{\sigma\mu\nu}(\xi,\eta)+\Psi_{\nu\theta\kappa}(\eta)]}&\widetilde{\varphi_{l}}(\Phi(\xi,\eta))\varphi(\kappa_\theta^{-1}\Omega_\eta\Phi(\xi,\eta)))\\
&\times\varphi_{Hi}(\nabla_\xi\Phi(\xi,\eta))\widehat{f^{\mu}_{j_1,k_1}}(\xi-\eta)h^q(\eta)d\eta.
\end{split}
\end{equation}
In view of \eqref{dtfType1Prelim} and \eqref{sw}, the functions $h^q=h^q_{\nu\theta\kappa}$ satisfy the properties
\begin{equation}\label{Alx48.1}
\begin{split}
&h^q(\eta)=h^q(\eta)\varphi_{\leq -m/21}(\Psi^\dagger_b(\eta)),\qquad \big\|D^\alpha_\eta h^q(s)\big\|_{L^\infty}\lesssim 2^{-m+50\delta m}2^{-q}2^{(m/2+q+2\delta^2m)\vert\alpha\vert},\\
&\Vert \partial_sh^q(s)\Vert_{L^\infty}\lesssim 2^{(6\delta-2)m}2^{q+42\delta q}.
\end{split}
\end{equation}
The contributions of exponents $q\geq 19m/40$ can be estimated as in \eqref{Alx48.2}, since the functions $h^q$ in this case have sufficiently small $L^2$ norm. Therefore we may assume that $q\leq 19m/40$.

The main observation we need is that either $|\Psi_{\nu\theta\kappa}(\eta)|\leq 2^{-m/22}$ or $|\Psi_{\nu\theta\kappa}(\eta)|\geq 2^{-\D}$ in the support of the integrals $\mathcal{C}_l$. This is due to the assumption $h^q(\eta)=h^q(\eta)\varphi_{\leq -m/21}(\Psi^\dagger_b(\eta))$ in \eqref{Alx48.1}. Therefore we can decompose $\mathcal{C}_l=\mathcal{C}_l^1+\mathcal{C}_l^2$, where $\mathcal{C}_l^1$ and $\mathcal{C}_l^2$ are defined by inserting the factors $\varphi_{\leq -m/25}(\Psi_{\nu\theta\kappa}(\eta))$ and $\varphi_{\geq -\D}(\Psi_{\nu\theta\kappa}(\eta))$ in \eqref{Alx48}. 

To control $\mathcal{C}_l^1$ we consider two cases: if $l\leq -1000\delta m$ then we use Proposition \ref{Separation2}. The restrictions $|\Psi_{\nu\theta\kappa}(\eta)|\leq 2^{-m/22}$, $|\Phi_{\sigma\mu\nu}(\xi,\eta)|\lesssim 2^{-1000\delta m}$, $|\nabla_\xi\Phi(\xi,\eta)|\gtrsim 2^{-400\delta m}$ show that
\begin{equation*}
\vert\nabla_\eta\left[\Phi(\xi,\eta)+\Psi(\eta)\right]\vert\gtrsim 2^{-500\delta m}.
\end{equation*}
Recalling that $j_1\le m/2$, $q\leq 19m/40$, and $l\geq l_-$, we can integrate by parts using Lemma \ref{tech5} to show that $|\mathcal{C}_l^1(\xi)|\lesssim 2^{-2m}$ in this case. On the other hand, if 
\begin{equation}\label{Alx48.5}
l\in[-1000\delta m,l_0]
\end{equation}
then we write $i\mathcal{C}_l^1(\xi)=\partial_sK_l^1(\xi)-\mathcal{E}_l^1(\xi)$, where
\begin{equation*}
\begin{split}
K_l^1(\xi):=\int_{\mathbb{R}^2}e^{is[\Phi_{\sigma\mu\nu}(\xi,\eta)+\Psi_{\nu\theta\kappa}(\eta)]}&\frac{\widetilde{\varphi_{l}}(\Phi(\xi,\eta))}{\Phi_{\sigma\mu\nu}(\xi,\eta)+\Psi_{\nu\theta\kappa}(\eta)}\varphi(\kappa_\theta^{-1}\Omega_\eta\Phi(\xi,\eta)))\\
&\times\varphi_{\leq -m/25}(\Psi_{\nu\theta\kappa}(\eta))\varphi_{Hi}(\nabla_\xi\Phi(\xi,\eta))\widehat{f^{\mu}_{j_1,k_1}}(\xi-\eta)h^q(\eta)d\eta
\end{split}
\end{equation*}
and
\begin{equation*}
\begin{split}
\mathcal{E}_l^1(\xi):=\int_{\mathbb{R}^2}&e^{is[\Phi_{\sigma\mu\nu}(\xi,\eta)+\Psi_{\nu\theta\kappa}(\eta)]}\frac{\widetilde{\varphi_{l}}(\Phi(\xi,\eta))}{\Phi_{\sigma\mu\nu}(\xi,\eta)+\Psi_{\nu\theta\kappa}(\eta)}\varphi(\kappa_\theta^{-1}\Omega_\eta\Phi(\xi,\eta)))\\
&\times\varphi_{\leq -m/25}(\Psi_{\nu\theta\kappa}(\eta))\varphi_{Hi}(\nabla_\xi\Phi(\xi,\eta))\partial_s\big[\widehat{f^{\mu}_{j_1,k_1}}(\xi-\eta,s)h^q(\eta,s)\big]d\eta.
\end{split}
\end{equation*}
Notice that, in view of \eqref{Alx48.1} (recall also that $q\leq 19m/40$) and Lemma \ref{dtfLemPrelim},
\begin{equation*}
|\widehat{f^{\mu}_{j_1,k_1}}(\xi-\eta)h^q(\eta)|\lesssim 2^{-m+50\delta m},\qquad \big|\partial_s\big[\widehat{f^{\mu}_{j_1,k_1}}(\xi-\eta,s)h^q(\eta,s)\big]\big|\lesssim 2^{-3m/2-m/100},
\end{equation*}
in the supports of the integrals. Therefore
\begin{equation*}
\begin{split}
|K_l^1(\xi)|&\lesssim 2^{\delta^2m}2^{-2l}\kappa_\theta 2^{-m+50\delta m}\lesssim 2^{-2l}2^{-3m/2+60\delta m},\\
|\mathcal{E}_l^1(\xi)|&\lesssim 2^{\delta^2m}2^{-2l}\kappa_\theta 2^{-3m/2-m/100}\lesssim 2^{-2l}2^{-2m}.
\end{split}
\end{equation*}
These are acceptable contributions as in \eqref{INCOK}, due to the condition \eqref{Alx48.5}.

The integral $\mathcal{C}_l^2$ containing the factor $\varphi_{\geq -\D}(\Psi_{\nu\theta\kappa}(\eta))$ can be controlled in a similar way, writing it as $\partial_sK_l^2(\xi)-\mathcal{E}_l^2(\xi)$ as in the case \eqref{Alx48.5} above. This completes the proof of \eqref{INCOK}.

{\bf{Case 2.}} We consider now $I^{lo}$ in the case
\begin{equation}\label{Alx95.5}
\min\{k,k_1,k_2\}\ge -\D,\quad j_1\leq\min(j_2,(1-200\delta)m),\quad
j_2\leq(1-50\delta)m.
\end{equation}
Integrations by parts, first in $\eta$ using Lemma \ref{tech5} then in $\Omega_\eta$ using Lemma \ref{RotIBP} show that
\begin{equation*}
\begin{split}
\Vert I^{lo}[f^\mu_{j_1,k_1},f^\nu_{j_2,k_2}]-&I^{S}[f^\mu_{j_1,k_1},f^\nu_{j_2,k_2}]\Vert_{L^2}\lesssim 2^{-2m},\\
\mathcal{F}\{I^{S}[f,g]\}(\xi):=&\int_{\mathbb{R}^2}e^{is\Phi(\xi,\eta)}\varphi_{lo}(\Phi(\xi,\eta))\varphi(\kappa_r^{-1}\nabla_\eta\Phi(\xi,\eta))\varphi(\kappa_\theta^{-1}\Omega_\eta\Phi(\xi,\eta))\widehat{f}(\xi-\eta)\widehat{g}(\eta)d\eta,
\end{split}
\end{equation*}
where $\kappa_r:=2^{\delta^2m}\big(2^{j_2-m}+2^{-m/2}\big)$, $\kappa_\theta:=2^{\delta^2m-m/2}$. Using  Proposition \ref{spaceres} and Proposition \ref{separation1} (iii), we see that
\begin{equation*}
I^S[f^\mu_{j_1,k_1},f^\nu_{j_2,k_2}]\equiv0\quad\text{ if }\quad\mu+\nu=0\quad\text{ and }\quad I^S[f^{\mu}_{j_1,k_1},f^\nu_{j_2,k_2}]\equiv I^S[f^{\mu}_{j_1,k_1,0},f^{\nu}_{j_2,k_2,0}].
\end{equation*}

We may rewrite the nontrivial terms as
\begin{equation*}
\begin{split}
\mathcal{F}\big\{I^S[f^\mu_{j_1,k_1},f^\nu_{j_2,k_2}]\big\}(\xi)&=e^{is\Psi(\xi)}g(\xi,s),\\
\end{split}
\end{equation*}
where
\begin{equation}\label{Alx96}
\begin{split}
g(\xi,s):=\int_{\mathbb{R}^2}e^{is\left[\Phi(\xi,\eta)-\Psi(\xi)\right]}\varphi(\kappa_\theta^{-1}\Omega_\eta\Phi(\xi,\eta))\varphi(\kappa_r^{-1}\nabla_\eta\Phi(\xi,\eta))\varphi_{lo}(\Phi(\xi,\eta))\\
\times\widehat{f^{\mu}_{j_1,k_1,0}}(\xi-\eta)\widehat{f^{\nu}_{j_2,k_2,0}}(\eta)\,d\eta.
\end{split}
\end{equation}
We claim that this gives a contribution as in \eqref{dtfType1} for $q:=\max\{0,j_2-m/2\}$. Indeed, using \eqref{FL1bd}, \eqref{FLinftybd}, and \eqref{cas10}, we see that
\begin{equation*}
\begin{split}
\vert \varphi_k(\xi)g(\xi,s)\vert&\lesssim 2^{\delta^2m}\Vert \widehat{f^{\mu}_{j_1,k_1,0}}\Vert_{L^\infty}\min\{\kappa_r^2\Vert \widehat{f^{\nu}_{j_2,k_2,0}}\Vert_{L^\infty},\kappa_\theta\kappa_r^{1/2}2^{-j_2+21\delta j_2}\}\lesssim 2^{7\delta^2m}2^{-m-q}.
\end{split}
\end{equation*}
The bound on $\xi$-derivatives follows from the fact that
\begin{equation*}
\vert\nabla_\xi s\left[\Phi(\xi,\eta)-\Psi(\xi)\right]\vert\lesssim \vert s\vert \vert \nabla_\xi\Psi(\xi,\eta)-\nabla_\xi\Phi(\xi,p(\xi))\vert\lesssim 2^{m}\kappa_r\lesssim (2^{m/2}+2^{j_2})2^{2\delta^2m}.
\end{equation*}
Finally, the bound on $\partial_sg$ follows in the same way as in the proof of Lemma \ref{dtfLemPrelim}, see \eqref{Alx99.6}. The bounds \eqref{Omegadtf} follow by examining the defining formulas above and the identities \eqref{Alx99.7}.
\end{proof}

\section{Dispersive control II: improved control of the $Z$ norm}\label{Sec:Z1Norm} In this section we show how to control the $Z$ component of the norms. Using the formulas \eqref{DuhamelDER2}, the symbol structure of the multipliers $\mathfrak{m}^a_{\sigma\mu\nu}$, and Lemma \ref{lem:CZ}, Proposition \ref{BootstrapZnorm} follows easily from Proposition \ref{ZNormProp} and Proposition \ref{CubicTermZnorm} below.

\begin{proposition}\label{ZNormProp}
Assume that $t\in[0,T]$ is fixed
\begin{equation*}
\sup_{0\le s\le t}\left\{\Vert f^\mu(s)\Vert_{H^{N_0/2}\cap Z_1^\mu\cap H^{N_1/2}_\Omega}+\Vert f^\nu(s)\Vert_{H^{N_0/2}\cap Z^\nu_1\cap H^{N_1/2}_\Omega}\right\}\le 1
\end{equation*}
and that $\partial_sf^\mu$, $\partial_sf^\nu$ satisfy the conclusions of Lemma \ref{dtfLemPrelim} and Lemma \ref{dtfLem}. For $\sigma,\mu,\nu\in \mathcal{P}$ and $m\in\{0,\ldots,L+1\}$ define
\begin{equation*}
\mathcal{F}\left\{T^{\sigma\mu\nu}_{m}[f,g]\right\}(\xi):=\int_{\mathbb{R}}q_m(s)\int_{\mathbb{R}^2}e^{is\Phi_{\sigma\mu\nu}(\xi,\eta)}\widehat{f}(\xi-\eta,s)\widehat{g}(\eta,s)d\eta ds.
\end{equation*}
Then
\begin{equation*}
\sum_{k_1,k_2\in\mathbb{Z}}2^{\max(k_1,k_2,0)}\Vert P_kT^{\sigma\mu\nu}_{m}[P_{k_1}f^\mu,P_{k_2}f^\nu]\Vert_{Z^\sigma_1}\lesssim 2^{-\delta^5m}.
\end{equation*}
\end{proposition}

The rest of this section is concerned with the proof of Proposition \ref{ZNormProp}. We consider first a few simple cases before moving to the main analysis in the next subsections. Recall that, for any $k\in\mathbb{Z}$ and $m\in\{0,1,\ldots\}$,
\begin{equation}\label{DirectBds1}
\begin{split}
\sup_{0\leq s\leq t}\left\{\Vert P_kf^\mu(s)\Vert_{L^2}+\Vert P_kf^\nu(s)\Vert_{L^2}\right\}&\lesssim \min\{2^{(1-20\delta)k},2^{-N_0/2k}\},\\
\sup_{2^{m}-1\le s\le t}\left\{\Vert P_ke^{-is\Lambda_\mu}f^\mu(s)\Vert_{L^\infty}+\Vert P_ke^{-is\Lambda_\nu}f^\nu(s)\Vert_{L^\infty}\right\}&\lesssim \min\{2^{(2-20\delta)k},2^{-(1-21\delta)m}\}.
\end{split}
\end{equation}
For simplicity of notation, we often omit the subscripts $\sigma\mu\nu$ and write $\Phi_{\sigma\mu\nu}=\Phi$ and $\Psi_{\sigma\mu\nu}=\Psi$. Let $I_m$ denote the support of the function $q_m$. 

\begin{lemma}\label{ZNormEstSimpleLem1}
Assume that $f^\mu,f^\nu$ are as in Proposition \ref{ZNormProp} and let $(k,j)\in\mathcal{J}$. Then
\begin{equation}\label{Alx21}
2^{6k_+}\sum_{\max\{k_1,k_2\}\ge \delta^2(j+m)-\D^2}2^{\max(k_1,k_2,0)}\Vert Q_{jk}T^{\sigma\mu\nu}_m[P_{k_1}f^\mu,P_{k_2}f^\nu]\Vert_{B^\sigma_j}\lesssim 2^{-\delta^4m},
\end{equation}
\begin{equation}\label{Alx22}
2^{6k_+}\sum_{\min\{k_1,k_2\}\le -(j+m)(1+11\delta)/2+\D^2}2^{\max(k_1,k_2,0)}\Vert Q_{jk}T^{\sigma\mu\nu}_m[P_{k_1}f^\mu,P_{k_2}f^\nu]\Vert_{B^\sigma_j}\lesssim 2^{-\delta^4m},
\end{equation}
\begin{equation}\label{Alx23}
\text{ if }j+k\leq 19\delta j-17\delta m\text{ then }\,\,\sum_{k_1,k_2\in\mathbb{Z}}2^{\max(k_1,k_2,0)}\Vert Q_{jk}T^{\sigma\mu\nu}_m[P_{k_1}f^\mu,P_{k_2}f^\nu]\Vert_{B^\sigma_j}\lesssim 2^{-\delta^4m},
\end{equation}
\begin{equation}\label{Alx24}
\text{ if }j\geq 3m\text{ then }\,\, 2^{6k_+}\sum_{-j\le k_1,k_2\le 2\delta^2j}2^{\max(k_1,k_2,0)}\Vert Q_{jk}T^{\sigma\mu\nu}_m[P_{k_1}f^\mu,P_{k_2}f^\nu]\Vert_{B^\sigma_j}\lesssim 2^{-\delta^4m}.
\end{equation}
\end{lemma}

\begin{proof} Using \eqref{DirectBds1}, the left-hand side of \eqref{Alx21} is dominated by
\begin{equation*}
\begin{split}
C\sum_{\max\{k_1,k_2\}\ge \delta^2 (m+j)-\D^2}2^{j+m}2^{8\max(k_1,k_2,0)}\sup_{s\in I_m}\Vert P_{k_1}f^\mu(s)\Vert_{L^2}\Vert P_{k_2}f^\nu(s)\Vert_{L^2}\lesssim 2^{-2m},
\end{split}
\end{equation*}
which is acceptable. Similarly, 
\begin{equation*}
\begin{split}
&2^j\Vert T^{\sigma\mu\nu}_m[P_{k_1}f^\mu,P_{k_2}f^\nu]\Vert_{L^2}\\
&\lesssim 2^{j+m}\sup_{s\in I_m}\min\big\{\Vert \widehat{P_{k_1}f^\mu}(s)\Vert_{L^1}\Vert P_{k_2}f^\nu(s)\Vert_{L^2},\Vert P_{k_1}f^\mu(s)\Vert_{L^2}\Vert \widehat{P_{k_2}f^\nu}(s)\Vert_{L^1}\big\}\\
&\lesssim 2^{j+m}2^{(2-20\delta)\min\{k_1,k_2\}}2^{-N_0\max(k_1,k_2,0)/2},
\end{split}
\end{equation*}
and the bound \eqref{Alx21} follows by summation over $2\min\{k_1,k_2\}\le -(j+m)(1+11\delta)+2\D^2$.

To prove \eqref{Alx23} we may assume that
\begin{equation}\label{Alx25}
j+k\leq 19\delta j-17\delta m,\qquad-2(j+m)/3\le k_1,k_2\le \delta^2(j+m)-\D^2.
\end{equation}
With $l:=-12\delta m-\D$ we decompose
\begin{equation*}
\begin{split}
T^{\sigma\mu\nu}_m[P_{k_1}f^\mu,P_{k_2}f^\nu]&=T^{hi}[P_{k_1}f^\mu,P_{k_2}f^\nu]+T^{lo}[P_{k_1}f^\mu,P_{k_2}f^\nu],\\
\widehat{T^\ast[f,g]}(\xi)&:=\int_{\mathbb{R}}q_m(s)\int_{\mathbb{R}^2}e^{is\Phi(\xi,\eta)}\varphi_\ast(\Phi(\xi,\eta))\widehat{f}(\xi-\eta,s)\widehat{g}(\eta,s)d\eta ds,\\
\varphi_{lo}(x)&:=\varphi_{\le l}(x),\qquad\varphi_{hi}(x):=1-\varphi_{lo}(x),\qquad\ast\in\{hi,lo\}.
\end{split}
\end{equation*}
We first examine $T^{hi}$. Integration by parts in time shows that
\begin{equation*}
\begin{split}
cT^{hi}[P_{k_1}f^\mu,P_{k_2}f^\nu]&=\mathcal{A}[P_{k_1}f^\mu,P_{k_2}f^\nu]+\mathcal{B}[P_{k_1}\partial_sf^\mu,P_{k_2}f^\nu]+\mathcal{B}[P_{k_1}f^\mu,P_{k_2}\partial_sf^\nu],\\
\widehat{\mathcal{A}[f,g]}(\xi)&:=\int_{\mathbb{R}}q^\prime_m(s)\int_{\mathbb{R}^2}e^{is\Phi(\xi,\eta)}\widetilde{\varphi}_{hi}(\Phi(\xi,\eta))\widehat{f}(\xi-\eta,s)\widehat{g}(\eta,s)d\eta ds,\\
\widehat{\mathcal{B}[f,g]}(\xi)&:=\int_{\mathbb{R}}q_m(s)\int_{\mathbb{R}^2}e^{is\Phi(\xi,\eta)}\widetilde{\varphi}_{hi}(\Phi(\xi,\eta))\widehat{f}(\xi-\eta,s)\widehat{g}(\eta,s)d\eta ds,
\end{split}
\end{equation*}
where $\widetilde{\varphi}_{hi}(x):=x^{-1}\varphi_{hi}(x)$. We observe that, using \eqref{DirectBds1} and \eqref{dtFTotalBds2Prelim},
\begin{equation*}
\begin{split}
\Vert \mathcal{F}\mathcal{A}[P_{k_1}f^\mu,P_{k_2}f^\nu]\Vert_{L^\infty}\lesssim 2^{12\delta m}\sup_{s}\Vert P_{k_1}f^{\mu}(s)\Vert_{L^2}\Vert P_{k_2}f^\nu(s)\Vert_{L^2}\lesssim 2^{13\delta m},\\
\Vert \mathcal{F}\mathcal{B}[P_{k_1}\partial_sf^\mu,P_{k_2}f^\nu]\Vert_{L^\infty}\lesssim 2^{(1+12\delta)m}\sup_s\Vert P_{k_1}\partial_sf^\mu\Vert_{L^2}\Vert P_{k_2}f^\nu(s)\Vert_{L^2}\lesssim 2^{16\delta m},\\
\Vert \mathcal{F}\mathcal{B}[P_{k_1}f^\mu,P_{k_2}\partial_sf^\nu]\Vert_{L^\infty}\lesssim 2^{(1+12\delta)m}\sup_s\Vert P_{k_1}f^\mu\Vert_{L^2}\Vert P_{k_2}\partial_sf^\nu(s)\Vert_{L^2}\lesssim 2^{16\delta m}.
\end{split}
\end{equation*}
Summing in $k_1,k_2$ as in \eqref{Alx25}, we obtain an acceptable contribution. 

To bound the contribution of $T^{lo}$ we examine first Proposition \ref{separation1} (i). In particular, since $k\leq -\D$, $P_kT^{lo}$ is nontrivial only when $\vert k_1\vert,\vert k_2\vert \lesssim 1$. We define, as before,
\begin{equation}\label{Alx25.5}
f^\mu_{j_1,k_1}=P_{[k_1-2,k_1+2]}Q_{j_1k_1}f^\mu,\qquad f^\nu_{j_2,k_2}=P_{[k_2-2,k_2+2]}Q_{j_2k_2}f^\nu,
\end{equation}
for $(k_1,j_1),(k_2,j_2)\in\mathcal{J}$. It suffices to show that
\begin{equation}\label{Alx26}
\sum_{(k_1,j_1),(k_2,j_2)\in\mathcal{J}}\Vert \mathcal{F}P_kT^{lo}[f^\mu_{j_1,k_1},f^\nu_{j_2,k_2}]\Vert_{L^\infty}\lesssim 2^{16\delta m}.
\end{equation}

If $\max\{j_1,j_2\}\le (1-\delta^2)m$ then integration by parts in $\eta$, using Proposition \ref{separation1} (i) and Lemma \ref{tech5}, gives an acceptable contribution. On the other hand, if $j_1=\max\{j_1,j_2\}\ge (1-\delta^2)m$ then 
\begin{equation*}
\|\widehat{f^\mu_{j_1,k_1}}(s)\|_{L^2}\lesssim 2^{-j_1+20\delta j_1},\qquad \|\widehat{f^\nu_{j_2,k_2}}(s)\|_{L^\infty}\lesssim 1,
\end{equation*}
as a consequence of Proposition \ref{separation1} (i) and Lemma \ref{LinEstLem}. Therefore, using also \eqref{cas4},
\begin{equation*}
\begin{split}
\Vert \mathcal{F}P_kT^{lo}[f^\mu_{j_1,k_1},f^\nu_{j_2,k_2}]\Vert_{L^\infty}&\lesssim 2^m\sup_{s}\Vert\widehat{f^\mu_{j_1,k_1}}(s)\Vert_{L^2}\Vert\widehat{f^\nu_{j_2,k_2}}(s)\Vert_{L^\infty}2^{-5.5\delta m}\lesssim 2^{15\delta m}2^{-\delta^2j_1}.
\end{split}
\end{equation*}
The desired bound \eqref{Alx26} follows.

Finally, to prove \eqref{Alx24} we may assume that
\begin{equation*}
j\ge \max(3m,\D),\qquad j+k\ge 2\delta j,\qquad -j\le k_1,k_2\le 2\delta^2j,
\end{equation*}
and define $f^\mu_{j_1,k_1},f^\nu_{j_2,k_2}$ as before. If $\min\{j_1,j_2\}\ge 99j/100-\D$ then
\begin{equation*}
\begin{split}
\Vert T^{\sigma\mu\nu}_m[f^\mu_{j_1,k_1},f^\nu_{j_2,k_2}]\Vert_{L^2}&\lesssim 2^m\sup_{s}\Vert \widehat{f^\mu_{j_1,k_1}}(s)\Vert_{L^1}\Vert f^\nu_{j_2,k_2}(s)\Vert_{L^2}\lesssim 2^m2^{-(1-20\delta)j_1-(1/2-\delta)j_2}
\end{split}
\end{equation*}
and therefore
\begin{equation*}
\begin{split}
\sum_{-j\le k_1,k_2\le 2\delta^2j}\sum_{\min\{j_1,j_2\}\ge 99j/100-\D}2^{8\max(k_1,k_2,0)}\Vert Q_{jk}T^{\sigma\mu\nu}_m[f^\mu_{j_1,k_1},f^\nu_{j_2,k_2}]\Vert_{B^\sigma_j}\lesssim 2^{-m/10}.
\end{split}
\end{equation*}
On the other hand, if $j_1\le 99j/100-\mathcal{D}$ then we rewrite
\begin{equation*}
\begin{split}
&Q_{jk}T_{m}^{\sigma\mu\nu}[f^\mu_{j_1,k_1},f^\nu_{j_2,k_2}](x)\\
&=C\phii_j^{(k)}(x)\cdot\int_{\mathbb{R}}q_m(s)\int_{\mathbb{R}^2}\left[\int_{\mathbb{R}^2}e^{i\left[s\Phi(\xi,\eta)+x\cdot \xi\right]}\varphi_k(\xi)\widehat{f^\mu_{j_1,k_1}}(\xi-\eta,s)d\xi\right] \widehat{f^\nu_{j_2,k_2}}(\eta,s)d\eta ds.
\end{split}
\end{equation*}
In the support of integration, we have the lower bound $\left\vert\nabla_\xi\left[s\Phi(\xi,\eta)+x\cdot\xi\right]\right\vert\approx\vert x\vert\approx 2^j$. Integration by parts in $\xi$ using Lemma \ref{tech5} gives
\begin{equation}\label{Alx26.5}
\left\vert Q_{jk}T_{m}^{\sigma\mu\nu}[f^\mu_{j_1,k_1},f^\nu_{j_2,k_2}](x)\right\vert\lesssim 2^{-10j}
\end{equation}
which gives an acceptable contribution. This finishes the proof.
\end{proof}

\subsection{The main decomposition}

We may assume that
\begin{equation}\label{Ass2}
\begin{split}
&-(j+m)(1+11\delta)/2\le k_1,k_2\le \delta^2(j+m),\qquad j+k\ge 19\delta j-17\delta m,\\
&j\le 3m,\qquad m\geq \D^2/8.
\end{split}
\end{equation}
Recall the definition \eqref{Alx80}. We fix $l_-:=\lfloor-(1-\delta/2) m\rfloor$ and $l_0:=\lfloor-12\delta m\rfloor$, and decompose
\begin{equation*}
\begin{split}
T^{\sigma\mu\nu}_m[f,g]&=\sum_{l_-\leq l\leq l_0}T_{m,l}[f,g],\\
\widehat{T_{m,l}[f,g]}(\xi)&:=\int_{\mathbb{R}}q_m(s)\int_{\mathbb{R}^2}e^{is\Phi(\xi,\eta)}\varphi_l^{[l_-,l_0]}(\Phi(\xi,\eta))\widehat{f}(\xi-\eta,s)\widehat{g}(\eta,s)d\eta ds.
\end{split}
\end{equation*}
When $l_-<l\leq l_0$, we may integrate by parts in time to rewrite $T_{m,l}[P_{k_1}f^\mu,P_{k_2}f^\nu]$,
\begin{equation}\label{Alx41}
\begin{split}
&T_{m,l}[P_{k_1}f^\mu,P_{k_2}f^\nu]=i\mathcal{A}_{m,l}[P_{k_1}f^\mu,P_{k_2}f^\nu]+i\mathcal{B}_{m,l}[P_{k_1}\partial_sf^\mu,P_{k_2}f^\nu]+i\mathcal{B}_{m,l}[P_{k_1}f^\mu,P_{k_2}\partial_sf^\nu],\\
&\mathcal{F}\mathcal{A}_{m,l}[P_{k_1}f,P_{k_2}g](\xi):=\int_{\mathbb{R}}q_m^\prime(s)\int_{\mathbb{R}^2}e^{is\Phi(\xi,\eta)}\widetilde{\varphi_{l}}(\Phi(\xi,\eta))\widehat{P_{k_1}f}(\xi-\eta,s)\widehat{P_{k_2}g}(\eta,s)\,d\eta ds,\\
&\mathcal{F}\mathcal{B}_{m,l}[P_{k_1}f,P_{k_2}g](\xi):=\int_{\mathbb{R}}q_m(s)\int_{\mathbb{R}^2}e^{is\Phi(\xi,\eta)}\widetilde{\varphi_{l}}(\Phi(\xi,\eta))\widehat{P_{k_1}f}(\xi-\eta,s)\widehat{P_{k_2}g}(\eta,s)\,d\eta ds,
\end{split}
\end{equation}
where $\widetilde{\varphi}_l(x)=x^{-1}\varphi_{l}(x)$ for $l<l_0$ and $\widetilde{\varphi}_{l_0}(x)=x^{-1}\varphi_{\geq l_0}(x)$. It is easy to see that the main Proposition \ref{ZNormProp} follows from Lemma \ref{ZNormEstSimpleLem1} and Lemmas \ref{FSPLem}--\ref{ResLem} below.

\begin{lemma}\label{FSPLem}
Assume that \eqref{Ass2} holds and, in addition, $m+\D\le j$. Then, for $l_-\le l\le  l_0$,
\begin{equation*}
2^{(1-20\delta)j}\Vert Q_{jk}T_{m,l}[P_{k_1}f^\mu,P_{k_2}f^\nu]\Vert_{L^2}\lesssim 2^{-50\delta^2m}.
\end{equation*}
\end{lemma}

\begin{lemma}\label{TimeNonRes}
Assume that \eqref{Ass2} holds and, in addition, $j\leq m+\D$. Then
\begin{equation*}
2^{(1-20\delta)j}\Vert Q_{jk}T_{m,l_0}[P_{k_1}f^\mu,P_{k_2}f^\nu]\Vert_{L^2}\lesssim  2^{-50\delta^2m}.
\end{equation*}
\end{lemma}

\begin{lemma}\label{StronglyResLem}
Assume that \eqref{Ass2} holds and, in addition, $j\leq m+\D$. Then, for $l_-<l<l_0$
\begin{equation*}
\Vert Q_{jk}T_{m,l_-}[P_{k_1}f^\mu,P_{k_2}f^\nu]\Vert_{B^\sigma_j}+\Vert Q_{jk}\mathcal{A}_{m,l}[P_{k_1}f^\mu,P_{k_2}f^\nu]\Vert_{B_j^\sigma}\lesssim 2^{-50\delta^2m}.
\end{equation*}
\end{lemma}

\begin{lemma}\label{ResLem}
Assume that \eqref{Ass2} holds and, in addition, $j\leq m+\D$. Then, for $l_-<l<l_0$
\begin{equation*}
\Vert Q_{jk}\mathcal{B}_{m,l}[P_{k_1}f^\mu,P_{k_2}\partial_sf^\nu]\Vert_{B^\sigma_j}\lesssim 2^{-50\delta^2m}.
\end{equation*}
\end{lemma}

\subsection{Approximate finite speed of propagation} In this subsection we prove Lemma \ref{FSPLem}. We define the functions $f^\mu_{j_1,k_1}$ and $f^\nu_{j_2,k_2}$ as before, see \eqref{Alx25.5}, and we further decompose 
\begin{equation}\label{Alx70}
f^\mu_{j_1,k_1}=\sum_{n_1=0}^{j_1+1} f^\mu_{j_1,k_1,n_1},\qquad f^\nu_{j_2,k_2}=\sum_{n_2=0}^{j_2+1} f^\nu_{j_2,k_2,n_2}
\end{equation}
as in \eqref{Alx100}. If $l\le -3m/4$ and $\min\{j_1,j_2\}\le j-\delta m$ then the same argument as in the proof of \eqref{Alx24} leads to rapid decay, as in \eqref{Alx26.5}. If $l\le -3m/4$ and $\min\{j_1,j_2\}\ge j-\delta m$, then we use Lemma \ref{Shur2Lem} (Schur's test) to estimate
\begin{equation*}
\begin{split}
\big\Vert T_{m,l}&[f^\mu_{j_1,k_1,n_1},f^\nu_{j_2,k_2,n_2}]\big\Vert_{L^2}\\
&\lesssim 2^m2^{l/2-n_1/2-n_2/2}\sup_{s}\big[\big\Vert \sup_\theta|\widehat{f^\mu_{j_1,k_1,n_1}}(r\theta,s)|\big\Vert_{L^2(rdr)}\big\Vert \sup_\theta|\widehat{f^\nu_{j_2,k_2,n_2}}(r\theta,s)|\big\Vert_{L^2(rdr)}\big].
\end{split}
\end{equation*}
The desired bound follows from \eqref{RadL2} and summation over $j_1,n_1,j_2,n_2$ in this case.

We now consider the case $-3m/4\le l\leq l_0$ and use the formula \eqref{Alx41}. The contribution of $\mathcal{A}_{m,l}$ can be estimated as above and we focus on the contribution of $\mathcal{B}_{m,l}$. We decompose
\begin{equation*}
\mathcal{B}_{m,l}[P_{k_1}f^\mu,P_{k_2}\partial_sf^\nu]=\sum_{j_1,n_1}\mathcal{B}_{m,l}[f^\mu_{j_1,k_1,n_1},P_{k_2}\partial_sf^\nu].
\end{equation*}
If $j_1\le j-\delta m$ then we integrate by parts in $\xi$ to prove rapid decay as in \eqref{Alx26.5}. If 
\begin{equation}\label{Alx72}
-3m/4\le l\leq l_0,\qquad j_1\ge j-\delta m,\qquad k_2\leq -m/2,
\end{equation}
then we notice that
\begin{equation*}
\sup_{|\lambda|\leq 2^{(1-\delta)m}}\|e^{-i(s+\lambda)\Lambda_\nu}(P_{k_2}\partial_sf^\nu)(s)\|_{L^{\infty}}\lesssim 2^{-2m+43\delta m},
\end{equation*}
as a consequence of \eqref{dtFTotalBds}. Therefore, using Lemma \ref{PhiLocLem},
\begin{equation}\label{Alx71}
2^{(1-20\delta)j}\big\|\mathcal{B}_{m,l}[f^\mu_{j_1,k_1,n_1},P_{k_2}\partial_sf^\nu]\big\|_{L^2}\lesssim 2^{(1-20\delta)j}2^m 2^{-l}\cdot 2^{-2m+43\delta m}2^{-j_1+20\delta j_1}2^{n_1/2-19\delta n_1}.
\end{equation}
Moreover, in view of Proposition \ref{separation1} (i),
\begin{equation*}
\mathcal{B}_{m,l}[f^\mu_{j_1,k_1,n_1},P_{k_2}\partial_sf^\nu]=0\qquad\text{ unless }n_1=0\text{ or }l=l_0.
\end{equation*}
The desired bound in the lemma follows from \eqref{Alx72} and \eqref{Alx71} if $n_1=0$ or if $l=l_0$ and $n_1\leq 3m/2$. On the other hand, if $l=l_0$ and $n_1\geq 3m/2$ then we estimate
\begin{equation*}
\begin{split}
2^{(1-20\delta)j}\big\|\mathcal{B}_{m,l}[f^\mu_{j_1,k_1,n_1},P_{k_2}\partial_sf^\nu]\big\|_{L^2}&\lesssim 2^{(1-20\delta)j}2^m 2^{-l_0}\cdot \sup_{s}\big[\|\widehat{f^\mu_{j_1,k_1,n_1}}(s)\|_{L^1}\|P_{k_2}\partial_sf^\nu(s)\|_{L^2}\big]\\
&\lesssim 2^{(1-20\delta)j}2^m 2^{12\delta m}\cdot 2^{-j_1+21\delta j_1}2^{-19\delta n_1}2^{-m+4\delta m},
\end{split}
\end{equation*}
using \eqref{FL1bd} and \eqref{dtFTotalBds2Prelim}. This completes the proof in the case \eqref{Alx72} (recall that $j\leq 3m$).

Finally, in the remaining case
\begin{equation}\label{Alx73}
-3m/4\le l\leq l_0,\qquad j_1\ge j-\delta m,\qquad k_2\geq -m/2,
\end{equation}
we decompose, according to Lemma \ref{dtfLemPrelim},
\begin{equation*}
\begin{split}
\partial_sf^\nu(s)&=\widetilde{f^\nu_C}(s)+\widetilde{f^\nu_{NC}}(s),\qquad \Vert P_{k_2}\widetilde{f^\nu_{NC}}(s)\Vert_{L^2}\lesssim 2^{-29/20m},\\
\widehat{\widetilde{f_{C}^\nu}}(\xi,s)&=\sum_{\alpha,\beta\in\mathcal{P},\,\alpha+\beta\neq 0} e^{is\Psi_{\nu\alpha\beta}(\xi)}g_{\nu\alpha\beta}(\xi,s),\qquad\Vert \varphi_{k_2}(\xi)D^\rho_\xi g_{\nu\alpha\beta}(\xi,s)\Vert_{L^\infty}\lesssim 2^{-m/2+(1-\delta)m\vert\rho\vert}.
\end{split}
\end{equation*}

We can rewrite
\begin{equation*}
\begin{split}
Q_{jk}&\mathcal{B}_{m,l}[f^\mu_{j_1,k_1,n_1},P_{k_2}\widetilde{f^\nu_C}](x)=c\sum_{\alpha,\beta\in\mathcal{P},\,\alpha+\beta\neq 0}\phii_j^{(k)}(x)\cdot\int_{\mathbb{R}}q_m(s)\int_{\mathbb{R}^2}\widehat{f^\mu_{j_1,k_1,n_1}}(\eta,s)\\
&\times \left[\int_{\mathbb{R}^2}e^{i[s\Phi_{\si\mu\nu}(\xi,\xi-\eta)+s\Psi_{\nu\alpha\beta}(\xi-\eta)+x\cdot \xi]}\widetilde{\varphi}_{l}(\Phi(\xi,\xi-\eta))\varphi_k(\xi)\varphi_{k_2}(\xi-\eta)\widehat{g_{\nu\alpha\beta}}(\xi-\eta,s)d\xi\right]d\eta ds.
\end{split}
\end{equation*}
Once again, integration by parts in $\xi$ using Lemma \ref{tech5} leads to an acceptable contribution. In addition, using Lemma \ref{Shur2Lem} and Lemma \ref{LinEstLem}, we find that
\begin{equation*}
\begin{split}
\big\Vert \mathcal{B}_{m,l}[f^\mu_{j_1,k_1,n_1},P_{k_2}\widetilde{f^\nu_{NC}}]\big\Vert_{L^2}&\lesssim 2^{m-l}2^{\frac{l-n_1}{2}}2^{2\delta^2m}\sup_{s}\big[\big\Vert \sup_{\theta\in\mathbb{S}^1}|\widehat{f^\mu_{j_1,k_1,n_1}}(r\theta,s)|\big\Vert_{L^2(rdr)}\Vert P_{k_2}\widetilde{f^\nu_{NC}}(s)\Vert_{L^2}\big]\\
&\lesssim 2^{m-l/2}2^{-j_1+21\delta j_1-19\delta n_1}2^{-29/20m}2^{4\delta^2m}.\\
\end{split}
\end{equation*}
The desired bound on this term follows from \eqref{Alx73}.

\subsection{Time-nonresonant interactions}\label{ProofZ1} In this subsection we prove Lemma \ref{TimeNonRes}. We define the functions $f^\mu_{j_1,k_1},f^\nu_{j_2,k_2}$ as before, and use \eqref{Alx41}. If $j_1\le j_2$, using Lemma \ref{PhiLocLem} we estimate
\begin{equation*}
\begin{split}
\Vert \mathcal{A}_{m,l_0}[f^\mu_{j_1,k_1},&f^\nu_{j_2,k_2}]\Vert_{L^2}\lesssim 2^{12\delta m}\sup_{s,\lambda\approx 2^m}\big[\Vert e^{-i\lambda\Lambda_\mu}f^\mu_{j_1,k_1}(s)\Vert_{L^\infty}\Vert f^\nu_{j_2,k_2}(s)\Vert_{L^2}+2^{-5m}\big]\\
&\lesssim
2^{-4m}+2^{12\delta m-m+2\delta^2m}\sup_{s}\Vert f^\mu_{j_1,k_1}(s)\Vert_{L^1}\Vert f^\nu_{j_2,k_2}(s)\Vert_{L^2}\\
&\lesssim 2^{-4m}+2^{12\delta m-m+2\delta^2m}2^{j_1}2^{-j_1/2+\delta j_1}2^{-j_2/2+\delta j_2}\\
\end{split}
\end{equation*}
and, using also Lemma \ref{LinEstLem},
\begin{equation*}
\begin{split}
\Vert \mathcal{A}_{m,l_0}[f^\mu_{j_1,k_1},f^\nu_{j_2,k_2}]\Vert_{L^2}&\lesssim 2^{12\delta m}\sup_{s\approx 2^m}\Vert f^\mu_{j_1,k_1}(s)\Vert_{L^2}\Vert \widehat{f^\nu_{j_2,k_2}}(s)\Vert_{L^1}\lesssim 2^{12\delta m-(1-21\delta)j_2-j_1/3}.
\end{split}
\end{equation*}
The desired bound on the $\mathcal{A}_{m,l_0}$ term follows by using the first estimate when $j_1\leq m$ and the second estimate when $j_1\geq m$.

Similarly, using also \eqref{dtFTotalBds} and Lemma \ref{L1easy},
\begin{equation*}
\begin{split}
\Vert \mathcal{B}_{m,l_0}[P_{k_1}\partial_sf^\mu,f^\nu_{j_2,k_2}]\Vert_{L^2}&\lesssim 2^m2^{100\delta m}\sup_{s}\Vert e^{-is\Lambda_\mu}\partial_sP_{k_1}f^\mu(s)\Vert_{L^\infty}\Vert f^\nu_{j_2,k_2}(s)\Vert_{L^2}\\
&\lesssim 2^{-m+150\delta m}2^{-j_2/3}.
\end{split}
\end{equation*}
Moreover, if $j_2\leq 500\delta m$ then we use Lemma \ref{PhiLocLem}, \eqref{dtFTotalBds2Prelim}, and \eqref{LinftyBd} to estimate
\begin{equation*}
\begin{split}
\Vert \mathcal{B}_{m,l_0}[P_{k_1}\partial_sf^\mu,f^\nu_{j_2,k_2}]\Vert_{L^2}&\lesssim 2^m2^{12\delta m}\sup_{s,\lambda\approx 2^m}\big[\Vert P_{k_1}\partial_sf^\mu(s)\Vert_{L^2}\Vert e^{-i\lambda\Lambda_\nu}f^\nu_{j_2,k_2}(s)\Vert_{L^\infty}+2^{-5m}\big]\\
&\lesssim 2^{m+12\delta m}2^{-m+4\delta m}2^{-m+2\delta m}.
\end{split}
\end{equation*}
The desired conclusion of the lemma follows from these two bounds.

\subsection{The case of strongly resonant interactions}\label{ProofZ2} In this subsection, we prove Lemma \ref{StronglyResLem}. This is where we need the localization operators $A^\sigma_{n,(j)}$ to control the output. It is an instantaneous estimate, in the sense that the time evolution will play no role. Hence, it suffices to show the following: let $\chi\in C^\infty_c(\mathbb{R}^2)$ be supported in $[-1,1]$ and assume that $j, l, s, m$ satisfy
\begin{equation}\label{Ass3}
-m+\delta m/2\le l\le-7\delta m,\qquad 2^{m-4}\le s\le 2^{m+4},\qquad j\leq m+\D.
\end{equation}
Assume that
\begin{equation}\label{Ass3.1}
\Vert f\Vert_{H^{N_0/2}\cap H^{N_1/2}_\Omega\cap Z_1^\mu}+ \Vert g\Vert_{H^{N_0/2}\cap H^{N_1/2}_\Omega\cap Z_1^\nu}\le 1,
\end{equation}
and define
\begin{equation*}
\begin{split}
\widehat{I[f,g]}(\xi)&:=\int_{\mathbb{R}^2}e^{is\Phi(\xi,\eta)}\chi_l(\Phi(\xi,\eta))\widehat{f}(\xi-\eta)\widehat{g}(\eta)d\eta,\qquad \chi_l(x)=\chi(2^{-l}x).
\end{split}
\end{equation*}
Assume also that $k,k_1,k_2,j,m$ satisfy \eqref{Ass2}. Then
\begin{equation}\label{NRCCLNewBd}
2^{\delta m/2}2^{-l}\Vert Q_{jk}I[P_{k_1}f,P_{k_2}g]\Vert_{B^\sigma_j} \lesssim 2^{-50\delta^2m}.
\end{equation}

To prove \eqref{NRCCLNewBd} we define $f_{j_1,k_1}, g_{j_2,k_2}, f_{j_1,k_1,n_1}, g_{j_2,k_2,n_2}$ as in \eqref{Alx100}, $(k_1,j_1),(k_2,j_2)\in\mathcal{J}$, $n_1\in[0,j_1+1]$, $n_2\in[0,j_2+1]$. We will analyze several cases depending on the relative sizes of the main parameters $m,l,k,j,k_1,j_1,k_2,j_2$. In most cases, such as \eqref{Ass3.2}, \eqref{Alx74}, \eqref{Alx74.2}, \eqref{Alx74.25}, \eqref{Alx74.3}, we will prove the stronger bound
\begin{equation}\label{SuffNRCCLNewBd}
2^{(1-20\delta)j}2^{\delta m/2}2^{-l}\Vert Q_{jk}I[f_{j_1,k_1},g_{j_2,k_2}]\Vert_{L^2}\lesssim 2^{-51\delta^2m}.
\end{equation}
However, in the main case \eqref{BulkCase}, we can only prove the weaker bound
\begin{equation}\label{SuffNRCCLNewBdPrev}
2^{\delta m/2}2^{-l}\Vert Q_{jk}I[f_{j_1,k_1},g_{j_2,k_2}]\Vert_{B^\sigma_{j}}\lesssim 2^{-51\delta^2m}.
\end{equation}
These bounds clearly suffice to prove \eqref{NRCCLNewBd}. We assume in the rest of the proof that $j_1\leq j_2$.

{\bf Case 1:} We prove first the bound \eqref{SuffNRCCLNewBdPrev} under the assumption
\begin{equation}\label{BulkCase}
\begin{split}
&j_2\le 9m/10,\qquad\min\{k,k_1,k_2\}\ge -\mathcal{D}.
\end{split}
\end{equation}
With $\kappa_\theta:=2^{-m/2+\delta^2m}$ and $\kappa_r:=2^{\delta^2m}\big(2^{-m/2}+2^{j_2-m}\big)$ we decompose,
\begin{equation*}
\begin{split}
&\mathcal{F}I[f_{j_1,k_1},g_{j_2,k_2}]=\mathcal{R}_{||}+\mathcal{R}_\perp+\mathcal{NR},\\
&\mathcal{R}_{||}(\xi):=\int_{\mathbb{R}^2}e^{is\Phi(\xi,\eta)}\chi_{l}(\Phi(\xi,\eta))\varphi(\kappa_r^{-1}\Xi(\xi,\eta))\varphi(\kappa_\theta^{-1}\Omega_\eta\Phi(\xi,\eta))\widehat{f_{j_1,k_1}}(\xi-\eta)\widehat{g_{j_2,k_2}}(\eta)d\eta,\\
&\mathcal{R}_{\perp}(\xi):=\int_{\mathbb{R}^2}e^{is\Phi(\xi,\eta)}\chi_{l}(\Phi(\xi,\eta))\varphi(\kappa_r^{-1}\Xi(\xi,\eta))(1-\varphi(\kappa_\theta^{-1}\Omega_\eta\Phi(\xi,\eta)))\widehat{f_{j_1,k_1}}(\xi-\eta)\widehat{g_{j_2,k_2}}(\eta)d\eta,\\
&\mathcal{NR}(\xi):=\int_{\mathbb{R}^2}e^{is\Phi(\xi,\eta)}\chi_{l}(\Phi(\xi,\eta))(1-\varphi(\kappa_r^{-1}\Xi(\xi,\eta)))\widehat{f_{j_1,k_1}}(\xi-\eta)\widehat{g_{j_2,k_2}}(\eta)d\eta.
\end{split}
\end{equation*}

With $\psi_1:=\varphi_{\leq (1-\delta/4)m}$ and $\psi_2:=\varphi_{>(1-\delta/4)m}$, we rewrite 
\begin{equation*}
\begin{split}
&\mathcal{NR}(\xi)=\mathcal{NR}_1(\xi)+\mathcal{NR}_2(\xi),\\
&\mathcal{NR}_i(\xi):=C2^{l}\int_{\mathbb{R}}\int_{\mathbb{R}^2}e^{i(s+\lambda)\Phi(\xi,\eta)}\widehat{\chi}(2^l\lambda)\psi_i(\lambda)(1-\varphi(\kappa_r^{-1}\Xi(\xi,\eta)))\widehat{f_{j_1,k_1}}(\xi-\eta)\widehat{g_{j_2,k_2}}(\eta)\,d\eta d\lambda.
\end{split}
\end{equation*}
Since $\widehat{\chi}$ is rapidly decreasing we have $\|\varphi_k\cdot\mathcal{NR}_2\|_{L^\infty}\lesssim 2^{-4m}$, which gives an acceptable contribution. On the other hand, in the support of the integral defining $\mathcal{NR}_1$, we have that $\vert s+\lambda\vert\approx 2^m$ and integration by parts in $\eta$ (using Lemma \ref{tech5}) gives $\|\varphi_k\cdot\mathcal{NR}_1\|_{L^\infty}\lesssim 2^{-4m}$.

The contribution of $\mathcal{R}=\mathcal{R}_{||}+\mathcal{R}_{\perp}$ is only present if we have a space-time resonance. In particular, we may assume that
\begin{equation}\label{Alx74.6}
\begin{split}
&-\D\le k,k_1,k_2\le \D,\qquad(\sigma,\mu,\nu)\in\{(b,e,e),(b,e,b),(b,b,e)\},\\
&f_{j_1,k_1}=f_{j_1,k_1,0},\qquad g_{j_2,k_2}=g_{j_2,k_2,0}.
\end{split}
\end{equation}
Notice that, if $\mathcal{R}(\xi)\neq 0$ then
\begin{equation}\label{EstimPsi}
\begin{split}
\vert \Psi(\xi)\vert &=\vert \Phi(\xi,p(\xi))\vert\lesssim\vert \Phi(\xi,\eta)\vert+\vert \Phi(\xi,\eta)-\Phi(\xi, p(\xi))\vert\lesssim 2^l+\kappa_r^2.
\end{split}
\end{equation}
Integration by parts using Lemma \ref{RotIBP} shows that $\big\|\phi_k\cdot\mathcal{R}_\perp\big\|\lesssim 2^{-4m}$, which gives an acceptable contribution. To bound the contribution of $\mathcal{R}_{||}$ we will show that
\begin{equation}\label{Alx81}
2^{\delta m/2}2^{-l}\sup_{|\xi|\approx 1} \vert (1+2^{m}\Psi^\dagger_b(\xi))\mathcal{R}_{||}(\xi)\vert\lesssim 2^{9\delta m/10},
\end{equation}
which is stronger than the bound we need in \eqref{SuffNRCCLNewBdPrev}. Indeed for $j$ fixed we estimate
\begin{equation}\label{Alx67.5}
\begin{split}
\sup_{0\leq n\leq j+1}&2^{(1-20\delta)j}2^{-n/2+19\delta n}\big\|A_{n,(j)}^bQ_{jk}\mathcal{F}^{-1}\mathcal{R}_{||}\big\|_{L^2}\\
&\lesssim \sup_{0\leq n\leq j+1}2^{(1-20\delta)j}2^{-n/2+19\delta n}\big\|\varphi_{-n}^{[-j-1,0]}(\Psi^\dagger_b(\xi))\mathcal{R}_{||}(\xi)\big\|_{L^2_\xi}\\
&\lesssim \sum_{n\geq 0}2^{(1-20\delta)j}2^{-n/2-(1/2-19\delta)\min(n,j)}\big\|\varphi_{-n}^{(-\infty,0]}(\Psi^\dagger_b(\xi))\mathcal{R}_{||}(\xi)\big\|_{L^\infty_\xi},
\end{split}
\end{equation} 
and notice that \eqref{SuffNRCCLNewBdPrev} would follow from \eqref{Alx81}.

Recall from Lemma \ref{LinEstLem} and \eqref{Alx74.6} that
\begin{equation}\label{NewBdinput}
\begin{split}
2^{(1/2-21\delta)j_1}\Vert \widehat{f_{j_1,k_1}}\Vert_{L^\infty}+2^{(1-21\delta)j_1}\sup_{\theta\in\mathbb{S}^1}\Vert \widehat{f_{j_1,k_1}}(r\theta)\Vert_{L^2(rdr)}&\lesssim 1,\\
2^{(1/2-21\delta)j_2}\Vert \widehat{g_{j_2,k_2}}\Vert_{L^\infty}+2^{(1-21\delta)j_2}\sup_{\theta\in\mathbb{S}^1}\Vert \widehat{g_{j_2,k_2}}(r\theta)\Vert_{L^2(rdr)}&\lesssim 1.
\end{split}
\end{equation}
We ignore first the factor $\chi_{l}(\Phi(\xi,\eta))$. In view of \eqref{cas10} the $\eta$ integration in the definition of $\mathcal{R}_{||}(\xi)$ takes place essentially over a $\kappa_\theta\times\kappa_r$ box in the neighborhood of $p(\xi)$. Using \eqref{EstimPsi} and \eqref{FL1bd}, and estimating $\Vert \widehat{f_{j_1,k_1}}\Vert_{L^\infty}\lesssim 1$, we have, if $j_2\geq m/2$,
\begin{equation*}
\vert (1+2^{m}\Psi(\xi))\mathcal{R}_{||}(\xi)\vert\lesssim 2^{m}(2^l+\kappa_r^2)2^{-j_2+21\delta j_2}\kappa_\theta\kappa_r^{1/2}\lesssim (2^l+\kappa_r^2)2^{-j_2(1/2-21\delta)}2^{2\delta^2m}.
\end{equation*}
On the other hand, if $j_2\leq m/2$ we estimate $\Vert \widehat{f_{j_1,k_1}}\Vert_{L^\infty}+\Vert \widehat{f_{j_2,k_2}}\Vert_{L^\infty}\lesssim 1$ and conclude that
\begin{equation*}
\vert (1+2^{m}\Psi(\xi))\mathcal{R}_{||}(\xi)\vert\lesssim 2^{m+l}\kappa_{\theta}\kappa_r\lesssim 2^l2^{2\delta^2m}.
\end{equation*}
The desired bound \eqref{Alx81} follows if $\kappa_r^22^{-l}\leq 2^{j_2/4}$.

Assume now that $\kappa_r^2\ge 2^l2^{j_2/4}$ (in particular $j_2\geq 11m/20$). In this case the restriction $|\Phi(\xi,\eta)|\leq 2^l$ is stronger and we have to use it. We decompose, with $p_-:=\lfloor\log_2(2^{l/2}\kappa_r^{-1})+\D\rfloor$,
\begin{equation*}
\begin{split}
\mathcal{R}_{||}(\xi)&=\sum_{p\in[p_-,0]}\mathcal{R}^p_{||}(\xi),\\
\mathcal{R}_{||}^p(\xi)&:=\int_{\mathbb{R}^2}e^{is\Phi(\xi,\eta)}\chi_{l}(\Phi(\xi,\eta))\varphi_p^{[p_-,1]}(\kappa_r^{-1}\nabla_\eta\Phi(\xi,\eta))\varphi(\kappa_\theta^{-1}\Omega_\eta\Phi(\xi,\eta))\widehat{f_{j_1,k_1}}(\xi-\eta)\widehat{g_{j_2,k_2}}(\eta)d\eta.
\end{split}
\end{equation*}
Notice that if $\mathcal{R}_{||}^p(\xi)\neq 0$ then $\vert\Psi(\xi)\vert\lesssim 2^{2p}\kappa_r^2$. The term $\mathcal{R}_{||}^{p_-}(\xi)$ can be bounded as before. Moreover, using Proposition \ref{spaceres} and the formula \eqref{cas7.1}, we notice that if $\xi=(s,0)$ is fixed then the set of points $\eta$ that satisfy the three restrictions $|\Phi(\xi,\eta)|\lesssim 2^l$, $|\nabla_\eta\Phi(\xi,\eta)|\approx 2^p\kappa_r$, $|\xi\cdot\eta^\perp|\lesssim \kappa_\theta$ is essentially contained in a union of two $\kappa_\theta\times 2^l2^{-p}\kappa_r^{-1}$ boxes. Using \eqref{EstimPsi} and \eqref{FL1bd}, and estimating $\Vert \widehat{f_{j_1,k_1}}\Vert_{L^\infty}\lesssim 1$, we have
\begin{equation*}
\vert (1+2^{m}\Psi(\xi))\mathcal{R}^p_{||}(\xi)\vert\lesssim 2^{m+2p}\kappa_r^22^{-j_2+21\delta j_2}\kappa_\theta(2^l2^{-p}\kappa_r^{-1})^{1/2}\lesssim 2^{3p/2}2^{-m+4\delta^2m}2^{l/2}2^{j_2/2+21\delta j_2}.
\end{equation*}
This suffices to prove \eqref{EstimPsi} since $2^p\leq 1$, $2^{-l/2}\leq 2^{m/2}$, and $2^{j_2}\leq 2^{9m/10}$, see \eqref{BulkCase}.

{\bf Case 2:} We prove now the bound \eqref{SuffNRCCLNewBd}. Assume first that
\begin{equation}\label{Ass3.2}
j_1\ge 7m/8.
\end{equation}
Using Lemma \ref{Shur2Lem} and \eqref{RadL2} we estimate
\begin{equation*}
\begin{split}
\Vert I[f_{j_1,k_1,n_1},&g_{j_2,k_2,n_2}]\Vert_{L^2}\\
&\lesssim 2^{\delta^2m}2^{l/2-n_1/2-n_2/2}\big\Vert\sup_{\theta\in\mathbb{S}^1}|\widehat{f_{j_1,k_1,n_1}}(r\theta)|\big\Vert_{L^2(rdr)}\big\Vert\sup_{\theta\in\mathbb{S}^1}|\widehat{g_{j_2,k_2,n_2}}(r\theta)|\big\Vert_{L^2(rdr)}\\
&\lesssim 2^{\delta^2m}2^{l/2}2^{-j_1+21\delta j_1}2^{-j_2+21\delta j_2},
\end{split}
\end{equation*}
and the desired bound follows. 

The bound \eqref{SuffNRCCLNewBd} also follows, using the $L^2\times L^\infty$ estimate (Lemma \ref{PhiLocLem}) and Lemma \ref{LinEstLem} if
\begin{equation}\label{Alx74}
j_2\ge 9m/10,\qquad l\geq -m/3.
\end{equation}

On the other hand, if
\begin{equation}\label{Alx74.2}
j_2\ge(1-30\delta)m\ge j_1,\qquad l\leq -m/3,\qquad\min\{k,k_1,k_2\}\le -200\delta m,
\end{equation}
then, using Proposition \ref{separation1} (i), we may assume that
\begin{equation*}
f_{j_1,k_1}=f_{j_1,k_1,0},\qquad g_{j_2,k_2}=g_{j_2,k_2,0},\qquad 2^{\max\{k,k_1,k_2\}}\approx 1.
\end{equation*}
We use Proposition \ref{separation1} (i), \eqref{cas5.55}, and Schur's test to estimate
\begin{equation*}
\begin{split}
\Vert P_kI[f_{j_1,k_1},g_{j_2,k_2}]\Vert_{L^2}&\lesssim 2^{2\delta^2m}2^{l}2^{\min(k,k_1,k_2)/2}\Vert \widehat{f_{j_1,k_1}}\Vert_{L^\infty}\Vert g_{j_2,k_2}\Vert_{L^2}\lesssim 2^{l-40\delta m}2^{-(1-20\delta) j_2}.
\end{split}
\end{equation*}
This completes the proof in the case \eqref{Alx74.2}.

Notice also that in the case
\begin{equation}\label{Alx74.25}
\min\{k,k_1,k_2\}\le -\D,\qquad j_2\le (1-30\delta)m
\end{equation}
we can use Proposition \ref{separation1} (i) and Lemma \ref{tech5} to obtain an acceptable contribution.

Assume now that
\begin{equation}\label{Alx74.3}
\begin{split}
&j_2\ge(1-30\delta)m\ge j_1,\qquad\qquad -200\delta m\le \min\{k,k_1,k_2\}\le -\D,\\
\hbox{or}\qquad&j_2\ge 9m/10\geq 7m/8\ge j_1,\,\,\qquad -\D\le  \min\{k,k_1,k_2\},\\
\end{split}
\end{equation}
We decompose, with $\kappa_\theta=2^{-m/4}$,
\begin{equation*}
\begin{split}
&I[f_{j_1,k_1},g_{j_2,k_2}]=\mathcal{A}_{||}[f_{j_1,k_1},g_{j_2,k_2}]+\mathcal{A}_{\perp}[f_{j_1,k_1},g_{j_2,k_2}],\qquad\\
&\widehat{\mathcal{A}_{||}[f,g]}(\xi)=\int_{\mathbb{R}^2}e^{is\Phi(\xi,\eta)}\chi_{l}(\Phi(\xi,\eta))\varphi(\kappa_\theta^{-1}\Omega_\eta\Phi(\xi,\eta))\widehat{f}(\xi-\eta)\widehat{g}(\eta)d\eta,\\
&\widehat{\mathcal{A}_{\perp}[f,g]}(\xi)=\int_{\mathbb{R}^2}e^{is\Phi(\xi,\eta)}\chi_{l}(\Phi(\xi,\eta))(1-\varphi(\kappa_\theta^{-1}\Omega_\eta\Phi(\xi,\eta)))\widehat{f}(\xi-\eta)\widehat{g}(\eta)d\eta.
\end{split}
\end{equation*}
Integration by parts using Lemma \ref{RotIBP} shows that $\big\|\mathcal{F}\mathcal{A}_\perp[f_{j_1,k_1},g_{j_2,k_2}]\big\|_{L^\infty}\lesssim 2^{-4m}$, while Proposition \ref{separation1} (ii) shows that
\begin{equation*}
\begin{split}
\mathcal{A}_{||}[f_{j_1,k_1,n_1},g_{j_2,k_2,n_2}]\equiv 0\qquad\text{ if }n_1\geq 1\text{ and }n_2\geq 1.
\end{split}
\end{equation*}
In addition, using Schur's test and Lemma \ref{Shur3Lem}
\begin{equation*}
\begin{split}
\Vert P_k\mathcal{A}_{||}[f_{j_1,k_1},g_{j_2,k_2,0}]\Vert_{L^2}&\lesssim 2^{2\delta^2m}2^{l-m/8}\Vert\widehat{f_{j_1,k_1}}\Vert_{L^\infty}\Vert g_{j_2,k_2,0}\Vert_{L^2}\lesssim 2^{l-m/10}2^{-(1-20\delta)j_2},
\end{split}
\end{equation*}
which gives an acceptable contribution. If $-200\delta m\le \min\{k,k_1,k_2\}\le -\D$ this covers all cases, in view of Proposition \ref{separation1} (i). 

On the other hand, if $\min\{k,k_1,k_2\}\ge-\D$ and $n_2\geq 1$ then we may assume that $\vert\nabla_\eta\Phi(\xi,\eta)\vert\gtrsim 1$ in the support of integration of $A_{||}[f_{j_1,k_1},g_{j_2,k_2,n_2}]$, in view of Proposition \ref{separation1} (iii). Integration by parts in $\eta$ using Lemma \ref{tech5} then gives an acceptable contribution unless $j_2\ge(1-\delta^2)m$. To summarize, it remains to estimate $\big\|Q_{jk}\mathcal{A}_{||}[f_{j_1,k_1,0},g_{j_2,k_2,n_2}]\big\|_{L^2}$ when
\begin{equation}\label{Alx74.4}
\begin{split}
j_2\ge(1-\delta^2)m\ge 7/8m\ge j_1,\qquad -\D\le k,k_1,k_2,\qquad n_2\geq 1.
\end{split}
\end{equation}

We decompose
\begin{equation*}
\begin{split}
\mathcal{A}_{||}[f_{j_1,k_1,0},&g_{j_2,k_2,n_2}]=\sum_{p\le \D}\mathcal{A}^{p}_{||}[f_{j_1,k_1,0},g_{j_2,k_2,n_2}],\\
\widehat{\mathcal{A}^{p}_{||}[f,g]}(\xi)&:=\int_{\mathbb{R}^2}e^{is\Phi(\xi,\eta)}\chi_{l}(\Phi(\xi,\eta))\varphi(\kappa_\theta^{-1}\Omega_\eta\Phi(\xi,\eta))\varphi_{p}(\nabla_\xi\Phi(\xi,\eta))\widehat{f}(\xi-\eta)\widehat{g}(\eta)d\eta.
\end{split}
\end{equation*}
As a consequence of Proposition \ref{volume} (iv), under our assumptions in \eqref{Alx74.4},
\begin{equation*}
\begin{split}
\sup_\xi\int_{\mathbb{R}^2}|\chi_{l}(\Phi(\xi,\eta))|\varphi(\kappa_\theta^{-1}\Omega_\eta\Phi(\xi,\eta))\varphi(\Psi^\dagger_b(\eta))\varphi_k(\xi)\varphi_{k_1}(\xi-\eta)d\eta&\lesssim 2^{2\delta^2m}2^{l}\kappa_\theta,\\
\end{split}
\end{equation*}
and
\begin{equation*}
\begin{split}
\sup_\eta\int_{\mathbb{R}^2}|\chi_{l}(\Phi(\xi,\eta))|\varphi(\kappa_\theta^{-1}\Omega_\eta\Phi(\xi,\eta))\varphi(\Psi^\dagger_b(\eta))\varphi_{p}(\nabla_\xi\Phi(\xi,\eta))\varphi_k(\xi)\varphi_{k_1}(\xi-\eta)d\xi\\
\lesssim 2^{2\delta^2m}\min\{2^p,2^{l-p}\}\kappa_\theta.
\end{split}
\end{equation*}
Using Schur's test we can then estimate, for $p\geq -20\delta m$
\begin{equation*}
\begin{split}
\Vert P_k\mathcal{A}^{p}_{||}[f_{j_1,k_1,0},g_{j_2,k_2,n_2}]\Vert_{L^2}&\lesssim 2^{-p/2}2^l2^{-m/2+4\delta^2m}\Vert \widehat{f_{j_1,k_1,0}}\Vert_{L^\infty}\Vert g_{j_2,k_2,n_2}\Vert_{L^2}\lesssim 2^l2^{-(1-15\delta)m},
\end{split}
\end{equation*}
which gives acceptable contributions. Similarly, the contributions of $Q_{jk}\mathcal{A}^{p}_{||}[f_{j_1,k_1,0},g_{j_2,k_2,n_2}]$ are acceptable if $p\leq l+20\delta m$, or if $j\leq (j_2+m)(1+8\delta)/2+p/2+2j_1/5$, or if $j\leq (j_2+m)(1+8\delta)/2+(l-p)/2+2j_1/5$ (using the bound $\Vert \widehat{f_{j_1,k_1,0}}\Vert_{L^\infty}\lesssim 2^{-2j_1/5}$, see \eqref{FLinftybd}). 

Therefore it remains to consider the case
\begin{equation}\label{Alx74.5}
\begin{split}
p\in [l+20\delta m,-20\delta m],\qquad j\geq (j_2+m)(1+8\delta)/2+2j_1/5+\max(p,l-p)/2.
\end{split}
\end{equation}
We use the approximate finite speed of propagation, showing that in this case,
\begin{equation}\label{FinalEstSTRes}
\Vert Q_{jk}\mathcal{A}^{p}_{||}[f_{j_1,k_1,0},g_{j_2,k_2,n_2}]\Vert_{L^2}\lesssim 2^{-3m}.
\end{equation}
Indeed, we decompose
\begin{equation*}
\begin{split}
\mathcal{F}\big\{P_k\mathcal{A}^{p}_{||}[f_{j_1,k_1,0},g_{j_2,k_2,n_2}]\big\}(\xi)=I_1(\xi)+I_2(\xi)
\end{split}
\end{equation*}
where, with $\psi_1:=\varphi_{\leq (1-\delta/4)m}$ and $\psi_2:=\varphi_{>(1-\delta/4)m}$,
\begin{equation*}
\begin{split}
I_i(\xi):=C\varphi_k(\xi)2^l\int_{\mathbb{R}}\psi_i(\lambda)\widehat{\chi}(2^l\lambda)\int_{\mathbb{R}^2}e^{i(s+\lambda)\Phi(\xi,\eta)}\varphi(\kappa_\theta^{-1}&\Omega_\eta\Phi(\xi,\eta))\varphi_{p}(\nabla_\xi\Phi(\xi,\eta))\\
&\times\widehat{f_{j_1,k_1,0}}(\xi-\eta)\widehat{g_{j_2,k_2,n_2}}(\eta)d\eta d\lambda.
\end{split}
\end{equation*}
Using the rapid decay of the function $\widehat{\chi}$ it is easy to see that $\|I_2\|_{L^\infty}\lesssim 2^{-4m}$, which gives an acceptable contribution. On the other hand, we may rewrite
\begin{equation*}
\begin{split}
\mathcal{F}^{-1}&(I_1)(x)=C2^l\int_{\mathbb{R}}\psi_1(\lambda)\widehat{\chi}(2^l\lambda)\widehat{g_{j_2,k_2,n_2}}(\eta)\\
&\times\Big[\int_{\mathbb{R}^2}\varphi_k(\xi)e^{i(s+\lambda)\Phi(\xi,\eta)+ix\cdot\xi}\varphi(\kappa_\theta^{-1}\Omega_\eta\Phi(\xi,\eta))\varphi_{p}(\nabla_\xi\Phi(\xi,\eta))\widehat{f_{j_1,k_1,0}}(\xi-\eta)\,d\xi\Big]\,d\eta d\lambda.
\end{split}
\end{equation*}
We examine the restrictions \eqref{Alx74.4} and \eqref{Alx74.5}. If $|x|\approx 2^j$ then we notice that in the support of $\xi$-integration we have $\vert\nabla_\xi\left[(s+\lambda)\Phi(\xi,\eta)+x\cdot\xi\right]\vert\approx \vert x\vert\approx 2^j$. Integration by parts in $\xi$ using Lemma \ref{tech5} gives  $\|\phii_j^{(k)}\mathcal{F}^{-1}(I_1)\|_{L^\infty}\lesssim 2^{-4m}$, as desired.

\subsection{The case of resonant interactions}\label{ProofZ3}
In this subsection we prove Lemma \ref{ResLem}. Using Lemma \ref{dtfLem} we may write, with $F^\nu=F^\nu_C+F^\nu_{NC}+F^\nu_{LO}$,
\begin{equation*}
\partial_sf^\nu(s)=f^\nu_C(s)+f^\nu_{SR}(s)+f^{\nu}_{NC}(s)+\partial_sF^\nu(s).
\end{equation*}

{\bf{Contribution of $f^\nu_C$.}} We start with the main term and decompose, according to \eqref{dtfType1},
\begin{equation}\label{Alx67.6}
\begin{split}
&f^\nu_{C}=\sum_{0\le q\le m/2-40\delta m}\sum_{\omega,\upsilon\in\mathcal{P},\omega+\upsilon\neq 0}f^{\nu\omega\upsilon}_{C,q},\quad\widehat{f^{\nu\omega\upsilon}_{C,q}}(\xi,s)=e^{is\Psi_{\nu\omega\upsilon}(\xi)}g_q(\xi,s),\\
&g_q(\xi,s)=g^q(\xi,s)\varphi_{\leq -3\delta m}(\Psi_{\nu\omega\upsilon}(\xi)),\quad \Vert D^\alpha_\xi g_q(s)\Vert_{L^\infty}\lesssim 2^{8\delta^2m-m-q}2^{\vert\alpha\vert(m/2+3\delta^2m+q)},\\
&\sup_{b\leq N_1/4}\|\Omega^bg_q(s)\|_{L^2}\lesssim 1,\qquad \|\partial_sg_q(s)\|_{L^\infty}\lesssim 2^{(5\delta-2)m+q}.
\end{split}
\end{equation}
We may assume $\nu=b$. It suffices to prove that for $k,j,m,l,k_1,k_2,q$ and $\sigma,\mu,\nu,\omega,\upsilon$ as before
\begin{equation}\label{ResLemMI1}
\begin{split}
2^{(1-20\delta)j}\Vert Q_{jk}\mathcal{B}_{m,l}[P_{k_1}f^\mu,P_{k_2}f^{\nu\omega\upsilon}_{C,q}]\Vert_{L^2}&\lesssim 2^{-51\delta^2m}.
\end{split}
\end{equation}
In the rest of this proof we set $\Phi=\Phi_{\sigma\mu\nu}$, $\Psi=\Psi_{\nu\omega\upsilon}$, $\mathfrak{p}(\xi,\eta):=\Phi_{\sigma\mu\nu}(\xi,\eta)+\Psi_{\nu\omega\upsilon}(\eta)$.

Notice that, in the support of integration we have
\begin{equation}\label{PropPhiSupp1}
|\Psi(\eta)|\lesssim 2^{-3\delta m}\ll 1,\qquad\vert\Phi(\xi,\eta)\vert\lesssim 2^{-7\delta m}\ll 1,\qquad\vert\nabla_\eta\Phi(\xi,\eta)\vert\gtrsim 1.
\end{equation}
In view of Proposition \ref{separation1}, we may assume that $\min\{k,k_1,k_2\}\ge-\D$. 
We define $f^\mu_{j_1,k_1}$ and $f^\mu_{j_1,k_1,n_1}$ as before. Assume first that $j_1\ge(1-100\delta^2)m$. Schur's test with Proposition \ref{volume} (i), (iii) give
\begin{equation*}
\begin{split}
\Vert P_k\mathcal{B}_{m,l}[f^\mu_{j_1,k_1,0},P_{k_2}f^{\nu\omega\upsilon}_{C,q}]\Vert_{L^2}&\lesssim 2^{2\delta^2m}2^{m-l}\cdot (2^l2^{-3\delta m/4})\sup_{s}\Vert f^\mu_{j_1,k_1,0}(s)\Vert_{L^2}\Vert \widehat{f^{\nu\omega\upsilon}_{C,q}}(s)\Vert_{L^\infty}.
\end{split}
\end{equation*}
Moreover, using \eqref{Shur2Lem3}, for $n_1\geq 1$
\begin{equation*}
\begin{split}
\Vert P_k\mathcal{B}_{m,l}[f^\mu_{j_1,k_1,n_1},P_{k_2}f^{\nu\omega\upsilon}_{C,q}]\Vert_{L^2}&\lesssim 2^{2\delta^2m}2^{m-l}2^{l-n_1/2}\sup_{s}\Vert \sup_{\theta\in\mathbb{S}^1}|\widehat{f^\mu_{j_1,k_1,n_1}}(r\theta,s)|\Vert_{L^2(rdr)}\Vert \widehat{f^{\nu\omega\upsilon}_{C,q}}(s)\Vert_{L^2}.
\end{split}
\end{equation*}
In both cases we use the estimates \eqref{RadL2} and \eqref{Alx67.6} to show that these are acceptable contributions as in \eqref{ResLemMI1}.

Assume now that $j_1\le (1-100\delta^2)m$, let 
$\kappa_\theta=2^{2\delta^2m-m/2}$, and decompose
\begin{equation*}
\begin{split}
&\mathcal{B}_{m,l}[f^\mu_{j_1,k_1},P_{k_2}f^{\nu\omega\upsilon}_{C,q}]=\mathcal{B}^\perp_{m,l,j_1}+\mathcal{B}^{||,1}_{m,l,j_1}+\mathcal{B}^{||,2}_{m,l,j_1},\\
&\widehat{\mathcal{B}_{m,l,j_1}^{\perp}}(\xi)=\int_{\mathbb{R}}q_m(s)\int_{\mathbb{R}^2}e^{is\mathfrak{p}(\xi,\eta)}\widetilde{\varphi}_l(\Phi(\xi,\eta))(1-\varphi(\kappa_\theta^{-1}\Omega_\eta\Phi(\xi,\eta)))\widehat{f^\mu_{j_1,k_1}}(\xi-\eta,s)g_q(\eta,s)\,d\eta ds,\\
&\widehat{\mathcal{B}_{m,l,j_1}^{||,i}}(\xi)=\int_{\mathbb{R}}q_m(s)\int_{\mathbb{R}^2}e^{is\mathfrak{p}(\xi,\eta)}\widetilde{\varphi}_l(\Phi(\xi,\eta))\varphi(\kappa_\theta^{-1}\Omega_\eta\Phi(\xi,\eta))\psi_i(\xi,\eta)
\widehat{f^\mu_{j_1,k_1}}(\xi-\eta,s)g_q(\eta,s)d\eta ds,
\end{split}
\end{equation*}
where $\psi_1(\xi,\eta):=\varphi_{\leq -\D}(\nabla_\xi\Phi(\xi,\eta))$ and $\psi_2(\xi,\eta):=\varphi_{> -\D}(\nabla_\xi\Phi(\xi,\eta))$.

We apply first Lemma \ref{RotIBP} (with $\widehat{g}(\eta)=e^{is\Psi(\eta)}g_q(\eta,s)$) to conclude that $\big\|P_k\mathcal{B}_{m,l,j_1}^{\perp}\big\|_{L^2}\lesssim 2^{-2m}$, which is an acceptable contribution. In view of Proposition \ref{separation1} (ii), when considering $\mathcal{B}_{m,l,j_1}^{||,i}$, $i=1,2$, we may assume that $f^\mu_{j_1,k_1}=f^\mu_{j_1,k_1,0}$. Moreover, Proposition \ref{Separation2} (i) shows that $\vert\nabla_\eta\mathfrak{p}(\xi,\eta)\vert\gtrsim 1$ in the support of integration of $\mathcal{B}_{m,l,j_1}^{||,2}$. Therefore integration by parts in $\eta$ using Lemma \ref{tech5} shows that $\big\|P_k\mathcal{B}_{m,l,j_1}^{||,2}\big\|_{L^2}\lesssim 2^{-2m}$, which leads to an acceptable contribution.

We now turn to $\mathcal{B}_{m,l,j_1}^{||,1}$. We use Lemma \ref{Shur3Lem} (iii) and the identity $f^\mu_{j_1,k_1}=f^\mu_{j_1,k_1,0}$ to estimate
\begin{equation*}
\vert \phi_k(\xi)\widehat{\mathcal{B}_{m,l,j_1}^{||,1}}(\xi)\vert\lesssim 2^{2\delta^2m}2^{m-l}\cdot 2^l\kappa_\theta\cdot \Vert \widehat{f^\mu_{j_1,k_1}}(s)\Vert_{L^\infty}\Vert g_q(s)\Vert_{L^\infty}\lesssim 2^{20\delta^2m-m/2}.
\end{equation*}
The desired bound \eqref{ResLemMI1} follows if $j\le m/2+10\delta m$. We may now assume $j\ge m/2+10\delta m$. In view of Proposition \ref{Separation2} (ii), we may decompose
\begin{equation*}
\begin{split}
\mathcal{B}_{m,l,j_1}^{||,1}&=\sum_{r\in(-\infty,-\D/2]}\mathcal{B}_{m,l,r,q}^{j_1},
\end{split}
\end{equation*}
where
\begin{equation}\label{Alx68.4}
\begin{split}
\widehat{\mathcal{B}_{m,l,r,q}^{j_1}}(\xi)=\int_{\mathbb{R}}q_m(s)&\int_{\mathbb{R}^2}e^{is\mathfrak{p}(\xi,\eta)}\widetilde{\varphi}_l(\Phi(\xi,\eta))\varphi(\kappa_\theta^{-1}\Omega_\eta\Phi(\xi,\eta))\\
&\times\varphi_{\leq -\D}(\nabla_\xi\Phi(\xi,\eta))\varphi_{r}(\nabla_\eta\mathfrak{p}(\xi,\eta))\widehat{f^\mu_{j_1,k_1}}(\xi-\eta,s)g_q(\eta,s)\,d\eta ds.\\
\end{split}
\end{equation}

Using Proposition \ref{Separation2} (i), we see that on the support of integration of $\mathcal{B}_{m,l,r,q}^{j_1}$
\begin{equation*}
\begin{split}
\mathfrak{p}(\xi,\eta)=\Lambda_\sigma(\xi)-\Lambda_\mu(\xi-\eta)-\Lambda_\sigma(p(\eta))+\Lambda_\mu(\eta-p(\eta)),
\end{split}
\end{equation*}
where $p:=p_{(-\mu)\sigma}$ (so $\nabla\Lambda_{\mu}(x-p(x))+\nabla\Lambda_{\sigma}(p(x))=0$). Moreover, for some $\gamma\in\{\gamma_1,\gamma_2\}$,
\begin{equation}\label{PropPhiSupp2}
\begin{split}
\vert \nabla_\eta\mathfrak{p}(\xi,\eta)\vert\approx\vert\xi-p(\eta)\vert\approx  \vert\nabla_\xi\Phi(\xi,\eta)\vert\approx 2^r,\\
\vert \mathfrak{p}(\xi,\eta)\vert\approx \vert \xi-p(\eta)\vert^2\approx\vert\nabla_\xi\Phi(\xi,\eta)\vert^2\approx 2^{2r},\\
\vert\vert\eta\vert-\gamma\vert\approx\vert\Psi(\eta)\vert\lesssim 2^l+2^{2r},\\
\big||\xi|-p_+(\gamma)\big|\lesssim 2^l+2^r.
\end{split}
\end{equation}

Integrating by parts, using Lemma \ref{tech5} and \eqref{PropPhiSupp2}, we observe that $\big\|P_k\mathcal{B}_{m,l,r,q}^{j_1}\big\|_{L^2}\lesssim 2^{-2m}$ if
\begin{equation}\label{ResLemMI1IBP}
\begin{split}
2^r\ge 2^{5\delta^2m}\big(2^{-l-m}+2^{j_1-m}+2^{q-m/2}\big)\quad\text{ or }\quad 2^j\ge 2^{5\delta^2m}(2^{m+r}+2^{q+m/2}+2^{-r}+2^{j_1}).
\end{split}
\end{equation}
Indeed, in the first case we integrate by parts in $\eta$, while in the second case we integrate by parts in $\xi$ after taking the inverse Fourier transform and restricting to $x\approx 2^j$ (notice that $\big|D^\alpha_\xi[\widetilde{\varphi}_l(\Phi(\xi,\eta))]\big|\lesssim_{|\alpha|} 2^{-l}2^{|\alpha|(r-l)}$ in the support of the integral, in view of \eqref{PropPhiSupp2}). 

For \eqref{ResLemMI1} it suffices to prove that if $2^J\le 2^{5\delta^2m}(2^{m+r}+2^{q+m/2}+2^{-r}+2^{j_1})$ then
\begin{equation}\label{Alx68.1}
\sum_{r\leq -\D/2}2^{J(1-20\delta)}\big\|P_k\mathcal{B}_{m,l,r,q}^{j_1}\big\|_{L^2}\lesssim 2^{-60\delta^2m},
\end{equation}
for any fixed parameters $m,l,r,q,j_1$ as before. Using Schur's test and \eqref{cas5.55}--\eqref{cas5.6}, and recalling that $f_{j_1,k_1}=f_{j_1,k_1,0}$,
\begin{equation*}
\big\|P_k\mathcal{B}_{m,l,r,q}^{j_1}\big\|_{L^2}\lesssim 2^{5\delta^2m}2^{m-l}\cdot 2^l\kappa_\theta^{1/2}\sup_{s}\| g_q(s)\|_{L^\infty}\|f^\mu_{j_1,k_1}(s)\|_{L^2}\lesssim 2^{20\delta^2m}2^{-(1-20\delta)j_1}2^{-m/4-q}
\end{equation*}
and
\begin{equation*}
\big\|P_k\mathcal{B}_{m,l,r,q}^{j_1}\big\|_{L^2}\lesssim 2^{5\delta^2m}2^{m-l}\cdot 2^{2r}\sup_{s}\| g_q(s)\|_{L^\infty}\|f^\mu_{j_1,k_1}(s)\|_{L^2}\lesssim 2^{20\delta^2m}2^{-(1-20\delta)j_1}2^{-l+2r-q}.
\end{equation*}
The desired bound \eqref{Alx68.1} follows if $r\leq-m$ or if $J\leq j_1+m/4$. Using \eqref{PropPhiSupp1}, \eqref{PropPhiSupp2}, and \eqref{Alx64.1} we estimate
\begin{equation}\label{EstimParamr2suiv}
\begin{split}
\vert \widehat{\mathcal{B}_{m,l,r,q}^{j_1}}(\xi)\vert&\lesssim 2^{5\delta^2m}2^{m-l}\cdot 2^l\kappa_\theta\sup_s\Vert \widehat{f^\mu_{j_1,k_1}}\Vert_{L^\infty}\Vert g_q\Vert_{L^\infty}\lesssim 2^{20\delta^2 m-m/2-q}.
\end{split}
\end{equation}
The desired bound \eqref{Alx68.1} follows also if $J\leq \delta m+m/2+q$. It remains to show that
\begin{equation}\label{EstimParamr2Suff}
\big[2^{(1-20\delta)(m+r)}+2^{-(1-20\delta)r}\big]\Vert \mathcal{B}_{m,l,r,q}^{j_1}\Vert_{L^2}\lesssim 2^{-70\delta^2m}.
\end{equation}

We use now \eqref{PropPhiSupp1}  and \eqref{PropPhiSupp2} to bound
\begin{equation}\label{Alx68.2}
\begin{split}
\vert \widehat{\mathcal{B}_{m,l,r,q}^{j_1}}(\xi)\vert&\lesssim 2^{5\delta^2m} 2^{m-l}\cdot \min\{2^l,2^r\}\min\{2^r,2^{-m/2}\}\sup_s\Vert \widehat{f^\mu_{j_1,k_1}}\Vert_{L^\infty}\Vert g_q\Vert_{L^\infty}\\
&\lesssim 2^{20\delta^2m}2^{-q}\min\{1,2^{r-l}\}\min\{2^r,2^{-m/2}\}.
\end{split}
\end{equation}
The desired bound \eqref{EstimParamr2Suff} follows if $r\leq-m/2$. It also follows if $-m/2\leq r\le -(m-2q)/3$. On the other hand, if $r\ge -(m-2q)/3$ and $j_1\geq m/2$ then we recall \eqref{PropPhiSupp1} and use Schur's test with \eqref{Alx64.1} to estimate
\begin{equation*}
\begin{split}
\Vert \mathcal{B}_{m,l,r,q}^{j_1}\Vert_{L^2}&\lesssim 2^{5\delta^2m}2^{m-l}\cdot \kappa_\theta2^{l-r/2}\sup_{s}\Vert f^\mu_{j_1,k_1}(s)\Vert_{L^2}\Vert g_q(s)\Vert_{L^\infty}\lesssim 2^{20\delta^2m}2^{-m/2-r/2}2^{-(1-20\delta)j_1},
\end{split}
\end{equation*}
and the desired bound \eqref{EstimParamr2Suff} follows in this case as well. 

Therefore, it remains to prove \eqref{EstimParamr2Suff} in the case
\begin{equation}\label{EstimParamr2}
r\ge -(m-2q)/3\qquad\text{ and }\qquad j_1\leq m/2.
\end{equation}
From \eqref{ResLemMI1IBP} and \eqref{EstimParamr2}, we may assume that $2^r\le 2^{6\delta^2m}2^{-l-m}$. In particular, $l\leq-m/2\leq r$. The main point is to notice that the modulation is rather large in this case, in view of \eqref{PropPhiSupp2}. We integrate by parts in time in the formula \eqref{Alx68.4} to rewrite
\begin{equation*}
\begin{split}
\widehat{\mathcal{B}_{m,l,r,q}^{j_1}}(\xi)=c\int_{\mathbb{R}}&\int_{\mathbb{R}^2}e^{is\mathfrak{p}(\xi,\eta)}\frac{\widetilde{\varphi}_l(\Phi(\xi,\eta))\varphi(\kappa_\theta^{-1}\Omega_\eta\Phi(\xi,\eta))}{\mathfrak{p}(\xi,\eta)}\\
&\times\varphi_{\leq -\D}(\nabla_\xi\Phi(\xi,\eta))\varphi_{r}(\nabla_\eta\mathfrak{p}(\xi,\eta))\partial_s\big\{q_m(s)\widehat{f^\mu_{j_1,k_1}}(\xi-\eta,s)g_q(\eta,s)\big\}\,d\eta ds.\\
\end{split}
\end{equation*}
We use the bounds, see Lemma \ref{dtfLemPrelim} and \eqref{Alx67.6},
\begin{equation*}
\begin{split}
\Vert \widehat{f^\mu_{j_1,k_1}}(s)\Vert_{L^\infty}\lesssim 1,&\qquad \partial_sf^\mu_{j_1,k_1}(s)=f^\mu_C+f^\mu_{NC},\\
\Vert \widehat{f^\mu_C}(s)\Vert_{L^\infty}\lesssim 2^{3\delta m+\delta^2m-m},&\qquad \Vert f^\mu_{NC}(s)\Vert_{L^2}\lesssim 2^{-11m/8},\\
\Vert g_q(s)\Vert_{L^\infty}\lesssim2^{(\delta-1)m},&\qquad\Vert \partial_sg_q(s)\Vert_{L^\infty}\lesssim2^{q-2m+5\delta m}.
\end{split}
\end{equation*}
Shur's Lemma allows to control the contribution of $f^\mu_{NC}$ as
\begin{equation*}
\begin{split}
\mathcal{B}_{m,l,r,q}^{j_1}&=\mathcal{B}_{m,l,r,q}^{j_1,\ast}+\mathcal{B}_{m,l,r,q}^{j_1,NC}\\
\widehat{\mathcal{B}_{m,l,r,q}^{j_1,NC}}(\xi):=c\int_{\mathbb{R}}&\int_{\mathbb{R}^2}e^{is\mathfrak{p}(\xi,\eta)}\frac{\widetilde{\varphi}_l(\Phi(\xi,\eta))\varphi(\kappa_\theta^{-1}\Omega_\eta\Phi(\xi,\eta))}{\mathfrak{p}(\xi,\eta)}\\
&\times\varphi_{\leq -\D}(\nabla_\xi\Phi(\xi,\eta))\varphi_{r}(\nabla_\eta\mathfrak{p}(\xi,\eta))q_m(s)\widehat{f^\mu_{NC}}(\xi-\eta,s)g_q(\eta,s)\,d\eta ds,\\
\Vert \mathcal{B}_{m,l,r,q}^{j_1,NC}\Vert_{L^2}&\lesssim 2^{2\delta^2m}2^{m-l}\cdot 2^l\kappa_\theta\cdot 2^{-2r}\cdot 2^{(\delta-1)m}\sup_{s}\Vert f^\mu_{NC}(s)\Vert_{L^2}\lesssim 2^{-\frac{13}{7}m-2r}.
\end{split}
\end{equation*}
In view of \eqref{EstimParamr2}, this gives an acceptable estimate. The rest of $\mathcal{B}_{m,l,r,q}^{j_1}$ is estimated crudely, using also \eqref{PropPhiSupp2} and \eqref{Alx64.1}, 
\begin{equation*}
\begin{split}
\Vert \widehat{\mathcal{B}_{m,l,r,q}^{j_1,\ast}}\Vert_{L^\infty}&\lesssim 2^{m-l}\cdot 2^l\kappa_\theta\cdot2^{-2r}2^{q-2m+5\delta m}\lesssim 2^{q+6\delta m-2r-3m/2},\\
2^{(1-20\delta)(m+r)}\Vert \mathcal{B}_{m,l,r,q}^{j_1,\ast}\Vert_{L^2}&\lesssim 2^{-(1/2+20\delta)(m+r-2q)+6\delta m-40\delta q}.
\end{split}
\end{equation*}
This completes the proof of \eqref{EstimParamr2Suff}, in view of the assumption \eqref{EstimParamr2}.

{\bf{Contribution of $f^\nu_{SR}$.}} We now show that
\begin{equation}\label{SecResEst}
\begin{split}
\Vert Q_{jk}\mathcal{B}_{m,l}[P_{k_1}f^\mu,P_{k_2}f^\nu_{SR}]\Vert_{B^\sigma_j}\lesssim 2^{-51\delta^2m}.
\end{split}
\end{equation}
Recall from \eqref{dtfSR} that $f^\nu_{SR}=\widetilde{P}_{SRI}f^\nu_{SR}$ and 
\begin{equation}\label{dtfSR2}
\begin{split}
\Vert D_\xi^\alpha \widehat{f^\nu_{SR}}(\xi,s)\Vert_{L^\infty}&\lesssim 2^{75\delta m-3m/2}2^{(1-300\delta) m\vert\alpha\vert},\quad\sup_{b\leq N_1/4}\|\Omega^bf^\nu_{SR}(s)\|_{L^2}\lesssim 1.
\end{split}
\end{equation}
In view of Proposition \ref{separation1} (i), we may assume that $\min\{k,k_1,k_2\}\ge-\D$. We make the change of variables $\eta\to\xi-\eta$ and\footnote{By a slight abuse of notation we also let $\Phi$ denote the phase $(\xi,\eta)\to\Phi(\xi,\xi-\eta)$.} decompose, with $\kappa_\theta=2^{2\delta^2m-m/2}$,
\begin{equation*}
\begin{split}
&\mathcal{B}_{m,l}[P_{k_1}f^\mu,P_{k_2}f^\nu_{SR}]=\mathcal{B}^{||}_{m,l}[P_{k_2}f^\nu_{SR},P_{k_1}f^\mu]+\mathcal{B}^\perp_{m,l}[P_{k_2}f^\nu_{SR},P_{k_1}f^\mu],\\
&\mathcal{F}\mathcal{B}^{||}_{m,l}[f,g](\xi):=\int_{\mathbb{R}}q_m(s)\int_{\mathbb{R}^2}e^{is\Phi(\xi,\eta)}\widetilde{\varphi}_l(\Phi(\xi,\eta))\varphi(\kappa_\theta^{-1}\Omega_\eta\Phi'(\xi,\eta))\widehat{f}(\xi-\eta,s)\widehat{g}(\eta,s)\,d\eta ds,\\
&\mathcal{F}\mathcal{B}^{\perp}_{m,l}[f,g](\xi):=\int_{\mathbb{R}}q_m(s)\int_{\mathbb{R}^2}e^{is\Phi(\xi,\eta)}\widetilde{\varphi}_l(\Phi(\xi,\eta))(1-\varphi(\kappa_\theta^{-1}\Omega_\eta\Phi(\xi,\eta)))\widehat{f}(\xi-\eta,s)\widehat{g}(\eta,s)\,d\eta ds.\\
\end{split}
\end{equation*}
Using Lemma \ref{RotIBP}, we see that $\big\|\mathcal{F}\mathcal{B}^{\perp}_{m,l}[f,g]\big\|_{L^\infty}\lesssim 2^{-2m}$, and this gives an acceptable contribution. Let $f_{j_1,k_1}$ be defined as before. If $j_1\ge 3m/4$ we use Lemma \ref{Shur3Lem} to estimate
\begin{equation*}
\begin{split}
\Vert P_k\mathcal{B}^{||}_{m,l}[P_{k_2}f^\nu_{SR},f^\mu_{j_1,k_1}]\Vert_{L^2}&\lesssim 2^{2\delta^2m}2^{m-l}\cdot(2^l\kappa_\theta^\frac{1}{2})\cdot \sup_{s}\Vert \widehat{P_{k_2}f^\nu_{SR}}\Vert_{L^\infty}\Vert f^\mu_{j_1,k_1}\Vert_{L^2}\\
&\lesssim 2^{80\delta m-3m/4}2^{-j_1/2},
\end{split}
\end{equation*}
and we obtain an acceptable contribution. On the other hand, if $j_1\le 3m/4$ then we decompose
\begin{equation*}
\begin{split}
\mathcal{B}^{||}_{m,l}[P_{k_2}f^\nu_{SR},f^\mu_{j_1,k_1}]&=\sum_{p\in[p_-,p_0]}\mathcal{B}_{m,l,p}[P_{k_2}f^\nu_{SR},f^\mu_{j_1,k_1}],
\end{split}
\end{equation*}
where $p_-:=\lfloor l/2\rfloor$ and $p_0:=\lfloor -250\delta m\rfloor$ and 
\begin{equation*}
\begin{split}
\mathcal{F}\mathcal{B}_{m,l,p}[f,g](\xi):=\int_{\mathbb{R}}&q_m(s)\int_{\mathbb{R}^2}e^{is\Phi(\xi,\eta)}\widetilde{\varphi}_l(\Phi(\xi,\eta))\\
&\times\varphi(\kappa_\theta^{-1}\Omega_\eta\Phi(\xi,\eta))\varphi_{p}^{[p_-,p_0]}(\nabla_\eta\Phi(\xi,\eta))\widehat{f}(\xi-\eta,s)\widehat{g}(\eta,s)\,d\eta ds.
\end{split}
\end{equation*}
Integration by parts using Lemma \ref{tech5} and \eqref{dtfSR2} show that $\Vert P_k\mathcal{B}_{m,l,p_0}[P_{k_2}f^\nu_{SR},f^\mu_{j_1,k_1}]\Vert_{L^2}\lesssim 2^{-2m}$. On the other hand, for $p<p_0$, $\mathcal{F}\mathcal{B}_{m,l,p}$ is supported in a small neighborhood of a space-time resonance, more precisely in the set of points $\xi$ with the property that
\begin{equation}\label{Alx67.1}
\begin{split}
\vert\Psi(\xi)\vert\lesssim\vert\Phi(\xi,\eta)\vert+\vert\eta-p(\xi)\vert^2\lesssim\vert\Phi(\xi,\eta)\vert+\vert\nabla_\eta\Phi(\xi,\eta)\vert^2\lesssim 2^l+2^{2p}.
\end{split}
\end{equation}
Given $\xi$ satisfying \eqref{Alx67.1}, the support of $\eta$ integration essentially and $\kappa_\theta\times 2^{l-p}$ box. Therefore we can estimate, for $p\in[p_-,p_0-1]$
\begin{equation*}
\begin{split}
(1+2^m|\Psi(\xi)|)\big|\mathcal{F}\mathcal{B}_{m,l,p}[P_{k_2}f^\nu_{SR},f^\mu_{j_1,k_1}](\xi)\big|&\lesssim 2^{m+2p}2^{m-l}\cdot2^{l-p}\kappa_\theta\sup_{s}\Vert\widehat{f^\nu_{SR}}(s)\Vert_{L^\infty}\Vert\widehat{f^\mu_{j_1,k_1}}(s)\Vert_{L^\infty}\\
&\lesssim 2^{p+80\delta m}.
\end{split}
\end{equation*}
Summing over $p\le -250\delta m$, we find that
\begin{equation*}
\sum_{p\in[p_-,p_0]} (1+2^m|\Psi(\xi)|)\big|\mathcal{F}\mathcal{B}_{m,l,p}[P_{k_2}f^\nu_{SR},f^\mu_{j_1,k_1}](\xi)\big|\lesssim 2^{-100\delta m}.
\end{equation*}
As remarked in the proof of Lemma \ref{StronglyResLem}, see \eqref{Alx67.5}, this gives an acceptable contribution.

{\bf{Contribution of $f^\nu_{NC}$.}} Recall that 
\begin{equation}\label{Alx45.1}
\|f^\nu_{NC}(s)\|_{L^2}\lesssim 2^{-19m/10}, \qquad \sup_{b\leq N_1/4}\|\Omega^bf^\nu_{NC}(s)\|_{L^2}\lesssim 1.
\end{equation}
We show that
\begin{equation}\label{ResLem11}
\begin{split}
2^{(1-20\delta)m}\Vert P_k\mathcal{B}_{m,l}[P_{k_1}f^\mu,P_{k_2}f^{\nu}_{NC}]\Vert_{L^2}&\lesssim 2^{-51\delta^2m}.
\end{split}
\end{equation}
We define $f^\mu_{j_1,k_1}$ and $f^\nu_{j_1,k_1,n_1}$ as before. If $j_1\ge 2m/3$, then using Lemma \ref{Shur2Lem} and Lemma \ref{LinEstLem},
\begin{equation*}
\begin{split}
\Vert \mathcal{B}_{m,l}[f^\mu_{j_1,k_1,n_1},P_{k_2}f^{\nu}_{NC}]\Vert_{L^2}&\lesssim 2^{\delta^2m} 2^{m-l}2^{\frac{l-n_1}{2}}\sup_{s}\big\Vert \sup_{\theta\in\mathbb{S}^1}|\widehat{f^\mu_{j_1,k_1,n_1}}(r\theta,s)|\big\Vert_{L^2(rdr)}\Vert P_{k_2}f^\nu_{NC}(s)\Vert_{L^2}\\
&\lesssim 2^{-9m/10-l/2}2^{-(1-21\delta)j_1},
\end{split}
\end{equation*}
which gives acceptable contributions. 

On the other hand, if $j_1\le 3m/4$ and $\underline{k}=\min\{k,k_1,k_2\}\le -m/5$ then we use Schur's test, \eqref{cas5.55}, and Lemma \ref{LinEstLem} to estimate
\begin{equation*}
\begin{split}
\Vert P_k\mathcal{B}_{m,l}[f^\mu_{j_1,k_1},P_{k_2}f^{\nu}_{NC}]\Vert_{L^2}&\lesssim 2^{2\delta^2m}2^{m-l}\cdot 2^{l+\underline{k}/2}\cdot\sup_{s}\Vert \widehat{f^\mu_{j_1,k_1}}(s)\Vert_{L^\infty}\Vert f^\nu_{NC}(s)\Vert_{L^2}\lesssim 2^{-m+5\delta m},
\end{split}
\end{equation*}
which gives an acceptable contribution. 

We may now decompose, for $\kappa_\theta$ to be chosen
\begin{equation*}
\begin{split}
&\mathcal{F}\mathcal{B}_{m,l}[f^\mu_{j_1,k_1},P_{k_2}f^{\nu}_{NC}](\xi)=\int_{\mathbb{R}}q_m(s)\big[I^{||}(\xi,s)+I^\perp(\xi,s)\big]\,ds,\\
&I^{||}(\xi,s):=\int_{\mathbb{R}^2}e^{is\Phi(\xi,\eta)}\widetilde{\varphi}_l(\Phi(\xi,\eta))\varphi(\kappa_\theta^{-1}\Omega_\eta\Phi(\xi,\eta))\widehat{f^\mu_{j_1,k_1}}(\xi-\eta,s)\widehat{P_{k_2}f^\nu_{NC}}(\eta,s)\,d\eta,\\
&I^\perp(\xi,s):=\int_{\mathbb{R}^2}e^{is\Phi(\xi,\eta)}\widetilde{\varphi}_l(\Phi(\xi,\eta))(1-\varphi(\kappa_\theta^{-1}\Omega_\eta\Phi(\xi,\eta)))\widehat{f^\mu_{j_1,k_1}}(\xi-\eta,s)\widehat{P_{k_2}f^\nu_{NC}}(\eta,s)\,d\eta.
\end{split}
\end{equation*}
If $j_1\le 3m/4$ and $-m/5\le \underline{k}\le -\D$ then we set $\kappa_\theta:=2^{-2m/5}$, and apply Lemma \ref{RotIBP} and Lemma \ref{Shur3Lem} to estimate $\Vert P_kI^\perp(s)\Vert_{L^2}\lesssim 2^{-3m}$ and
\begin{equation*}
\Vert P_kI^{||}(s)\Vert_{L^2}\lesssim 2^{2\delta^2m}2^{-l}\cdot\kappa_\theta 2^l2^{-\underline{k}}\Vert \widehat{f^\mu_{j_1,k_1}}(s)\Vert_{L^\infty}\Vert f^\nu_{NC}(s)\Vert_{L^2}\lesssim 2^{-2m}.
\end{equation*}
This gives acceptable contributions. Assume finally that $j_1\le 3m/4$ and $\underline{k}\ge -\D$, and set $\kappa_\theta:=2^{2\delta^2m-m/2}$. Proceeding as before, we find that $\Vert P_kI^\perp(s)\Vert_{L^2}\lesssim 2^{-3m}$ and
\begin{equation*}
\begin{split}
\Vert P_kI^{||}(s)\Vert_{L^2}&\lesssim 2^{2\delta^2m}2^{-l}\cdot \kappa_\theta^{1/4} 2^l\Vert \widehat{f^\mu_{j_1,k_1}}(s)\Vert_{L^\infty}\Vert f^\nu_{NC}(s)\Vert_{L^2}\lesssim 2^{-2m}.
\end{split}
\end{equation*}
This gives again an acceptable contribution, which completes the proof in this case.

{\bf{Contribution of $\partial_sF^\nu$.}} It remains to show that
\begin{equation}\label{ResLem3}
\begin{split}
2^{(1-20\delta)m}\Vert P_k\mathcal{B}_{m,l}[P_{k_1}f^\mu,P_{k_2}\partial_sF^\nu_\alpha]\Vert_{L^2}&\lesssim 2^{-51\delta^2m},\qquad\alpha\in\{C,NC,LO\}.
\end{split}
\end{equation}
We define also $f^\mu_{j_1,k_1}$ and $f^\mu_{j_1,k_1,n_1}$ as before. We integrate by parts in $s$ to rewrite
\begin{equation*}
\mathcal{F}\mathcal{B}_{m,l}[P_{k_1}f^\mu,P_{k_2}\partial_sF^\nu_\alpha]=-B_1[P_{k_1}f^\mu,P_{k_2}F^\nu_\alpha]-iB_2[P_{k_1}f^\mu,P_{k_2}F^\nu_\alpha]-B_3[P_{k_1}\partial_sf^\mu,P_{k_2}F^\nu_\alpha],
\end{equation*}
\begin{equation*}
\begin{split}
B_1[P_{k_1}f^\mu,g](\xi)&:=\int_{\mathbb{R}}q^\prime_m(s)\int_{\mathbb{R}^2}e^{is\Phi(\xi,\eta)}\widetilde{\varphi}_l(\Phi(\xi,\eta))\widehat{P_{k_1}f^\mu}(\xi-\eta,s)\widehat{g}(\eta,s)\,d\eta ds,\\
B_2[P_{k_1}f^\mu,g](\xi)&:=\int_{\mathbb{R}}q_m(s)\int_{\mathbb{R}^2}e^{is\Phi(\xi,\eta)}\widetilde{\varphi}_l(\Phi(\xi,\eta))\Phi(\xi,\eta)\widehat{P_{k_1}f^\mu}(\xi-\eta,s)\widehat{g}(\eta,s)\,d\eta ds,\\
B_3[P_{k_1}\partial_sf^\mu,g](\xi)&:=\int_{\mathbb{R}}q_m(s)\int_{\mathbb{R}^2}e^{is\Phi(\xi,\eta)}\widetilde{\varphi}_l(\Phi(\xi,\eta))\partial_s\widehat{P_{k_1}f^\mu}(\xi-\eta,s)\widehat{g}(\eta,s)\,d\eta ds.
\end{split}
\end{equation*}

Using Lemma \ref{PhiLocLem}, we first see that, for $\beta\in\{1,2\}$,
\begin{equation*}
\begin{split}
\big\|B_\beta[P_{k_1}f^\mu,P_{k_2}F^\nu_{NC}]\big\|_{L^2}&\lesssim 2^m\big[\sup_{\lambda,s \approx 2^m}\Vert e^{-i\lambda\Lambda_\mu}P_{k_1}f^\mu(s)\Vert_{L^\infty}\Vert P_{k_2}F^\nu_{NC}(s)\Vert_{L^2}+2^{-5m}\big]\lesssim 2^{-m},
\end{split}
\end{equation*}
which gives an acceptable contribution. To bound $B_3$, we notice first that
\begin{equation}\label{ShurBB3}
\begin{split}
&\sup_\xi\int_{\mathbb{R}^2}\varphi_{\leq l}(\Phi(\xi,\eta))\varphi_k(\xi)\varphi_{k_2}(\eta)\vert\partial_s\widehat{P_{k_1}f^\mu}(\xi-\eta)\vert d\eta\\
&+\sup_\eta\int_{\mathbb{R}^2}\varphi_{\leq l}(\Phi(\xi,\eta))\varphi_k(\xi)\varphi_{k_2}(\eta)\vert\partial_s\widehat{P_{k_1}f^\mu}(\xi-\eta)\vert d\xi\lesssim 2^{100\delta m}2^{l-m}.
\end{split}
\end{equation}
This is a consequence of Lemma \ref{dtfLemPrelim} and \eqref{cas5.55}. Using Schur's test we have
\begin{equation*}
\begin{split}
\big\|\varphi_k\cdot B_3[P_{k_1}f^\mu,P_{k_2}F^\nu_{NC}]\big\|_{L^2}\lesssim 2^{m-l}\cdot 2^{100\delta m}2^{l-m}\cdot\sup_{s}\Vert P_{k_2}F^\nu_{NC}(s)\Vert_{L^2}\lesssim 2^{-m},
\end{split}
\end{equation*}
which gives an acceptable contribution.

We consider now the contribution of $F^\nu_{LO}$. We first remark that, using Proposition \ref{separation1} (i), we may assume that $k_1\geq-\D$ and $f^\mu_{j_1,k_1}=f^\mu_{j_1,k_1,0}$. 
Therefore, if $\beta\in\{1,2\}$,
\begin{equation*}
\begin{split}
\Vert B_\beta[f^\mu_{j_1,k_1},P_{k_2}F^\nu_{LO}]\Vert_{L^2}&\lesssim 2^m\sup_{s}\Vert \widehat{P_{k_2}F_{LO}}(s)\Vert_{L^1}\Vert f^\mu_{j_1,k_1}(s)\Vert_{L^2}\lesssim 2^{-4m/5}2^{-(1-20\delta)j_1},
\end{split}
\end{equation*}
which gives an acceptable contribution if $j_1\ge m/2$. Otherwise, using \eqref{LinftyBd} and Lemma \ref{PhiLocLem}
\begin{equation*}
\begin{split}
\Vert B_\beta[f^\mu_{j_1,k_1},P_{k_2}F^\nu_{LO}]\Vert_{L^2}
&\lesssim 2^m\big[\sup_{s,\lambda\approx 2^m}\Vert e^{-i\lambda\Lambda_\mu}f^\mu_{j_1,k_1}(s)\Vert_{L^\infty}\Vert P_{k_2}F^\nu_{LO}(s)\Vert_{L^2}+2^{-5m}\big]\lesssim 2^{-(1-10\delta)m},
\end{split}
\end{equation*}
which also gives an acceptable contribution. Finally, we use Lemma \ref{dtfLemPrelim} to write $\partial_sf^\mu=\widetilde{f^\mu_C}+\widetilde{f^\mu_{NC}}$. Then we use Schur's test with \eqref{cas5.55} to estimate
\begin{equation*}
\begin{split}
&\Vert B_3[P_{k_1}\widetilde{f^\mu_C},P_{k_2}F^\nu_{LO}]\Vert_{L^2}
\lesssim 2^{m-l}\cdot 2^{2\delta^2m}2^l\sup_{s}\Vert \widehat{P_{k_1}\widetilde{f^\mu_C}}(s)\Vert_{L^\infty}\Vert P_{k_2}F^\nu_{LO}(s)\Vert_{L^2}\lesssim 2^{-m+10\delta m},\\
&\Vert B_3[P_{k_1}\widetilde{f^\mu_{NC}},P_{k_2}F^\nu_{LO}]\Vert_{L^2}
\lesssim 2^{m-l}\cdot \sup_{s}\Vert \widehat{P_{k_1}\widetilde{f^\mu_{NC}}}(s)\Vert_{L^2}\Vert \widehat{P_{k_2}F^\nu_{LO}}(s)\Vert_{L^1}\lesssim 2^{-m}.
\end{split}
\end{equation*}
These are acceptable contributions, as in \eqref{ResLem3}.

We consider now the contribution of $F^\nu_C$. Using Schur's test, \eqref{cas5.55}, \eqref{dtfType3}, and \eqref{dtFTotalBds2Prelim},
\begin{equation*}
\begin{split}
\Vert \varphi_k\cdot B_3[P_{k_1}\partial_sf^\mu,P_{k_2}F^\nu_{C}]\Vert_{L^2}&\lesssim 2^{m-l}\cdot 2^l2^{2\delta^2m}\sup_{s}\Vert \widehat{F^\nu_C}(s)\Vert_{L^\infty}\Vert \partial_sP_{k_1}f^\mu(s)\Vert_{L^2}\lesssim \vert l\vert 2^{-(1-8\delta)m}.
\end{split}
\end{equation*}
This is an acceptable contribution as in \eqref{ResLem3}. In order to bound the operators $B_1$ and $B_2$, we use Lemma \ref{Shur2Lem} to estimate, for all $n_1$ and $\beta\in\{1,2\}$,
\begin{equation*}
\begin{split}
\Vert B_\beta[f^\mu_{j_1,k_1,n_1},P_{k_2}F^\nu_C]\Vert_{L^2}&\lesssim 2^m2^{(l-n_1)/2}2^{2\delta^2m}\cdot \sup_{s}\Vert \widehat{P_{k_2}F^\nu_C}(s)\Vert_{L^\infty}\big\Vert \sup_{\theta\in\mathbb{S}^1}|\widehat{f^\mu_{j_1,k_1,n_1}}(r\theta,s)|\big\Vert_{L^2(rdr)}\\
&\lesssim 2^{3.5\delta m}2^{l/2}2^{-19\delta n_1}2^{-(1-21\delta) j_1}.
\end{split}
\end{equation*}
Recalling that $2^l\lesssim 2^{-12\delta m}$, this gives an acceptable contribution if $j_1\ge m-\delta m$. On the other hand, if $j_1\le m-\delta m$ and $k_1\geq-\delta m/2$, then we use \eqref{LinftyBd} and Lemma \ref{PhiLocLem} to obtain
\begin{equation}\label{Alx68.7}
\begin{split}
\Vert B_\beta[f^\mu_{j_1,k_1},P_{k_2}F^\nu_C]\Vert_{L^2}
&\lesssim 2^m\big[\sup_{s,\lambda\approx 2^m}\Vert e^{-i\lambda\Lambda_\mu}f^\mu_{j_1,k_1}(s)\Vert_{L^\infty}\Vert P_{k_2}F^\nu_C(s)\Vert_{L^2}+2^{-5m}\big]\\
&\lesssim 2^m\cdot 2^{-m+3\delta m}2^{-m+4\delta m},
\end{split}
\end{equation}
which is an acceptable contribution.

Finally, it remains to consider the case
\begin{equation}\label{Alx68.5}
j_1\le m-\delta m,\qquad k_1\leq-\delta m/2,\qquad\beta\in\{1,2\}.
\end{equation}
In this case we need to improve slightly on both estimates above. Notice first that 
\begin{equation*}
B_\beta[f^\mu_{j_1,k_1},P_{k_2}F^\nu_C]=B_\beta[f^\mu_{j_1,k_1},P_{k_2}G^\nu_C],\quad\text{ where }\quad\widehat{G^\nu_C}(\xi,s):=\mathbf{1}_{\{\vert \vert\xi\vert-\gamma_0\vert\lesssim 2^l+2^{k_1}\}}(\xi)\widehat{F^\nu_C}(\xi,s).
\end{equation*}
This is a consequence of Proposition \ref{separation1} (i). Moreover, $f_{j_1,k_1}=f_{j_1,k_1,0}$ and we estimate, using Schur's test and \eqref{cas5.55}, for $\beta\in\{1,2\}$,
\begin{equation*}
\begin{split}
2^{(1-20\delta)m}\Vert B_\beta[f^\mu_{j_1,k_1},P_{k_2}G^\nu_C]\Vert_{L^2}&\lesssim 2^{(1-20\delta)m}2^m2^{l}2^{2\delta^2m}\cdot \sup_{s}\Vert \widehat{P_{k_2}G^\nu_C}(s)\Vert_{L^\infty}\big\Vert f^\mu_{j_1,k_1}\big\Vert_{L^2}\\
&\lesssim 2^{3\delta m}2^{l}2^{15\delta^2m}2^{(1-20\delta) (m-j_1)}.
\end{split}
\end{equation*}
This suffices if $m-j_1\leq 9\delta m$. On the other hand, if $j_1\leq m-9\delta m$ and $k_1\geq -9\delta m+10\delta^2m$ then the estimate \eqref{Alx68.7} still holds, since we can apply the strong $L^\infty$ bound in the last line of \eqref{LinftyBd}. Finally, if $j_1\leq m-9\delta m$ and $k_1\leq -9\delta m+10\delta^2m$ then we notice that 
\begin{equation*}
\|G^\nu_C\|_{L^2}\lesssim (2^l+2^{k_1})^{1/2}\|\widehat{G^\nu_C}\|_{L^\infty}\lesssim 2^{-m-\delta m}.
\end{equation*}
Using the bound in the first line of \eqref{LinftyBd} we can improve the estimate \eqref{Alx68.7}, 
\begin{equation*}
\begin{split}
\Vert B_\beta[f^\mu_{j_1,k_1},P_{k_2}G^\nu_C]\Vert_{L^2}
&\lesssim 2^m\big[\sup_{s,\lambda\approx 2^m}\Vert e^{-i\lambda\Lambda_\mu}f^\mu_{j_1,k_1}(s)\Vert_{L^\infty}\Vert P_{k_2}G^\nu_C(s)\Vert_{L^2}+2^{-5m}\big]\\
&\lesssim 2^m\cdot 2^{-m+20\delta j_1}2^{-m-\delta m},
\end{split}
\end{equation*}
which gives an acceptable contribution. This completes the proof of \eqref{ResLem3}.

\subsection{Control of cubic terms}

In this subsection, we control of the cubic terms due to the general pressure law. For simplicity of notation let $H_3:=H_3(\rho)$ and recall that
\begin{equation*}
\rho=(-i/2)|\nabla|\Lambda_e^{-1}[U_e-\overline{U_e}]=(-i/2)|\nabla|\Lambda_e^{-1}[e^{-it\Lambda_e}V_e-e^{it\Lambda_e}\overline{V_e}].
\end{equation*}

\begin{lemma}\label{dtVCubic}
Assume the hypothesis of Proposition \ref{bootstrap}, $m\ge 0$, $s\in [2^m-1,2^{m+1}]$, $k\in\mathbb{Z}$, then
\begin{equation*}
\begin{split}
\Vert P_{\le k}\vert\nabla\vert H_3(s)\Vert_{H^{N_0}}+\sup_{a\le N_1}\Vert P_{\le k}\Omega^a \vert\nabla\vert H_3(s)\Vert_{L^2}&\lesssim \epsilon_1^32^{k}2^{-2m+50\delta m},\\
\sup_{a\le N_1/2}\Vert P_k\Omega^a \vert\nabla\vert H_3(s)\Vert_{L^\infty}&\lesssim\epsilon_1^3\min\{2^{-3m+75\delta m},2^{3k}\},\\
\sup_{a\le N_1/2}\Vert  P_k\Omega^a\vert\nabla\vert H_3(s)\Vert_{L^1}&\lesssim \eps_1^3\min\{2^{k},1\}2^{-m+25\delta m}.
\end{split}
\end{equation*}
\end{lemma}
Indeed, these bounds follow directly from the fact that $H_3$ is a cubic function of $\rho$.

\begin{proposition}\label{CubicTermZnorm}

Under the hypothesis of Proposition \ref{bootstrap}, for any $t\in[0,T]$ we have
\begin{equation*}
\sup_{a\leq N_1/2}\Big\Vert \int_0^te^{is\Lambda_e}\vert\nabla\vert \Omega^aH_3(s)ds\Big\Vert_{Z^1_e}\lesssim \epsilon_1^3.
\end{equation*}

\end{proposition}

\begin{proof}

It suffices to show that, for $m\geq 0$ and $a\leq N_1/2$,
\begin{equation}\label{SuffCubicAdded0}
\begin{split}
I:=2^{6k_+}2^{(1-20\delta)j}2^k\Big\Vert \int_{\mathbb{R}}q_m(s)Q_{jk}e^{is\Lambda_e}\Omega^aH_3(s)ds\Big\Vert_{L^2}\lesssim\epsilon_1^32^{-\delta m}.
\end{split}
\end{equation}
Using Lemma \ref{dtVCubic}, we see that we can assume that $k\le \delta^2(j+m)$ and $j\geq 2m/3$. Moreover, if $k\le -j/2-10\delta m$, we find that
\begin{equation*}
\begin{split}
I&\lesssim 2^{6k_+}2^{(1-20\delta)j}2^{2k}2^{m}\sup_{s}\Vert \mathcal{F}P_k\Omega^aH_3(s)\Vert_{L^\infty},
\end{split}
\end{equation*}
which also gives an acceptable contribution.

We assume now that $j\ge (1+\delta)m$ and $k\ge -j/2-10\delta m$. Observe from Shur's test that
\begin{equation*}
\begin{split}
\Vert Q_{jk}e^{is\Lambda_e}Q_{j'k'}f\Vert_{L^2}\lesssim (\mathbf{1}_{\{\vert j-j'\vert\le 8\}}+2^{-10(j+j')})\|Q_{j'k'}f\|_{L^2}
\end{split}
\end{equation*}
if $|k-k'|\leq 4$. Since, as a consequence of Lemma \ref{LinEstLem},
\begin{equation*}
\begin{split}
\Vert Q_{j'k'}\Omega^aH_3\Vert_{L^2}&\lesssim \epsilon_1^32^{-(1-25\delta)(m+j')}2^{-j'/4},
\end{split}
\end{equation*}
we obtain an acceptable contribution.

We may now assume that $j\in[2m/3,(1+\delta)m]$. For \eqref{SuffCubicAdded0} it suffices to prove that
\begin{equation}\label{SuffCubicAdded}
\Vert \Omega^a H_3\Vert_{L^2}\lesssim \epsilon_1^32^{-2m+12\delta m}.
\end{equation}
For this, we decompose $H_3$ into a cubic term and a quartic and higher order term,
\begin{equation*}
\begin{split}
H_3=c\rho^3+H_4,\qquad\Vert \Omega^a H_4\Vert_{L^2}&\lesssim \epsilon_1^42^{-5m/2},
\end{split}
\end{equation*}
so it suffices to consider the cubic term. For this, we write
\begin{equation*}
\begin{split}
\rho^3=\sum_{j_1+k_1\ge0}\sum_{j_2+k_2\ge0}\sum_{j_3+k_3\ge0}(Q_{j_1k_1}\rho)(Q_{j_2k_2}\rho)(Q_{j_3k_3}\rho).
\end{split}
\end{equation*}
In view of Lemma \ref{LinEstLem}, for $0\le a\le N_1/2$,
\begin{equation*}
\Vert \Omega^a Q_{jk}\rho\Vert_{L^\infty}\lesssim\eps_1\min\{2^{2k-21\delta k},2^{-(N_0/3)k},2^{-m+22\delta m}\}.
\end{equation*}
Since $\Vert \Omega^a Q_{jk}\rho\Vert_{L^2}\lesssim \eps_12^{-j+20\delta j}$, the contribution is acceptable if $\min\{k_1,k_2,k_3\}\le -2m/3$ or if $\max\{k_1,k_2,k_3\}\ge \delta m$ or if $\max\{j_1,j_2,j_3\}\ge m/100$. On the other hand, if
\begin{equation*}
-2m/3\le k_1,k_2,k_3\le \delta m,\qquad \max\{j_1,j_2,j_3\}\leq m/100,
\end{equation*}
then we use \eqref{LinftyBd} to bound $\Vert \Omega^a Q_{j_lk_l}\rho\Vert_{L^\infty}\lesssim\eps_1 2^{-m+3\delta m}$, $l\in\{1,2\}$. It follows that 
\begin{equation*}
\Vert (\Omega^{a_1}Q_{j_1k_1}\rho)(\Omega^{a_2}Q_{j_2k_2}\rho)(\Omega^{a_3}Q_{j_3k_3}\rho)\Vert_{L^2}\lesssim \eps_1^32^{-2m+6\delta m},
\end{equation*}
which gives an acceptable contribution.
\end{proof}

\section{Estimates on phase functions}\label{phacolle}In this section we gather some important facts about the phase functions $\Phi$. In this section,
\begin{equation*}
\Phi(\xi,\eta)=\Phi_{\sigma\mu\nu}(\xi,\eta)=\Lambda_\sigma(\xi)-\Lambda_\mu(\xi-\eta)-\Lambda_\nu(\eta),\qquad \sigma,\mu,\nu\in\{e,b,-e,-b\},
\end{equation*}
is as in (\ref{phasedef}), and we often omit the subscripts. We sometimes let $+e=e$, $+b=b$. Let $\D_0$ denote a sufficiently large constant that may depend only on the parameter $d\in(0,1)$. 

\subsection{Resonant phases} We first observe that, among all phases, only few can be time resonant.  For clarity, let $\Phi_{\sigma;\mu,\nu}:=\Phi_{\sigma\mu\nu}$. We define three sets of phases
\begin{equation}\label{EllP}
\begin{split}
\mathcal{P}_{Ell}&:=\{\Phi_{e;+e,+e},\Phi_{e; +e, -e},\Phi_{e;-e,-e},\Phi_{e; +e, +b},\Phi_{e; +e, -b},\Phi_{e;-e, -b},\Phi_{e; +b, +b},\Phi_{e; -b, -b},\\
&\qquad \Phi_{b; +e, -e},\Phi_{b; -e, -e},\Phi_{b;+e,-b},\Phi_{b;-e,-b},\Phi_{b;+b,+b},\Phi_{b;+b,-b},\Phi_{b;-b,-b}\},\\
\mathcal{P}_{Hyp}^1&:=\{\Phi_{e;-e,+b},\Phi_{b;+e,+e}\},\\
\mathcal{P}_{Hyp}^2&:=\{\Phi_{e;+b,-b},\Phi_{b;+e,+b},\Phi_{b;-e,+b}\}.
\end{split}
\end{equation}

\begin{lemma}\label{PhaseClass}
If $\Phi\in\mathcal{P}_{Ell}$, then
\begin{equation}\label{EllPhiBdd}
\vert \Phi(\xi,\eta)\vert\gtrsim (1+\min\{\vert\xi\vert,\vert\xi-\eta\vert,\vert\eta\vert\})^{-1}.
\end{equation}
If $\Phi\in\mathcal{P}_{Hyp}^1$, the same bound holds unless
\begin{equation}\label{ResCond1}
\vert\xi\vert\approx \vert\xi-\eta\vert\approx\vert\eta\vert\quad\hbox{ or }\max\{\vert\xi\vert,\vert\xi-\eta\vert,\vert\eta\vert\}\lesssim 1.
\end{equation}
If $\Phi\in\mathcal{P}_{Hyp}^2$ then the bound \eqref{EllPhiBdd} is satisfied unless
\begin{equation}\label{ResCond2}
\begin{split}
&\Phi=\Phi_{e;b+,b-}\qquad{ and }\qquad \vert\xi-\eta\vert +1\approx\vert\eta\vert+1,\\
\text{ or }\quad&\Phi=\Phi_{b;e\pm,b+}\qquad{ and }\qquad\vert\xi\vert+1\approx\vert\eta\vert+1.
\end{split}
\end{equation}
\end{lemma}

\begin{proof}[Proof of Lemma \ref{PhaseClass}]

Simple computations show that if $a,b\in[0,\infty)$ and $\sigma\in\{e,b\}$ then
\begin{equation}\label{ener42.5}
\lambda_\sigma(a)+\lambda_\sigma(b)-\lambda_\sigma(a+b)\gtrsim (1+\min\{a,b\})^{-1}.
\end{equation}
Therefore the bound \eqref{EllPhiBdd} holds when $(\sigma;\mu,\nu)\in\{(e;\pm e,\pm e),(b;\pm b,\pm b)\}$. Since $\lambda_b(r)\geq\lambda_e(r)$, the bound \eqref{EllPhiBdd} also holds for the remaining phases in $\mathcal{P}_{Ell}$.

Now to show \eqref{ResCond1}, it suffices to consider $\Phi_{b;e+,e+}$ and we may assume that $\vert\xi-\eta\vert\ge\vert \eta\vert$. We first observe that 
\begin{equation*}
\hbox{if }\quad\vert\xi\vert\le d^{-1/2}\vert\xi-\eta\vert\quad\hbox{ then }\quad\Phi_{b;+e,+e}(\xi,\eta)=\left[\Lambda_b(\xi)-\Lambda_e(\xi-\eta)\right]-\Lambda_e(\eta).
\end{equation*}
The terms in the right-hand side of this equality are negative so that \eqref{EllPhiBdd} holds in this case. On the other hand,
\begin{equation*}
\Phi_{b;+e,+e}(\xi,\eta)\ge \vert\xi\vert-2-d(\vert\xi-\eta\vert+\vert\eta\vert)\ge (1-d)\vert\xi-\eta\vert-2-(d+1)\vert\eta\vert,
\end{equation*}
so that if $1+\vert\eta\vert\ll\vert\xi-\eta\vert$ then \eqref{EllPhiBdd} holds again. The proof for the other phase in $\mathcal{P}_{Hyp}^1$ follows similarly after switching the variables.

Finally, for \eqref{ResCond2}, by symmetry it suffices to observe that
\begin{equation*}
\Lambda_b(\xi_1)-\Lambda_b(\xi_2)\pm\Lambda_e(\xi_1+\xi_2)\ge \vert\xi_1\vert-1-\vert\xi_2\vert-1-d\vert\xi_1+\xi_2\vert\ge (1-d)\vert\xi_1\vert-2-(d+1)\vert\xi_2\vert ,
\end{equation*}
so that if $\vert\xi_1\vert\gg\vert\xi_2\vert +1$, then $\Lambda_b(\xi_1)-\Lambda_b(\xi_2)\pm\Lambda_e(\xi_1+\xi_2)\gtrsim 1$.
\end{proof}

\subsection{Resonant sets} We start with a proposition describing the geometry of resonant sets.

\begin{proposition}(Structure of resonance sets)\label{spaceres} The following claims hold:

(i) If either $\nu+\mu=0$ or $\max(|\xi|,|\eta|,|\xi-\eta|)\geq 2^{\D_0}$ or $\min(|\xi|,|\eta|,|\xi-\eta|)\leq 2^{-\D_0}$ then 
\begin{equation}\label{res00}|\Phi(\xi,\eta)|\gtrsim (1+|\xi|+|\eta|)^{-1}\quad\mathrm{or}\quad|\nabla_{\eta}\Phi(\xi,\eta)|\gtrsim (1+|\xi|+|\eta|)^{-3}.
\end{equation}

(ii) For any $\xi,\eta\in\mathbb{R}^2$ we have
\begin{equation}\label{res01}(1+|\xi|+|\eta|)^3|\Phi(\xi,\eta)|+|\nabla_{\eta}\Phi(\xi,\eta)|+|\nabla_{\xi}\Phi(\xi,\eta)|\gtrsim 1.
\end{equation}

(iii) If $\nu+\mu\neq 0$, then there exists a function $p=p_{\mu\nu}:\mathbb{R}^{2}\to\mathbb{R}^{2}$ such that $|p(\xi)|\lesssim|\xi|$ and $|p(\xi)|\approx |\xi|$ for small $\xi$, and \[\nabla_{\eta}\Phi(\xi,\eta)=0\quad\Leftrightarrow\quad \eta=p(\xi).\] There is an odd smooth function $p_{+}:\mathbb{R}\to\mathbb{R}$, such that $p(\xi)=p_{+}(|\xi|)\xi/|\xi|$. Moreover,
\begin{equation}\label{cas10}
\text{ if }\quad|\eta|+|\xi-\eta|\leq U\in[1,\infty)\quad\text{ and }\quad|\nabla_{\eta}\Phi(\xi,\eta)|\leq\varep\quad\text{ then }\quad|\eta-p(\xi)|\lesssim\varep U^4.
\end{equation}
and, for any $s\in\mathbb{R}$,
\begin{equation}\label{cas10.1}
|D^\alpha p_+(s)|\lesssim_{\alpha} 1,\qquad |p'_+(s)|\gtrsim (1+|s|)^{-3},\qquad |1-p'_+(s)|\gtrsim (1+|s|)^{-3}.
\end{equation}

(iv) If $\nu+\mu\neq 0$, we define $p$ as above and $\Psi(\xi):=\Phi(\xi,p(\xi))$. Then $\Psi$ is a radial function, and there exist two positive constants $\gamma_{1}<\gamma_{2}$, such that $\Psi(\xi)=0$ if and only if either\[\pm(\sigma,\mu,\nu)=(b,e,e)\quad\mathrm{and}\quad |\xi|=\gamma_{1},\] or \[\pm(\sigma,\mu,\nu)\in\{(b,e,b),(b,b,e)\}\quad\mathrm{and}\quad |\xi|=\gamma_{2}.\]
\end{proposition}

\begin{remark}
We remark the following interesting cancellation property: when $(\sigma,\mu,\nu)=(b,e,e)$ and the frequencies are parallel, the multiplier $\mathfrak{m}_{\sigma\mu\nu}$ vanishes, thus providing an unexpected ``null-form'' at $\gamma_1$. This can be seen from \eqref{ener4s} and \eqref{ABCD}. Such cancellation does not seem to be present at $\gamma_2$ and to keep the symmetry of our analysis, we will not use this fact.

\end{remark}

\begin{proof} We prove first (i). If $\mu=-\nu=\pm\sigma$ or $\mu=-\nu=\pm e$, then
$|\Phi(\xi,\eta)|\gtrsim (1+|\xi|+|\eta|)^{-1}$ due to \eqref{EllPhiBdd}. If $\mu=-\nu=b$ and $\sigma=e$ (the other choices are equivalent by symmetry), then \[\nabla_{\eta}\Phi=\frac{\xi-\eta}{\sqrt{1+|\xi-\eta|^{2}}}+\frac{\eta}{\sqrt{1+|\eta|^{2}}}.\] Thus $|\nabla_{\eta}\Phi|\gtrsim 1$ if $|\xi|\geq 1/2$ and $|\eta|\leq 1/8$. If $\min(|\xi|,|\eta|)\geq 1/8$ then
\begin{equation*}
|\nabla_{\eta}\Phi|\gtrsim (1+|\xi|+|\eta|)^{-3},
\end{equation*}
using the Lipschitz norm of the inverse map of $\xi\mapsto\xi/\sqrt{1+|\xi|^{2}}$. Finally, if $|\xi|\leq 1/2$ then \[|\Phi|\geq\sqrt{1+d|\xi|^{2}}-\left|\sqrt{1+|\xi-\eta|^{2}}-\sqrt{1+|\eta|^{2}}\right|\geq \sqrt{1+d|\xi|^{2}}-|\xi|\geq 1/2.\] This completes the case $\nu+\mu=0$.

Now suppose $\max(|\xi|,|\eta|,|\xi-\eta|)\geq 2^{\mathcal{D}/10}$, and assume the contrary. We may assume $\nu+\mu\neq 0$; if $\nu=\mu$, then the only possibility is $\nu=\mu=e$ and $\sigma=b$, due to Lemma \ref{PhaseClass}. In this case we must also have $|\eta|\approx|\xi-\eta|\approx|\xi|$ and $|2\eta-\xi|\ll 1$, by a similar argument as before. This implies that $|\Phi|\gtrsim 1$, contradiction. 

Now if $\nu\neq\pm\mu$, then we have $|\xi|\geq 2^{\D/10-2}$ and $\min(|\eta|,|\xi-\eta|)\leq 2^{\D/10-10}$, since otherwise $|\nabla_{\eta}\Phi|\gtrsim 1$. By symmetry assume $|\eta|\leq 2^{\D/10-10}$. The only possibility is $\sigma=\mu=\pm b$, since $\sigma\neq\mu$  would imply $|\Phi|\gtrsim |\xi|$, and $\sigma=\mu=\pm e$ would imply $|\Phi|\gtrsim |\xi|^{-1}$ due to \eqref{ener42.5}. Therefore $\nu=\pm e$, but then again $|\nabla_{\eta}\Phi|\gtrsim 1$, contradiction.

The proof in the case $\min(|\xi|,|\eta|,|\xi-\eta|)\leq 2^{-\mathcal{D}/10}$ is similar.

The claim (iii) is proved in \cite[Lemma 5.6]{IoPa2}. Moreover, (ii) follows from (i), (iii), and \cite[Lemma 5.8]{IoPa2} if $\max(|\xi|,|\eta|,|\xi-\eta|)\lesssim 1$. On the other hand, if $\max(|\xi|,|\eta|,|\xi-\eta|)\gg 1$ then a similar argument as before, using \eqref{ener42.5}, leads to the conclusion.

Finally, to prove (iv), we need to solve the equation $\Phi=\nabla_{\eta}\Phi=0$. Clearly we may assume $\xi=\alpha e$ and $\eta=\beta e$, where $e$ is a unit vector. By Lemma \ref{PhaseClass} we only need to consider the cases 
\begin{equation}\label{cas1}
(\sigma,\mu,\nu)\in\{(b,e,e),(b,b,e),(*,+b,-e)\},
\end{equation}
where $*$ represents $e$ or $b$.

In the third case above we claim that space-time resonance is impossible; in fact, $\nabla_{\eta}\Phi=0$ would imply\[\frac{d\beta}{\sqrt{1+d\beta^{2}}}=\frac{\beta-\alpha}{\sqrt{1+(\beta-\alpha)^{2}}}=\pm\sqrt{k},\] where $0\leq k<d<1$. We may then assume $\beta\geq0$, and hence $\theta=\beta-\alpha\geq0$, so\[\beta=\sqrt{\frac{k}{d(d-k)}},\qquad\theta=\sqrt{\frac{k}{1-k}}.\] 
Therefore \[\sqrt{1+d\beta^{2}}=\sqrt{\frac{d}{d-k}}\geq\sqrt{\frac{d}{d-dk}}=\sqrt{\frac{1}{1-k}}=\sqrt{1+\theta^{2}},\] thus $\Phi\geq \sqrt{1+d\alpha^{2}}>0$. 

In the first case in \eqref{cas1}, the equation $\nabla_{\eta}\Phi=0$ simply reduces to $\beta=\alpha-\beta$, or $\alpha=2\beta$. Using also $\Phi=0$, we obtain the only solution\[|\alpha|=\gamma_{1}:=\sqrt{\frac{3}{1-d}}.\] 

In the second case in \eqref{cas1}, we reduce to the system
\begin{equation}\label{resequation}
\frac{d\beta}{\sqrt{1+d\beta^{2}}}=\frac{\alpha-\beta}{\sqrt{1+|\alpha-\beta|^{2}}},\qquad \sqrt{1+\alpha^{2}}=\sqrt{1+d\beta^{2}}+\sqrt{1+(\alpha-\beta)^{2}}.
\end{equation}
Introducing
\begin{equation}\label{Falpha}
F_\alpha(\beta)=\sqrt{1+\alpha^{2}}-\sqrt{1+d\beta^{2}}-\sqrt{1+(\alpha-\beta)^{2}},
\end{equation}
we see that $F_\alpha$ is concave and achieves its maximum at some unique point $\beta_\ast(\alpha)\in(0,\alpha)$. Thus $\gamma_2$ corresponds to a point where
$F_{\gamma_2}(\beta_\ast(\gamma_2))=0$. To show existence and uniqueness of $\gamma_2$, it suffices to show that
the function $\alpha\to F_\alpha(\beta_\ast(\alpha))$ vanishes exactly once for $\alpha\in(0,\infty)$. Uniqueness follows from the computation that
\begin{equation*}
\frac{d}{d\alpha}(F_\alpha(\beta_\ast(\alpha)))=\frac{d }{d\alpha}_{|\beta=\beta_\ast(\alpha)}F_\alpha(\beta)>0.
\end{equation*}
To show the existence of $\gamma_2$, we remark that
\begin{equation*}
F_\alpha(\beta_\ast(\alpha))\ge F_\alpha(\alpha)=\sqrt{1+\alpha^2}-\sqrt{1+d\alpha^2}-1
\end{equation*}
which is positive for $\alpha$ large enough, while we see that when $\alpha=\gamma_1$,
\begin{equation}\label{Ordergamma}
F_{\gamma_1}(\beta_\ast(\gamma_1))<0,
\end{equation}
which also shows that $\gamma_1<\gamma_2$. To prove \eqref{Ordergamma}, it suffices to see that
\begin{equation*}
F_{\gamma_1}(\beta)=\sqrt{1+\gamma_1^2}-\sqrt{1+d(\gamma_1-\beta)^2}-\sqrt{1+d\beta^2}+\left[\sqrt{1+d(\gamma_1-\beta)^2}-\sqrt{1+(\gamma_1-\beta)^2}\right].
\end{equation*}
The term in bracket is negative unless $\beta=\gamma_1$, while we directly see by concavity that
\begin{equation*}
\sqrt{1+\gamma_1^2}-\sqrt{1+d(\gamma_1-\beta)^2}-\sqrt{1+d\beta^2}\le 0
\end{equation*}
with equality only in the case $\beta=\gamma_1/2$.
\end{proof}

\begin{remark}\label{largeres} (i) For $\D_0$ sufficiently large we define the function\[\Psi_{\sigma}^{\dagger}(\xi):=2^{\D_0}(1+|\xi|)\inf_{\mu,\nu\in\mathcal{P};\nu+\mu\neq 0}\left|\Psi_{\sigma\mu\nu}(\xi)\right|\] as in Section \ref{notation}. Using the conclusions of Proposition \ref{spaceres} we can easily prove 
\begin{equation}\label{cas2}
\Psi_{\pm b}^{\dagger}(\xi)\approx (1+|\xi|)^{-1}\cdot\min\big(\big||\xi|-\gamma_{1}\big|,\big||\xi|-\gamma_{2}\big|\big)\qquad\text{ and }\qquad10\leq \Psi_{\pm e}^{\dagger}(\xi)\lesssim 1.
\end{equation}

(ii) Let $\alpha_{1,e}$, $\alpha_{2,e}$, and $\alpha_{2,b}$ denote the absolute values of the space-time resonant inputs corresponding to the values $\gamma_1$ and $\gamma_2$. The analysis in the proof above shows that
\begin{equation}\label{Alx68.68}
\alpha_{1,e}=\gamma_1/2,\qquad 0<\alpha_{2,b}<\gamma_1,\qquad 0<\alpha_{2,e}.
\end{equation}
\end{remark}

In the next lemma we prove a separation property concerning the space-time resonances.

\begin{proposition}\label{separation1} (Separation property) (i) There exists $\gamma_0>\gamma_2>\gamma_1$ such that
\begin{equation}\label{Res0Est}
\vert\Phi(0,\eta)\vert\gtrsim \vert \vert\eta\vert-\gamma_0\vert/(1+\vert \eta\vert).
\end{equation}
In particular, we have that, when $\vert \vert\eta\vert-\gamma_0\vert\ll1$,
\begin{equation*}
\vert\nabla_\xi\Phi(0,\eta)\vert\gtrsim 1,\qquad \vert\nabla_\eta\Phi(0,\eta)\vert\gtrsim 1.
\end{equation*}

(ii) Assume that $\vert\xi-\eta\vert,\vert\eta\vert\in\{\gamma_1,\gamma_2\}$, that $\vert\mu\vert=\vert\nu\vert=b$ and that $\xi$ and $\eta$ are aligned. Then
\begin{equation*}
\vert\Phi(\xi,\eta)\vert\gtrsim 1.
\end{equation*}

(iii) Assume that $\vert\eta\vert\in\{\gamma_1,\gamma_2\}$ and $\vert\nu\vert=b$. Then, for any $\xi\in\mathbb{R}^2$,
\begin{equation}\label{Separation}
\vert\Phi(\xi,\eta)\vert+\vert\nabla_\eta\Phi(\xi,\eta)\vert\gtrsim 1.
\end{equation}
\end{proposition}

\begin{proof}
We start with the proof of (i). Using \eqref{ener42.5}, letting $r=\vert\eta\vert$, to prove \eqref{Res0Est}, it suffices to show that
\begin{equation*}
f(r)=\sqrt{1+r^2}-\sqrt{1+dr^2}-1
\end{equation*}
vanishes only once where its derivative does not vanish. But this is obvious since $f(0)=-1$, $f(r)\to\infty$ as $r\to\infty$ and
\begin{equation*}
f^\prime(r)=r\left[\frac{1}{\sqrt{1+r^2}}-\frac{d}{\sqrt{1+dr^2}}\right]
\end{equation*}
is strictly positive where $r>0$. It remains to show that $\gamma_0>\gamma_2>\gamma_1$. We use $F_\alpha$ as defined in \eqref{Falpha}, which attains its maximum at a unique point $\beta_\ast(\alpha)\in(0,\alpha)$. We see that
\begin{equation*}
F_{\gamma_0}(\beta_\ast(\gamma_0))>F_{\gamma_0}(\gamma_0)=0.
\end{equation*}
Since $\alpha\mapsto F_\alpha(\beta_\ast(\alpha))$ is increasing and vanishes only when $\alpha=\gamma_2$, we see that $\gamma_2<\gamma_0$.

We now prove (ii). If $\vert\xi-\eta\vert=\vert\eta\vert$, this is clear. If $\vert\sigma\vert =b$, this is clear as well. We may now assume by contradiction that $\vert\xi-\eta\vert=\gamma_1$, $\vert\eta\vert=\gamma_2$ and
\begin{equation*}
\Lambda_b(\eta)=\Lambda_b(\xi-\eta)+\Lambda_e(\xi),
\end{equation*}
so that
\begin{equation*}
F_{\gamma_2}(\gamma_2\pm\gamma_1)=0,
\end{equation*}
and $\gamma_2\pm\gamma_1=\beta_\ast(\gamma_2)\in(0,\gamma_2)$. The only possibility is that $\gamma_2-\gamma_1=\beta_\ast(\gamma_2)$. Assuming this, however, we get a contradiction,
\begin{equation*}
\sqrt{d}\geq\frac{d\beta_\ast(\gamma_2)}{\sqrt{1+d\beta_\ast(\gamma_2)^2}}=\frac{\gamma_2-\beta_\ast(\gamma_2)}{\sqrt{1+(\gamma_2-\beta_\ast(\gamma_2))^2}}=\frac{\gamma_1}{\sqrt{1+\gamma_1^2}}=\frac{\sqrt{3}}{\sqrt{4-d}},
\end{equation*}
which is false if $d\in(0,1)$.

We now turn to the proof of (iii). Since $\vert\nu\vert=b$, it follows from Proposition \ref{spaceres} that the only possibility for a space-time resonance is $(\sigma,\mu,\nu)=\pm(b,e,b)$ and $\vert\xi\vert=\gamma_2$. In case $\vert\eta\vert=\gamma_2$, we see that $\vert\xi\vert>\vert\eta\vert=\gamma_2$ and we get \eqref{Separation}. If $\vert\eta\vert=\gamma_1$, we deduce easily that $\vert\nabla_\eta\Phi(\xi,\eta)\vert\gtrsim 1$.
\end{proof}

Our next lemma is used to control the second order interaction of space-time resonances and time resonances.

\begin{proposition}\label{Separation2} (Iterated resonances) (i) Assume that $\xi,\eta\in\mathbb{R}^2$, and $\Phi_{\sigma\mu\nu}$, $\Phi_{\nu\theta\kappa}$ and phases as in \eqref{phasedef}, $\theta+\kappa\neq 0$. Let $\Psi_{\nu\theta\kappa}(\eta)=\Phi_{\nu\theta\kappa}(\eta,p_{\theta\kappa}(\eta))$ as before, and assume that
\begin{equation}\label{Sep2Hyp}
\begin{split}
&|\Phi_{\sigma\mu\nu}(\xi,\eta)|\leq 2^{-\D/10},\qquad|\Psi_{\nu\theta\kappa}(\eta)|\leq 2^{-\D/10},\\
&\big|\nabla_\eta\left[\Phi_{\sigma\mu\nu}(\xi,\eta)+\Psi_{\nu\theta\kappa}(\eta)\right]\big|\leq\eps\leq 2^{-\D/10}.
\end{split}
\end{equation}
Then
\begin{equation}\label{Sep2CCL}
\begin{split}
&\sigma=\kappa,\qquad\mu=-\theta,\qquad |\xi-p_{\theta\kappa}(\eta)|\lesssim \eps,\\
\hbox{or}\qquad&\sigma=\theta,\qquad\mu=-\kappa,\qquad |\xi-\eta+p_{\theta\kappa}(\eta)|\lesssim\eps.\\
\end{split}
\end{equation}
In particular,
\begin{equation*}
\big|\nabla_\xi\Phi_{\sigma\mu\nu}(\xi,\eta)\big|\lesssim \eps.
\end{equation*}

(ii) Conversely, if
\begin{equation}\label{Sep2Hyp2}
\begin{split}
&|\Phi_{\sigma\mu\nu}(\xi,\eta)|\leq 2^{-\D/10},\qquad|\Psi_{\nu\theta\kappa}(\eta)|\leq 2^{-\D/10},\qquad\big|\nabla_\xi\Phi_{\sigma\mu\nu}(\xi,\eta)\big|\leq\eps\leq 2^{-\D/10},
\end{split}
\end{equation}
then \eqref{Sep2CCL} holds and 
\begin{equation*}
\big|\nabla_\eta\left[\Phi_{\sigma\mu\nu}(\xi,\eta)+\Psi_{\nu\theta\kappa}(\eta)\right]\big|\lesssim\eps.
\end{equation*}
\end{proposition}

\begin{proof} We prove (i) by examining all cases. Let $o$ denote small quantities, $|o|\lesssim_d2^{-\D/10}$. Let $\chi:=p_{\theta\kappa}(\eta)$. We start by remarking that, in view of \eqref{Sep2Hyp},
\begin{equation}\label{EqualityOfGrad}
\big|\nabla\Lambda_\mu(\xi-\eta)-\nabla\Lambda_\theta(\eta-\chi)\big|\lesssim_d\eps,\qquad\nabla\Lambda_\theta(\eta-\chi)=\nabla\Lambda_\kappa(\chi).
\end{equation}
In view of Proposition \ref{spaceres} (iv) and Proposition \ref{separation1} (i), we may assume that $|\eta|,\,|\xi|,|\xi-\eta|,|\chi|,|\eta-\chi|\approx_d 1$. Up to multiplying by $-1$ we may assume that $\nu=b$.

The smallness of the phases gives
\begin{equation}\label{TimeResSep2}
\begin{split}
-\Lambda_\sigma(\xi)+\Lambda_\mu(\xi-\eta)+\Lambda_b(\eta)=o,\qquad -\Lambda_b(\eta)+\Lambda_\theta(\eta-\chi)+\Lambda_\kappa(\chi)&=o.
\end{split}
\end{equation}
In particular, we directly see that $\mu\ne b$.

{\bf Case 1}: $(\nu,\theta,\kappa)=(b,e,e)$. In this case $2\chi=\eta$. Assume first that $\mu=e$, then by \eqref{EqualityOfGrad}, $|\xi+\chi-2\eta|\lesssim_d\eps$ and we find that
\begin{equation*}
-\Lambda_\sigma(3\eta/2)+\Lambda_e(\eta/2)+\Lambda_b(\eta)=o,\qquad -\Lambda_b(\eta)+2\Lambda_e(\eta/2)=o,
\end{equation*}
and therefore
\begin{equation*}
\Lambda_\sigma(3\eta/2)-3\Lambda_e(\eta/2)=o,\qquad \Lambda_\sigma(3\eta/2)-(3/2)\Lambda_b(\eta)=o.
\end{equation*}
This is impossible since $|\Lambda_\sigma(t x)-t\Lambda_\sigma(x)|\gtrsim_d 1$ if $t\in\{3/2,3\}$, $|x|\approx_d 1$, and $\sigma\in\{e,b\}$.

Assume now that $\mu=-e$. Then \eqref{EqualityOfGrad} gives that $\chi=\eta/2$, $|\xi-\eta/2|\lesssim_d\eps$, and we find that
\begin{equation*}
\begin{split}
-\Lambda_\sigma(\eta/2)-\Lambda_e(\eta/2)+\Lambda_b(\eta)&=o,\qquad -\Lambda_b(\eta)+2\Lambda_e(\eta/2)=o.
\end{split}
\end{equation*}
Therefore $\sigma=e$ and we obtain \eqref{Sep2CCL}.

Assume finally that $\mu=-b$. Then $\sigma=\pm e$. If $\sigma=e$, \eqref{TimeResSep2} gives that
\begin{equation*}\label{TimeResSep3}
\begin{split}
\Lambda_b(\eta)+o=\Lambda_e(\xi)+\Lambda_b(\xi-\eta)=\Lambda_e(\xi)+\Lambda_e(\xi-\eta)+\left[\Lambda_b(\xi-\eta)-\Lambda_e(\xi-\eta)\right].
\end{split}
\end{equation*}
Since the last bracket is $\gtrsim_d1$ and also $\Lambda_b(\eta)+o=2\Lambda_e(\eta/2)$ (using \eqref{TimeResSep2} again), we obtain that
\begin{equation*}
2\Lambda_e(\eta/2)-\big[\Lambda_e(\xi)+\Lambda_e(\xi-\eta)\big]\gtrsim_d1.
\end{equation*}
Since $\lambda_e$ is strictly convex, we obtain a contradiction. On the other hand, if $\sigma=-e$, using that $\vert\nabla\Lambda_b(\xi)\vert-\vert\nabla\Lambda_e(\xi)\vert\gtrsim_d1$, \eqref{EqualityOfGrad} gives that $\vert\eta\vert/2-\vert\xi-\eta\vert\gtrsim_d1$, while \eqref{TimeResSep2} gives that
\begin{equation*}
\Lambda_e(\xi)+\Lambda_b(\eta)=\Lambda_b(\xi-\eta)+o.
\end{equation*}
This is impossible in view of the fact that $\vert\xi-\eta\vert\le \vert\eta\vert$. 

{\bf Case 2}: $(\nu,\theta,\kappa)=(b,e,b)$. Assume that $\mu=e$, then \eqref{EqualityOfGrad} gives that $|\xi+\chi-2\eta|\lesssim_d\eps$, while
\eqref{TimeResSep2} gives that
\begin{equation}\label{Alx67.7}
-\Lambda_\sigma(\xi)+\Lambda_e(\xi-\eta)+\Lambda_b(\eta)=o,\qquad -\Lambda_b(\eta)+\Lambda_e(\eta-\chi)+\Lambda_b(\chi)=o.
\end{equation}
We directly see that the only possibility is $\sigma=b$, but in this case, subtracting the two equations, we obtain that
\begin{equation*}
2\Lambda_b(\frac{\xi+\chi}{2})-\Lambda_b(\chi)-\Lambda_b(\xi)=o.
\end{equation*}
Since $\Lambda_b$ is strictly convex, this implies that $\xi-\chi=o,\eta-\chi=o$, in contradiction with \eqref{Alx67.7}.

Assume now that $\mu=-e$,  then \eqref{EqualityOfGrad} gives that $|\xi-\chi|\lesssim_d\eps$, while \eqref{TimeResSep2} gives that
\begin{equation*}
-\Lambda_\sigma(\xi)-\Lambda_e(\xi-\eta)+\Lambda_b(\eta)=o,\qquad -\Lambda_b(\eta)+\Lambda_e(\eta-\chi)+\Lambda_b(\chi)=o.
\end{equation*}
Then $\sigma=b$ and we obtain the conclusion.

Assume finally that $\mu=-b$. In this case, \eqref{EqualityOfGrad} gives that $|\xi-\eta+\chi|\lesssim_d\eps$, while
\eqref{TimeResSep2} gives
\begin{equation*}
-\Lambda_\sigma(\xi)-\Lambda_b(\xi-\eta)+\Lambda_b(\eta)=o,\qquad -\Lambda_b(\eta)+\Lambda_e(\eta-\chi)+\Lambda_b(\chi)=o.
\end{equation*}
Therefore $\sigma=e$ and we obtain the conclusion.

{\bf Case 3}: $(\nu,\theta,\kappa)=(b,b,e)$. Changing variables $\chi\to\eta-\chi$, we get back to {\bf Case 2}.

We prove now (ii). We may assume that $\mu\neq\sigma$ and $|\xi-p_{-\mu,\sigma}(\eta)|\lesssim_d\eps$. Moreover, $\eta$ is close to a space-time resonance for the phases $\Phi_{\nu\theta\kappa}$ and $\Phi_{\nu(-\mu)\sigma}$. The result follows by Proposition \ref{spaceres}, using also the fact that $\gamma_1\ne\gamma_2$.
\end{proof}

\subsection{Bounds on sublevel sets} We prove first a general upper bound on the size of sublevel sets of functions. 

\begin{lemma}\label{lemma00}
Suppose $L,R,M\in\mathbb{R}$, $M\geq \max(1,L,L/R)$, and $Y:B_R:=\{x\in\mathbb{R}^n:|x|<R\}\to\mathbb{R}$ is a function satisfying $\|\nabla Y\|_{C^{l}(B_R)}\leq M$, for some $l\geq 1$. Then, for any $\epsilon>0$,
\begin{equation}\label{scale1}
\big|\big\{x\in B_R:|Y(x)|\leq\epsilon\text{ and }\sum_{|\alpha|\leq l}|\partial_{x}^{\alpha}Y(x)|\geq L\big\}\big|\lesssim R^{n}ML^{-1-1/l}\epsilon^{1/l}.
\end{equation} 
Moreover, if $n=l=1$, $K$ is a union of at most $A$ intervals, and $|Y'(x)|\geq L$ on K, then
\begin{equation}\label{scale2}\left|\{x\in K:|Y(x)|\leq\epsilon\}\right|\lesssim AL^{-1}\epsilon.\end{equation}
\end{lemma}

\begin{proof} First we consider a simple case, in which $n=1$ and $|Y^{(l)}(\xi)|\geq 1$ for all $\xi$ in some interval $J$. A simple argument shows that
\begin{equation}\label{scale3}
|\{x\in J:|Y(x)|\leq\epsilon\}|\leq C_l\epsilon^{1/l},
\end{equation} 
for some constant $C_l\geq 1$. Note that by scaling, this also proves (\ref{scale2}).

We prove now \eqref{scale1}. We may assume $\epsilon\ll L$ and $M=1$. Choose a small absolute constant $\rho$, and cover $B_R$ by $\approx (R/L)^{n}$ balls of radius $\rho L$. For each ball we can find a vector $e\in \mathbb{S}^{n-1}$ and $0\leq j\leq l$ such that $|\partial_{e}^{j}Y|\gtrsim L$ at the center of that ball. We may assume $j\geq 1$; by the $C^{l+1}$ bound, this inequality also holds within that ball (with different constants). By an orthogonal transformation we may assume that $e$ is a coordinate vector, so the volume of the part of the set contained in that ball is bounded by $CL^{n-1}(\epsilon/L)^{1/l}$ by (\ref{scale3}). Adding up, we obtain the desired bound.
\end{proof}

We prove now several bounds on the sets of time resonances.

\begin{proposition}[Volume bounds of sublevel sets, I]\label{volume} (i) Let $k\geq 0$ and $\epsilon\leq 1/2$, and \[E=\{(\xi,\eta):\max(|\xi|,|\eta|)\leq 2^{k},|\Phi(\xi,\eta)|\leq 2^{-k}\epsilon\}.\]
Then
\begin{equation}\label{cas4} \sup_{\xi}\int_{\mathbb{R}^{2}}\mathbf{1}_{E}(\xi,\eta)\,d\eta+\sup_{\eta}\int_{\mathbb{R}^{2}}\mathbf{1}_{E}(\xi,\eta)\,d\xi\lesssim 2^{7k}\epsilon\log(1/\epsilon).
\end{equation} 

(ii) Assume now that $\epsilon\leq\epsilon'\leq 1/2$, and consider the sets
\begin{equation*}
\begin{split}
&E'_1=\{(\xi,\eta):\max(|\xi|,|\eta|)\leq 2^{k},\,|\Phi(\xi,\eta)|\leq 2^{-k}\epsilon,\,|{\Upsilon}(\xi,\eta)|\leq 2^{-3k}\epsilon',\,|\nabla_\eta\Phi(\xi,\eta)|\geq 2^{-\D}\},\\
&E'_2=\{(\xi,\eta):\max(|\xi|,|\eta|)\leq 2^{k},\,|\Phi(\xi,\eta)|\leq 2^{-k}\epsilon,\,|{\Upsilon}(\xi,\eta)|\leq 2^{-3k}\epsilon',\,|\nabla_\xi\Phi(\xi,\eta)|\geq 2^{-\D}\},
\end{split}
\end{equation*}
 where ${\Upsilon}$ is defined by \[\Upsilon(\xi,\eta)=\nabla_{\xi,\eta}^{2}\Phi(\xi,\eta)\left[\nabla_{\xi}^{\perp}\Phi(\xi,\eta),\nabla_{\eta}^{\perp}\Phi(\xi,\eta)\right].\] 
Then
\begin{equation}\label{cas5} 
\sup_{\xi}\int_{\mathbb{R}^{2}}\mathbf{1}_{E_{1}'}(\xi,\eta)\,d\eta+\sup_{\eta}\int_{\mathbb{R}^{2}}\mathbf{1}_{E_{2}'}(\xi,\eta)\,d\xi\lesssim 2^{12k}\epsilon\log(1/\eps)\cdot(\epsilon')^{1/8}.
\end{equation}

(iii) Assume that $\epsilon\leq\epsilon''\leq 1/2$, $r_0\in [2^{-\D},2^{\D}]$ and consider the sets
\[E''=\{(\xi,\eta):\max(|\xi|,|\eta|)\leq 2^{k},|\Phi(\xi,\eta)|\leq 2^{-k}\epsilon,\big||\xi-\eta|-r_0\big|\leq \epsilon''\},\]
Then we can write $E''=E_{1}''\cup E_{2}''$ such that
\begin{equation}\label{cas5.5} \sup_{\xi}\int_{\mathbb{R}^{2}}\mathbf{1}_{E_{1}''}(\xi,\eta)\,d\eta+\sup_{\eta}\int_{\mathbb{R}^{2}}\mathbf{1}_{E_{2}''}(\xi,\eta)\,d\xi\lesssim 2^{12k}\epsilon\log(1/\eps)\cdot(\epsilon'')^{1/2},
\end{equation}

(iv) Assume that $k\geq 0$, $\eps\leq 2^{-k-\D}$, $\kappa\leq \eps^{1/2} 2^{-5k-\D}$, and $\Phi=\Phi_{\si\mu\nu}$ with $|\nu|=b$. Then
\begin{equation}\label{cas5.7}
\begin{split}
\sup_\xi\int_{\mathbb{R}^2}\varphi(\eps^{-1}2^k\Phi(\xi,\eta))\varphi(\kappa^{-1}\xi\cdot\eta^\perp)\varphi(\Psi^\dagger_b(\eta))\varphi(2^{-k}\xi) d\eta\lesssim 2^{12k} \eps\kappa
\end{split}
\end{equation}
and, if $p\in\mathbb{Z}$,
\begin{equation}\label{cas5.8}
\begin{split}
\sup_\eta\int_{\mathbb{R}^2}\varphi(\eps^{-1}2^k\Phi(\xi,\eta))\varphi(\kappa^{-1}\xi\cdot\eta^\perp)\varphi(\Psi^\dagger_b(\eta))\varphi_{p}(\nabla_\xi\Phi(\xi,\eta))\varphi(2^{-k}\xi) d\xi\\
\lesssim 2^{12k}\min\{2^p,2^{-p}\eps\}\kappa.
\end{split}
\end{equation}
\end{proposition}

\begin{proof} Note that all derivatives of $\Phi$ and $\Upsilon$ are uniformly bounded, except for $\Phi$ itself.  

{\bf{Proof of (i).}} We may assume $\eps\leq 2^{-\D^2}$. By symmetry, it suffices to control the first term in the left-hand side of \eqref{cas4}. Fix $\xi$ and define
\begin{equation*}
\begin{split}
&K_{1,\xi}:=\{\eta:|\eta|\leq 2^{k},\,|\nabla_\eta\Phi(\xi,\eta)|\geq 2^{-3\D}2^{-3k}\},\\
&K_{2,\xi}:=\{\eta:|\eta|\leq 2^{k},\,|\nabla_\eta\Phi(\xi,\eta)|\leq 2^{-3\D}2^{-3k}\}.
\end{split}
\end{equation*}
It follows from Lemma \ref{lemma00} with $l=1$ that
\begin{equation}\label{cas6}
\big|E_{1,\xi}\big|\lesssim \eps 2^{7k}\qquad\text{ where }E_{1,\xi}:=\{\eta\in K_{1,\xi}:|\Phi(\xi,\eta)|\leq 2^{-k}\epsilon\}.
\end{equation}

For \eqref{cas4} it remains to prove a similar bound on the measure of the set
\begin{equation}\label{cas7}
E_{2,\xi}:=\{\eta\in K_{2,\xi}:|\Phi(\xi,\eta)|\leq 2^{-k}\epsilon\}.
\end{equation}
In view of Proposition \ref{spaceres} (i), we may assume that $|\xi|,|\eta|,|\xi-\eta|\in[2^{-\D},2^{\D}]$ and $\mu+\nu\neq 0$. Assuming that $\xi=(s,0),\eta=(r\cos\theta,r\sin\theta)$ we have
\begin{equation}\label{cas7.1}
-\Phi(\xi,\eta)=-\iota_\sigma\sqrt{1+d_\sigma s^2}+\iota_\nu\sqrt{1+d_\nu r^2}+\iota_\mu\sqrt{1+d_\mu (s^2+r^2-2sr\cos\theta)},
\end{equation}
where $\iota_\sigma,\iota_\mu,\iota_\nu\in\{+,-\}$ and $d_\sigma,d_\mu,d_\nu\in\{d,1\}$. Let
\begin{equation*}
Z_{\mp}(r):=-\iota_\sigma\sqrt{1+d_\sigma s^2}+\iota_\nu\sqrt{1+d_\nu r^2}+\iota_\mu\sqrt{1+d_\mu (s\mp r)^2}.
\end{equation*}
Recalling that $r,s\in [2^{-\D},2^{\D}]$, it is easy to see that, for any $r$ and $\xi$ fixed,
\begin{equation*}
\big|\{\theta\in [0,2\pi]:\,\eta=(r\cos\theta,r\sin\theta)\in K_{2,\xi},\,|\Phi(\xi,\eta)|\leq 2^{-k}\epsilon\}\big|\lesssim \sum_{\iota\in\{-,+\}}\frac{2^{-k}\eps}{[|Z_\iota(r)|+2^{-k}\eps]^{1/2}}.
\end{equation*}
Moreover, using \eqref{cas10}, $|Z'_{\mp}(r)|+|Z''_{\mp}(r)|\gtrsim 1$ if $s,r\in[2^{-\D},2^{\D}]$. Therefore, using Lemma \ref{lemma00},
\begin{equation*}
\big|\{r\in[2^{-\D},2^{\D}]:|Z_\iota(r)|\leq\kappa\}\big|\lesssim\kappa^{1/2}\qquad\text{ for any }\iota\in\{-,+\}\text{ and }\kappa>0.
\end{equation*}
We combine the last two inequalities, with $\kappa=(2^{-k}\eps)2^j$, $j=0,1,\ldots$, to conclude that
\begin{equation*}
\big|E_{2,\xi}\big|\lesssim \sum_{j\geq 0,\,2^j\lesssim 2^k\eps^{-1}}\frac{2^{-k}\eps}{(2^{-k}\eps 2^j)^{1/2}}\cdot (2^{-k}\eps2^j)^{1/2}\lesssim \eps\log(1/\eps).
\end{equation*}
The desired bound on the first term in \eqref{cas4} follows, using also \eqref{cas6}.

{\bf{Proof of (ii).}} By symmetry, it suffices to prove the bound on the first term in the left-hand side of \eqref{cas5}. We may assume that $\eps,\eps'$ are sufficiently small, i.e. $\eps'\leq 2^{-40\D}2^{-40k}$. Notice that
\begin{equation*}
\nabla_{\xi,\eta}^{2}\Phi(\xi,\eta)\big[\partial_i,\partial_j\big]=\lambda''_\mu(|\xi-\eta|)\frac{(\xi_i-\eta_i)(\xi_j-\eta_j)}{|\xi-\eta|^2}+\lambda'_\mu(|\xi-\eta|)\frac{\delta_{ij}|\xi-\eta|^2-(\xi_i-\eta_i)(\xi_j-\eta_j)}{|\xi-\eta|^3},
\end{equation*}
for $i,j\in\{1,2\}$. Also
\begin{equation*}
\nabla_\xi\Phi(\xi,\eta)=-\frac{\lambda'_\mu(|\xi-\eta|)}{|\xi-\eta|}(\xi-\eta)+\frac{\lambda'_\sigma(|\xi|)}{|\xi|}\xi,\qquad \nabla_\eta\Phi(\xi,\eta)=\frac{\lambda'_\mu(|\xi-\eta|)}{|\xi-\eta|}(\xi-\eta)-\frac{\lambda'_\nu(|\eta|)}{|\eta|}\eta.
\end{equation*}
Therefore
\begin{equation}\label{cas13}
\begin{split}
\Upsilon(\xi,\eta)&=\frac{\lambda'_\mu(|\xi-\eta|)-|\xi-\eta|\lambda''_\mu(|\xi-\eta|)}{|\xi-\eta|^3}\frac{\lambda'_\sigma(|\xi|)}{|\xi|}\frac{\lambda'_\nu(|\eta|)}{|\eta|}(\eta\cdot\xi^\perp)^2\\
&+\frac{\lambda'_\mu(|\xi-\eta|)}{|\xi-\eta|}\nabla_\xi\Phi(\xi,\eta)\cdot \nabla_\eta\Phi(\xi,\eta).
\end{split}
\end{equation}
Using this formula, Proposition \ref{spaceres}, and the smallness of $\eps'$, it is easy to see that
\begin{equation}\label{cas14}
\begin{split}
\text{ if }\quad (\xi,\eta)\in E'_1\quad&\text{ then }\quad \min(|\xi|,|\eta|,|\xi-\eta|)\geq 2^{-\D}\\
\text{ if }\quad d_\sigma=d_\mu=d_\nu\quad&\text{ then }\quad E^\prime_1=\emptyset.
\end{split}
\end{equation}

Let $E'_{1,\xi}:=\{\eta\in\mathbb{R}^2:\,(\xi,\eta)\in E'_1\}$. We may assume that $\xi=(s,0)$, $\eta=(r\cos\theta,r\sin\theta)$, $r,s\in[2^{-\D},2^k]$. Let
\begin{equation*}
E^{\prime}_{1,\xi,1}:=\{\eta\in E^{\prime}_{1,\xi}:|\sin\theta|\leq(\eps')^{1/8}\},\quad E^{\prime}_{1,\xi,2}:=\{\eta\in E^{\prime}_{1,\xi}:|\sin\theta|\geq(\eps')^{1/8}\}.
\end{equation*}
It follows easily from Lemma \ref{lemma00} that $\big|E^{\prime}_{1,\xi,1}\big|\lesssim 2^{12k}\epsilon\cdot(\epsilon')^{1/8}$. It remains to prove that
\begin{equation}\label{cas15.1} 
\big|E^{\prime}_{1,\xi,2}\big|\lesssim 2^{12k}\epsilon\cdot(\epsilon')^{1/8}.
\end{equation}

To prove \eqref{cas15.1} we have to understand more precisely the function $\Upsilon$. Notice that
\begin{equation}\label{cas20}
\frac{\lambda'_\gamma(x)}{x}=\frac{d_\gamma }{\lambda_\gamma(x)},\qquad \frac{\lambda_\gamma'(x)-x\lambda''_\gamma(x)}{x^3}=\frac{d_\gamma^2}{\lambda_\gamma(x)^3},\qquad d_\gamma x^2=\lambda_\gamma(x)^2-1
\end{equation}
for any $\gamma \in \mathcal{P}$ and $x\in\mathbb{R}$. Moreover, letting $\rho:=|\xi-\eta|=\sqrt{r^2+s^2-2rs\cos\theta}$ we have
\begin{equation}\label{cas21}
\begin{split}
&(2\eta\cdot \xi^\perp)^2=4r^2s^2-(r^2+s^2-\rho^2)^2,\qquad 2\xi\cdot\eta=r^2+s^2-\rho^2,\\
&2(\xi-\eta)\cdot\xi=s^2+\rho^2-r^2,\qquad 2(\xi-\eta)\cdot\eta=s^2-\rho^2-r^2.
\end{split}
\end{equation}
Using also \eqref{cas13} it follows that
\begin{equation*}
\begin{split}
4\Upsilon&(\xi,\eta)=\frac{d_\mu^2 d_\sigma d_\nu}{\lambda_\mu(\rho)^3\lambda_\sigma(s)\lambda_\nu(r)}\big[4r^2s^2-(r^2+s^2-\rho^2)^2\big]\\
&+\frac{d_\mu}{\lambda_\mu(\rho)}\Big[-4\frac{d_\mu^2\rho^2}{\lambda_\mu(\rho)^2}-2\frac{d_\sigma d_\nu(r^2+s^2-\rho^2)}{\lambda_\sigma(s)\lambda_\nu(r)}+2\frac{d_\sigma d_\mu(s^2+\rho^2-r^2)}{\lambda_\sigma(s)\lambda_\mu(\rho)}+2\frac{d_\mu d_\nu(s^2-\rho^2-r^2)}{\lambda_\mu(\rho)\lambda_\nu(r)}\Big].
\end{split}
\end{equation*}
Therefore
\begin{equation*}
\begin{split}
d_\sigma d_\nu\lambda_\mu(\rho)^3&\lambda_\sigma(s)\lambda_\nu(r)\cdot 4\Upsilon(\xi,\eta)=d_\sigma^2d_\mu^2  d_\nu^2\big[2r^2s^2+2r^2\rho^2+2s^2\rho^2-r^4-s^4-\rho^4\big]\\
&-4d_\sigma d_\mu^3 d_\nu\lambda_\sigma(s)\lambda_\nu(r)\rho^2-2d_\sigma^2d_\mu  d_\nu^2\lambda_\mu(\rho)^2(r^2+s^2-\rho^2)\\
&-2d_\sigma^2d_\mu^2 d_\nu\lambda_\mu(\rho)\lambda_\nu(r)(r^2-s^2-\rho^2)-2d_\sigma d_\mu^2 d_\nu^2\lambda_\sigma(s)\lambda_\mu(\rho)(r^2+\rho^2-s^2).
\end{split}
\end{equation*}

Let
\begin{equation*}
\rho_\ast:=\lambda_\mu(\rho),\qquad r_\ast:=\lambda_\nu(r),\qquad s_\ast:=\lambda_\sigma(s).
\end{equation*}
In view of \eqref{cas20}, $d_\mu\rho^2=\rho_\ast^2-1,\,d_\nu r^2=r_\ast^2-1,\,d_\sigma s^2=s_\ast^2-1$. Therefore
\begin{equation}\label{cas30}
-4d_\sigma d_\nu\lambda_\mu(\rho)^3\lambda_\sigma(s)\lambda_\nu(r)\Upsilon(\xi,\eta)=F(s_\ast,r_\ast,\rho_\ast),
\end{equation}
where
\begin{equation}\label{cas30.1}
F=F_{(4)}+F_{(2)}+F_{(0)},
\end{equation}
\begin{equation*}
\begin{split}
F_{(4)}(s_\ast,r_\ast,\rho_\ast)&:=d_\sigma d_\mu^2 d_\nu (2s_\ast^2r_\ast^2 -4s_\ast r_\ast \rho_\ast^2+2s_\ast^2r_\ast\rho_\ast-2s_\ast r_\ast^2\rho_\ast)+d_\sigma d_\mu d_\nu^2 (-2s_\ast\rho_\ast^3)\\
&+d_\sigma^2 d_\mu d_\nu(2r_\ast\rho_\ast^3)+d_\mu^2d_\nu^2(-s_\ast^4+2s_\ast^3\rho_\ast)+d_\sigma^2d_\mu^2(-r_\ast^4-2r_\ast^3\rho_\ast)+d_\sigma^2d_\nu^2(\rho_\ast^4),
\end{split}
\end{equation*}
\begin{equation*}
\begin{split}
F_{(2)}(s_\ast,r_\ast,\rho_\ast)&:=d_\sigma d_\mu^2 d_\nu\left[ 2(s_\ast-r_\ast)(\rho_\ast-s_\ast+r_\ast)      \right]+d_\sigma^2 d_\mu d_\nu\big[ -2r_\ast(r_\ast+\rho_\ast)\big]\\
&+ d_\sigma d_\mu d_\nu^2(-2s_\ast(s_\ast-\rho_\ast))+d_\sigma^2 d_\mu^2(2r_\ast(\rho_\ast+r_\ast))+d_\mu^2d_\nu^2(2s_\ast(s_\ast-\rho_\ast)),
\end{split}
\end{equation*}
\begin{equation*}
\begin{split}
F_{(0)}&:=2d_\sigma d_\mu d_\nu(d_\sigma+d_\mu+d_\nu)-d_\sigma^2d_\mu^2-d_\mu^2 d_\nu^2-d_\sigma^2 d_\nu^2.
\end{split}
\end{equation*}
Since $\Phi(\xi,\eta)=s_\ast-r_\ast-\rho_\ast$, we notice that
\begin{equation}\label{cas40}
\text{ if }\quad \eta\in E^{\prime}_{1,\xi}\quad\text{ then }\quad |F(s_\ast,r_\ast,s_\ast-r_\ast)|\lesssim 2^{2k}\eps'.
\end{equation}
We see that
\begin{equation*}
G(r_\ast):=F(s_\ast,r_\ast,s_\ast-r_\ast)=r_\ast^4d_\sigma^2\big(d_\mu -d_\nu\big)^2+G_{\leq 3}(r_\ast),
\end{equation*}
where $G_{\le 3}(r_\ast)$ denotes a polynomial in $r_\ast$ of degree at most $3$ (notice that $s_\ast$ is fixed). Therefore, if $d_\mu\ne d_\nu$, using Lemma \ref{lemma00},
\begin{equation}\label{cas41}
|K_{s_\ast}|\lesssim 2^{2k}(\eps')^{1/4}\qquad\text{ where }\qquad K_{s_\ast}:=\{r_\ast\in [0,2^{k+\D}]:\,|G(r_\ast)|\leq 2^{2k+\D}\eps'\}.
\end{equation}
It now follows from \eqref{cas40} that, in this case
\begin{equation*}
E'_{1,\xi,2}\subseteq\{\eta=(r\cos\theta,r\sin\theta):\,\lambda_\nu(r)\in K_{s_\ast},\,|\sin\theta|\geq (\eps')^{1/8},\,|\Phi(\xi,\eta)|\leq 2^{-k}\eps\}.
\end{equation*}
The desired estimate \eqref{cas15.1} follows using the formula \eqref{cas7.1}, and the bounds \eqref{cas14} and \eqref{cas41}.

It remains to estimate $E^{\prime}_{1,\xi,2}$ in the case $d_\mu= d_\nu$. In this case, we see that
\begin{equation*}
\begin{split}
F_{(4)}(s_\ast,r_\ast,s_\ast-r_\ast)&=d_\mu^2(d_\sigma-d_\mu)^2(s_\ast^4-2s_\ast^3r_\ast),\\ 
F_{(2)}(s_\ast,r_\ast,s_\ast-r_\ast)&=2d_\mu^3 (d_\mu-d_\sigma)s_\ast r_\ast,\\
F_{(0)}&=d_\mu^3(4d_\sigma-d_\mu). 
\end{split}
\end{equation*}
Therefore $G$ is a linear function in $r_\ast$. Moreover, since $s_\ast\geq 1$ by definition,
\begin{equation*}
G(0)=d_\mu^2(d_\sigma-d_\mu)^2(s_\ast^4-1)+d_\mu^2d_\sigma^2+2d_\sigma d_\mu^3\gtrsim 1.
\end{equation*}
Therefore \eqref{cas41} holds in this case as well. The proof then finishes as before.

{\bf{Proof of (iii).}} In proving \eqref{cas5.5} we may assume $\eps''\leq 2^{-10k-4\D}$. Define
\begin{equation}\label{cas50}
\begin{split}
&E''_{1}:=\{(\xi,\eta)\in E'':|\xi|\geq 2^{-\D-2}\text{ and }\,|\nabla_{\eta}\Phi(\xi,\eta)|\geq 2^{-3k-\D}\},\\
&E''_{2}:=\{(\xi,\eta)\in E'':|\eta|\geq 2^{-\D-2}\text{ and }\,|\nabla_{\xi}\Phi(\xi,\eta)|\geq 2^{-3k-\D}\}.
\end{split}
\end{equation}
It is easy to see that $E''=E''_1\cup E''_2$. By symmetry, it suffices to prove \eqref{cas5.5} for the first term in the left-hand side. Let, as before, $\xi=(s,0)$, $\eta=(r\cos\theta,r\sin\theta)$, and 
\begin{equation}\label{cas51}
\begin{split}
E''_{1,\xi,1}:&=\{\eta:(\xi,\eta)\in E''_{1},\,r\geq 2^{-2\D},\,|\sin\theta|\leq(\eps'')^{1/2}\},\\
E''_{1,\xi,2}:&=\{\eta:(\xi,\eta)\in E''_{1},\,r\geq 2^{-2\D},\,|\sin\theta|\geq(\eps'')^{1/2}\},\\
E''_{1,\xi,3}:&=\{\eta:(\xi,\eta)\in E''_{1},\,r\leq 2^{-2\D}\}.
\end{split}
\end{equation}
As before, it follows from Lemma \ref{lemma00} that $\big|E''_{1,\xi,1}\big|\lesssim 2^{12k}\epsilon\cdot(\epsilon'')^{1/2}$. To estimate $\big|E''_{1,\xi,2}\big|$ we use the formula \eqref{cas7.1}. With $r_\ast=\lambda_\nu(r)$ as before, it follows from definitions that 
\begin{equation*}
E''_{1,\xi,2}\subseteq\{\eta:r\geq 2^{-2\D},\,\lambda_\nu(r)\in K'_{s_\ast,\rho_\ast},\,|\sin\theta|\geq (\eps'')^{1/2},\,|\Phi(\xi,\eta)|\leq 2^{-k}\eps\},
\end{equation*}
where $K'_{s_\ast,\rho_\ast}$ is an interval of length $\lesssim \eps''$. Therefore, using the formula \eqref{cas7.1} as before,
\begin{equation*}
\big|E''_{1,\xi,2}\big|\lesssim 2^{6k}\eps(\eps'')^{1/2},
\end{equation*}
as desired. Finally, the estimate for the set $E''_{1,\xi,3}$ follows by reversing the roles of $\eta$ and $\xi-\eta$.

{\bf{Proof of (iv).}} We prove first \eqref{cas5.7}. In view of \eqref{cas2} we may assume that $\big||\eta|-\gamma _1\big|\ll 1$ or $\big||\eta|-\gamma _2\big|\ll 1$. In particular, using Proposition \ref{separation1} (i), we may assume that $|\xi|\gtrsim 1$, $|\xi-\eta|\gtrsim 1$, and $|\nabla_{\eta}\Phi(\xi,\eta)|\gtrsim 2^{-3k}$. Let $\xi=(s,0)$, $\eta=(r\cos\theta,r\sin\theta)$, and recall the formula \eqref{cas7.1}. The hypothesis shows that, in the support of the integral,  for any $\theta$ fixed with $|\theta|\lesssim \kappa$ the absolute value of the derivative in $r$ of the function $r\to \Phi(\xi,r\theta)$ is $\gtrsim 1$. The bound \eqref{cas5.7} follows.

To prove \eqref{cas5.8}, we may assume that $|\xi|\gtrsim 1$, $|\xi-\eta|\gtrsim 1$ in the support of the integral, in view of Proposition \ref{separation1}. Letting $\eta=(s,0)$ and $\xi=(r\cos\theta,r\sin\theta)$, we have
\begin{equation}\label{cas52}
-\Phi(\xi,\eta)=-\iota_\sigma\sqrt{1+d_\sigma r^2}+\iota_\nu\sqrt{1+d_\nu s^2}+\iota_\mu\sqrt{1+d_\mu (s^2+r^2-2sr\cos\theta)}.
\end{equation}
Assuming $\theta$ fixed with $|\theta|\lesssim \kappa$, and recalling that $\kappa\leq \eps^{1/2} 2^{-5k-\D}$, we let 
\begin{equation*}
Z(r)=Z_{\theta,s}(r):=-\Phi((r\cos\theta,r\sin\theta),\eta).
\end{equation*}
We have $|Z'(r)|+|Z''(r)|\gtrsim 1$ in the support of the integral. The conclusion follows from Lemma \ref{lemma00}, by considering the two cases $2^p\leq \eps^{1/2}$ and $2^p\geq\eps^{1/2}$.
\end{proof}

We prove now several bounds concerning simultaneous phase and angular localization.

\begin{lemma}\label{Shur3Lem}
[Volume bounds of sublevel sets, II] (i) Assume $l,p,q\leq-\D/10$ and define
\begin{equation*}
E=\{(\xi,\eta):\,\vert\Phi(\xi,\eta)\vert\le 2^l,\,\vert\Omega_\eta\Phi(\xi,\eta)\vert\le\kappa_\theta,\,\vert\nabla_\xi\Phi(\xi,\eta)\vert\ge 2^p\gg\kappa_\theta,\,\vert\nabla_\eta\Phi(\xi,\eta)\vert\ge 2^q\gg \kappa_\theta\}.
\end{equation*}
If $1\lesssim 2^{\min(k,k_2)}\leq 2^{\max(k,k_1,k_2)}\leq U\in[1,\infty)$ then
\begin{equation}\label{Alx64.1}
\begin{split}
\sup_\xi\int_{\mathbb{R}^2}\mathbf{1}_E(\xi,\eta)\varphi_{k}(\xi)\varphi_{k_1}(\xi-\eta)\varphi_{k_2}(\eta)d\eta&\lesssim U^4\kappa_\theta 2^{l-q},\\
\sup_\eta\int_{\mathbb{R}^2}\mathbf{1}_E(\xi,\eta)\varphi_{k}(\xi)\varphi_{k_1}(\xi-\eta)\varphi_{k_2}(\eta)d\xi&\lesssim U^4\kappa_\theta 2^{l-p}.
\end{split}
\end{equation}

(ii) Assume that $2^l,\kappa_\theta\leq2^{-\D/10}$ and define
\begin{equation*}
E'=\{(\xi,\eta):\,\vert\Phi(\xi,\eta)\vert\le 2^l,\,\vert\Omega_\eta\Phi(\xi,\eta)\vert\le\kappa_\theta\}.
\end{equation*}
If $k_2\leq -\D/10$ then
\begin{equation}\label{Alx64.2}
\begin{split}
\sup_\xi\int_{\mathbb{R}^2}\mathbf{1}_{E'}(\xi,\eta)\varphi_{k}(\xi)\varphi_{k_1}(\xi-\eta)\varphi_{k_2}(\eta)\,d\eta&\lesssim \kappa_\theta 2^{l}|l|,\\
\sup_\eta\int_{\mathbb{R}^2}\mathbf{1}_{E'}(\xi,\eta)\varphi_{k}(\xi)\varphi_{k_1}(\xi-\eta)\varphi_{k_2}(\eta)\,d\xi&\lesssim \kappa_\theta 2^{-k_2}2^{l}|l|.
\end{split}
\end{equation}

Moreover, if $2^k+2^{k_1}+2^{k_2}\leq U\in[1,\infty)$ then
\begin{equation}\label{cas5.55}
\begin{split}
\sup_{\xi}\int_{\mathbb{R}^2}\varphi(2^{-l}\Phi(\xi,\eta))\varphi_{k}(\xi)\varphi_{k_1}(\xi-\eta)\varphi_{k_2}(\eta)\,d\eta&\lesssim U^{8}|l|2^{l}2^{\min(k_1,k_2)},\\
\sup_{\eta}\int_{\mathbb{R}^2}\varphi(2^{-l}\Phi(\xi,\eta))\varphi_{k}(\xi)\varphi_{k_1}(\xi-\eta)\varphi_{k_2}(\eta)\,d\xi&\lesssim U^{8}|l|2^{l}2^{\min(k_1,k)}.
\end{split}
\end{equation}

(iii) Assume that $2^l,\kappa_\theta\leq 2^{-\D/10}$, $U\geq 1$, and consider the sets
\[E''=\{(\xi,\eta):|\xi|,|\eta|,|\xi-\eta|\in[2^{-2\D},U],\,|\Phi(\xi,\eta)|\leq 2^l,\big|\Omega_\eta\Phi(\xi,\eta)\vert\le\kappa_\theta\}.\]
Then we can write $E''=E_{1}''\cup E_{2}''$ such that
\begin{equation}\label{cas5.6} 
\sup_{\xi}\int_{\mathbb{R}^{2}}\mathbf{1}_{E_{1}''}(\xi,\eta)\,d\eta+\sup_{\eta}\int_{\mathbb{R}^{2}}\mathbf{1}_{E_{2}''}(\xi,\eta)\,d\xi\lesssim U^{10}2^l|l|\kappa_\theta.
\end{equation}
\end{lemma}

\begin{proof} For \eqref{Alx64.1}, it suffices to show the first estimate. Assume $\xi=(s,0)$, $\eta=(r\cos\alpha,r\sin\alpha)$. We first use the fact that
\begin{equation*}
\left\vert\Omega_\eta\Phi(\xi,\eta)\right\vert=\left\vert\frac{\lambda^\prime_\mu(\vert\xi-\eta\vert)}{\vert\xi-\eta\vert}(\xi\cdot\eta^\perp)\right\vert\approx\frac{sr}{1+\vert\xi-\eta\vert}\vert\sin\alpha\vert.
\end{equation*}
In addition, remark that
\begin{equation}\label{Alx64.4}
\vert\partial_r\Phi(\xi,\eta)\vert\ge\vert\nabla_\eta\Phi(\xi,\eta)\vert-\vert\eta\vert^{-1}\vert\Omega_\eta\Phi(\xi,\eta)\vert\gtrsim 2^{q},
\end{equation}
and the estimate now follows from Lemma \ref{lemma00}.

The proof of \eqref{Alx64.2} is similar, using also Proposition \ref{separation1} (i) and the bounds \eqref{cas4}. The bounds \eqref{cas5.55} follow from \eqref{cas4} if $2^{\min(k,k_1,k_2)}\geq 2^{-\D}$, or from \eqref{Alx64.1} if $2^{k_1}\leq 2^{-\D/10}$, or from \eqref{Alx64.2} with $\kappa_\theta\approx 2^{\min(k,k_2)}$ if $2^{\min(k,k_2)}\leq 2^{-\D/10}$.

To prove \eqref{cas5.6} we may assume $\kappa_\theta\ll U^{-2}$ and $2^l\ll U^{-4}$. Define, as in \eqref{cas50}
\begin{equation*}
\begin{split}
&E''_{1}:=\{(\xi,\eta)\in E'':\,|\nabla_{\eta}\Phi(\xi,\eta)|\geq 2^{-\D}\},\quad E''_{2}:=\{(\xi,\eta)\in E'':\,|\nabla_{\xi}\Phi(\xi,\eta)|\geq 2^{-\D}\}.
\end{split}
\end{equation*}
Clearly $E''=E''_1\cup E''_2$, as a consequence of \eqref{res01}. By symmetry, it suffices to prove \eqref{cas5.6} for the first term in the left-hand side. Let, as before, $\xi=(s,0)$, $\eta=(r\cos\theta,r\sin\theta)$. The restriction $\big|\Omega_\eta\Phi(\xi,\eta)\vert\le\kappa_\theta$ gives $|\sin\theta|\lesssim \kappa_\theta$, while the lower bound on $|\nabla_{\eta}\Phi(\xi,\eta)|$ gives, as in \eqref{Alx64.4}, $\vert\partial_r\Phi(\xi,\eta)\vert\gtrsim 1$. The desired conclusion follows from Lemma \ref{lemma00}.
\end{proof}

Finally, we prove some bilinear estimates that involve localization.

\begin{lemma}\label{Shur2Lem}
Assume that $l,n,p\leq -\D/10$. Then
\begin{equation}\label{Shur2Lem1}
\Big\Vert \int_{\mathbb{R}^2}\varphi_l(\Phi_{\sigma\mu\nu}(\xi,\eta))\varphi_n(\Psi_\mu^\dagger(\xi-\eta))\widehat{f}(\xi-\eta)\widehat{g}(\eta)d\eta\Big\Vert_{L^2_\xi}\lesssim 2^{\frac{l+n}{2}}\big\Vert \sup_{\theta\in\mathbb{S}^1}|\widehat{f}(r\theta)|\,\big\Vert_{L^2(rdr)}\Vert g\Vert_{L^2},
\end{equation}
\begin{equation}\label{Shur2Lem3}
\begin{split}
\Big\Vert \int_{\mathbb{R}^2}\varphi_l(\Phi_{\sigma\mu\nu}(\xi,\eta))\varphi_n(\Psi_\mu^\dagger(\xi-\eta))&\varphi_{p}(\Psi_\nu^\dagger(\eta))\widehat{f}(\xi-\eta)\widehat{g}(\eta)d\eta\Big\Vert_{L^2_\xi}\\
&\lesssim \min\{2^{l/2},2^{p/2}\}2^{(l+n)/2}\big\Vert \sup_{\theta\in\mathbb{S}^1}|\widehat{f}(r\theta)|\,\big\Vert_{L^2(rdr)}\Vert g\Vert_{L^2},
\end{split}
\end{equation}
and, assuming $2^k,2^{k_1},2^{k_2}\leq U$,
\begin{equation}\label{Shur2Lem2}
\begin{split}
\Big\Vert \int_{\mathbb{R}^2}&\varphi_l(\Phi(\xi,\eta))\varphi_k(\xi)\varphi_{k_2}(\eta)\varphi_{k_1}(\xi-\eta)\widehat{f}(\xi-\eta)\widehat{g}(\eta)d\eta\Big\Vert_{L^2_\xi}\\
&\lesssim U^4\big\Vert \sup_{\theta\in\mathbb{S}^1}|\widehat{f}(r\theta)|\,\big\Vert_{L^2(rdr)}\Vert g\Vert_{L^2}
\begin{cases}
2^{3l/4}(1+|l|)\qquad&\text{ if }1\lesssim 2^{\min(k,k_1,k_2)},\\
2^{l/2}(1+|l|)\qquad&\text{ in all cases }.
\end{cases}
\end{split}
\end{equation}
\end{lemma}

\begin{proof} We record the following identity: if $\xi=(s,0)$, $\eta=(r\cos\alpha,r\sin\alpha)$, then
\begin{equation}\label{SimpleComp}
\partial_\alpha\vert\xi-\eta\vert=\frac{sr}{\vert \xi-\eta\vert}\sin\alpha=-\frac{\xi\cdot\eta^\perp}{\vert\xi-\eta\vert}.
\end{equation}
We may assume that $\big\Vert \sup_{\theta\in\mathbb{S}^1}|\widehat{f}(r\theta)|\,\big\Vert_{L^2(rdr)}=1$.

We start with \eqref{Shur2Lem1}. Note that by Proposition \ref{separation1} (i), we may freely assume that $\vert\xi\vert\gtrsim 1$ and $\vert\eta\vert\gtrsim 1$.  By Schur's test, it suffices to show that
\begin{equation}\label{Shur2suff}
\begin{split}
\sup_{\xi}\int_{\mathbb{R}^2}\varphi_l(\Phi_{\sigma\mu\nu}(\xi,\eta))\varphi_n(\Psi_\mu^\dagger(\xi-\eta))\vert \widehat{f}(\xi-\eta)\vert d\eta&\lesssim 2^{\frac{l+n}{2}},\\
\sup_\eta\int_{\mathbb{R}^2}\varphi_l(\Phi_{\sigma\mu\nu}(\xi,\eta))\varphi_n(\Psi_\mu^\dagger(\xi-\eta))\vert \widehat{f}(\xi-\eta)\vert d\xi&\lesssim 2^{\frac{l+n}{2}}.
\end{split}
\end{equation}
We focus on the first inequality. Fix $\xi\in\mathbb{R}^2$ and introduce polar coordinates. The left-hand side is dominated by
\begin{equation*}
C\sum_{i\in\{1,2\}}\int_{\theta\in\mathbb{S}^1}\int_0^\infty \varphi_l(\Phi(\xi,\xi-r\theta))\varphi_{\leq n+C}(|r-\gamma_i|)\vert \widehat{f}(r\theta)\vert rdrd\theta,
\end{equation*}
for a constant $C$ sufficiently large. Therefore, it suffices to show that
\begin{equation}\label{EstimThetaInt}
\sup_{r,\vert\xi\vert\gtrsim 1}\int_{\theta\in\mathbb{S}^1}\varphi_l(\Phi(\xi,\xi-r\theta))\varphi_{\geq -\D}(\xi-r\theta)d\theta\lesssim 2^{l/2},
\end{equation}
which is readily verified (see Proposition \ref{volume} (i) for similar arguments). The second inequality in \eqref{Shur2suff} follows similarly.

We now turn to \eqref{Shur2Lem3}. We may assume that $|\mu|=|\nu|=b$. Using Proposition \ref{separation1}, we know that $\vert\angle\xi,\eta\vert\gtrsim 1$, $\vert\xi\vert\gtrsim 1$ and therefore, using \eqref{SimpleComp}, we have $\big|\partial_\alpha\vert\xi-\eta\vert\big|\approx 1$. We proceed as for \eqref{Shur2suff} but replace \eqref{EstimThetaInt} by
\begin{equation*}
\sup_{r\approx 1}\int_{\theta\in\mathbb{S}^1}\varphi_l(\Phi_{\sigma\mu\nu}(\xi,\xi-r\theta))\varphi_p(\Psi_\nu^\dagger(\xi-r\theta))d\theta\lesssim \min\{2^l,2^p\}.
\end{equation*}

We now turn to \eqref{Shur2Lem2}. To prove the inequality when $1\lesssim 2^{\min(k,k_1,k_2)}$ it suffices to show that
\begin{equation}\label{Alx30}
\begin{split}
\sup_{\xi}\int_{\mathbb{R}^2}\varphi_l(\Phi(\xi,\eta))\varphi_k(\xi)\varphi_{k_1}(\xi-\eta)\varphi_{k_2}(\eta)\vert \widehat{f}(\xi-\eta)\vert d\eta & \lesssim U^42^{3l/4}(1+|l|),\\
\sup_\eta\int_{\mathbb{R}^2}\varphi_l(\Phi(\xi,\eta))\varphi_k(\xi)\varphi_{k_1}(\xi-\eta)\varphi_{k_2}(\eta)\vert \widehat{f}(\xi-\eta)\vert d\xi& \lesssim U^42^{3l/4}(1+|l|).
\end{split}
\end{equation}
We show the first inequality. Introducing polar coordinates, we estimate
\begin{equation*}
\begin{split}
\int_{\mathbb{R}^2} &\varphi_l(\Phi(\xi,\xi-r\theta))\varphi_k(\xi)\varphi_{k_1}(r\theta)\varphi_{k_2}(\xi-r\theta)\vert \widehat{f}(r\theta)\vert rdr d\theta\\
&\lesssim \big\Vert\sup_\theta |\widehat{f}(r\theta)|\,\big\Vert_{L^2(rdr)}\Big\Vert \int_\theta\varphi_l(\Phi(\xi,\xi-r\theta))\varphi_k(\xi)\varphi_{k_1}(r\theta)\varphi_{k_2}(\xi-r\theta)\Big\Vert_{L^2(rdr)}\\
&\lesssim \Vert \varphi_l(\Phi(\xi,\xi-\eta))\Vert_{L^2_\eta} \sup_{r,|\xi|\gtrsim 1}\left\{\int_\theta\varphi_l(\Phi(\xi,\xi-r\theta))\varphi_{k_2}(\xi-r\theta)d\theta\right\}^\frac{1}{2}\\
&\lesssim U^42^{3l/4}(1+|l|),
\end{split}
\end{equation*}
using Proposition \ref{volume} (i) and \eqref{EstimThetaInt}. The second inequality in \eqref{Alx30} follows in a similar way. Also, the inequality in \eqref{Shur2Lem2} corresponding to $2^{\min(k,k_1,k_2)}\ll 1$ follows in the same way.
\end{proof}

\end{document}